\documentclass[10pt]{amsart}
\usepackage[nohug,heads=littlevee]{diagrams_arxiv}
\usepackage{mathrsfs}
\usepackage{bm}
\usepackage{stmaryrd}
\usepackage[bbgreekl]{mathbbol}
\usepackage{amssymb}
\usepackage[latin1]{inputenc}
\usepackage{xr-hyper}
\usepackage[plainpages=false,pdfpagelabels]{hyperref}
\usepackage[left=2cm,top=2.5cm,right=2cm,bottom=1.4in,asymmetric]{geometry}
\usepackage{mathtools}
\usepackage[backend=bibtex]{biblatex}
\usepackage{enumitem}
\usepackage{color}
\usepackage[all]{xy}
\usepackage{tensor}
\usepackage{xcolor}

   \DeclareFontFamily{U}{wncy}{}
    \DeclareFontShape{U}{wncy}{m}{n}{<->wncyr10}{}
    \DeclareSymbolFont{mcy}{U}{wncy}{m}{n}
    \DeclareMathSymbol{\Sha}{\mathord}{mcy}{"58}

\newcommand{\defnword}[1]{\textbf{#1}}
\newcommand{\comment}[1]{}
\newcommand{\on}[1]{\operatorname{#1}}

\DeclareSymbolFontAlphabet{\mathbbl}{bbold}
\DeclareSymbolFontAlphabet{\mathbb}{AMSb}

\numberwithin{equation}{subsubsection}
\newtheorem{introthm}{Theorem}

\newtheorem{introprop}{Proposition}
\newtheorem{proposition}[subsubsection]{Proposition}
\newtheorem{theorem}[subsubsection]{Theorem}

\newtheorem*{thm*}{Theorem}
\newtheorem{lemma}[subsubsection]{Lemma}
\newtheorem{sublemma}[subsubsection]{Sublemma}
\newtheorem*{lem*}{Lemma}
\newtheorem{corollary}[subsubsection]{Corollary}
\newtheorem{observation}[subsubsection]{Observation}

\theoremstyle{definition}
\newtheorem{definition}[subsubsection]{Definition}
\newtheorem{construction}[subsubsection]{Construction}
\newtheorem{example}[subsubsection]{Example}

\newtheorem{introrem}{Remark}

\newtheorem{remark}[subsubsection]{Remark}
\newtheorem{notation}[subsubsection]{Notation}

\newtheorem{assumption}[subsubsection]{Assumption}

\newtheorem*{assump*}{Assumption}

\newarrow{Equals}{=}{=}{}{=}{}
\newarrow{Onto}{-}{-}{-}{-}{>>}
\newcommand{\Square}[8]{\begin{diagram}
  #1&\rTo^{#2}&#3\\
  \dTo^{#4}&&\dTo_{#5}\\
  #6&\rTo_{#7}&#8
\end{diagram}
}

\newcommand{\Int}{\mathbb{Z}}

\newcommand{\Aff}{\mathbb{A}}
\newcommand{\Adele}{\mathbb{A}}
\newcommand{\Field}{\mathbb{F}}
\newcommand{\Rat}{\mathbb{Q}}
\newcommand{\Lie}{\on{Lie}}
\newcommand{\Dieu}{\mathbb{D}}

\newcommand{\Map}{\on{Map}}

\newcommand{\BT}[2][G,\mu]{\mathrm{BT}^{#1}_{#2}}

\newcommand{\Wind}[3][G,\mu]{\mathsf{Wind}^{#1}_{#2,#3}}
\newcommand{\Disp}[2][G,\mu]{\mathsf{Disp}^{#1}_{#2}}
\newcommand{\dbqt}[2]{\tensor[_{#1}]{\sslash}{_{#2}}}
\newcommand{\Prism}{\mathbbl{\Delta}}

\newcommand{\mup}[1][p]{\mbox{$\raisebox{-0.59ex}
  {$l$}\hspace{-0.18em}\mu\hspace{-0.88em}\raisebox{-0.98ex}{\scalebox{2}
  {$\color{white}.$}}\hspace{-0.416em}\raisebox{+0.88ex}
  {$\color{white}.$}\hspace{0.46em}_{#1}$}{}}
\newcommand{\Rees}{\mathcal{R}}
\newcommand{\defn}{\overset{\mathrm{defn}}{=}}
\newcommand{\oneb}{1\text{-bdd}}

\newcommand{\modpI}{\heartsuit}
\newcommand{\preoneb}{\diamondsuit}
\newcommand{\defbase}{\hbar}
\newcommand{\degzero}[1]{{#1}^\flat}

\newcommand{\colim}{\on{colim}}

\newcommand{\im}{\on{im}}
\newcommand{\gr}{\on{gr}}

\newcommand{\Rg}{\mathscr{O}}
\newcommand{\Reg}[1]{\Rg_{#1}}

\newcommand{\Hom}{\on{Hom}}
\newcommand{\End}{\on{End}}

\newcommand{\Ext}{\on{Ext}}
\newcommand{\Aut}{\on{Aut}}

\newcommand{\Sym}{\on{Sym}}

\newcommand{\Gm}{\mathbb{G}_m}

\newcommand{\Gmh}[1]{\mathbb{G}_{m,#1}}
\newcommand{\Ga}{\mathbb{G}_a}

\newcommand{\Pic}{\on{Pic}}

\newcommand{\crys}{\on{crys}}
\newcommand{\dR}{\on{dR}}

\newcommand{\pow}[1]{[\vert#1\vert]}

\newcommand{\Mod}[2][ ]{\on{Mod}_{#2}^{#1}}

\newcommand{\Spec}{\on{Spec}}
\newcommand{\Spf}{\on{Spf}}

\newcommand{\et}{\on{\acute{e}t}}

\newcommand{\Fzip}{F\mathrm{Zip}}

\newcommand{\Fil}{\on{Fil}}

\newcommand{\mx}{\mathfrak{m}}

\newcommand{\rank}{\on{rank}}

\newcommand{\Sh}{\on{Sh}}
\newcommand{\Ss}{\mathcal{S}}

\newcommand{\Res}{\on{Res}}

\newcommand{\hker}{\operatorname{hker}}
\newcommand{\hcoker}{\operatorname{hcoker}}

\newcommand{\ad}{\on{ad}}
\newcommand{\GSp}{\on{GSp}}

\newcommand{\GL}{\on{GL}}

\newcommand{\op}{\on{op}}

\bibliography{gdisplays}

\begin{document}
\title[Conjectures of Drinfeld]{An algebraicity conjecture of Drinfeld and the moduli of $p$-divisible groups}
\author{Zachary Gardner}
\address{Zachary Gardner\\
Department of Mathematics\\
Maloney Hall\\
Boston College\\
Chestnut Hill, MA 02467\\
USA}
\email{gardneza@bc.edu}
\author{Keerthi Madapusi}
\address{Keerthi Madapusi\\
Department of Mathematics\\
Maloney Hall\\
Boston College\\
Chestnut Hill, MA 02467\\
USA}
\email{madapusi@bc.edu}

\begin{abstract}
We use the newly developed stacky prismatic technology of Drinfeld and Bhatt-Lurie to give a uniform, group-theoretic construction of smooth stacks $\mathrm{BT}^{G,\mu}_{n}$ attached to a smooth affine group scheme $G$ over $\mathbb{Z}_p$ and $1$-bounded cocharacter $\mu$, verifying a recent conjecture of Drinfeld. This can be viewed as a refinement of results of B\"ultel-Pappas, who gave a related construction using $(G,\mu)$-displays defined via rings of Witt vectors. We show that, when $G = \mathrm{GL}_h$ and $\mu$ is a minuscule cocharacter, these stacks are isomorphic to the stack of truncated $p$-divisible groups of height $h$ and dimension $d$ (the latter depending on $\mu$). This gives a generalization of results of Ansch\"utz-Le Bras, yielding a linear algebraic classification of $p$-divisible groups over very general $p$-adic bases, and verifying another conjecture of Drinfeld.

The proofs use deformation techniques from derived algebraic geometry, combined with an animated variant of Lau's theory of higher frames and displays, and---with a view towards applications to the study of local and global Shimura varieties---actually prove representability results for a wide range of stacks whose tangent complexes are $1$-bounded in a suitable sense. As an immediate application, we prove algebraicity for the stack of perfect $F$-gauges of Hodge-Tate weights $0,1$ and level $n$.
\end{abstract}

\maketitle

\tableofcontents

\section{Introduction}

The goal of this paper is to prove two recent conjectures of Drinfeld~\cite{drinfeld2023shimurian}. The first of these has to do with a Dieudonn\'e theory for $p$-divisible groups over arbitrary $p$-adic formal schemes; that is, we aim to describe $p$-divisible groups, or more generally truncated $p$-divisible groups or Barsotti-Tate groups, in terms of linear algebraic data. For the purposes of this paper, this last phrase means a subcategory of vector bundles on a formal stack, though it has historically taken the form of a description in terms of modules equipped with a Frobenius semi-linear map along with certain additional structures. 

The stacks we consider here arose in recent work of Bhatt-Lurie~\cites{bhatt2022absolute,bhatt2022prismatization,bhatt_lectures} and Drinfeld~\cite{drinfeld2022prismatization}. These authors have shown that one can associate with every $p$-adic formal scheme $X$ a $p$-adic formal stack\footnote{This is actually a \emph{derived} formal stack that is in general not a classical object. We will attempt to ignore this fact in this introduction.} $X^{\mathrm{syn}}$, its \emph{syntomification}, whose coherent cohomology computes the $p$-adic syntomic cohomology of $X$. If $X = \Spf R$ is affine, we will also denote this by $R^{\mathrm{syn}}$.\footnote{We have adopted this notation from the lecture notes of Bhatt~\cite{bhatt_lectures}.} Vector bundles of rank $h$ on this stack  and its mod-$p^n$ fibers---which are examples of objects known as $F$-\emph{gauges over $X$}---have a natural $h$-tuple of locally constant integer valued functions on $X$ associated with them: these are the \emph{Hodge-Tate weights}. We prove:

\begin{introthm}
\label{introthm:dieudonne}
Let $\mathcal{BT}_n(X)$ be the category of $n$-truncated Barsotti-Tate groups over $X$~\cite{MR0801922}, and let $\mathrm{Vect}_{\{0,1\}}(X^{\mathrm{syn}})$ be the category of vector bundles on $X^{\mathrm{syn}}\otimes\Int/p^n\Int$ with Hodge-Tate weights in $\{0,1\}$. Then there is a canonical equivalence of categories
\[
\mathcal{G}_n:\mathrm{Vect}_{\{0,1\}}(X^{\mathrm{syn}}\otimes\Int/p^n\Int)\xrightarrow{\simeq}\mathcal{BT}_n(X)
\]
compatible with Cartier duality.
\end{introthm}

\begin{introrem}
\label{introrem:historical}
Here are some (very incomplete) historical remarks, though see also~\cite[\S 7]{drinfeld2024shimurian}. We will write $\mathcal{BT}(X)$ for the category of $p$-divisible groups over $X$:
\begin{itemize}
   \item The first attempt at a complete description of this category for a particular $X$ was probably by Dieudonn\'e-Manin~\cite{MR0157972}, where $X = \Spec \kappa$ with $\kappa$ a perfect field of characteristic $p$. 
   \item A uniform proof of a description of $\mathcal{BT}(\Spec\kappa)$ in terms of Dieudonn\'e modules was given by Fontaine~\cite{MR0498610}.
   \item A general construction of a \emph{crystalline} Dieudonn\'e functor was given in~\cite{bbm:cris_ii}, and this was used by de Jong~\cite{dejong:formal_rigid}---building on subsequent work of Berthelot-Messing~\cite{Berthelot1990-gp}---to exhibit an equivalence between $p$-divisible groups and Dieudonn\'e $F$-crystals over formally smooth formal schemes over $\Field_p$ whose reduced scheme is of finite type over a field with finite $p$-basis.
   \item When $X = \Spec R$ with $R$ a perfect $\Field_p$-algebra, a form of Theorem~\ref{introthm:dieudonne} is due to Gabber and Lau~\cite{MR2983008}: One can show that $\mathrm{Vect}_{\{0,1\}}(X^{\mathrm{syn}}\otimes\Int/p^n\Int)$ is equivalent to a category of finite locally free $W_n(R)$-modules equipped with certain additional structures appearing in \emph{loc. cit.}
   \item When $X = \Spf R$ for $p$-complete $R$, $p$-divisible \emph{formal} groups have been classified by Zink~\cite{MR1827031} and Lau~\cite{MR2983008} in terms of Witt vector displays.

   \item When $X$ is quasisyntomic, Ansch\"utz and Le Bras demonstrated in~\cite{MR4530092} an equivalence of categories between $\mathcal{BT}(X)$ and a certain category of \emph{admissible} $\varphi$-modules over a sheaf of rings $\mathcal{O}^{\mathrm{pris}}$ obtained using prismatic cohomology. One can once again reformulate their result as proving Theorem~\ref{introthm:dieudonne} for such rings. See also the recent papers of Guo-Li~\cite{guo2023frobenius} and Mondal~\cite{Mondal2024-cy}, where a similar connection is made. Mondal actually proves a classification theorem for all finite locally free $p$-power torsion commutative group schemes over quasisyntomic $X$ in terms of $F$-gauges.

   \item Perhaps the most general existing results are those of Lau in~\cite{lau2018divided}, where one finds a classification of $p$-divisible groups over many $\Field_p$-algebras, including all schemes of finite type over a field with finite $p$-basis, and, more generally, any Noetherian $F$-finite $\Field_p$-algebra. For $p>3$, we can bootstrap this to a classification over $p$-nilpotent bases lifting such $\Field_p$-algebras.

   \item Therefore, the main content of the above theorem is its validity for all $p$-adic formal schemes, as well as for not necessarily formal $p$-divisible groups. Our proofs are uniform, without consideration of special cases, and are largely independent of previous classifications: see Remark~\ref{rem:mondal_not_needed} in the body of the paper.
\end{itemize}
\end{introrem}

\begin{introrem}
The compatibility with Cartier duality takes the following shape: There is a canonical object $\mathcal{O}^{\mathrm{syn}}\{1\}$ in $\mathrm{Vect}_{\{0,1\}}(X^{\mathrm{syn}})$ of rank $1$, the \emph{Breuil-Kisin twist}, which we can tensor with any vector bundle $\mathcal{M}$ over $X^{\mathrm{syn}}\otimes\Int/p^n\Int$ to obtain the twist $\mathcal{M}\{1\}$. If $\mathcal{M}$ has Hodge-Tate weights $0,1$, then so does $\mathcal{M}^\vee\{1\}$, and we now have a canonical isomorphism of truncated Barsotti-Tate groups
\[
\mathcal{G}_n(\mathcal{M}^\vee\{1\})\xrightarrow{\simeq}\mathcal{G}_n(\mathcal{M})^*,
\]
where the right hand side is the Cartier dual of $\mathcal{G}_n(\mathcal{M})$.
\end{introrem}

\begin{introrem}
\label{introrem:syn_cohom}
One striking feature of the theorem to those familiar with Dieudonn\'e theory hitherto is the natural direction of the functor realizing the equivalence. Usually, one associates linear algebraic objects with $p$-divisible groups and their truncations. Here, our functor $\mathcal{G}_n$ goes in the \emph{other} direction and associates truncated Barsotti-Tate groups with objects that are more linear algebraic in nature. Its definition is in terms of syntomic cohomology: That is, for any map $f:\Spf C\to X$, we set
\[
\mathcal{G}_n(\mathcal{M})(C) = \tau^{\le 0}R\Gamma(C^{\mathrm{syn}}\otimes\Int/p^n\Int,(f^{\mathrm{syn}})^*\mathcal{M}).
\]
As such it is completely canonical, compatible with arbitrary base-change and satisfies quasisyntomic descent. The proof will in fact show that it satisfies fpqc descent.
\end{introrem}

\begin{introrem}
As noted above, in~\cite{Mondal2024-cy}, Mondal extends the results of Ansch\"utz-Le Bras and shows that the category of finite flat group schemes over quasisyntomic formal schemes is equivalent to a certain subcategory of the category of perfect $F$-gauges of Hodge-Tate weights $\{0,1\}$ and Tor amplitude $[-1,0]$. In forthcoming work~\cite{madapusi_mondal}, this will be generalized to the same context as that found in Theorem~\ref{introthm:dieudonne}.
\end{introrem}

\begin{introrem}
As a prior footnote observed, $X^{\mathrm{syn}}$ is not in general a classical object, and, correspondingly, the category of vector bundles on $X^{\mathrm{syn}}$ is in general an \emph{$\infty$-category} that is not classical. However, the theorem shows that the subcategory spanned by the objects with Hodge-Tate weights in $\{0,1\}$ \emph{is} classical. 

It is possible that this is because this category depends only on the classical truncation of $X^{\mathrm{syn}}$, though we have not been able to verify this here. We do know that this is the case if one restricts to the subcategory spanned by the $F$-gauges that satisfy a certain nilpotence condition: this follows from Remark~\ref{rem:nilpotent_locus_general} in the body of the paper. This subcategory corresponds via the equivalence of Theorem~\ref{introthm:dieudonne} to that spanned either by the truncated Barsotti-Tate groups that do not admit any non-trivial \'etale quotients---or (up to Cartier duality) by those that do not admit any non-trivial multiplicative subgroups---at any geometric point.
\end{introrem}

\subsection{Method of proof}

Our proof is geometric in nature. Its starting point is the fundamental result of Grothendieck that the stack $\mathrm{BT}_n$ of $n$-truncated Barsotti-Tate groups is a smooth $p$-adic formal Artin stack~\cite{MR0801922}.\footnote{One can circumvent the use of Grothendieck's theorem, and in fact get an alternate proof of it, by making use of the classification results of Lau from~\cite{lau2018divided}. See Remark~\ref{rem:mondal_not_needed}.} We begin by showing the following analogue of Grothendieck's theorem:
\begin{introthm}
\label{introthm:vect_repble}
The assignment\footnote{We write $\mathcal{C}_\simeq$ for the underlying groupoid of any ($\infty$-)category $\mathcal{C}$.}
\[
X\mapsto \mathrm{Vect}_{\{0,1\}}(X^{\mathrm{syn}}\otimes\Int/p^n\Int)_{\simeq}
\]
is represented by a smooth $p$-adic formal Artin stack over $\Int_p$.
\end{introthm}

Theorem~\ref{introthm:dieudonne} can be reduced to the assertion that this $p$-adic formal Artin stack---which we will denote for the purposes of this introduction by $\mathrm{Vect}^{\mathrm{syn}}_{\{0,1\},n}$---is canonically isomorphic to $\mathrm{BT}_n$. To construct this isomorphism, we need another representability result.

\begin{introthm}
\label{introthm:vect_sections_repble}
For any $\mathcal{M}$ in $\mathrm{Vect}_{\{0,1\}}(X^{\mathrm{syn}})$ the functor $\mathcal{G}_n(\mathcal{M})$ on formal schemes over $X$ given for $f:\Spf C\to X$ by
\[
\mathcal{G}_n(\mathcal{M})(C) = \tau^{\le 0}R\Gamma(C^{\mathrm{syn}}\otimes\Int/p^n\Int,(f^{\mathrm{syn}})^*\mathcal{M})
\]
is represented by a truncated Barsotti-Tate group scheme over $X$.
\end{introthm} 

Theorems~\ref{introthm:vect_repble} and~\ref{introthm:vect_sections_repble} together now give us a map of smooth $p$-adic formal Artin stacks
\[
\mathcal{G}_n:\mathrm{Vect}^{\mathrm{syn}}_{\{0,1\},n}\to \mathrm{BT}_n.
\]

To get a map in the other direction, by the smoothness of the stacks involved, and quasisyntomic descent, it suffices to define a canonical map $\mathcal{M}:\mathrm{BT}_n(X)\to \mathrm{Vect}^{\mathrm{syn}}_{\{0,1\},n}(X)$ when $X = \Spf R$ with $R$ quasiregular semiperfectoid (qrsp). For this, we use the functor defined by Mondal~\cite{Mondal2024-cy}, which is a reinterpretation of that of Ansch\"utz-LeBras~\cite{MR4530092}. We could have also used the Dieudonn\'e functor of Berthelot-Breen-Messing~\cite{bbm:cris_ii} for characteristic $p$ inputs and the results of Lau~\cite{lau2018divided}, which would have the consequence of actually giving an alternate proof of Grothendieck's smoothness theorem; see Remark~\ref{rem:mondal_not_needed}

With the functor $\mathcal{M}$ in hand, the verification that it is indeed an inverse proceeds via a direct and quite simple argument that comes down to the compatibility of the functor $\mathcal{G}_n$ with Cartier duality. This in turn relies on two things:
\begin{enumerate}
   \item A computation of Bhatt-Lurie showing that we have a canonical isomorphism $\mathcal{G}_n(\mathcal{O}^{\mathrm{syn}}_n\{1\}) \simeq \mup[p^n]$;
   \item Results of Berthelot-Messing~\cite{Berthelot1990-gp} and de Jong~\cite{de-Jong1998-ki} on crystalline Dieudonn\'e theory. In terms of classification, only the full faithfulness of the crystalline Dieudonn\'e functor for complete DVRs in characteristic $p$ is needed.
\end{enumerate}

\subsection{Truncated $(G,\mu)$-apertures}
\label{subsec:trunc_Gmu}
The representability result in Theorem~\ref{introthm:vect_repble} is a special case of a more general result that proves another conjecture of Drinfeld from~\cite{drinfeld2023shimurian}. Here is the setup for this: We start with a smooth affine group scheme $G$ over $\Int_p$ (not necessarily reductive!) and a cocharacter $\mu:\Gm\to G_{\mathcal{O}}$ defined over the ring of integers $\mathcal{O}$ of a finite unramified extension of $\Rat_p$ that is \defnword{$1$-bounded} in the sense of Lau~~\cite{MR4355476}, so that the weights of the adjoint action of $\mu$ on the Lie algebra $\mathfrak{g}$ are bounded above by $1$. For example, if $G$ is reductive, then $\mu$ will simply be a minuscule cocharacter of $G_{\mathcal{O}}$. A standard example, for non-negative integers $d\leq h$ with $h>0$, is $G = \GL_h$ with $\mu = \mu_d$ given by $z\mapsto \mathrm{diag}(\underbrace{z,\ldots,z}_{d},1,\ldots,1)$. 

When $\mu$ is defined over $\Int_p$, Drinfeld has given a definition for a stack $\BT{n}$ associated with the pair $(G,\mu)$ that specializes to an open and closed substack of $\mathrm{Vect}^{\mathrm{syn}}_{\{0,1\},n}$ when $(G,\mu) = (\GL_h,\mu_d)$. He conjectured that this should be representable by a smooth $0$-dimensional $p$-adic formal Artin stack over $\Int_p$. We generalize this to the case where $\mu$ is defined over a finite unramified ring of integers.\footnote{In fact, one can do this over an arbitrary base, but we restrict ourselves to this case here, since it appears to suffice for global applications.}

\begin{introrem}
The purpose of the stacks $\BT{n}$ is to give a \emph{group-theoretic} construction of a putative stack of truncated $p$-divisible groups equipped with `$G$-structure'. In particular, such a construction should apply even in the case of the exceptional groups of type $E_6$ and $E_7$, which admit minuscule cocharacters, but do not admit any faithful representations in which such cocharacters remain minuscule: This means that there is no direct way to access `motives' of such type through $p$-divisible groups or abelian varieties. Even in the case of a group like $\GSp_{2g}$, the correct interpretation of what a `symplectic structure' on a $p$-divisible group should be is a subtle point. Furthermore, studying functoriality for group homomorphisms is somewhat annoying from this perspective, since such maps in general do not in general have any compatibility with the faithful representations giving rise to $p$-divisible groups. All of these issues are addressed cleanly and systematically by the stacks $\BT{n}$ studied in this article. Theorem~\ref{introthm:dieudonne} shows that the theory is indeed a generalization of the classical story of $p$-divisible groups.
\end{introrem}

To get to the definition of the stacks, we begin with a cartoon of how the syntomification is constructed. For any $p$-complete commutative ring $R$, the stack $R^{\mathrm{syn}}$ is obtained as follows. We have (derived) $p$-adic formal stacks $R^{\Prism},R^{\mathcal{N}}$: These are the \emph{prismatization} of $R$ and the \emph{(Nygaard) filtered prismatization} of $R$, respectively. The second of these is a \emph{filtered stack}: it lives naturally over $\mathbb{A}^1/\Gm$. The open locus lying over the point $\Gm/\Gm$ can be identified with $R^\Prism$: this is the \emph{de Rham} embedding of $R^\Prism$ into $R^{\mathcal{N}}$. There is another open immersion of $R^\Prism$ into $R^{\mathcal{N}}$, called the \emph{Hodge-Tate} embedding, that is physically disjoint from the de Rham embedding. The syntomification is obtained by gluing these two copies of $R^\Prism$ together.

\begin{introrem}
\label{introrem:perfect_case}
When $R$ is a perfect $\Field_p$-algebra, we can identify $R^\Prism$ with $\Spf W(R)$ and describe $R^{\mathcal{N}}$ via the Rees construction applied to the $p$-adic filtration on $W(R)$: this yields a stack isomorphic to $[\Spf W(R)[u,t]/(ut-p)/\Gm]$. Here, $u$ has degree $1$ and $t$ has degree $-1$, and the de Rham and Hodge-Tate embeddings correspond respectively to the loci $\{t\neq  0\}$ and $\{u\neq 0\}$ (though the latter appears with a Frobenius twist). Objects over the mod-$p$ fiber of the syntomification can be interpreted as giving two filtrations on objects over $R$---a decreasing Hodge filtration and an increasing conjugate filtration---along with an identification of their associated gradeds up to Frobenius twist. In other words, vector bundles over this stack are the $F$-zips of Moonen-Pink-Wedhorn-Ziegler~\cite{Pink2015-ye}.
\end{introrem}

Let us return to the question of defining $\BT{n}$. Using the Breuil-Kisin twist $\mathcal{O}^{\mathrm{syn}}_n\{1\}$ and the cocharacter $\mu$, we can produce a canonical $G$-torsor $\mathcal{P}_\mu$ over $\mathcal{O}^{\mathcal{N}}$. We now define $\BT{n}$ as the groupoid-valued functor on $\mathrm{CRing}^{p\text{-comp}}_{\heartsuit,\mathcal{O}/}$, the category of $p$-complete commutative $\mathcal{O}$-algebras $R$: $\BT{n}(R)$ is the ($\infty$)-groupoid of flat local $G$-torsors on $R^{\mathrm{syn}}\otimes\Int/p^n\Int$ whose restriction to $R^{\mathcal{N}}\otimes\Int/p^n\Int$ is isomorphic \emph{flat locally on $\Spec R$} to $\mathcal{P}_\mu$.

\begin{introrem}
This local triviality condition should be viewed as an analogue of the possibly familiar Kottwitz signature condition appearing in the moduli description of Shimura varieties of PEL type: When $G = \GL_h$, the definition is essentially concerned with vector bundles on $R^{\mathrm{syn}}\otimes\Int/p^n\Int$. Any such $F$-gauge gives rise to a filtered vector bundle over $\Spec(R\otimes\Int/p^n\Int)$ equipped with a \emph{Hodge filtration}. The triviality condition imposed here fixes the type of this filtration.
\end{introrem}

The next theorem proves~\cite[Conjecture C.3.1]{drinfeld2023shimurian}.\footnote{Drinfeld takes the cocharacter $\mu$ to be a map $\Gm\to \Aut(G)$ defined over $\Int_p$ and gives a slightly different definition for $\BT{n}$, so we are technically proving something very closely related to Drinfeld's conjecture. See Remark~\ref{rem:drinfeld_definition} for a discussion of this.}

\begin{introthm}
\label{introthm:main}
The formal prestack $\BT{n}$ is represented by a zero-dimensional quasi-compact smooth $p$-adic formal Artin stack over $\mathcal{O}$ with affine diagonal. Moreover, the natural map $\BT{n+1}\to \BT{n}$ is smooth and surjective.
\end{introthm}

\begin{introrem}
\label{introrem:bp}
As far as we are aware, the work of B\"ultel-Pappas~\cite{MR4120808} was the first to attempt to construct such stacks in generality. However, their construction---which involves working with a more direct generalization of the perfect case explained in Remark~\ref{introrem:perfect_case} using Witt vectors---has the expected properties only when restricted to what the authors there call the `adjoint nilpotent' locus. When considering the stack of $p$-divisible groups, this amounts to working only with the connected ones. We show that the more elaborate syntomic construction here recovers that of B\"ultel-Pappas when restricted to this nilpotent locus; see Remark~\ref{rem:adjoint_nilpotent_locus}.

We should also make note of the work of K. Ito~\cite{ito2024prismaticgdisplaysdescenttheory}: He defines the notion of a \emph{prismatic $G$-display} using the prismatic site of Bhatt-Scholze. See the discussion in Section 7 of \emph{loc. cit.} for the connection to the definitions here. A closely related notion is studied by Hedayatzadeh-Partofard~\cite{hedayatzadeh_partofard}, and their main result can be viewed as a special case of Theorem~\ref{introthm:groth_mess} below.
\end{introrem}

\begin{introrem}
\label{introrem:parahoric}
One should formulate and prove versions of the Theorem~\ref{introthm:main} `with coefficients' (see for instance~\cite{ito2023deformation} or~\cite{marks2023prismatic}), allowing smooth group schemes over the ring of integers of finite extensions of $\Rat_p$. This will be considered in forthcoming work of the first author, Z. G. 

There is also the very interesting question of finding the correct analogues of $\BT{n}$ associated with parahoric group schemes. In particular, these analogues should somehow be aware of the corresponding local models as appearing for instance in~\cite{anschutz2022padictheorylocalmodels}. However, this appears to require a genuinely new idea.
\end{introrem}

Let us now record some other results about $\BT{n}$ that are of independent interest, and give some idea of the proof of Theorem~\ref{introthm:main} along the way.

Following Drinfeld, we first obtain a somewhat explicit description of the mod-$p$ fiber $\BT{1}\otimes\Field_p$. To explain this, recall that we can associate with the pair $(G,\mu)$ the algebraic $k$-stack $\Disp{1}$ of \defnword{$F$-zips with $G$-structure and type $\mu$}; see~\cite{Pink2015-ye}: It is a smooth zero-dimensional Artin stack over $k$ with affine diagonal, and $\Disp{1}(R)$ is obtained by replacing $R^{\mathrm{syn}}\otimes\Field_p$ with the \emph{$F$-zip stack} $R^{\Fzip}$ in the definition of $\BT{1}(R)$. The stack $R^{\Fzip}$ is a sort of toy model for the mod-$p$ syntomification, and will play a significant technical role in our proofs. In any case, we now have:

\begin{introthm}\label{introthm:map_to_f-zips}
There is a natural map $\BT{1}\otimes \Field_p\to \Disp{1}$ that is a relatively representable by a smooth zero-dimensional Artin stack with relatively affine diagonal: in fact, it is a gerbe banded by a finite flat commutative $p$-group scheme of height one, the \emph{Lau group scheme}. In particular, $\BT{1}\otimes\Field_p$ is a smooth zero-dimensional Artin stack over $k$ with affine diagonal.
\end{introthm}

\begin{introrem}
When restricted to smooth inputs and $\mu$ is defined over $\Int_p$, this result is due to Drinfeld~\cite{drinfeld2023shimurian}. We verify here that his description continues to hold in general.
\end{introrem}

With Theorem~\ref{introthm:map_to_f-zips} in hand, the rest of the proof of Theorem~\ref{introthm:main} comes down to a double bootstrapping argument. First, we inductively establish  representability for $\BT{n}\otimes\Field_p$ for $n\ge 1$. For this, note that, given an object $\mathcal{P}\in \BT{n}(R)$, we can twist the adjoint representation on $\mathfrak{g}$ by $\mathcal{P}$ to obtain a vector bundle $(\mathfrak{g})_{\mathcal{P}}$ over $R^{\mathrm{syn}}\otimes\Int/p^n\Int$. It is not difficult now to see that the fibers of $\BT{n+1}\to \BT{n}$ over $\mathcal{P}$ are controlled by the syntomic cohomology of this $F$-gauge. The main property that makes this $F$-gauge tractable is that it has Hodge-Tate weights bounded by $1$: this is a direct consequence of the fact that $\mu$ is $1$-bounded. The inductive argument therefore comes down to a special case of the following theorem, which is also an input into the proof of Theorem~\ref{introthm:vect_sections_repble}:

\begin{introthm}\label{introthm:f_gauges}
Suppose that $R\in \mathrm{CRing}^{p\text{-comp}}$ and suppose that $\mathcal{M}$ is an $F$-gauge over $R$ corresponding to a perfect complex on $R^{\mathrm{syn}}\otimes\Int/p^n\Int$ with Tor amplitude in $[-r,\infty)$ and Hodge-Tate weights bounded by $1$. Then the assignment on $p$-complete $R$-algebras given by
\[
C\mapsto \tau^{\le 0}R\Gamma(C^{\mathrm{syn}}\otimes\Int/p^n\Int,\mathcal{M}\vert_{C^{\mathrm{syn}}\otimes\Int/p^n\Int})
\]
is represented by a locally finitely presented $p$-adic formal derived algebraic $r$-stack over $R$.
\end{introthm}

The second bootstrapping argument involves a derived descent statement, encapsulated by:
\begin{introprop}
\label{introprop:derived_descent}
The natural map
\[
\BT{n}(R) \to \mathrm{Tot}\left(\BT{n}(R\otimes^{\mathbb{L}}\Field_p^{\otimes_{\Int}^{\mathbb{L}}(\bullet + 1)})\right)
\]
is an equivalence.
\end{introprop}
Note that, even to state this result, one needs to be working with \emph{animated} commutative rings. We will do so systematically in the body of the paper.

To make full use of the proposition, we also need some finer control of the deformation theory of $\BT{n}$. This involves an interesting (and in a sense elementary) technical tool: Weil restriction from $\Int/p^n\Int$ to $\Int_p$, an operation that is only fully sensible in the derived realm. This yields, for any $p$-adic formal Artin stack $X$, a new \emph{derived} $p$-adic formal Artin stack $X^{(n)}$, whose values are characterized by
\[
X^{(n)}(R) = X(R\otimes^{\mathbb{L}}\Int/p^n\Int).
\]
Using this, for any animated divided power thickening $(R'\twoheadrightarrow R,\gamma)$ in $\mathrm{CRing}^{p\text{-comp}}_{\mathcal{O}/}$ we can write down a canonical commuting diagram
\begin{equation}\label{eqn:commuting_square}
\Square{\BT{n}(R')}{}{BP^{-,(n)}_\mu(R')}{}{}{\BT{n}(R)}{}{BP^{-,(n)}_\mu(R)\times_{BG^{(n)}(R)}BG^{(n)}(R')}.
\end{equation}
Here, $P^{-}_\mu\subset G_{\mathcal{O}}$ is the parabolic subgroup associated with the non-negative eigenspaces of the adjoint action of $\mu$, and $BH$ for any group scheme $H$ denotes its classifying stack. The obstruction theory for $\BT{n}$ is now captured by the following result: 

\begin{introthm}
[Grothendieck-Messing theory]
\label{introthm:groth_mess}
The above commuting square is \emph{Cartesian} when the divided powers are \emph{nilpotent}.
\end{introthm}

This should be viewed as a truncated analogue of classical Grothendieck-Messing theory, which classifies liftings of $p$-divisible groups across classical nilpotent divided power thickenings in terms of lifts of the Hodge filtration on its crystalline realization. We first prove this when $R'$ is an $\Field_p$-algebra, and then lift it to general inputs using Proposition~\ref{introprop:derived_descent}. It is now not hard to deduce the general case of Theorem~\ref{introthm:main} from its mod-$p$ version (at least when $p>2$) by applying Theorem~\ref{introthm:groth_mess} to the canonical nilpotent divided power thickening $R\to R/{}^{\mathbb{L}}p$. A very slightly more involved argument also works when $p=2$. 

One can obtain a more general Grothendieck-Messing theory for not necessarily nilpotent divided power thickenings by restricting to the nilpotent locus, which---as explained in Remark~\ref{introrem:bp}---can also be described in terms of the $(G,\mu)$-displays of B\"ultel-Pappas. 

\subsection{Higher animated frames}
A key technical device we use is a generalization to the animated realm of the notion of a \emph{higher frame} introduced by Lau~\cite{MR4355476}. We call this an \emph{animated higher frame} or simply \emph{frame}. This is combined with an important structural result due to Bhatt-Lurie that---in the terminology introduced here---says that the syntomification of a semiperfectoid algebra can be realized from (and in fact determines) a canonical frame structure on its absolute prismatic cohomology. We can now combine this with quasisyntomic descent in order to translate questions about stacks over syntomifications to assertions about objects living over frames.

The flexibility afforded by this translation turns out to be very useful, since the category of frames permits various constructions that are not visible on the level of the cohomological stacks. We exploit this flexibility to prove Proposition~\ref{prop:def_theory_frames}, a technical frame-theoretic progenitor of Theorem~\ref{introthm:groth_mess} applying to somewhat general thickenings of frames, using two tools:
\begin{enumerate}
   \item Derived deformation theory in the filtered context.
   \item An explicit understanding of the stack $R^{\Fzip}$ associated with the $1$-truncated Witt frame (also termed the \emph{zip frame} by Lau~\cite{MR4355476}) $\underline{W_1(R)}$.
\end{enumerate}
This method can be viewed as an animated refinement (and a substantial generalization) of a unique lifting principle that is (by now) quite classical, is due essentially to Zink, and appears in some form or other already in various papers on related topics, including those of Lau~\cite{MR4355476}, B\"ultel-Pappas~\cite{MR4120808}, and also the recent work of Bartling~\cite{bartling2022mathcalgmudisplayslocalshtuka} and Hedayatzadeh-Partofard~\cite{hedayatzadeh_partofard}.

With this technical backup in our pockets, the proof of Theorem~\ref{introthm:groth_mess} is reduced to a nilpotence result for the divided Frobenius on the fiber of the map between prismatic cohomologies of a nilpotent divided power thickening, which we prove at the end of Section~\ref{sec:bld_stacks}. 

One interesting point here is that the map of frames to which one would like to apply this reasoning to---given by Nygaard filtered prismatic cohomology---is \emph{not} surjective for square-zero extensions of classical rings. In our applications in \S\ref{subsec:1-bounded_stacks_def_theory}, we use derived algebraic geometry again to reduce to the case of certain square-zero extensions of \emph{animated} commutative rings where the map in question is in fact surjective.

We also use similar techniques---derived deformation theory and reduction to the case of the zip frame---to prove Proposition~\ref{prop:abstract_devissage_to_fzips}, the technical base for the proof of Theorem~\ref{introthm:map_to_f-zips}.

\subsection{Further remarks on the proofs}
All the results above are special cases of theorems about objects that we call \emph{$1$-bounded stacks}, whose precise definition is a bit technical and can be found in~\S\ref{subsec:1_bounded_stacks}. Roughly speaking, a $1$-bounded stack is a(n almost) finitely presented stack over the syntomification of a $p$-complete ring, equipped with additional `bounding data' for the Hodge-Tate weights of its cotangent complex. This bounding condition ensures that the deformation theory is controlled by the sections of an $F$-gauge with Hodge-Tate weights bounded by $1$.\footnote{This condition appears essential in order to obtain representable objects: The syntomic cohomology of Breuil-Kisin twists of Hodge-Tate weights greater than $1$ is known to not yield representable functors.} 

Given such an object $\mathcal{X}$ over $R^{\mathrm{syn}}\otimes\Int/p^n\Int$, we can define a functor on animated $p$-complete $R$-algebras by:
\[
\Gamma_{\mathrm{syn}}(\mathcal{X}):C\mapsto \Map_{/R^{\mathrm{syn}}\otimes\Int/p^n\Int}(C^{\mathrm{syn}}\otimes\Int/p^n\Int,\mathcal{X}).
\]

The condition of $1$-boundedness is essentially the one that ensures that $\Gamma_{\mathrm{syn}}(\mathcal{X})$ is representable: The arguments sketched in~\S\ref{subsec:trunc_Gmu} go through when applied to this prestack. Examples of $1$-bounded stacks include:
\begin{itemize}
   \item The `stack' over $\mathcal{O}^{\mathrm{syn}}\otimes\Int/p^n\Int$ parameterizing $G$-torsors that are isomorphic to $\mathcal{P}_\mu$ when restricted to $B\Gm\times\Spec \kappa$ for any algebraically closed field $\kappa$: this is of course relevant for Theorem~\ref{introthm:main};
   \item Total spaces of vector bundles (and perfect complexes) with Hodge-Tate weights bounded by $1$: this is relevant for Theorem~\ref{introthm:vect_sections_repble}.
   \item The stack $\mathrm{Perf}\times (\Int_p^{\mathrm{syn}}\otimes\Int/p^n\Int)$ of perfect complexes, equipped with bounding data that picks out perfect $F$-gauges of Hodge-Tate weights $0,1$.
\end{itemize}

The general representability result boils down via the technical inputs from animated frames explained above to knowing the representability of certain Artin-Milne type cohomology groups, generalizing the fppf cohomology of finite flat group schemes of height one. For this, we use representability results of Bragg and Olsson~\cite{bragg2021representability}, which we present and amplify into a somewhat broader context in Section~\ref{sec:technical_repbility}.

This level of generality is responsible for some of the bulk of this paper. Our justification for indulging in it is that it will be required for future applications, including, for instance, the construction of spaces of \emph{isogenies} between objects in $\BT{\infty}$~\cite{lee_madapusi}, leading to a general construction of Rapoport-Zink spaces as well as $p$-Hecke correspondences \emph{without} the direct involvement of $p$-divisible groups. It is also used for the construction of special cycles on Shimura varieties in~\cite{Madapusi2022-rp}. We expect that it will help address some of the difficulty in constructing the correct analogues of $\BT{n}$ when $G$ is a parahoric, non-reductive group scheme.

As a more immediate consequence, we are able to obtain an extension of Theorem~\ref{introthm:vect_repble} to perfect $F$-gauges.
\begin{introthm}
\label{introthm:perf_fgauges}
The prestack $\mathrm{Perf}^{\mathrm{syn}}_{\{0,1\},n}$ assigning to every $p$-complete ring $R$ the $\infty$-groupoid of perfect complexes on $R^{\mathrm{syn}}\otimes\Int/p^n\Int$ with Hodge-Tate weights in $\{0,1\}$ is represented by a locally finitely presented derived $p$-adic formal Artin stack over $\Int_p$. Moreover, the prestack $\mathrm{Perf}^{\mathrm{syn}}_{0,n}$ classifying perfect complexes on $R^{\mathrm{syn}}\otimes\Int/p^n\Int$ with Hodge-Tate weights $0$ is canonically isomorphic to the $p$-adic formal stack of perfect complexes of lisse $\Int/p^n\Int$-sheaves.
\end{introthm}

This result will be used in~\cite{madapusi_mondal} to extend Theorem~\ref{introthm:dieudonne} to a classification of all finite flat $p$-power torsion commutative group schemes over $X$ in terms of certain perfect $F$-gauges.

\subsection{Application to Shimura varieties}
Theorem~\ref{introthm:main} also has a global application, which was the main motivation for one of us (K.M.) to pursue the work here, and is now explored in~\cite{MadYoucis}. Suppose that $(G,X)$ is a Shimura datum with reflex field $E$, and that $G$ is unramified at $p$ with reductive model $G_{\Int_p}$: This implies in particular that $E$ is unramified over $p$. Fix a place $v\mid p$ of $E$, and choose an $\Reg{E_v}$-rational representative $\mu^{-1}:\Gm\to G_{\Reg{E_v}}$ for the (inverse of the) conjugacy class of Shimura cocharacters underlying $X$. 

Consider a neat level subgroup $K\subset G(\Adele_f)$ of the form $K_pK^p$ with $K_p = G_{\Int_p}(\Int_p)$ and $K^p\subset G(\Adele_f^p)$ a compact open subgroup. Associated with this is the Shimura variety $\Sh_K$ over $E$, and over it we find a canonical pro-\'etale $G^c(\Int_p)$-torsor $\mathbf{Et}_{K,p}$. 

In~\cite{MadYoucis}, we find the definition of an \emph{integral canonical model} $\Ss_K$ for $\Sh_K$ over $\Reg{E,(v)}$, which asks for the existence of a formally \'etale map of formal stacks over $\Reg{E_v}$
\[
  \varpi: \widehat{\Ss}_K\to \BT[G^c_{\Int_p},\mu^{-1}]{\infty},
\]
whose \'etale realization (via a functor such as the one described in~\cite[\S 6.3]{bhatt_lectures}) is $\mathbf{Et}_{K,p}$, and also requires that the adic generic fiber $\widehat{\Ss}_{K,\eta}$ contain all the classical points of the analytification of $\Sh_{K,E_v}$ where $\mathbf{Et}_{K,p}$ is crystalline.

Furthermore, we now know the following facts about such integral canonical models (see \emph{op.\@ cit.}):
\begin{enumerate}
   \item Integral canonical models are unique up to unique isomorphism, and in fact satisfy a strong mapping property with respect to \emph{normal} excellent algebraic spaces over $\Reg{E,(v)}$.

   \item The integral models of abelian type constructed by Kisin~\cite{kisin:abelian} and Kim--Madapusi~\cite{Kim2016-fb} are integral canonical models in this sense.

   \item If either $p>2$ and $(G,X)$ is \emph{pre}-abelian, or if $p$ is sufficiently large (the bound depending only on the adjoint datum $(G^{\ad},X^{\ad})$), an integral canonical model $\Ss_K$ exists.
 \end{enumerate}
The above results amount to a unification---and a substantial generalization---of the results of Imai--Kato--Youcis~\cite{imai2023prismatic} and Bakker--Shankar--Tsimerman~\cite{bst}, and provide the required setup for applying the results of this paper to global questions.

\subsection{A note on the terminology}

Various categories of objects associated with the pair $(G,\mu)$ show up in this paper, and we have tried our best to find some coherent way for distinguishing between them. Here are some possibly helpful remarks for the reader:
\begin{itemize}
   \item For objects appearing over $p^n$-torsion bases, we have used the adjective \defnword{$n$-truncated}: this is compatible via Theorem~\ref{thm:dieudonne} with the corresponding terminology for Barsotti-Tate groups.
   \item For objects associated with the (higher animated) frames appearing in Section~\ref{sec:higher_frames}, we have used the term $(G,\mu)$-\defnword{windows}: This harkens to Zink's original terminology in~\cite{MR1827031}.
   \item Upon the advice of Drinfeld, we have reserved the term $(G,\mu)$-\defnword{display} for objects associated with the Witt vector frame: this is compatible with the terminology in~\cite{MR4120808}.
   \item Finally, for the fundamental objects living over the syntomification stacks, we have coined the term $(G,\mu)$-\defnword{aperture}: The (admittedly vague) inspiration behind this choice is the aperture in a camera, which directs light onto the lens, a lens that can occasionally be a \emph{prismatic} one. 
   \item Given a particular frame, one can often produce $(G,\mu)$-windows over that frame from $(G,\mu)$-apertures over a quotient ring (see Remark~\ref{rem:pullback_to_frames}): in this sense, frames can be viewed as a device for expanding apertures into windows.
\end{itemize}

\subsection{Structure of the paper}

\begin{itemize}
   \item We begin in Section~\ref{sec:prelim} with some background on derived stacks. We also recall the notion of derived Weil restriction, and some facts about divided powers in the animated context.
   \item In Section~\ref{sec:filtered_rings_stacks}, we recall the story of filtered animated rings and as well as of filtered derived stacks via C. Simpson's perspective of viewing such gadgets as objects over $\Aff^1/\Gm$. We give an account of our notion of a $1$-bounded stack, give examples of such objects and prove some general facts about them.
   \item Section~\ref{sec:higher_frames} contains the technical latticework undergirding this whole enterprise. Here, we present our generalization of Lau's theory of higher frames and displays from~\cite{MR4355476} in an animated context (though, as mentioned above, we use the term `window' instead of `display'). We then use this to prove an abstract version of Grothendieck-Messing theory for $1$-bounded stacks in \S~\ref{subsec:abstract_def_theory}, and we also prove an abstract version of the `reduction to $F$-zips', Theorem~\ref{introthm:map_to_f-zips}, in \S~\ref{subsec:torsor_struct}.
   \item In Section~\ref{sec:bld_stacks}, we review the stack-theoretic constructions of Drinfeld and Bhatt-Lurie from~\cite{bhatt_lectures},~\cite{bhatt2022absolute},~\cite{bhatt2022prismatization} and~\cite{drinfeld2022prismatization}. Our treatment of the Nygaard filtered prismatization here---arising from conversations with Juan Esteban Rodr\'iguez Camargo---appears to be new and works cleanly for animated inputs. Using this perspective, we recall in \S~\ref{subsec:nygaard_filtered_affineness} the \emph{filtered affineness} of the various stacks when working with \emph{semiperfectoid} rings, where the stacks of Drinfeld and Bhatt-Lurie are now obtained---via the Rees construction---from Nygaard filtered prismatic cohomology. We end with an important nilpotence result on the first divided Frobenius on the fiber between the prismatic cohomologies of a nilpotent divided power extension.
   \item Section~\ref{sec:technical_repbility} recalls a result of Bragg-Olsson on the representability of derived stacks that parameterize the fppf cohomology of certain `perfect complexes' of finite flat group schemes of height one and extends it to the almost perfect case.
   \item We then prove our general representability theorems for stacks of sections associated with $1$-bounded stacks: this takes up Section~\ref{sec:abstract}. We follow the strategy sketched above: Representability on the level of $F$-zips is first lifted to representability of the stack of sections over the mod-$p$ syntomification of $\Field_p$-algebras using filtered affineness for semiperfectoid inputs and the results of \S~\ref{subsec:torsor_struct}. This is then bootstrapped to representability over the syntomification of $\Field_p$-algebras, followed by a further bootstrapping up to arbitrary $p$-nilpotent algebras. We give some applications of our general representability results for stacks of $F$-gauges, and prove Theorems~\ref{introthm:f_gauges} and~\ref{introthm:perf_fgauges}.
   \item Section~\ref{sec:main} is where we define the stacks $\BT{n}$ and prove Theorems~\ref{introthm:main},~\ref{introthm:map_to_f-zips} and~\ref{introthm:groth_mess} as consequences of the general results of the previous section.

   \item In Section~\ref{sec:explicit}, we use deformation theory and a strategy introduced by Ito~\cite{ito2023deformation} to give explicit descriptions of the points of $\BT{n}$ valued in certain regular complete local Noetherian rings, and show that the deformation rings defined by Faltings in~\cite[\S 7]{faltings:very_ramified} in fact provide explicit coordinates for the complete local rings of $\BT{\infty} = \varprojlim_n\BT{n}$.
   \item Finally, in Section~\ref{sec:bt_classification}, we gather our results together to prove Theorem~\ref{introthm:dieudonne}. The reader will also find some complements dealing (among other things) with polarizations and compatibility with the classical de Rham and crystalline realizations.
   \item The short appendix~\ref{app:completeness} collects some completeness results in the context of graded and filtered commutative rings that are used in Section~\ref{sec:filtered_rings_stacks}.
\end{itemize}

\section*{Acknowledgments}
Both Z. G. and K. M. were partially supported by NSF grant DMS-2200804.

We thank Bhargav Bhatt, Juan Esteban Rodr\'iguez Camargo, Vladimir Drinfeld, Manuel Hoff, Naoki Imai, Hiroki Kato, Shubhodip Mondal, Ayan Nath and Alex Youcis for helpful remarks, corrections, and conversations.

We also thank Eike Lau for for many essential corrections and clarifications of our understanding of frames.

Very special thanks are due to Akhil Mathew for many enlightening discussions on different aspects of the paper. Crucially, it was his idea that one should use syntomic cohomology to define the functor inducing the equivalence in Theorem~\ref{introthm:dieudonne}, and this was a key motivator for much of the rest of the content of this paper as well. 

\section{Notation and other conventions}
\label{sec:notation}

\begin{enumerate}
  \item We adopt a resolutely $\infty$-categorical approach. This means that all operations, including (but not limited to) limits, colimits, tensor products, exterior powers etc. are always to be understood in a derived sense, unless otherwise stated.

  \item We will use $\mathrm{Spc}$ to denote the $\infty$-category of spaces, anima, or homotopy types: roughly speaking, this is the localization of the Quillen model category of simplicial sets with respect to homotopy equivalences.

  \item A map $X\to Y$ in $\mathrm{Spc}$ is \defnword{surjective} if the induced map $\pi_0(X)\to \pi_0(Y)$ is a surjective map of sets; we will denote surjective maps with $\twoheadrightarrow$.

  \item For any $\infty$-category $\mathcal{C}$ and an object $c$ of $\mathcal{C}$, we will write $\mathcal{C}_{c/}$ (resp. $\mathcal{C}_{/c}$) for the comma $\infty$-categories of arrows $c\to d$ (resp. $d\to c$).
 
  \item We will in a few places make reference to the process of \defnword{animation}, as described say in~\cite[Appendix A]{Mao2021-jt}. This is a systematic way to get well-behaved $\infty$-categories and functors between them, starting from `nice' classical categories $\mathcal{C}$ with a set $\mathcal{C}_0$ of compact, projective generators. The animation of such a category is the $\infty$-category $\mathcal{P}_{\Sigma}(\mathcal{C}_0)$ of presheaves of spaces on $\mathcal{C}_0$ that preserve finite products.

  \item We will denote by $\mathrm{CRing}$ the $\infty$-category of \defnword{animated commutative rings}, obtained via the process of animation from the usual category of commutative rings. Objects here can be viewed as being simplicial commutative rings up to homotopy equivalence.

  \item We will follow homological notation for $\mathrm{CRing}$: For any $n\in \Int_{\ge 0}$, $\mathrm{CRing}_{\leq n}$ will be the subcategory of $\mathrm{CRing}$ spanned by those objects $R$ with $\pi_k(R) = 0$ for $k>n$; that is, by the \defnword{$n$-truncated objects}. If $n = 0$, we will write $\mathrm{CRing}^{\heartsuit}$ instead of $\mathrm{CRing}_{\leq 0}$: its objects are the \defnword{discrete} or classical commutative rings, and the category can be identified with the usual category of commutative rings. 

  \item Any animated commutative ring $R$ admits a \defnword{Postnikov tower} $\{\tau_{\leq n}R\}_{n\in \Int_{\ge 0}}$ where $R\to \tau_{\leq n}R$ is the universal arrow from $R$ into $\mathrm{CRing}_{\leq n}$ and the natural map $R\to \varprojlim_n \tau_{\leq n}R$ is an equivalence.

  \item We will also need the notion of a \defnword{stable $\infty$-category} from~\cite{Lurie2017-oh}: this is the $\infty$-category analogue of a triangulated category. The basic example is the $\infty$-category $\Mod{R}$, the derived $\infty$-category of $R$-modules. We will use \emph{cohomological} conventions for these objects and so will write for instance $H^{-1}(M)$ instead of $\pi_1(M)$.

  \item An important feature of a stable $\infty$-category $\mathcal{C}$ is that it has an initial and final object $0$, and, for any map $f:X\to Y$ in $\mathcal{C}$, we have the \defnword{homotopy cokernel} $\hcoker(f)$ defined as the pushout of $0\to Y$ along $f$. We will sometimes abuse notation and write $Y/X$ for this object.

  \item If $R\in \mathrm{CRing}^{\heartsuit}$ is a classical commutative ring, $M\in \Mod{R}$ is a complex of $R$-modules, and $a_1,\ldots,a_m\in R$ form a regular sequence, we will write $M/{}^{\mathbb{L}}(a_1,\ldots,a_m)$ for the derived tensor product
  \[
    M\otimes^{\mathbb{L}}_RR/(a_1,\ldots,a_m).
  \]

  \item In any stable $\infty$-category $\mathcal{C}$ and an object $X$ in $\mathcal{C}$, we set $X[1]=\hcoker(X\to 0)$: this gives a shift functor $\mathcal{C}\to \mathcal{C}$ with inverse $X\mapsto X[-1]$, and we set $\hker(f:X\to Y)  = \hcoker(f)[-1]$. 

  \item Given an animated commutative ring $R$, we will write $\Mod[\mathrm{cn}]{R}$ for the sub $\infty$-category spanned by the connective objects (that is, objects  with no cohomology in positive degrees), and $\mathrm{Perf}(R)$ for the sub $\infty$-category spanned by the perfect complexes. 

  \item We have a truncation operator $\tau^{\leq 0}:\Mod{R}\to \Mod[\mathrm{cn}]{R}$ defined as the right adjoint to the natural functor in the other direction. This leads to truncation operators $\tau^{\leq n}$ and cotruncation operators $\tau^{\geq n}$ for any $n\in \Int$ in the usual way.

  \item If $f:X\to Y$ is a map in $\Mod[\mathrm{cn}]{R}$, we set $\hker^{\mathrm{cn}}(f) = \tau^{\leq 0}\hker(f)$: this is the \defnword{connective (homotopy) kernel}.

  \item For any stable $\infty$-category $\mathcal{C}$, the mapping spaces $\Map_{\mathcal{C}}(X,Y)$ between any two objects have canonical lifts to the $\infty$-category of connective spectra. We will be interested in stable $\infty$-categories like $\Mod{R}$, which are $\Mod[\mathrm{cn}]{\Int}$-modules, in the sense that the mapping spaces have canonical lifts to $\Mod[\mathrm{cn}]{\Int}$. In this case, we can extend the mapping spaces $\Map_{\mathcal{C}}(X,Y)$ from $\Mod[\mathrm{cn}]{\Int}$ to objects $\mathrm{RHom}_{\mathcal{C}}(X,Y)$ in $\Mod{\Int}$ by taking 
  \[
    \mathrm{RHom}_{\mathcal{C}}(X,Y) = \colim_{k\geq 0}\Map_{\mathcal{C}}(X,Y[k])[-k]\in \Mod{\Int}.
  \]
  When $\mathcal{C} = \Mod{R}$ for an animated commutative ring $R$, this lifts to the \defnword{internal Hom} in $\Mod{R}$, which we will denote by $\mathrm{RHom}_R(X,Y)$.

  \item We will write $\bm{\Delta}$ for the usual \defnword{simplex} category with objects the sets $\{0,1,\ldots,n\}$ and morphisms given by the non-decreasing functions between them.

  \item A \defnword{cosimplicial object} $S^{(\bullet)}$ in an $\infty$-category  $\mathcal{C}$ is a functor
  \begin{align*}
  \bm{\Delta}&\mapsto \mathcal{C}\\
  [n]&\mapsto S^{(n)}.
  \end{align*}
  If $\mathcal{C}$ admits limits, we will write $\mathrm{Tot}S^{(\bullet)}$ for the limit of the corresponding functor: this is the \defnword{totalization} of $S^{(\bullet)}$.

  \item Given any $\infty$-category $\mathcal{C}$ with finite coproducts, and any object $S$ in $\mathcal{C}$ there is a canonical cosimplicial object $S^{(\bullet)}$ in $\mathcal{C}$, the \defnword{\u{C}ech conerve} with $S^{(n)} = \bigsqcup_{i\in [n]}S$. 

  \item If $X$ is a (derived) stack (resp. an object of $\Mod{R}$ for some $R$), and $N\in \Int\backslash\{0\}$, we will write $X[N^{-1}]$ for the base-change $\Spec \Int[N^{-1}]\times X\to \Spec \Int[N^{-1}]$ (resp. for the base-change $\Int[N^{-1}]\otimes_{\Int}X$ in $\Mod{\Int[N^{-1}]\otimes_{\Int}R}$). On the rare occasions when these notations collide, context will make the usage clear.
\end{enumerate}

\section{Stacks and other preliminaries}
\label{sec:prelim}

\subsection{Square-zero extensions and differential conditions}
\label{subsec:square_zero}

Given a pair $(R,M)$ with $R\in \mathrm{CRing}$ and $M\in \Mod[\mathrm{cn}]{R}$, we have a canonical object $R\oplus M\in \mathrm{CRing}_{R/{}/R}$, the \defnword{trivial square-zero extension} of $R$ by $M$: This is obtained by animating the construction on such pairs with $R$ a polynomial algebra and $M$ a finite free $R$-module to the usual square-zero extension $R\oplus M$.

If $R\in \mathrm{CRing}_{A/}$, we set
\[
\mathrm{Der}_A(R,M) = \Map_{A/{}/R}(R,R\oplus M).
\]
This is the space of $A$-\defnword{derivations} of $R$ \defnword{valued in $M$}. We always have the trivial $A$-derivation $d_{\mathrm{triv}} = (\mathrm{id},0)$.

A \defnword{square-zero extension} of $R$ by $M$ in $\mathrm{CRing}_{A/}$ is a surjective map $R'\twoheadrightarrow R$ in $\mathrm{CRing}_{A/}$ such that there exists an $A$-derivation $d:R\to R\oplus M[1]$ and an equivalence of $A$-algebras
\[
R' \xrightarrow{\simeq} R\times_{d,R\oplus M[1],d_{\mathrm{triv}}}R.
\]

We have the \defnword{cotangent complex} $\mathbb{L}_{R/A}\in \Mod[\mathrm{cn}]{R}$: this is obtained by animating the functor taking maps $S\to S'$ of polynomial rings over $\Int$ in finitely many variables to the module of differentials $\Omega^1_{S'/S}$, and is characterized by the property that, for any trivial square zero extension $R\oplus M\twoheadrightarrow R$, there is a canonical equivalence
\[
\Map_{{R}}(\mathbb{L}_{R/A},M)\xrightarrow{\simeq}\mathrm{Der}_A(R,M).
\]

\begin{definition}\label{defn:differential_conditions}
An $R$-algebra $C\in \mathrm{CRing}_{R/}$ is \defnword{finitely presented} (over $R$) if the functor $S\mapsto \mathrm{Map}_{\mathrm{CRing}_{R/}}(C,S)$ respects filtered colimits. For any such finitely presented $C$, the cotangent complex $\mathbb{L}_{C/R}\in \Mod[\mathrm{cn}]{C}$ is perfect; see~\cite[(17.4.3.18)]{Lurie2017-oh}.

If in addition $\mathbb{L}_{C/R}$ is $1$-connective, we say that $C$ is \defnword{unramified} over $R$; if $\mathbb{L}_{C/R}\simeq 0$, we say that $C$ is \defnword{\'etale} over $R$.

We say that a finitely presented $C\in \mathrm{CRing}_{R/}$ is \defnword{smooth} over $R$ if $\mathbb{L}_{C/R}\in \Mod[\mathrm{cn}]{C}$ is locally free of finite rank. It is \defnword{quasi-smooth} if $\mathbb{L}_{C/R}$ is perfect with Tor amplitude $[-1,0]$.
\end{definition}

\subsection{Derived (pre)stacks}

Suppose that $\mathcal{C}$ is an $\infty$-category admitting all finite and sequential limits, totalizations of cosimplicial objects, and filtered colimits. A $\mathcal{C}$-\defnword{valued prestack over} $R\in \mathrm{CRing}$ is a functor
\[
F:\mathrm{CRing}_{R/}\to \mathcal{C}.
\]
If $\mathcal{C} = \mathrm{Spc}$,  we will simply call $F$ a \defnword{prestack over} $R$. Such objects organize into an $\infty$-category $\mathrm{PStk}_R$.

We view such prestacks as presheaves on the $\infty$-category of derived affine schemes $\Spec R'$ over $R$ (by definition opposite to $\mathrm{CRing}_{R/}$), and we can consider the subcategory of prestacks that are fpqc (resp. \'etale) sheaves---that is, presheaves satisfying descent along faithfully flat (resp. faithfully flat and \'etale) maps $\Spec R'\to \Spec R$.

\begin{definition}
\label{defn:artin_stacks}
Following To\"en-Vezzosi~\cite{TV2}, we will say that $F$ is $0$-\defnword{geometric} if we have $F \simeq \Spec R'$ for some $R'\in \mathrm{CRing}_{R/}$, and, inductively, that it is an $n$-\defnword{geometric derived Artin stack over $R$} for an integer $n\ge 1$ if it is an \'etale sheaf and admits a surjective cover $f:U = \sqcup_{i\in I}\Spec R'_i\to F$ of \'etale sheaves with $R'_i\in \mathrm{CRing}_{R/}$ satisfying the following condition: For every $S\in \mathrm{CRing}_{R/}$ and $x\in F(S)$, the base-change $U\times_{f,F,x}\Spec S\to \Spec S$ is represented by an $(n-1)$-geometric derived Artin stack over $S$. 

Following Lurie~\cite[\S 5]{lurie_thesis}, we will say that $F$ is a \defnword{derived Artin $n$-stack over $R$} if it is $m$-geometric for some $m$ and is such that $F(R')$ is $n$-truncated for all discrete $R'\in \mathrm{CRing}_{R/,\heartsuit}$. A derived Artin $0$-stack over $R$ will be called a \defnword{derived algebraic space over $R$}.

A \defnword{derived Artin stack over $R$} is a prestack $F$ that is a derived Artin $n$-stack for some $n\ge 0$. If $R = \Int$, then we will simply say `derived Artin stack' instead.

A map of $X\to Y$ of prestacks over $R$ is a \defnword{relative derived Artin stack} if, for every $R$-algebra $C$ and every $y\in Y(C)$, the base-change $X_y\to \Spec C$ is a derived Artin stack over $C$.
\end{definition}

\begin{definition}
\label{defn:almost_fp}
A prestack $F$ over $R$ is \defnword{locally of finite presentation} or \defnword{locally finitely presented} if for every filtered system $\{C_i\}_{i\in I}$ in $\mathrm{CRing}_{R/}$ with colimit $C\in \mathrm{CRing}_{R/}$, the natural map
\[
\colim_{i\in I}F(C_i)\to F(C)
\]
is an equivalence.

It is \defnword{almost locally of finite presentation} or \emph{almost locally finitely presented} if the above map is an equivalence for filtered colimits of $k$-truncated animated commutative $R$-algebras for all $k\ge 0$.
\end{definition}

\begin{definition}
\label{defn:formally_smooth}
A prestack $F$ over $R$ is \defnword{formally smooth} if for every square-zero extension $C'\twoheadrightarrow C$ in $\mathrm{CRing}_{R/}$, the map $F(C')\to F(C)$ is surjective.

A derived Artin stack over $R$ is \defnword{smooth} if is locally finitely presented and formally smooth.
\end{definition}

\begin{definition}
A prestack $F$ over $A\in \mathrm{CRing}$ that is an fpqc sheaf is \defnword{classical} if it is equivalent as an fpqc sheaf to the left Kan extension to $\mathrm{CRing}_{A/}$ of its classical truncation $F_{\mathrm{cl}}:\mathrm{CRing}_{\pi_0(A)/}\to \mathrm{Spc}$: That is, it is a colimit of derived affine schemes $\Spec B$ with $B\in \mathrm{CRing}_{\pi_0(A)/}$ in the $\infty$-category of fpqc sheaves on $\mathrm{CRing}_{A/}^{\op}$. The functor $F\mapsto F_{\mathrm{cl}}$ is fully faithful when restricted to classical prestacks. 
\end{definition}

\subsubsection{}\label{subsubsec:cotangent_complex}
For any prestack $F\in \mathrm{PStk}_{R}$, we have an $\infty$-category $\mathrm{QCoh}(F)$ of \defnword{quasi-coherent sheaves on $F$}. The precise definition can be found in~\cite[\S 6.2.2]{Lurie2018-kh}: roughly speaking, it is obtained by right Kan extension of the contravariant functor sending $S\in \mathrm{CRing}_{R/}$ to $\Mod{S}$. One can think of an object $\mathcal{M}$ in $\mathrm{QCoh}(F)$ as a way of assigning to every point $x\in F(S)$ an object $\mathcal{M}_x\in\Mod{S}$ compatible with base-change. This $\infty$-category is particularly well-behaved when $F$ is \defnword{quasi-geometric}~\cite[\S 9.1]{Lurie2018-kh}: this means that $F$ is an fpqc sheaf with quasi-affine diagonal admitting a flat cover by an affine derived scheme. Most of the prestacks we will encounter in this paper will be quasi-geometric or instances of a formal analogue of this notion.

\begin{definition}
We will say that $\mathcal{M}$ is \defnword{connective} if $\mathcal{M}_x$ belongs to $\Mod[\mathrm{cn}]{S}$ for each $x\in F(S)$ as above. We will say that it is \defnword{almost connective} if, for every $x\in F(S)$, there exists $n\in \Int_{\geq 0}$ such that $\mathcal{M}_x[n]$ is connective. We will say that it is \defnword{perfect} if, for every $x\in F(S)$, $\mathcal{M}_x$ is perfect. It is \defnword{almost perfect} if, for every $x\in F(S)$, $\mathcal{M}_x$ is almost perfect: That is, there exists $m\ge 0$ such that $\mathcal{M}_x[m]$ is connective and such that, for all $k \ge 1$, $\tau^{\ge -k}(\mathcal{M}_x[m])$ is a finitely presented object in the $\infty$-subcategory spanned by the $k$-truncated connective objects in $\Mod{S}$. Concretely, any such object is up to shifting the geometric realization of a simplicial object valued in projective $S$-modules of finite rank. See~\cite[\S 7.2.4]{Lurie2017-oh},\cite[Appendix A]{MR4560539}.

Write $\mathrm{QCoh}^{\mathrm{cn}}(F)$ (resp. $\mathrm{QCoh}^{\mathrm{acn}}(F)$, resp. $\mathrm{Perf}(F)$, resp. $\mathrm{APerf}(F)$) for the  $\infty$-category spanned by the connective (resp. almost connective, resp. perfect, resp. almost perfect) objects in $\mathrm{QCoh}(F)$.
\end{definition}

\begin{definition}
\label{defn:cotangent_complex}
Following~\cite[\S 17.2.4]{Lurie2018-kh}, we will say that a morphism $f:F\to G$ in $\mathrm{PStk}_{R}$ \defnword{admits a cotangent complex} if there exists $\mathbb{L}_{F/G}\in \mathrm{QCoh}^{\mathrm{acn}}(X)$ such that, for every $C\in \mathrm{CRing}_{R/}$, every $M\in \Mod[\mathrm{cn}]{R}$, and every $x\in F(C)$, we have a canonical equivalence
\[
\mathrm{Map}_{C}(\mathbb{L}_{F/G,x},M)\xrightarrow{\simeq}\mathrm{fib}_{(f(x)[M],x)}(F(C\oplus M)\to G(C\oplus M)\times_{G(C)}F(C)).
\]
Here, $f(x)[M]\in G(C\oplus M)$ is the image of $f(x)$ along the natural section $G(C)\to G(C\oplus M)$.

If $F = \Spec C$ and $G = \Spec D$, then by Yoneda, any morphism $f:F\to G$ corresponds to an arrow $D\to C$ in $\mathrm{CRing}_{R/}$, and $f$ admits a cotangent complex, namely $\mathbb{L}_{C/D}$.
\end{definition}

\begin{remark}
\label{rem:classicality}
Suppose that $F$ is a locally finitely presented derived Artin stack over $R\in \mathrm{CRing}_{\heartsuit}$ such that the cotangent complex $\mathbb{L}_{F/R}$ is a perfect complex of \emph{non-negative} Tor-amplitude. Then $F$ is smooth and classical. By an argument via induction on $n$ where $F$ is an $n$-geometric derived Artin stack, this reduces to the fact that smooth $R$-algebras are flat over $R$ and thus classical; see~\cite[Prop. 3.4.9]{lurie_thesis}.
\end{remark}

\subsection{Derived vector stacks}
\label{subsec:vector stacks}

We have the classical construction associating with every finite locally free $R$-module $M$ the affine $R$-scheme $\mathbf{V}(M)$ with ring of functions $\Sym_R(M^\vee)$ the symmetric algebra of the $R$-dual $M^\vee$ of $M$.\footnote{This is Grothendieck's convention.} Its functor of points is given by $S\mapsto \Hom_R(M,S)$.

One can now consider, for any $R\in \mathrm{CRing}$ and any almost perfect complex $M\in \Mod{R}$, the prestack
\[
\mathrm{CRing}_{R/}\xrightarrow{S\mapsto \Map_{{R}}(M,S)}\mathrm{Spc}.
\]
It is represented by an almost finitely presented derived Artin $n$-stack $\mathbf{V}(M)$ over $R$ where $n$ is such that $M[n]$ is connective. When $M$ is connective, this is derived affine and represented by the spectrum of the derived symmetric algebra $\Sym_R(M)$, and the general case is obtained by taking iterated classifying stacks.

It is easy to see from the definition that $\mathbf{V}(M)$ has cotangent complex given by
\[
\mathbb{L}_{\mathbf{V}(M)/R}\simeq \Reg{\mathbf{V}(M)}\otimes_RM.
\]

\subsection{$p$-adic formal stacks}
\label{subsec:padic_formal_stacks}

Let $\mathrm{CRing}^{p\text{-nilp}}$ be the subcategory of $\mathrm{CRing}$ spanned by those objects $R$ such that $p$ is nilpotent in $\pi_0(R)$. 

\begin{definition}
A \defnword{$p$-adic formal prestack} over $R\in \mathrm{CRing}$ is simply a $\mathrm{Spc}$-valued functor on $\mathrm{CRing}^{p\text{-nilp}}_{R/}$. 
\end{definition}

\begin{definition}
For any $R$-algebra $S$, the restriction of the affine scheme $\Spec S$ to $\mathrm{CRing}^{p\text{-nilp}}_{R/}$ yields a $p$-adic formal prestack, which, since $p$ will be fixed in this paper, we will denote simply by $\Spf S$. This depends only on the $p$-completion of $S$.
\end{definition}

\begin{definition}
A $p$-adic formal prestack is a \defnword{derived $p$-adic formal Artin stack} if, for each $n\ge 1$, its restriction to $\mathrm{CRing}_{(\Int/p^n\Int)/}$ is represented by a derived Artin stack. Given such a derived $p$-adic formal Artin stack $F$, we will say that it is \defnword{foo}, if `foo' is an attribute applicable to derived Artin stacks, and, if for each $n\ge 1$, the restriction of $F$ to $\mathrm{CRing}_{(\Int/p^n\Int)/}$ is a derived Artin stack that is foo.
\end{definition}

\begin{definition}
Suppose that we have a surjective map $A\twoheadrightarrow\overline{A}$ in $\mathrm{CRing}$ with fiber $J$ such that $\pi_0(\overline{A})_{\mathrm{red}}$ is an $\Field_p$-algebra. Then we can consider the $p$-adic formal prestack $\Spf(A,J)$ given for each $C\in \mathrm{CRing}^{p\text{-nilp}}$ by
\[
\Spf(A,J)(C) = \Map_{\mathrm{CRing}}(A,C)\times_{\Map_{\mathrm{CRing}}(\pi_0(A),\pi_0(C)_{\mathrm{red}})}\Map_{\mathrm{CRing}}(\pi_0(\overline{A})_{\mathrm{red}},\pi_0(C)_{\mathrm{red}}).
\]
In other words, we are looking at maps $A\to C$ such that the image of $J$ is locally nilpotent in $\pi_0(C)$. If $J$ is clear from context, we will sometimes just write $\Spf(A)$ instead.
\end{definition}

\begin{remark}
Let $\mathrm{CRing}^{p\text{-comp}}$ be the subcategory of $\mathrm{CRing}$ spanned by the (derived) $p$-complete animated commutative rings. Then any $p$-adic formal prestack $\mathcal{Y}$ can be extended to a functor on $\mathrm{CRing}^{p\text{-comp}}$. For any $p$-complete $A$, we have
\[
\mathcal{Y}(A) \defn \varprojlim_{n}\mathcal{Y}(A/{}^{\mathbb{L}}p^n).
\]
More precisely, this definition allows us to evaluate $\mathcal{Y}$ on any animated commutative ring $A$, but its value at $A$ depends only on the $p$-completion of $A$.
\end{remark}

\subsection{Weil restrictions}
\label{subsec:weil}
A very useful aspect of derived geometry is the ability to construct well-behaved Weil restrictions along certain non-flat maps. This will enable us to correctly identify the local models for our stacks from Theorem~\ref{introthm:main}.

\begin{definition}
Given $R\in \mathrm{CRing}$, for any prestack $X$ over $R/{}^{\mathbb{L}}p^n$, we will define its \defnword{Weil restriction} $\Res_{(R/{}^{\mathbb{L}}p^n)/R}X$ to be the $p$-adic formal prestack over $R$ given by the composition
\[
\mathrm{CRing}^{p\text{-nilp}}_{R/}\xrightarrow{C\mapsto C/{}^{\mathbb{L}}p^n}\mathrm{CRing}_{R/{}^{\mathbb{L}}p^n/}\xrightarrow{X}\mathrm{Spc}.
\]
\end{definition}

\begin{definition}
If $Y$ is a $p$-adic formal prestack over $R$, we will set
\[
Y^{(n)} = \Res_{(R/{}^{\mathbb{L}}p^n)/R}(Y\vert_{\mathrm{CRing}_{R/{}^{\mathbb{L}}p^n}/}).
\]
There is then a canonical map $\alpha^{(n)}:Y\to Y^{(n)}$ given on points by $Y(C)\to Y(C/{}^{\mathbb{L}}p^n)$. 
\end{definition}

\begin{remark}
There is a natural functor 
\[
\mathrm{QCoh}(X)\xrightarrow{\mathcal{F}\mapsto \mathcal{F}^{(n)}} \mathrm{QCoh}(\Res_{(\Int/p^n\Int)/\Int_p}X).
\]
With any $\tilde{x}\in (\Res_{(\Int/p^n\Int)/\Int_p}X)(C)$ corresponding to $x\in X(C/{}^{\mathbb{L}}p^n)$ it associates the object $\mathcal{F}^{(n)}_{\tilde{x}}$, which is the image of $\mathcal{F}_{x}[-1]\in \Mod{C/{}^{\mathbb{L}}p^n}$ in $\Mod{C}$.
\end{remark}

\begin{proposition}
\label{prop:weil_restriction}
Suppose that we have a map $Y\to Z$ of prestacks over $\Int/p^n\Int$ that is a relative locally almost finitely presented (resp. smooth, resp. \'etale) derived Artin $r$-stack with cotangent complex $\mathcal{L} \defn \mathbb{L}_{Y/Z}$. Then 
\[
\Res_{(\Int/p^n\Int)/\Int_p}Y\to \Res_{(\Int/p^n\Int)/\Int_p}Z
\]
is once again a relative locally almost finitely presented (resp. smooth, resp. \'etale) derived $p$-adic formal Artin $(r+1)$-stack, and we have a canonical identification
\[
\mathbb{L}_{(\Res_{(\Int/p^n\Int)/\Int_p}Y)/(\Res_{(\Int/p^n\Int)/\Int_p}Z)} \xrightarrow{\simeq}\mathcal{L}^{(n)}.
\]
\end{proposition}
\begin{proof}
Set 
\[
\tilde{Y} = \Res_{\Int/p^n\Int/\Int_p}Y; \tilde{Z} =  \Res_{\Int/p^n\Int/\Int_p}Z.
\]
To begin, suppose that we have $\tilde{y}\in\tilde{Y}(C)$ corresponding to $y\in Y(C/{}^{\mathbb{L}}p^n)$ with image $\tilde{z}\in \tilde{Z}(C)$ corresponding to $z\in Z(C/{}^{\mathbb{L}}p^n)$. Let $\tilde{z}'\in \tilde{Z}(C\oplus M)$ be the trivial lift of $\tilde{z}$ corresponding to $z'\in Z((C\oplus M)/{}^{\mathbb{L}}p^n)$. Then we have:
\begin{align*}
\mathrm{fib}_{(\tilde{y},\tilde{z}')}(\tilde{Y}(C\oplus M)\to \tilde{Y}(C)\times_{\tilde{Z}(C)}\tilde{Z}(C\oplus M)) & = \mathrm{fib}_{(y,z')}(Y((C\oplus M)/{}^{\mathbb{L}}p^n)\to Y(C/{}^{\mathbb{L}}p^n)\times_{Z(C/{}^{\mathbb{L}}p^n)}Z((C\oplus M)/{}^{\mathbb{L}}p^n))\\
&\simeq \Map_{C/{}^{\mathbb{L}}p^n}(\mathcal{L}_y,M/{}^{\mathbb{L}}p^n)\\
&\simeq \Map_{C/{}^{\mathbb{L}}p^n}(\mathcal{L}_y,\mathrm{RHom}_{C}(C/{}^{\mathbb{L}}p^n,M[1]))\\
&\simeq \Map_{{C}}(i_{n,*}\mathcal{L}_y[-1],M).
\end{align*}
This proves that the cotangent complex is as claimed.

Now, given a map $\Spec R\to \tilde{Z}$ with $R\in \mathrm{CRing}^{p\text{-nilp}}$ corresponding to a map $\Spec R/{}^{\mathbb{L}}p^n\to Z$, we see that
\[
\tilde{Y}\times_{\tilde{Z}}\Spec R \simeq \Res_{(R/{}^{\mathbb{L}}p^n)/R}(Y\times_Z\Spec R/{}^{\mathbb{L}}p^n).
\]
Therefore, the first assertion amounts to showing that $\tilde{V} \defn \Res_{(R/{}^{\mathbb{L}}p^n)/R}V$ is a locally almost finitely presented derived Artin $(r+1)$-stack over $R$ whenever $V$ is a locally almost finitely presented derived Artin $r$-stack over $R/{}^{\mathbb{L}}p^n$. If $\pi_0(R)$ is a $G$-ring, then, given our description of the cotangent complex in the first paragraph, this follows quite easily from Artin-Lurie representability~\cite[Theorem 7.1.6]{lurie_thesis}. The general case can be deduced from this via standard approximation techniques.

For the remaining assertions, note that, if $Y$ is smooth over $Z$, so that $\mathcal{L}$ is a perfect complex with Tor-amplitude in $[0,\infty)$, then $\mathcal{L}^{(n)}$ is also perfect with Tor-amplitude in $[0,\infty)$, showing that $\tilde{Y}$ is a smooth Artin stack over $\tilde{Z}$. The same argument shows that $\tilde{Y}$ is \'etale over $\tilde{Z}$ when $Y$ is \'etale over $Z$.
\end{proof}

\subsection{Divided powers}
\label{subsec:divided_powers}

We will also need the notion of animated divided powers, which is an additional structure $\gamma$ on surjective maps $R'\twoheadrightarrow R$ in $\mathrm{CRing}$ that `animates' the classical notion. 

\subsubsection{}
We follow the presentation from~\cite[\S 3.2]{Mao2021-jt}, where one obtains an $\infty$-category $\mathrm{AniPDPair}$ via the process of animation: one takes the full subcategory $\mathcal{E}^{0}$ of the classical category $\mathrm{PDPair}$ of divided power thickenings $(R'\twoheadrightarrow R,\gamma)$ spanned by those thickenings of the form
\[
(D_{(Y)}\Int[X,Y]\twoheadrightarrow \Int[X],\gamma)
\]
where $X,Y$ are finite sets of variables and $D_{(Y)}\Int[X,Y]$ is the divided power envelope of $\Int[X,Y]\xrightarrow{Y\mapsto 0}\Int[X]$ equipped with its tautological divided powers, and then takes $\mathrm{AniPDPair} = \mathcal{P}_{\Sigma}(\mathcal{E}^0)$ to be the $\infty$-category of finite product preserving presheaves on $\mathcal{E}^0$. The natural functor $\mathrm{PDPair}\to \mathrm{AniPDPair}$ obtained via the Yoneda map is fully faithful; see~\cite[Lemma 3.13]{Mao2021-jt}.

\subsubsection{}
There is a forgetful functor $\mathrm{AniPDPair}\to \mathrm{AniPair}$ that preserves all limits and finite colimits to the $\infty$-category $\mathrm{AniPair}$ of surjective maps $R'\twoheadrightarrow R$, and we can view a divided power structure $\gamma$ on such a surjective map as being a lift along the forgetful functor. The forgetful functor admits a left adjoint, the \defnword{divided power envelope}, carrying $f:R'\twoheadrightarrow R$ to a surjection $D(f)\twoheadrightarrow R$ equipped with a divided power structure. For all this, see the discussion in~\cite[\S 3.2]{Mao2021-jt}.

\begin{definition}
A \defnword{divided power extension (or thickening)} is a pair $(R'\twoheadrightarrow R,\gamma)$, where $\gamma$ is a divided power structure on $R'\twoheadrightarrow R$.
\end{definition}

\subsubsection{}
We now explain the relationship with divided power algebras. For any $R\in \mathrm{CRing}$ and any $M\in \Mod[\mathrm{cn}]{R}$, we have the \defnword{animated divided power algebra} $\Gamma_R(M)$: this is obtained by animating the usual divided power algebra on pairs $(R,M)$ with $R$ a polynomial algebra in finitely many variables and $M$ is a finite free $R$-module defined for instance in~\cite[App. A]{berthelot_ogus:notes}. This is in some sense a classical construction that goes back to the seminal work of Dold-Puppe~\cite{MR0150183}. 

By construction, $\Gamma_R(M)$ is an object in $\mathrm{CRing}_{R/}$, equipped with a map of $R$-modules $M\to \Gamma_R(M)$ that satisfies a certain universal property. To explain this, write $\mathbb{G}_a^\sharp$ for the affine scheme $\Spec\Gamma_{\Int}(\Int)$: this is the divided power envelope of the origin in the additive group $\Ga$. 
\begin{lemma}
\label{lem:divided_powers_univ}
For any $M\in\Mod[\mathrm{cn}]{R}$ and for any other $R$-algebra $C$, we have a canonical equivalence
\[
\Map_{\mathrm{CRing}_{R/}}(\Gamma_R(M),C)\xrightarrow{\simeq}\Map_{{R}}(M,\mathbb{G}_a^\sharp(C)).
\]
\end{lemma}
\begin{proof}
Both sides of the purported equivalence are evaluations on $M$ of functors $\Mod[\mathrm{cn}]{R}\to \mathrm{Spc}^{\op}$ that preserve sifted colimits. Therefore, by~\cite[Prop. 5.5.8.15]{Lurie2009-oc}, it suffices to construct a canonical equivalence between these functors when evaluated on free $R$-modules of finite rank. That is, we want to construct isomorphisms
\[
\Map_{\mathrm{CRing}_{R/}}(\Gamma_R(R^n),C)\xrightarrow{\simeq}\mathbb{G}_a^\sharp(C)^n;
\]
or in other words isomorphisms
\[
\Gamma_R(R^n)\xrightarrow{\simeq}\underbrace{\Gamma_R(R)\otimes_R\otimes\cdots\otimes_R\Gamma_R(R)}_{n}
\]
in $\mathrm{CRing}_{R/}$. This is classical; see for instance~\cite[Prop. (A2)]{berthelot_ogus:notes}.

\end{proof} 

\begin{remark}
\label{rem:divided_powers_div_algebra}
For each $m\ge 1$, there is a canonical map of schemes $u_m:\Ga^\sharp\to \Ga$ induced by the canonical generator for the degree-$m$ piece $\Gamma^m_{\Int}(\Int)\subset \Gamma_{\Int}(\Int)$: this is the $m$-th divided power map. If $R'\twoheadrightarrow R$ is in $\mathrm{AniPair}$ with homotopy kernel $I$, then a divided power structure on $I$ gives rise to a map of $R'$-algebras $\Gamma_{R'}(I)\to R'$---equivalently a map of $R'$-modules $I \to \Ga^\sharp(R')$---equipped with: 
\begin{itemize}
   \item A homotopy equivalence between the induced map $I\to \Gamma_{R'}(I)\to R'$ with the tautological one;
   \item For each $m\ge 1$, a lift $\gamma_m:I\to I$ of the composition
   \[
     I \to \Ga^\sharp(R')\xrightarrow{u_m}R'.
   \]
\end{itemize}
In the classical setting, the divided power operators $\gamma_m$ determine the divided power structure on $I$ completely.
\end{remark}

\begin{definition}
\label{defn:nilpotent_divided_powers}
Suppose that $(R'\twoheadrightarrow R,\gamma)$ is a divided power extension with $\pi_0(R')$ a $p$-nilpotent ring. Then we will say that $R'\twoheadrightarrow R$ is a \defnword{locally nilpotent} (resp. \defnword{nilpotent}) divided power extension if $\gamma_p:I\to I$ is a locally nilpotent\footnote{That is, a map that is a filtered colimit of nilpotent endomorphisms. Here, `nilpotence' is being used in the context of pointed spaces, where a self-map $f:(X,*)\to (X,*)$ is nilpotent if some power $f^m$ is nullhomotopic, that is, is homotopic to the constant map with value $*$.} (resp. nilpotent) endomorphism. If the divided power structure is nilpotent, its \defnword{nilpotence degree} is the smallest integer $m\ge 1$ such that $\gamma_p^m$ is nullhomotopic.
\end{definition}

\begin{remark}
\label{rem:nilpotent_divided_powers_classical}
When $(R'\twoheadrightarrow R,\gamma)$ is a classical divided power extension of $p$-nilpotent rings (that is, it lies in the image of $\mathrm{PDPair}$ with $R'$ $p$-nilpotent), then the local nilpotence condition simply says that every element $x\in I = \ker(R'\to R)$ satisfies $\gamma_m(x) = 0$ for all $m$ sufficiently large. Indeed, if $\gamma_p^r(x) = 0$, then a short calculation with the properties of divided powers and $p$-adic valuations of factorials shows that, for all $m\ge 1$, and all $k<p^r$, there exists a unit $u\in \Int_{(p)}^\times$ such that
\[
\gamma_{p^rm+k}(x) = u\gamma_m(\gamma_{p}^r(x))\gamma_k(x) = 0.
\] 
\end{remark} 

\begin{remark}
\label{rem:nilpotent_finitely_generated_divided_powers}
If $I$ is finitely generated, then local nilpotence is equivalent to saying that there exists an integer $m\ge 1$ such that we have $I^{[m]} = 0$. Here, $I^{[m]}$ is the ideal generated by $\gamma_r(x)$ for $r\ge m$ and $x\in I$. Indeed, let $x_1,\ldots,x_k$ be generators for $I$, and let $n_0$ be such that $\gamma_m(x_i) = 0$ for all $m\ge n_0$. Then the identities
\[
\gamma_n(y_1+\cdots+y_k) = \sum_{n_1+\ldots+n_k = n}\prod_{i=1}^k\gamma_{n_i}(y_i)\;;\; \gamma_n(ay) = a^n\gamma_n(y)
\]  
show that we have $\gamma_m(x) = 0$ for all $x\in I$ and all $m\ge kn_0$.
\end{remark}

\section{Filtered abstractions}
\label{sec:filtered_rings_stacks}

The purpose of this section is to introduce enough background about filtered animated commutative rings and modules so that we can discuss the key notion of a \emph{$1$-bounded stack}. We also give examples, as well as record some important properties, of these objects.

\subsection{Graded rings and modules}
\label{subsec:graded}

\subsubsection{}
As usual, a graded ring or module can be viewed as a $\Gm$-equivariant object. Therefore, given $R\in \mathrm{CRing}$, we will define the $\infty$-category of  \defnword{graded animated commutative $R$-algebras} to be the opposite to the $\infty$-category of relatively affine map of derived stacks $X\to B\Gm\times \Spec R$. Let $\mathcal{O}(1)$ be the inverse tautological bundle over $B\Gm$, and set $\mathcal{O}(i) = \mathcal{O}(1)^{\otimes i}$. Then, given a relatively affine map $X\to B\Gm\times \Spec R$, we will denote the coresponding graded animated ring symbolically by $B_\bullet = \oplus_i B_i$, where $B_i = R\Gamma(X,\mathcal{O}(i))$, so that $X = (\Spec B_\bullet)/\Gm$.

\subsubsection{}
The $\infty$-category $\mathrm{GrMod}_{B_\bullet}$ of \defnword{graded $B_\bullet$-modules} is the category $\mathrm{QCoh}(X)$. Symbolically, if $\mathcal{F}$ is a quasicoherent sheaf over $X$, we can write the associated graded module in the form $M_\bullet = \oplus_iM_i$, where $M_i = R\Gamma(X,\mathcal{F}\otimes\mathcal{O}(i))$. Note that by construction this is a symmetric monoidal $\infty$-category, and we denote the associated graded tensor product of $M_\bullet$ and $N_\bullet$ in $\mathrm{GrMod}_{B_\bullet}$ by
\[
M_\bullet\otimes_{B_\bullet}N_\bullet.
\]
We will usually write $\Map_{B_\bullet}(M_\bullet,N_\bullet)$ for mapping spaces in this category.

Given any graded module $M_\bullet$ over $B_\bullet$ and an integer $i$, we obtain the \defnword{$i$-shifted module} $M_\bullet(i)$: If $M_\bullet$ is associated with a quasicoherent sheaf $\mathcal{F}$ over $(\Spec B_\bullet)/\Gm$, $M_\bullet(i)$ is associated with $\mathcal{F}\otimes \mathcal{O}(i)$ and satisfies $(M_\bullet(i))_m = M_{m+i}$.

\subsubsection{}
Note that we can use this optic to speak of \defnword{graded perfect} $B_\bullet$-modules and \defnword{graded vector bundles} over $B_\bullet$: they will correpond to perfect complexes (resp. vector bundles) over $X$.

For any $R$-algebra $C$, the relatively affine map $B\Gm\times \Spec C\to B\Gm\times \Spec R$ corresponds to $C$ with its \emph{trivial} grading. In this case, we can speak simply of graded $C$-modules, etc.

\subsection{Filtered objects via the Rees construction}
\label{subsec:rees_construction}

We will make frequent use of the quotient stack $\mathbb{A}^1/\Gm$, where we view $\Gm$ as acting on the affine line via $(t,z)\mapsto tz^{-1}$. Explicitly, this stack parameterizes line bundles $\mathcal{L}$ equipped with a \emph{co}section $t:\mathcal{L}\to \Rg$.

\subsubsection{}
This stack gives a geometric method for dealing with filtered objects~\cite{moulinos}. More precisely, for any $R\in \mathrm{CRing}$, there is a canonical equivalence 
\[
\mathrm{QCoh}(\mathbb{A}^1/\Gm\times \Spec R)\xrightarrow{\simeq}\mathrm{FilMod}_R,
\]
where the right hand side is the stable $\infty$-category of filtered objects in $\Mod{R}$: classically, if $R = \pi_0(R)$ is discrete, then its associated triangulated derived category is the usual filtered derived category.

Symbolically, under this equivalence, a filtered module $\Fil^\bullet M$ on the right is associated with the $\Gm$-equivariant $R[t]$-module
\[
\mathrm{Rees}(\Fil^\bullet M) = \bigoplus_{i\in \Int}\Fil^iM\cdot t^{-i}.
\] 
Our convention is that $t$ lives in graded degree $1$. For the functor in the other direction, note that we have a canonical family of line bundles $\mathcal{O}(n) = \mathcal{O}(1)^{\otimes n}$ over $\Aff^1/\Gm$ indexed by integers $n\in \Int$: Here, $\mathcal{O}(1)$ is the inverse tautological line bundle $\mathcal{L}^{\otimes -1}$. Note that we have canonical maps $t:\mathcal{O}(i) \to \mathcal{O}(i+1)$. Given a quasi-coherent sheaf $\mathcal{F}$ over $\Aff^1/\Gm\times \Spec R$, we now obtain a filtered module $\Fil^\bullet M$ by setting $\Fil^iM = R\Gamma(\Aff^1/\Gm\times \Spec R,\mathcal{F}\otimes \mathcal{O}(-i))$ with the transition maps given by $t$. 

\begin{definition}
Any $R$-module $M$, viewed as a quasi-coherent sheaf on $\Spec R$ pulls back to a quasi-coherent sheaf on $\Aff^1/\Gm\times \Spec R$, and this yields a filtered $R$-module $\Fil^\bullet_{\mathrm{triv}}M$ with underlying $R$-module $M$. This filtration is just the \defnword{trivial filtration} with $\Fil^i_{\mathrm{triv}}M = M$ if $i\leq 0$ and $0$ otherwise.
\end{definition}

\begin{definition}
A \defnword{filtered stack over $R$} is an $R$-stack $X$ equipped with a map to $\Aff^1/\Gm\times\Spec R$; we will view it as a filtration on the $R$-stack $X_{(t\neq 0)}$ with associated graded $X_{(t=0)}\to B\Gm$.
\end{definition}

\subsection{Filtered animated commutative rings and filtered modules}

\subsubsection{}
The Rees equivalence also gives us a compact way of defining \defnword{filtered animated commutative $R$-algebras}. These correspond to relatively \emph{affine} stacks over $\Aff^1/\Gm\times \Spec R$. Symbolically, given a filtered $R$-algebra $\Fil^\bullet S$, the $\Gm$-equivariant $R[t]$-module $\mathrm{Rees}(\Fil^\bullet S)$ has a canonical $\Gm$-equivariant structure of an animated commutative $R[t]$-algebra, and taking the quotient of the associated affine scheme over $\Aff^1\times \Spec R$ yields the associated affine morphism
\[
\mathcal{R}(\Fil^\bullet S)\to \Aff^1/\Gm \times \Spec R.
\]
We will call the source of this map the associated \defnword{Rees stack}. Note that the fiber of this stack over the open point $\Gm/\Gm$ is canonically isomorphic to $\Spec S$, and its fiber over $B\Gm$ is the relatively affine stack associated with the graded ring $\bigoplus_i\gr^{-i}S$.

Note that we can now give a precise meaning to the $\infty$-category of filtered animated commutative algebras over the filtered animated commutative ring $\Fil^\bullet S$: it is opposite to the category of relatively affine stacks over $\mathcal{R}(\Fil^\bullet S)$. Moreover, fiber products in this opposite category correspond to filtered tensor products of filtered $\Fil^\bullet S$-algebras.

Observe also that any animated commutative ring $R$ admits a lift $\Fil^\bullet_{\mathrm{triv}}R$ to a \emph{trivially} filtered animated commutative ring corresponding to $\Aff^1/\Gm\times\Spec R$: We have $\Fil^i_{\mathrm{triv}}R = R$ if $i\leq 0$ and $\Fil^i_{\mathrm{triv}}R = 0$ otherwise.

\subsubsection{}
A \defnword{filtered module} over $\Fil^\bullet S$ is now just a quasi-coherent sheaf $\mathcal{F}$ over the associated Rees stack. Once again, concretely, one can write it in the form $\Fil^\bullet M$ where the $S$-modules $\Fil^iM$ are obtained as global sections of suitable twists of $\mathcal{F}$. Write $\mathrm{FilMod}_{\Fil^\bullet S}$ for the associated $\infty$-category. We will write mapping spaces in this category in the form $\Map_{\Fil^\bullet S}(\_,\_)$. If $\Fil^\bullet_{\mathrm{triv}}S$ is the \emph{trivial} filtration, then we will also write $\Map_{\mathrm{FilMod}_S}(\_,\_)$ for this mapping space.

Note that this gives us a symmetrical monoidal $\infty$-category by definition, where the tensor product corresponds to that of quasicoherent sheaves on the Rees stack. We will denote the associated product between filtered $\Fil^\bullet S$-modules $\Fil^\bullet M$ and $\Fil^\bullet N$ by
\[
\Fil^\bullet M\otimes_{\Fil^\bullet S}\Fil^\bullet N.
\]

Using this optic, we can also systematically talk about \defnword{filtered perfect} complexes as well as \defnword{filtered vector bundles} over $\Fil^\bullet S$: these correspond to perfect complexes (resp. vector bundles) on the associated Rees stacks.

Pullback from $\Rees(\Fil^\bullet S)$ to the closed substack $\Rees(\Fil^\bullet S)_{(t=0)}$ yields a symmetric monoidal functor from $\mathrm{FilMod}_{\Fil^\bullet S}$ to $\mathrm{GrMod}_{\gr^\bullet S}$: this is just the functor taking a filtered module to its associated graded.

\subsection{Increasing filtrations}

There is a variant of the above that looks at objects over the stack $\mathbb{A}^1_+/\Gm$ classifying \emph{sections} of line bundles $u:\Rg\to L$: this corresponds to the `usual' action of $\Gm$ on $\Aff^1$. We will write $\Aff^1_+\times\Spec R = \Spec R[u]$ where $u$ has graded degree $-1$.

Quasi-coherent sheaves over this stack are now equivalent to \emph{increasingly} filtered modules $\Fil_\bullet M$, and relatively affine schemes over it are now equivalent to increasingly filtered animated commutative rings $\Fil_\bullet S$. We will denote the corresponding Rees construction by $\Rees_+(\Fil_\bullet S)$. Symbolically, we have
\[
\Rees_+(\Fil_\bullet S) = \Spec\left(\bigoplus_i\Fil_iS\cdot u^i\right)/\Gm.
\]

Observe that any animated commutative ring $R$ admits a lift $\Fil^{\mathrm{triv}}_{\bullet} R$ to a \emph{trivially} increasingly filtered animated commutative ring corresponding to $\Aff^1_+/\Gm\times\Spec R$: We have $\Fil^{\mathrm{triv}}_{i}R = R$ if $i\ge 0$ and $\Fil^{\mathrm{triv}}_iR = 0$ otherwise.

\subsection{Filtered deformation theory}
\label{subsec:filtered_square-zero}

\subsubsection{}
Every filtered animated commutative algebra $\Fil^\bullet S$ over a filtered animated commutative ring $\Fil^\bullet R$ admits a filtered cotangent complex $\mathbb{L}_{\Fil^\bullet S/\Fil^\bullet R}$: this is a filtered $\Fil^\bullet S$-module corresponding to the cotangent complex of the associated Rees stacks. This controls the filtered deformation theory as follows:

A map of filtered animated commutative rings $\Fil^\bullet S' \to \Fil^\bullet S$ is a \defnword{filtered square-zero extension} if the corresponding map of $\Gm$-equivariant affine schemes over $\Aff^1/\Gm$ is a square zero thickening. In this case the fiber of the map of filtered rings is a filtered $\Fil^\bullet S$-module $\Fil^\bullet M$. 

Given a connective filtered module $\Fil^\bullet M$ over $\Fil^\bullet S$, we can consider the \emph{trivial} square-zero extension $\Fil^\bullet S\oplus \Fil^\bullet M$. We then have a canonical equivalence:
\[
\Map_{\Fil^\bullet R/}(\Fil^\bullet S,\Fil^\bullet S\oplus \Fil^\bullet M) \simeq \Map_{\Fil^\bullet S}(\mathbb{L}_{\Fil^\bullet S/\Fil^\bullet R},\Fil^\bullet M).
\]
Sections of either equivalent space will be called \defnword{$\Fil^\bullet R$-derivations} from $\Fil^\bullet S$ to $\Fil^\bullet M$. 

One way to obtain square-zero extensions with fiber $\Fil^\bullet M$ therefore is as the left vertical arrow of a Cartesian diagram of the form
\[
\begin{diagram}
\Fil^\bullet S'&\rTo&\Fil^\bullet S\\
\dTo&&\dTo_{d_{\mathrm{triv}}}\\
\Fil^\bullet S&\rTo_{d}&\Fil^\bullet S\oplus \Fil^\bullet M[1]
\end{diagram}
\]
where the right vertical arrow is the trivial map and the horizontal one on the bottom is a $\Fil^\bullet R$-derivation.

\subsubsection{}
Now suppose that $X\to \Rees(\Fil^\bullet S)$ is an fppf (or even smooth) sheaf admitting a relative cotangent complex $\mathbb{L}_X\defn \mathbb{L}_{X/\Rees(\Fil^\bullet S)}$. For any filtered $\Fil^\bullet S$-algebra $\Fil^\bullet A$, set
\[
X(\Fil^\bullet A) = \Map_{/\Rees(\Fil^\bullet S)}(\Rees(\Fil^\bullet A),X).
\]
Then we obtain a Cartesian diagram
\[
\begin{diagram}
X(\Fil^\bullet S')&\rTo&X(\Fil^\bullet S)\\
\dTo&&\dTo_{d_{\mathrm{triv}}}\\
X(\Fil^\bullet S)&\rTo_{d}&X(\Fil^\bullet S\oplus \Fil^\bullet M[1]).
\end{diagram}
\]

Moreover, for any $x\in X(\Fil^\bullet S)$, pulling $\mathbb{L}_X$ along $x$ yields a filtered module $\Fil^\bullet \mathbb{L}_{X,x}$ over $\Fil^\bullet S$, and we have a canonical equivalence:
\[
\mathrm{fib}_x(X(\Fil^\bullet S\oplus \Fil^\bullet M[1])\to X(\Fil^\bullet S))\xrightarrow{\simeq}\Map_{\Fil^\bullet S}(\Fil^\bullet \mathbb{L}_{X,x},\Fil^\bullet M[1]).
\]

\subsection{The attractor stack}
\label{subsec:attractor_stacks}

The terminology we will use here is borrowed (with a sign difference) from~\cite{drinfeld2015algebraic}.

\begin{definition}
Suppose that we have a prestack $\mathcal{Y}\to B\Gm\times\Spec R$. The associated \defnword{fixed point locus} is the functor $X^0$ on $R$-algebras given by
\[
Y^0(C) =  \Map_{B\Gm\times \Spec R}(B\Gm \times \Spec C,\mathcal{Y}).
\]
 \end{definition} 

\begin{definition}
Suppose that we have a prestack $\mathcal{X}\to \Aff^1/\Gm\times \Spec R$; its associated \defnword{attractor stack} or simply \defnword{attractor} is the functor $X^-$ on $R$-algebras given by:
\[
X^-(C) = \Map_{/\Aff^1/\Gm\times \Spec R}(\Aff^1/\Gm \times \Spec C,\mathcal{X}).
\]
We define its fixed point prestack $X^0$ to be that of the restriction $\mathcal{X}_{(t=0)}$ of $\mathcal{X}$ over the closed substack $B\Gm\times\Spec R$. 
\end{definition}

\subsubsection{}
In other words, $X^-$ (resp. $X^0$) is the Weil restriction of $\mathcal{X}$ (resp. $\mathcal{X}_{(t=0)}$) from $\Aff^1/\Gm\times \Spec R$ (resp. $B\Gm\times \Spec R$) down to $\Spec R$. Note that the sequence of natural maps
\[
B\Gm\times\Spec R \hookrightarrow \Aff^1/\Gm\times \Spec R \to B\Gm\times\Spec R
\]
yields maps
\[
X^0\leftarrow X^-\leftarrow X^0
\]
whose composition is the identity.

\subsubsection{}
If we have a prestack $\mathcal{X}\to \Aff^1_+/\Gm$ then we have the analogous notion of the \defnword{repeller} $X^+$ associated with $\mathcal{X}$, which also admits maps $X^0\to X^+\to X^0$ whose composition is the identity.

If $\mathcal{Y}\to B\Gm$ is a graded prestack, then we will define its attractor and repeller to be those associated with its pullback over $\Aff^1/\Gm$ and $\Aff^1_+/\Gm$, respectively.

\begin{remark}
If $\mathcal{X}$ is the pullback of an algebraic space over $B\Gm$, these notions are studied by Drinfeld in~\cite{drinfeld2015algebraic}.
\end{remark}

\subsubsection{}
\label{subsec:attractor_cotangent_complex}
Suppose that $\mathcal{X}$ is locally almost finitely presented over $\Aff^1/\Gm\times \Spec R$ and that it admits an almost perfect relative cotangent complex. Note that, over $X^-$, we have a canonical filtered almost perfect complex $\Fil^\bullet\mathbb{L}^-_{\mathcal{X}}$: This associates with every $x\in X^-(C)$ the filtered module corresponding to the pullback of the cotangent complex $\mathbb{L}_{\mathcal{X}/(\Aff^1/\Gm\times\Spec R)}$ to $\Aff^1/\Gm\times \Spec C$ along $x$. 

Similarly, over $X^0$, we have a canonical graded almost perfect complex $\mathbb{L}^0_{\mathcal{X},\bullet}$: it is isomorphic to the associated graded of the restriction of $\Fil^\bullet\mathbb{L}^-_{\mathcal{X}}$ along $X^0\to X^-$.

\begin{lemma}
\label{lem:attractor_cotangent_complex}
The prestack $X^-$ admits an almost perfect cotangent complex over $X$; we have
\[
\mathbb{L}_{X^-/R} \simeq \mathbb{L}^-_{\mathcal{X}}/\Fil^1\mathbb{L}^-_{\mathcal{X}}.
\]
Similarly, the prestack $X^0$ admits an almost perfect cotangent complex over $X$ with
\[
\mathbb{L}_{X^0/R} \simeq \mathbb{L}^0_{\mathcal{X},0}.
\]
In particular, we have
\[
\mathbb{L}_{X^-/X}\simeq \Fil^1\mathbb{L}^-_{\mathcal{X}}[1]
\]
\end{lemma}
\begin{proof}
From the discussion in~\S\ref{subsec:filtered_square-zero}, we see that, for $C\in \mathrm{CRing}_{R/}$, $M\in \Mod[\mathrm{cn}]{C}$, and $x\in X^-(C)$, we have
\[
\mathrm{fib}_x(X^-(C\oplus M)\to X^-(C))\simeq \Map_{\mathrm{FilMod}_C}(\Fil^\bullet\mathbb{L}^-_{\mathcal{X},x},\Fil^\bullet_{\mathrm{triv}}M)\simeq \Map_{{C}}(\mathbb{L}^-_{\mathcal{X},x}/\Fil^1\mathbb{L}^-_{\mathcal{X},x},M).
\]
This proves the first part of the lemma. The proof of the second is entirely analogous, and the last follows from the canonical fiber sequence
\[
\mathbb{L}_{X/R}\vert_{X^-}\to \mathbb{L}_{X^-/R}\to \mathbb{L}_{X^-/X}.
\]
\end{proof}

We will now give a general criterion for representability of $X^-$, $X^+$ and $X^0$ due to Halpern-Leistner and Preygel~\cite[Example 1.2.2]{MR4560539}.

\begin{proposition}
\label{prop:t-integrable_repble}
Suppose that $\pi_0(R)$ is a $G$-ring and that $\mathcal{X}\to \Aff^1/\Gm\times \Spec R$ is a locally almost finitely presented derived Artin $1$-stack with quasi-affine (resp. affine) diagonal. Then $X^-,X^0,X^+$ are locally almost finitely presented derived Artin $1$-stacks over $R$, and if $\mathcal{X}$ is \emph{flat} over $\Aff^1/\Gm\times\Spec R$, then $X^-,X^0,X^+$ all have quasi-affine (resp. affine) diagonal.
\end{proposition}
\begin{proof}
We recall some key points of the proof, which uses Lurie's derived generalization of Artin's representability theorem~\cite[Theorem 7.1.6]{lurie_thesis}. 

It is straightforward to see that $X^-,X^0,X^+$ are all \'etale sheaves that are locally almost finitely presented, nilcomplete and infinitesimally cohesive. We have already seen that they admit almost perfect cotangent complexes, and it is clear that their classical truncations are valued in $1$-truncated spaces. 

The main difficulty now is to show that they are \emph{integrable} (condition (3) in~\emph{loc. cit.}). The authors of~\cite{MR4560539} appeal to a very general argument that applies to a wide class of quotient stacks, which are shown to be \emph{cohomologically projective} and hence \emph{formally proper}. We can translate this into rather concrete assertions in the particular cases we are dealing with here. 

For $X^0$, one uses Proposition~\ref{prop:BGm_complete}.

For $X^-$ (the argument for $X^+$ is identical), we need to know that the map
\begin{align}\label{eqn:attractor_integrability_map}
\Map_{/\Aff^1/\Gm\times \Spec R}(\Aff^1/\Gm \times \Spec C,\mathcal{X})\to\varprojlim_m\Map_{/\Aff^1/\Gm\times \Spec R}(\Aff^1/\Gm \times \Spec C/\mathfrak{m}^m,\mathcal{X}).
\end{align}
is an isomorphism. Here, $C$ is a complete local Noetherian $R$-algebra with maximal ideal $\mx$. 

For this, one first finds from Proposition~\ref{prop:A1Gm_complete} that, for any Noetherian $B\in \mathrm{CRing}_{\heartsuit,R/}$, we have
\[
\Map_{/\Aff^1/\Gm\times \Spec R}(\Aff^1/\Gm \times \Spec B,\mathcal{X})\xrightarrow{\simeq}\varprojlim_n\Map_{/\Aff^1/\Gm\times \Spec R}((\Aff^1/\Gm)_{(t^n=0)} \times \Spec B,\mathcal{X}).
\]

Via filtered deformation theory, the desired integrability for $X^-$ now reduces to the already known assertion for $X^0$.\footnote{This argument is closely related to one appearing in~\cite{drinfeld2015algebraic}.}

It remains only to check the assertion about the diagonal, which is~\cite[Proposition 5.1.15]{MR4560539}.
\end{proof}

\subsection{$1$-bounded fixed points}
\label{subsec:1-bounded_fixed_points}

Suppose that $\mathcal{Z}\to B\Gm\times\Spec R$ is a relative locally almost finitely presented derived Artin stack and let $Z^0\to \Spec R$ be the corresponding fixed point locus. 

\begin{definition}
As observed in~\S\ref{subsec:attractor_cotangent_complex}, over $Z^0$ we have the canonical graded almost perfect complex $\mathbb{L}^0_{\mathcal{Z},\bullet}$. The locus $Z^0_{\oneb}$ of \defnword{$1$-bounded fixed points} is the locus of $Z^0$ where we have $\mathbb{L}^0_{\mathcal{Z},i}\simeq 0$ for $i<-1$. If $\mathbb{L}^0_{\mathcal{Z},i}$ is perfect for all $i<-1$, and nullhomotopic for almost all but finitely many such $i$, then this locus is in fact an \emph{open} substack of $Z^0$.
\end{definition}

\begin{example}
\label{ex:sections_1_bounded_perfect_complexes}
Suppose that we have $\mathcal{M}\in \mathrm{APerf}(B\Gm\times \Spec R)$ corresponding to a graded $R$-module $M_\bullet$. Then we can take $\mathcal{Z} = \mathbf{V}(\mathcal{M})\to B\Gm\times \Spec R$ to be the associated vector stack. 

One checks that the corresponding fixed point locus $Z^0$ now is just the vector stack $\mathbf{V}(M_0)\to \Spec R$, while the graded almost perfect complex $\mathbb{L}^0_{\mathcal{Z},\bullet}$ corresponds simply to the restriction of $\mathcal{M}$ to $B\Gm\times \mathbf{V}(M_0)$. This implies that
\[
Z^0_{\oneb} = \mathbf{V}(M_0)\times_{\Spec R}(\Spec R)_{\oneb},
\]
where $(\Spec R)_{\oneb}\subset \Spec R$ is the open locus $M_i$ becomes nullhomotopic for $i<-1$. 
\end{example}  

\begin{example}
\label{ex:perfect_complexes_HTwts_01}
Consider the stack $\mathcal{P}:R\mapsto \mathrm{Perf}(R)_{\simeq}$ on $\mathrm{CRing}$: this is represented by a locally finitely presented derived Artin stack over $\Int$; see~\cite[\S~3]{TOEN2007387}. 

Now take $\mathcal{Z} = \mathcal{P}\times B\Gm\to B\Gm$: the fixed point locus $Z^0$ associates with every $R\in  \mathrm{CRing}$ the $\infty$-groupoid of graded perfect $R$-modules.

The cotangent complex of $\mathcal{P}$ is $M_{\mathrm{taut}}^\vee\otimes M_{\mathrm{taut}}$, where $M_{\mathrm{taut}}\in \mathrm{Perf}(\mathcal{P})$ is the tautological perfect complex. From this, one finds that the graded perfect complex $\mathbb{L}^0_{\mathcal{Z},\bullet}$ is $M_{\mathrm{taut},\bullet}^\vee\otimes M_{\mathrm{taut},\bullet}$, where $M_{\mathrm{taut},\bullet}$ is the tautological graded perfect complex over $Z^0$.

Now, $Z^0_{\oneb}$ is precisely the locus where $M^\vee_{\mathrm{taut},i}\otimes M_{\mathrm{taut},j}\simeq 0$ for all $i,j\in \Int$ with $|j-i|>1$.

In particular, the locus where $M_{\mathrm{taut},i} \simeq 0$ for $i\neq 0,1$ is an open substack of $Z^0_{\oneb}$.
\end{example}

\begin{remark}
If $\mathcal{Z}$ is relatively locally finitely presented, then $\mathbb{L}^0_{\mathcal{Z},\bullet}$ is a graded perfect complex, and the condition of being $1$-bounded can be framed in terms of its dual graded tangent complex $\mathbb{T}^0_{\mathcal{Z},\bullet}$ by requiring that we have $\mathbb{T}^0_{\mathcal{Z},i}\simeq 0$ for $i>1$.
\end{remark}

\subsection{$1$-bounded stacks}
\label{subsec:1_bounded_stacks}

\begin{definition}
\label{defn:pointed_graded}
Suppose that $A\in \mathrm{CRing}$ and $R\in \mathrm{CRing}_{A/}$. An \defnword{$R$-pointed graded prestack} over $A$ is a prestack $\mathcal{Y}\to B\Gmh{A}$ equipped with a morphism $\iota:B\Gm\times\Spec R\to \mathcal{Y}$ of graded prestacks. In particular, for such a prestack, any relative derived Artin stack $\mathcal{Z}\to \mathcal{Y}$ has an associated fixed point locus $Z^0\to \Spec R$ obtained from the base-change of $\mathcal{Z}$ over $B\Gm\times\Spec R$.

Usually $R$ will be implicit, and we will simply call the pair $(\mathcal{Y},\iota)$ a \emph{pointed} graded prestack. If the `point' $\iota$ is also clear from context, we will just refer to $\mathcal{Y}$ as a pointed graded prestack.
\end{definition}

\begin{definition}
A \defnword{$1$-bounded stack} $\mathcal{X} = (\mathcal{X}^\preoneb,X^0)\to (\mathcal{Y},\iota)$ over $(\mathcal{Y},\iota)$ (or simply $\mathcal{Y}$ if $\iota$ is clear from context) consists of the following data: 
\begin{enumerate}
   \item A relative locally almost finitely presented derived Artin $r$-stack $\mathcal{X}^\preoneb\to \mathcal{Y}$;
   \item An open immersion $X^0\hookrightarrow X^{\preoneb,0}$ factoring through $X^{\preoneb,0}_{\oneb}$, which we will refer to as the \defnword{fixed point locus} of $\mathcal{X}$.
\end{enumerate}

One can speak of maps between $1$-bounded stacks over $\mathcal{Y}$ in the obvious way. 
\end{definition}

\subsubsection{}
Suppose that $\mathcal{Y}\to \Aff^1/\Gm\times\Spec A$ is in fact a filtered prestack, and that $\iota$ lifts to a map of filtered stacks $\Aff^1/\Gm\times\Spec R\to \mathcal{Y}$ over $A$. Then we will associate with any $1$-bounded stack $\mathcal{X}\to \mathcal{Y}$ its \defnword{attractor} $X^-\to \Spec R$ by setting $X^- \defn X^{\preoneb,-}\times_{X^{\preoneb,0}}X^0$. Here $X^{\preoneb,-}$ is the attractor of the base-change of $\mathcal{X}^\preoneb$ over $\Aff^1/\Gm\times \Spec R$. Analogously, given a lift $\Aff^1_+/\Gm\times \Spec R\to \mathcal{Y}$, we can define an associated repeller $X^+$.

\subsubsection{}
If $(\mathcal{Z},\eta)\to (\mathcal{Y},\iota)$ is a map of pointed graded prestacks with $\eta:B\Gm\times \Spec C\to \mathcal{Z}$ and $\mathcal{X} = (\mathcal{X}^\preoneb,X^0)$ is a $1$-bounded stack over $\mathcal{Y}$, we set
\[
\Map_{/(\mathcal{Y},\iota)}((\mathcal{Z},\iota),\mathcal{X}) = \Map_{/\mathcal{Y}}(\mathcal{Z},\mathcal{X}^\preoneb)\times_{X^{\preoneb,0}(C)}X^0(C).
\]
Here, the map 
\[
\Map_{/\mathcal{Y}}(\mathcal{Z},\mathcal{X}^\preoneb)\to X^{\preoneb,0}(C) =\Map_{/\mathcal{Y}}(B\Gm\times\Spec C,\mathcal{X}^{\preoneb})
\]
is obtained via restriction along $\eta$.

If $\eta$ and $\iota$ are clear from context, we will simply write $\Map_{/\mathcal{Y}}(\mathcal{Z},\mathcal{X})$ for this space.

Here are our main examples:

\begin{example}
\label{ex:1_bounded_perfect_complexes_stack}
Example~\ref{ex:sections_1_bounded_perfect_complexes} shows that, if $\mathcal{M}\in \mathrm{APerf}(\mathcal{Y})$ is an almost perfect complex whose restriction over $B\Gm\times\Spec R$ yields a graded complex concentrated in degrees $\ge -1$, then the stack $\mathcal{X}^\preoneb \defn \mathbf{V}(\mathcal{M})\to \mathcal{Y}$ underlies a $1$-bounded stack over $\mathcal{Y}$ with $X^0 = X^{\preoneb,0}$.
\end{example} 

\begin{example}
\label{ex:perfect_HT_wts_01_stack}
Example~\ref{ex:perfect_complexes_HTwts_01} shows that, when $\mathcal{Y} = B\Gm$ (viewed as a pointed graded stack in the tautological sense), then we obtain a $1$-bounded stack $\mathcal{P}_{\{0,1\}}\to B\Gm$ with $\mathcal{P}_{\{0,1\}}^\preoneb = \mathcal{P}\times B\Gm$, and $P_{\{0,1\}}^0\subset P_{\{0,1\}}^{\preoneb,0}$ is the open substack parameterizing graded perfect complex $M_\bullet$ with $M_i\simeq 0$ for $i\neq 0,1$.

For any pointed graded stack $(\mathcal{Z},\eta)\to B\Gm$, the space
\[
\mathrm{Perf}_{\{0,1\}}((\mathcal{Z},\eta))\defn \Map_{/B\Gm}(\mathcal{Z},\mathcal{P}_{\{0,1\}})
\]
is the $\infty$-groupoid of perfect complexes on $\mathcal{Z}$ whose restriction along $\eta$ is in graded degrees $0,1$.

If we take $\Aff^1/\Gm\to B\Gm$ to be the canonical map, then the associated attractor is the stack 
\[
R\mapsto \mathrm{Perf}_{\{0,1\}}(\Aff^1/\Gm\times\Spec R) 
\]
of filtered perfect complexes $\Fil^\bullet M$ with $\gr^iM \simeq 0$ for $i\neq 0,-1$. 

Similarly, the repeller associated with $\Aff^1_+/\Gm\to B\Gm$ is the stack of ascendingly filtered perfect complexes $\Fil_\bullet M$ with $\gr_iM\simeq 0$ for $i\neq 0,1$.

Both these stacks are locally finitely presented relative derived Artin stacks over $\Spec \Int$.  This can be verified using Artin-Lurie representability: the only difficulty is integrability, but this is easily checked using results from Appendix~\ref{app:completeness}.
\end{example}    

\begin{example}
 \label{ex:vect_HT_wts_01_stack}
We have an `open substack' $\mathcal{V}_{\{0,1\}}$ of $\mathcal{P}_{\{0,1\}}$ by restricting to the open locus $\mathcal{V}_{\{0,1\}}^\preoneb\subset \mathcal{P}_{\{0,1\}}^\preoneb$, where the tautological perfect complex is in fact a vector bundle. 

For any pair of non-negative integers $d\leq h$, we can further refine this to the $1$-bounded stack $\mathcal{V}_{\{0,1\}}^{h,d} = (\mathcal{V}_{\{0,1\}}^\preoneb,V^{0,h,d}_{\{0,1\}})$, where $V^{0,h,d}_{\{0,1\}}$ is the open and closed substack of the fixed point locus parameterizing graded vector bundles $M_\bullet$ such that $M_i\simeq 0$ for $i\neq 0,1$, and such that $M_1$ is a vector bundle of rank $d$ and $M_0$ is a vector bundle of rank $h-d$.

The attractor $V^{-,h,d}_{\{0,1\}}$ is the stack of filtered vector bundles $\Fil^\bullet V$ where: $V$ has rank $h$; $\gr^i V\simeq 0$ for $i\neq 0,-1$; and $\gr^{-1}V$ has rank $d$.
\end{example} 

\subsection{Cocharacters of group schemes and twisted group stacks}
\label{subsec:cochar}
This following discussion is essentially from~\cite[\S 2.3]{drinfeld2023shimurian}. Suppose that $G$ is a smooth affine group scheme over a classical commutative ring $R$ and let $\mu: \Gmh{R'}\to G_{R'}$ be a cocharacter defined over some $R'\in \mathrm{CRing}_{R/}$, inducing a $\Gmh{R'}$-action on $G_{R'}$ via the adjoint action.

\begin{remark}
\label{rem:drinfeld_definition_cochars}
In~\cite{drinfeld2023shimurian}, Drinfeld considers a seemingly more general situation where $\Gmh{R}$ acts on $G$ via a map $\Gmh{R'}\to \Aut(G)_{R'}$. For instance, one can consider the case of $G = \Ga$ equipped with the natural action of $\Gm$. 

This case can be subsumed---perhaps a bit unnaturally---into the setup here by viewing such a map as a cocharacter of the $R'$-group scheme $G\rtimes_\mu \Gm$, where the semi-direct product is defined by the action of $\Gmh{R'}$ on $G_{R'}$ via $\mu$.

See also Remark~\ref{rem:drinfeld_definition} below.
\end{remark}

\subsubsection{}
The fpqc quotient of $G_{R'}$ by the action of $\mu$ yields a group stack $G\{\mu\}$ over $B\Gmh{R'}$. We have subgroups $U_\mu^{\pm}\subset P_{\mu}^{\pm} \subset G_{R'}$ with $P_{\mu}^{\pm}/U_\mu^{\pm}\simeq M_\mu$ independent of sign: Namely, $P_{\mu}^-$ (resp. $P_{\mu}^+$) is the attractor (resp. repeller) of the base-change of $G\{\mu\}$ over $\Aff^1/\Gm\times\Spec R'$, and $M_\mu$ is the fixed point locus of $G\{\mu\}$.  

Explicitly, given an $R'$-algebra $S$, we have
\begin{align*}
P_{\mu}^-(S) = \Map_{B\Gmh{R}}\left(\Aff^1/\Gm\times \Spec S,G\{\mu\}\right)\;&;\;  P_{\mu}^+(S) = \Map_{B\Gmh{R}}\left(\Aff^1_+/\Gm\times \Spec S,G\{\mu\}\right);\\
M_\mu(S) &= \Map_{B\Gmh{R}}\left(B\Gm\times\Spec S,G\{\mu\}\right).
\end{align*}
Restriction to the open point $\Gm/\Gm\subset \Aff^1/\Gm$ now gives a closed immersion of group schemes $P^{\pm}_\mu\hookrightarrow G$. The section $M_\mu\to P_\mu^{\pm}$ exhibits it as the centralizer in $P_\mu^{\pm}$ (and in $G$) of $\mu$.

The subgroup $U_\mu^{\pm}\subset P_\mu^{\pm}$ is the kernel of the map to $M_\mu$.

\subsubsection{}
In terms of Lie algebras, the action of $\mu$ gives us a grading of the base-change of $\mathfrak{g} \defn \Lie G$ over $R'$:
\[
\mathfrak{g}_{R'} = \bigoplus_{i\in \Int}\mathfrak{g}_i,
\]
where $\Gm$ acts on $\mathfrak{g}_i$ via $z\mapsto z^{-i}$. We now have:\footnote{We are following the sign conventions from~\cite{MR4355476}}
\[
\Lie P^{\pm}_\mu = \bigoplus_{\pm i\ge 0}\mathfrak{g}_i\;;\; \Lie U^\pm_\mu = \bigoplus_{\pm i>0}\mathfrak{g}_i\;;\; \Lie M_\mu = \mathfrak{g}_0.
\]

When $G$ is reductive, then what we have defined here are the parabolic and unipotent subgroups of $G$ associated with $\mu$.

\subsubsection{}
\label{subsubsec:BGmu_equivalence}
Note that the cocharacter $\mu:\Gmh{R'}\to G_{R'}$ yields a map
\[
B\mu:B\Gmh{R'}\to BG_{R'}
\]
yielding a canonical $G$-torsor $\mathcal{P}_\mu\to B\Gmh{R'}$: The automorphisms of this torsor are represented by the group stack $G\{\mu\}\to B\Gmh{R'}$, and twisting by $\mathcal{P}_\mu$ yields an isomorphism
\[
BG\times B\Gmh{R'}\xrightarrow{\simeq}BG\{\mu\}
\]
of $B\Gmh{R'}$-stacks carrying $\mathcal{P}_\mu$ to the trivial $G\{\mu\}$-torsor.

In particular, we can canonically view every $G\{\mu\}$-torsor over a $B\Gmh{R'}$-stack as a $G$-torsor, and \emph{vice versa}.

\subsubsection{}
\label{subsubsec:fixed_points_BG}
Consider the pointed graded stack $(B\Gmh{R},\iota_{R'})$, where $\iota_{R'}:B\Gmh{R'}\to B\Gmh{R}$ is the structure morphism. If we take the stack $\mathcal{Z} = BG\times B\Gm\to B\Gmh{R}$, the associated fixed point locus $Z^0$ over $R'$ parameterizes, for any $C\in \mathrm{CRing}_{R'/}$, $G$-torsors over $B\Gm\times \Spec C$. 

Note that by the discussion in~\eqref{subsubsec:BGmu_equivalence}, for any $R'$-algebra $C$, we can also view $Z^0(C)$ as the $\infty$-groupoid of $G\{\mu\}$-torsors over $B\Gm\times \Spec C$. Unwinding definitions, one finds that, for any $C\in \mathrm{CRing}_{R'/}$, $Z^0(C)$ is the $\infty$-groupoid of the following equivalent kinds of objects:
\begin{itemize}
   \item $G\{\mu\}$-torsors over $B\Gm\times\Spec C$;
   \item  $G\rtimes_\mu \Gm$-equivariant schemes $\mathcal{P}\to \Spec C$ such that the underlying $G$ action presents $\mathcal{P}$ as a $G$-torsor over $C$.
\end{itemize}

In the next lemma, given $C\in \mathrm{CRing}_{R'/}$ and a $G$-torsor $\mathcal{Q}\to B\Gm\times \Spec C$, we will write $\mathcal{Q}^\mu$ for the corresponding $G\{\mu\}$-torsor.
\begin{lemma}
\label{lem:BGmu_trivial_locus}
There is a canonical open and closed immersion $BM_\mu\to Z^0$ mapping isomorphically onto the locus of $G$-torsors $\mathcal{Q}\to B\Gm\times \Spec C$ satisfying the following equivalent conditions when $\Spec C$ is connected:
\begin{itemize}
    \item[(1a)] There exists an \'etale cover $C\to C'$ such that the restriction of $\mathcal{Q}$ over $B\Gm\times \Spec C'$ is isomorphic to $\mathcal{P}_\mu\otimes_{\mathcal{O}}C'$;

    \item[(1b)]There exists an \'etale cover $C\to C'$ such that the restriction of $\mathcal{Q}^\mu$ over $B\Gm\times \Spec C'$ is trivial;
    \item[(2a)]For every geometric point $C\to \kappa$ of $\Spec C$, the $G$-torsor $x^*\mathcal{Q}$ over $B\Gm\times\Spec \kappa$ is isomorphic to $\mathcal{P}_\mu\otimes_{\mathcal{O}}\kappa$;
     \item[(2b)]For every geometric point $C\to \kappa$ of $\Spec C$, the $G\{\mu\}$-torsor $x^*\mathcal{Q}^\mu$ over $B\Gm\times\Spec \kappa$ is trivial;
    \item[(3a)]For \emph{some} geometric point $C\to \kappa$ of $\Spec C$, the $G$-torsor $x^* \mathcal{Q}$ over $B\Gm\times\Spec \kappa$ is isomorphic to $\mathcal{P}_\mu\otimes_{\mathcal{O}}\kappa$;
    \item[(3b)]For \emph{some} geometric point $C\to \kappa$ of $\Spec C$, the $G\{\mu\}$-torsor $x^* \mathcal{Q}^\mu$ over $B\Gm\times\Spec \kappa$ is trivial.
 \end{itemize}
\end{lemma}
\begin{proof}
The (a),(b) counterparts in each numbered pair are equivalent, so we will replace $\mathcal{Q}$ by $\mathcal{Q}^\mu$ and prove the (b) sides of each pair.

The map $BM_\mu\to Z^0$ associates with each $M_\mu$-torsor $\mathcal{P}^0$ the $G$-torsor obtained via pushforward along the map $M_\mu\to G_{\mathcal{O}}$: Such a $G$-torsor is equipped with a canonical extension to an action of $G\rtimes_{\mu}\Gm$.

Let us now show that this yields an isomorphism of $BM_\mu(C)$ with the space of $G\{\mu\}$-torsors over $B\Gm\times \Spec C$ satisfying any of the three given conditions (1b), (2b) and (3b). For condition (1b), it is easy: Giving such an object over $B\Gm\times \Spec C$ is the same as giving an \'etale torsor over $C$ for the fixed point group scheme of $G\{\mu\}$, which is of course $M_\mu$. 

To finish, it is enough to see that $BM_\mu$ is an open and closed substack of $Z^0$. The quickest way to see this is to observe, as Drinfeld does in~\cite[\S C.2.3]{drinfeld2023shimurian} that we have discrete invariants on $Z^0$ given by $G$-conjugacy classes of cocharacters $\Gm\to G\rtimes_\mu \Gm$ lifting the identity map of $\Gm$. Now, $BM_\mu$ is the open and closed substack of $Z^0$ associated with the trivial such lift.
\end{proof}

\begin{remark}
\label{rem:BGmu_simple}
Let $Z^-$ be the attractor (of the base-change over $\Aff^1/\Gm$ of) $\mathcal{Z}$. Then we find that $Z^-(C)\times_{Z^0(C)}BM_\mu(C)$ is spanned by $G\{\mu\}$-torsors $\mathcal{Q}^\mu$ over $\Aff^1/\Gm\times\Spec C$ satisfying the following condition: There exists an \'etale cover $C\to C'$ such that the restriction of $\mathcal{Q}^\mu$ over $\Aff^1/\Gm\times \Spec C'$ is trivial. Indeed, this amounts to checking that a $G\{\mu\}$-torsor $\mathcal{Q}^\mu$ over $\Aff^1/\Gm\times \Spec C$ with trivial restriction over $B\Gm\times\Spec C$ is itself trivial. This is because $\mathcal{Q}^\mu$ is smooth over $\Aff^1/\Gm\times \Spec C$, and we have
\[
\Map_{\Aff^1/\Gm\times \Spec C}(\Aff^1/\Gm\times\Spec C,\mathcal{Q}^\mu)\xrightarrow{\simeq}\varprojlim_n\Map_{\Aff^1/\Gm\times \Spec C}((\Aff^1/\Gm)_{(t^m=0)}\times\Spec C,\mathcal{Q}^\mu).
\]
See the proof of Proposition~\ref{prop:X_attractor_repble}.

In particular, $Z^-\times_{Z^0}BM_\mu$ is isomorphic to the stack of \'etale torsors for the attractor group scheme associated with $G\{\mu\}$, which is of course $P^-_{\mu}$. In other words, we have $Z^-\times_{Z^0}BM_\mu\simeq BP^-_\mu$.
\end{remark}

\subsection{$1$-bounded cocharacters}
\label{subsec:1-bounded_case}
Here, we will present a key example of a $1$-bounded stack relevant to Theorem~\ref{introthm:main}. The notation will be as above.

\begin{definition}
Following~\cite{MR4355476}, we will say that $\mu$ is \defnword{$1$-bounded} if, under the adjoint action of $\mu$, we have $\mathfrak{g}_i = 0$ for $i>1$. In this case, we will set $\mathfrak{g}^+_\mu = \mathfrak{g}_{1}$. 

If $G$ is reductive, then this condition is equivalent to asking that $\mu$ be \emph{minuscule}.
\end{definition}

\begin{lemma}
\label{lem:1-bounded_exp}
If $\mu$ is $1$-bounded, the exponential map induces an equivalence:
\[
\mathbf{V}(\mathfrak{g}^{+,\vee}_\mu)\xrightarrow{\simeq}U^{+}_\mu.
\]
In particular, for any $C\in \mathrm{CRing}_{R'/}$, we have an equivalence
\[
\mathfrak{g}^+_\mu\otimes_{R}(C/{}^{\mathbb{L}}p^n)\xrightarrow[\simeq]{\exp}U^{+,(n)}_\mu(C).
\]
\end{lemma}
\begin{proof}
See~\cite[Lemma 6.3.2]{MR4355476}.
\end{proof}

\subsubsection{}
We now isolate a particular open substack of $Z^0_{\oneb}$ using the $1$-bounded condition on $\mu$. For this, we begin by noting that, for every representation $W$ of $G$ defined over $R'$, the usual twisting process by the $G$-torsor $\mathcal{P}$ over $B\Gm\times \Spec C$ yields a canonical graded vector bundle $\mathcal{M}_\bullet(W)_{\mathcal{P}}$ over $\Spec C$. 

In this way, we obtain a graded vector bundle $\mathcal{M}_\bullet(W)$ over $Z^0$. Now, the graded perfect complex $\mathbb{T}^0_{\mathcal{Z},\bullet}$ over $Z^0$ is simply $\mathcal{M}_\bullet(\mathfrak{g})[-1]$. 

The $1$-bounded locus $Z^0_{\oneb}$ is the open and closed locus over which we have $\mathcal{M}_i(\mathfrak{g})\simeq 0$ for $i>1$.

Over $BM_\mu$, for each $i\in \Int$, we have the vector bundles $\mathcal{M}^0(\mathfrak{g}_i)$ obtained by twisting the representation $\mathfrak{g}_i$ by the universal $M_\mu$-torsor. The restriction of $\mathbb{T}^0_{\mathcal{Z},\bullet}$ to $BM_\mu$ is then seen to be isomorphic to the graded complex $\bigoplus_{i\in\Int}\mathcal{M}^0(\mathfrak{g}_i)(-i)[-1]$. In particular, the $1$-boundedness of $\mu$ ensures exactly that this is a $1$-bounded complex. That is, we have defined a map $BM_\mu\to Z^0_{\oneb}$.

\begin{definition}
\label{defn:BGmu_1-bounded_stack}
$\mathcal{B}(G,\mu)$ will be the $1$-bounded stack over the pointed graded stack $(B\Gm\times\Spec R,\iota_{R'})$ given by the pair $(BG\times B\Gm,BM_\mu)$.

Remark~\ref{rem:BGmu_simple} shows that the attractor $B(G,\mu)^-$ is simply $BP^-_\mu$, and an analogous argument shows that the repeller is $BP^+_\mu$.
\end{definition}

\begin{remark}
\label{rem:independence_of_mu}
It is easy to see from the definitions that $\mathcal{B}(G,\mu)$ depends only on the isomorphism class of the map $B\mu:B\Gmh{R'}\to BG_{R'}$. In particular, it depends on the cocharacter $\mu$ only up to conjugacy. 
\end{remark}

\subsection{Deformations of $1$-bounded fixed points}
\label{subsec:deform_1bounded_fixed_pts}
Suppose that $\mathcal{Y} = \Spec(B_\bullet)/\Gm$ for a non-positively graded animated commutative ring $B_\bullet$.

\subsubsection{}
For any $m\ge 0$, we have the animated commutative graded `quotient' $B_\bullet\to B_{\ge -m}$ with underlying graded $B_0$-module $\bigoplus_{i\ge -m}B_i(-i)$. To construct this, we need a different perspective on the $\infty$-category of non-positively graded animated commutative rings: They can also be obtained as the \emph{animation} of the category of non-positively graded commutative rings, which admits a set of compact projective generators given by graded polynomial algebras in finitely many homogeneous variables in non-positive degrees. From this perspective, the quotient map $B_\bullet\to B_{\ge -m}$ is simply the animation of the usual construction for non-positively graded commutative rings. Note in particular, that $B_{\ge -1}$ is the \emph{graded} trivial square-zero extension of $B_0$ by $B_{-1}(1)$. 
 
\subsubsection{}
Let $\mathcal{X} = (\mathcal{X}^\preoneb,X^0)\to (\mathcal{Z},\iota)$ be a $1$-bounded stack over a pointed graded prestack with $\iota:B\Gm\times\Spec R\to \mathcal{Z}$. Suppose that we have a map of pointed prestacks $\mathcal{Y}\to \mathcal{Z}$. 

\begin{proposition}
\label{prop:1bounded_graded_deformation}
Suppose that $\mathcal{X}^\preoneb\to \mathcal{Z}$ is graded integrable (Definition~\ref{defn:graded_integrable}). Then the natural map
\[
\Map_{/\mathcal{Z}}(\mathcal{Y},\mathcal{X})\to \Map_{/\mathcal{Z}}((\Spec B_{\ge-1})/\Gm,\mathcal{X})
\]
is an equivalence.
\end{proposition}
\begin{proof}
Consider the following general situation: Suppose that $C'_\bullet \to C_\bullet$ is a square-zero extension of animated non-positively graded commutative $B_\bullet$-algebras with fiber $I_\bullet$. By the graded analogue of the discussion in~\S\ref{subsec:filtered_square-zero}, one sees that the fiber of the map
\begin{align*}
\Map_{/\mathcal{Z}}((\Spec C'_\bullet)/\Gm,\mathcal{X}^\preoneb)\to \Map_{/\mathcal{Z}}((\Spec C)/\Gm,\mathcal{X}^\preoneb)
\end{align*}
over a section $x$ has the following features:
\begin{itemize}
   \item The obstruction to its being non-empty is given by a section of
   \[
    \Map_{C_\bullet}(\mathbb{L}_{\mathcal{X}^\preoneb,x,\bullet},I_\bullet[1])
   \]
    where $\mathbb{L}_{\mathcal{X}^\preoneb,x,\bullet}$ is the graded $C_\bullet$-module obtained via pulling the relative cotangent complex of $\mathcal{X}^\preoneb$ over $\mathcal{Y}$ along $x$. 
    \item If the obstruction is nullhomotopic, then the fiber is equivalent to
    \[
     \Map_{C_\bullet}(\mathbb{L}_{\mathcal{X}^\preoneb,x,\bullet},I_\bullet).
    \]
\end{itemize}

Let $x_0\in \Map_{/\mathcal{Y}}((\Spf C_0)/\Gm,\mathcal{X}^\preoneb) = X^{\preoneb,0}(C_0)$ be the image of $x$. Then Lemma~\ref{lem:graded_weight_filtration} below gives us a canonical increasing and complete filtration $\Fil_\bullet^{\mathrm{wt}}\mathbb{L}_{\mathcal{X}^\preoneb,x,\bullet}$ with
\[
\gr_{-i}^{\mathrm{wt}}\mathbb{L}_{\mathcal{X}^\preoneb,x,\bullet}\simeq C_\bullet(-i)\otimes_{C_0}\mathbb{L}_{\mathcal{X}^\preoneb,x_{0},i},
\]
and we have
\[
\Map_{C_\bullet}(\gr_{-i}^{\mathrm{wt}}\mathbb{L}_{\mathcal{X}^\preoneb,x,\bullet},I_\bullet)\simeq \Map_{\mathrm{GrMod}_{C_0}}(\mathbb{L}_{\mathcal{X}^\preoneb,x_{0},i},I_\bullet(i))\simeq \Map_{{C_0}}(\mathbb{L}_{\mathcal{X}^\preoneb,x_{0},i},I_i).
\]

The right hand side here is is non-trivial only when $\mathbb{L}_{\mathcal{X}^\preoneb,x_{0},i}$ is not nullhomotopic and when $i\leq 0$.

Now, if $x_0$ is in the image of $X^0(C_0)$, then it maps to a $1$-bounded fixed point, and so $\mathbb{L}_{\mathcal{X}^\preoneb,x_{0},i} \simeq 0$ for $i\le -2$. From this, we deduce that the fiber over $x$ depends only on the quotient $I_{-1}(1)\oplus I_0$ of $I_\bullet$. More precisely, the following square is Cartesian
\[
\begin{diagram}
\Map_{/\mathcal{Z}}((\Spec C'_\bullet)/\Gm,\mathcal{X})&\rTo&\Map_{/\mathcal{Z}}((\Spec C_\bullet)/\Gm,\mathcal{X})\\
\dTo&&\dTo\\
\Map_{/\mathcal{Z}}((\Spec C'_{\ge -1})/\Gm,\mathcal{X})&\rTo&\Map_{/\mathcal{Z}}((\Spec C_{\ge -1})/\Gm,\mathcal{X})
\end{diagram}
\]

Applying this with $C'_\bullet \to C_\bullet$ the map $\tau_{\leq (k+1)}B_\bullet\to \tau_{\leq k}B_\bullet$ for $k\ge 0$ and using the fact that $\mathcal{X}^\preoneb$ is a nilcomplete smooth sheaf reduces to the situation where $B_\bullet$ is a \emph{discrete} graded commutative ring, so that $\mathcal{Y}$ is now a classical stack. 

We can now complete the proof by applying the same reasoning again to the square-zero thickenings $B_{\ge -m}\to B_{\ge -m+1}$ for $m\ge 2$ and using graded integrability.
\end{proof}

\begin{remark}
\label{rem:use_of_graded_completeness}
Proposition~\ref{prop:graded_completeness} shows that $\mathcal{X}\to (\mathcal{Z},\iota)$ is graded integrable whenever $\mathcal{X}^\preoneb\to \mathcal{Z}$ has quasi-affine diagonal. 

However, graded integrability holds under weaker hypotheses. For instance, if $\mathcal{X}^\preoneb = \mathcal{P}\times B\Gm\to \mathcal{Z} = B\Gm$, where $\mathcal{P}$ is as in Example~\ref{ex:perfect_complexes_HTwts_01}, then, even though $\mathcal{P}$ does not have quasi-affine diagonal over $\Spec \Int$, we still know that the map
\[
\Map((\Spec B_\bullet)/\Gm,\mathcal{P}) = \mathrm{Perf}((\Spec B_\bullet)/\Gm)\to \varprojlim_m \mathrm{Perf}((\Spec B_{\ge -m})/\Gm) = \varprojlim_m\Map((\Spec B_{\ge -m})/\Gm,\mathcal{P}) 
\]
is an equivalence. In other words, the $1$-bounded stack $\mathcal{P}_{\{0,1\}}\to B\Gm$ from Example~\ref{ex:perfect_HT_wts_01_stack} is graded integrable.

Similarly, the $1$-bounded stack from Example~\ref{ex:1_bounded_perfect_complexes_stack} is also graded integrable.
\end{remark}

\subsection{A useful cartesian square}

\subsubsection{}
Suppose that $\Fil^\bullet S$ is a non-negatively filtered animated commutative ring, and set $\overline{S} = \gr^0S$. The map $S\to \overline{S}$ underlies an arrow $\Fil^\bullet S \to \Fil^\bullet_{\mathrm{triv}}\overline{S}$ of filtered animated commutative rings corresponding to a map of stacks
\[
\Aff^1/\Gm\times \Spec \overline{S}\to \Rees(\Fil^\bullet S)
\]
whose restriction over the open point of $\Aff^1/\Gm$ is the closed immersion $\Spec \overline{S}\to \Spec S$.

We will view $\mathcal{Y}\defn \Rees(\Fil^\bullet S)$ as a pointed graded stack via the composition
\[
B\Gm\times \Spec \overline{S}\to \Aff^1/\Gm\times \Spec \overline{S}\to \mathcal{Y}.
\]
\subsubsection{}
Let $\mathcal{X} = (\mathcal{X}^\preoneb,X^0)\to (\mathcal{Z},\iota)$ be a $1$-bounded stack over a pointed graded prestack with $\iota:B\Gm\times\Spec R\to \mathcal{Z}$, and let $X^-\to \Spec R$ be its associated attractor.

Suppose that we have a map of pointed prestacks $\mathcal{Y}\to \mathcal{Z}$. We now have a commutative diagram
\begin{align}\label{diagram:1bounded_comm_square}
\begin{diagram}
\Map_{/\mathcal{Z}}(\mathcal{Y},\mathcal{X})&\rTo&\Map_{/\mathcal{Z}}(\Spec S,\mathcal{X}^\preoneb)\\
\dTo&&\dTo\\
X^-(S)=\Map_{/\mathcal{Z}}(\Aff^1/\Gm\times\Spec \overline{S},\mathcal{X})&\rTo&\Map_{/\mathcal{Z}}(\Spec\overline{S},\mathcal{X}^\preoneb).
\end{diagram}
\end{align}

\begin{proposition}
\label{prop:1_bounded_cartesian}
Suppose that the kernel of the map
\[
\pi_0(S)\to \pi_0(\overline{S})
\]
is locally nilpotent and that $\mathcal{X}^\preoneb\to \mathcal{Z}$ is filtered integrable (Definition~\ref{defn:filtered_integrable}) and also satisfies the following additional condition: For any classical ring $R$ and any locally nilpotent ideal $I\subset R$, and for any map $\Spec R \to \mathcal{Z}$, we have
\[
  \mathcal{X}^\preoneb(R) = \varprojlim_m \mathcal{X}^\preoneb(R/I^m).
\] 
Then~\eqref{diagram:1bounded_comm_square} is a Cartesian square.
\end{proposition}
\begin{proof}
Let $\Fil^\bullet_{(S\twoheadrightarrow\overline{S})}S$ be the non-negatively filtered animated commutative ring with $\Fil^0_{(S\twoheadrightarrow\overline{S})}S\simeq S$ and 
\[
\gr^i_{(S\twoheadrightarrow\overline{S})}S\simeq \begin{cases}
\overline{S}&\text{if $i=0$};\\
0&\text{otherwise}.
\end{cases}
\]
The associated Rees algebra sits in a Cartesian square of graded animated commutative rings
\begin{align}\label{eqn:rees_cartesian_construction}
\begin{diagram}
\bigoplus_i\Fil^i_{(S\twoheadrightarrow\overline{S})}S\cdot t^{-i}&\rTo&S[t,t^{-1}]\\
\dTo&&\dTo\\
\overline{S}[t]&\rTo&\overline{S}[t,t^{-1}].
\end{diagram}
\end{align}
One can obtain this construction for instance by animating the obvious one for surjections of polynomial algebras over $\Int$.

For any non-negatively filtered animated commutative ring $\Fil^\bullet A$, view $\Rees(\Fil^\bullet A)$ as a pointed graded prestack with
\[
\iota:B\Gm\times \Spec \gr^0A \to \Rees(\Fil^\bullet A).
\]
Also, set
\[
\Fil^\bullet A' = \Fil^\bullet_{(A\twoheadrightarrow\gr^0(A))}A,
\]
where the right hand side is defined as above. If $\Rees(\Fil^\bullet A)$ is a pointed graded stack over $(\mathcal{Z},\iota)$, set
\[
\mathcal{X}(\Fil^\bullet A) \defn \Map_{/\mathcal{Z}}(\Rees(\Fil^\bullet A),\mathcal{X})
\]

The diagram~\eqref{diagram:1bounded_comm_square} is Cartesian with $\Rees(\Fil^\bullet S)$ replaced by $\Rees(\Fil^\bullet S')$. This is because of the Cartesian nature of~\eqref{eqn:rees_cartesian_construction}, and the \emph{surjectivity} of the vertical maps; see (4) of~\cite[Theorem 7.5.1]{lurie_thesis}.

Note that there is a natural map $\Rees(\Fil^\bullet S')\to \Rees(\Fil^\bullet S)$. To complete the proof of the proposition it now suffices to show that the corresponding map
\[
\mathcal{X}(\Fil^\bullet S)\to \mathcal{X}(\Fil^\bullet S')
\]
is an equivalence. We will prove this using filtered deformation theory. 

Suppose quite generally that $\Fil^\bullet B\to \Fil^\bullet A$ is a square-zero extension of non-negatively filtered animated commutative rings with fiber $\Fil^\bullet K$, and write $\Fil^\bullet K'$ for the fiber of the induced map $\Fil^\bullet B'\to \Fil^\bullet A'$. If $\Rees(\Fil^\bullet B)$ is a pointed graded stack over $(\mathcal{Z},\iota)$, by the discussion in~\S\ref{subsec:filtered_square-zero}, we obtain a Cartesian square
\[
\begin{diagram}
\mathcal{X}(\Fil^\bullet B)&\rTo&\mathcal{X}(\Fil^\bullet A)\times_{X^0(\gr^0A)}X^0(\gr^0B)\\
\dTo&&\dTo\\
\mathcal{X}(\Fil^\bullet A)\times_{X^0(\gr^0A)}X^0(\gr^0B)&\rTo&\mathcal{X}(\Fil^\bullet A\oplus\Fil^\bullet K[1])\times_{X^0(\gr^0A)}X^0(\gr^0B).
\end{diagram}
\]
Moreover, if $\mathbb{L}_{\mathcal{X}^\preoneb/\mathcal{Y}}$ is the relative cotangent complex, then for any $x\in \mathcal{X}(\Fil^\bullet A)$, we obtain a filtered $\Fil^\bullet A$-module $\Fil^\bullet \mathbb{L}_{\mathcal{X}^\preoneb/\mathcal{Y},x}$, and we have a canonical equivalence:
\begin{align*}
\mathrm{fib}_x(\mathcal{X}(\Fil^\bullet A\oplus \Fil^\bullet K[1])\to \mathcal{X}(\Fil^\bullet A))&\simeq\Map_{\Fil^\bullet A}(\Fil^\bullet \mathbb{L}_{\mathcal{X}^\preoneb/\mathcal{Y},x},\Fil^\bullet K[1]).
\end{align*}
Similarly, if $x'\in \mathcal{X}(\Fil^\bullet A')$ is the image of $x$, we have
\begin{align*}
\mathrm{fib}_{x'}(\mathcal{X}(\Fil^\bullet A'\oplus \Fil^\bullet K'[1])\to \mathcal{X}(\Fil^\bullet A'))&\simeq\Map_{\Fil^\bullet A}(\Fil^\bullet \mathbb{L}_{\mathcal{X}^\preoneb/\mathcal{Y},x},\Fil^\bullet K'[1]).
\end{align*}

We claim that the natural map between these fibers is an equivalence. For this, set 
\[
\Fil^\bullet J  = \hcoker(\Fil^\bullet K \to \Fil^\bullet K').
\]
The claim would follow if we knew that
\[
\Map_{\Fil^\bullet A}(\Fil^\bullet \mathbb{L}_{\mathcal{X}^\preoneb/\mathcal{Y},x},\Fil^\bullet J[1])\simeq 0.
\]
For this, first note that by construction we have $\Fil^i J \simeq 0$ for $i\leq 1$. Now, the $1$-bounded condition essentially ensures that $\Fil^\bullet \mathbb{L}_{\mathcal{X}^\preoneb/\mathcal{Y},x}$ is generated in filtered degrees $\leq 1$, and so the mapping space above vanishes as desired. 

To see this precisely, one can argue as follows. If $\Fil^\bullet \overline{\mathbb{L}}_{\mathcal{X}^\preoneb/\mathcal{Y},x}$ is the filtered base-change over $\Fil^\bullet_{\mathrm{triv}}\gr^0A$, then the $1$-bounded condition tells us that $\gr^i\overline{\mathbb{L}}_{\mathcal{X}^\preoneb/\mathcal{Y},x}\simeq 0$ for $i\ge 2$. By Lemma~\ref{lem:graded_weight_filtration}, $\gr^\bullet \mathbb{L}_{\mathcal{X}^\preoneb/\mathcal{Y},x}$ admits an increasing filtration with $i$-th graded piece isomorphic to
\[
\gr^i\overline{\mathbb{L}}_{\mathcal{X}^\preoneb/ \mathcal{Y},x}(i)\otimes_{\gr^0A}\gr^\bullet A,
\]
and, for each $i$ and each $j\ge 1$, we have
\[
\Map_{\gr^\bullet A}(\gr^i\overline{\mathbb{L}}_{\mathcal{X}^\preoneb/ \mathcal{Y},x}(i)\otimes_{\gr^0A}\gr^\bullet A,\gr^\bullet J[1](-j))\simeq \Map_{\gr^0A}(\gr^i\overline{\mathbb{L}}_{\mathcal{X}^\preoneb/ \mathcal{Y},x},\gr^{i+j}J[1])\simeq 0
\]
since either the source or the target of this mapping space has to be zero. Using this, one finds that 
\[
\Map_{\Fil^\bullet A}(\mathbb{L}_{\mathcal{X}^\preoneb/\mathcal{Y},x},\gr^\bullet J[1](-j))\simeq \Map_{\gr^\bullet A}(\gr^\bullet\mathbb{L}_{\mathcal{X}^\preoneb/\mathcal{Y},x},\gr^\bullet J[1](-j))\simeq 0.
\]
Here we are viewing graded modules over $\gr^\bullet A$ as filtered modules over $\Fil^\bullet A$ with trivial transition maps: Geometrically, this corresponds to pushforward along the closed immersion $\Rees(\Fil^\bullet A)_{(t=0)}\to \Rees(\Fil^\bullet A)$.

Using the cofiber sequence
\[
\Fil^\bullet J(-j+1)\to \Fil^\bullet J(-j)\to (\gr^\bullet J)(-j)
\]
now shows that we have
\[
\Map_{\Fil^\bullet A}(\Fil^\bullet \mathbb{L}_{\mathcal{X}^\preoneb/\mathcal{Y},x},\Fil^\bullet J[1])\simeq \Map_{\Fil^\bullet A}(\Fil^\bullet \mathbb{L}_{\mathcal{X}^\preoneb/\mathcal{Y},x},\Fil^\bullet J[1](-j))
\]
for all $j\ge 1$. Now, for every $m\in\Int$, the $m$-th filtered piece of $\Fil^\bullet J[1](-m+1)$ is zero. Therefore, we conclude that the mapping spaces in question are also zero.

The claim that we just verified shows that, if $\mathcal{X}(\Fil^\bullet A)\to \mathcal{X}(\Fil^\bullet A')$ is an isomorphism, then so is $\mathcal{X}(\Fil^\bullet B)\to \mathcal{X}(\Fil^\bullet B')$. We will use this principle repeatedly in what follows.

For every $k$, we have the truncated filtered animated commutative ring $\tau_{\leq k}(\Fil^\bullet S)$ obtained by taking the corresponding truncation for the associated Rees algebra: this is a square-zero extension of $\tau_{\leq(k-1)}\Fil^\bullet S$ by a filtered module $\pi_k(\Fil^\bullet S)[k]$. Via the deformation argument above, combined with induction on $k$ and nilcompleteness, one therefore reduces to showing that the map
\[
\mathcal{X}(\pi_0(\Fil^\bullet S))\to \mathcal{X}(\pi_0(\Fil^\bullet S'))
\]
is an isomorphism

So we can assume that $\Fil^\bullet S$ corresponds to a classical graded ring $\mathrm{Rees}(\Fil^\bullet S)$. Set $I = \mathrm{im}(\Fil^1S)\subset S$: this is locally nilpotent by hypothesis. For $m\ge 1$, let $\Fil^\bullet S_m = \Fil^\bullet S/I^m\Fil^\bullet S$ be the induced filtered structure on the ring $S_m = S/I^m$. We now claim that we have
\[
\mathcal{X}(\Fil^\bullet S) = \varprojlim_m \mathcal{X}(\Fil^\bullet S_m).
\]
Via smooth affine descent, this follows from our hypotheses: Indeed, they ensure that $\Map_{/\mathcal{Z}}(\Spec A,\mathcal{X}) = \varprojlim_m \Map_{/\mathcal{Z}}(\Spec A/J^m,\mathcal{X})$ for any discrete commutative ring $A$ with ideal $J\subset A$, complete for the $J$-adic topology, equipped with $\Spec A\to \mathcal{Z}$. We apply this with $A = \mathrm{Rees}(\Fil^\bullet S)\otimes_{\Int}\Int[t,t^{-1}]^{\otimes_{\Int} r}$ for $r\ge 0$ and $J\subset A$ the ideal generated by $I$, and appeal to descent to deduce the claimed equality.

Therefore, using the deformation argument again, we are reduced to the case $S = S_1$, where the map $\Fil^1S\to S$ is $0$. 

In the notation of~\eqref{subsubsec:filtered_completion}, we now have $\Fil^\bullet_{(1)}S = \Fil^\bullet_{\mathrm{triv}}S = \Fil^\bullet_{(1)}S'$, and $\Fil^\bullet_{(m)}S' = \Fil^\bullet S'$ for $m\ge 2$. Therefore, the proposition holds with $\Fil^\bullet S$ replaced with $\Fil^\bullet_{(m)}S$ for any $m\ge 1$. We now conclude using filtered integrability.
\end{proof}

\begin{remark}
The above result can be viewed as a generalization of~\cite[Remark 6.3.3]{MR4355476}.
\end{remark}

\begin{definition}
\label{defn:strongly_integrable}
Given the above result and Proposition~\ref{prop:1bounded_graded_deformation}, we will find the following notion useful: A $1$-bounded stack $\mathcal{X}$ over a pointed graded stack $(\mathcal{Z},\iota)$ is \defnword{strongly integrable} if it is filtered integrable, graded integrable, and also satisfies the following condition: For any classical ring $R$ and any locally nilpotent ideal $I\subset R$, and for any map $\Spec R \to \mathcal{Z}$, we have
\[
  \mathcal{X}^\preoneb(R) = \varprojlim_m \mathcal{X}^\preoneb(R/I^m).
\]
\end{definition}  

\begin{remark}
\label{rem:use_of_filtered_completeness}
Proposition~\ref{prop:filtered_completeness} tells us that $\mathcal{X}$ is filtered integrable whenever $\mathcal{X}^\preoneb\to \mathcal{Z}$ has quasi-affine diagonal. Moreover, by Noetherian approximation, and Theorem 1.5 from~\cite{BHATT2017576}, it is in fact strongly integrable: Indeed, we already observed in Remark~\ref{rem:use_of_graded_completeness} that it is graded integrable.

These conditions are also valid more generally: For instance, if $\mathcal{X}= \mathcal{P}_{\{0,1\}}\to B\Gm$ with $\mathcal{X}^\preoneb = \mathcal{P}\times B\Gm$ (notation as in Example~\ref{ex:1_bounded_perfect_complexes_stack}), then we still know that the map
\[
\Map(\Rees(\Fil^\bullet S),\mathcal{P}) = \varprojlim_m\Map(\Rees(\Fil^\bullet_{(m)}S),\mathcal{P}) 
\]
is an equivalence. Moreover, for any $I$-adically complete ring $R$, we also have
\[
\mathcal{P}(R) = \mathrm{Perf}(R) = \varprojlim_m \mathrm{Perf}(R/I^m) = \varprojlim_m \mathcal{P}(R/I^m).
\]
See~\cite[Lemma 8.2]{BHATT2017576}. Therefore, combined with Remark~\ref{rem:use_of_graded_completeness}, we find that $\mathcal{X}$ is strongly integrable.

Similarly, the $1$-bounded stack from Example~\ref{ex:sections_1_bounded_perfect_complexes} is also strongly integrable.
\end{remark}

\section{Animated higher frames and windows}
\label{sec:higher_frames}

The purpose of this section is to give an account of the theory of~\cite{MR4355476} in an animated context, and to use it to prove two crucial technical results, Propositions~\ref{prop:def_theory_frames} and~\ref{prop:abstract_devissage_to_fzips}. They play an essential role in the establishment of the representability theorems of Section~\ref{sec:abstract}.

\subsection{Generalized Cartier divisors}

\begin{definition}
Recall that a \defnword{generalized Cartier divisor} for an animated commutative ring $R$ is a surjective map $R\twoheadrightarrow\overline{R}$ whose homotopy kernel $I$ is an invertible $R$-module. By abuse of notation we will refer to such an object via the cosection $s:I\to R$, which is the same as a map $s:\Spec R\to \Aff^1/\Gm$.
\end{definition}

\subsubsection{}
Any generalized Cartier divisor lifts $R$ to a filtered animated commutative ring $\Fil^\bullet_I R$ where the filtration is the $I$-adic one given by
\[
\Fil^k_I R = \begin{cases}
I^{\otimes k}&\text{if $k\ge 0$}\\
R&\text{if $k<0$},
\end{cases}
\]
and the transition maps are the identity for $k\leq 0$ and given by 
\[
I^{\otimes k} \simeq I\otimes_RI^{\otimes(k-1)}\xrightarrow{s\otimes 1}R\otimes_RI^{\otimes(k-1)}\simeq I^{\otimes (k-1)}
\]
for $k>0$. We will also have occasion to consider the \defnword{two-sided $I$-adic filtration} given by $\Fil^k_{I,\pm} R = I^{\otimes k}$ for all $k\in \Int$, which once again underlies a filtered animated commutative ring with
\[
\Rees(\Fil^\bullet_{I,\pm}R)\simeq \Spec R.
\]

To verify the assertions in the previous paragraph, using the classifying map $s:\Spec R\to \Aff^1/\Gm$, one reduces everything to the case where $R = \Int[x]$ with $I=x\Int[x]$, and here everything can be checked explicitly.

\begin{remark}
\label{rem:rees_stack_points}
This also gives a concrete way of thinking of a point $\Spec R\to \Rees(\Fil^\bullet S)$ of the Rees stack corresponding to a filtered commutative ring $\Fil^\bullet S$: it is equivalent to giving a generalized Cartier divisor $I\to R$, along with a map of filtered animated commutative rings $\Fil^\bullet S\to \Fil^\bullet_{I}R$.
\end{remark}

\subsubsection{}
For any $M\in \Mod{R}$, we will set $M[I^{-1}] = \colim_{k\ge 0} I^{-k}\otimes_RM$, where the transition maps are induced by $s$.

When we have an isomorphism $R\xrightarrow{\simeq}I$ of $R$-modules given by a section $\xi$ of $I$, we will write $\Fil^\bullet_\xi R$ and $\Fil^\bullet_{\xi,\pm}R$ for these filtered rings.

For any $R$-module $M$, we will write $M/{}^{\mathbb{L}}(p,I)$ for $M/{}^{\mathbb{L}}p\otimes_R\overline{R}$.

If $R'\in \mathrm{CRing}_{R/}$ is an $R$-algebra, and $s':I' = R'\otimes_RI\xrightarrow{1\otimes s}R'$, then we will sometimes also denote the $I'$-adic filtrations on $R'$ by $\Fil^\bullet_I R'$ and $\Fil^\bullet_{I,\pm}R'$.

\subsection{Formal Rees stacks}
\label{subsec:formal_rees}

Suppose that $I\to A$ is a generalized Cartier divisor with $p$-complete quotient $\overline{A}$, and suppose that $A$ underlies a filtered animated commutative ring $\Fil^\bullet A$.

\begin{definition}
The \defnword{formal Rees stack} associated with this datum is the one associating with each $R\in \mathrm{CRing}^{p\text{-nilp}}$ the space of generalized Cartier divisors $J\to R$ along with maps $\Fil^\bullet A\to \Fil^\bullet_{J,\pm}R$ (see Remark~\ref{rem:rees_stack_points}) such that the underlying map $A\to R$ is in $\Spf(A,I)(R)$. 
\end{definition}

In the sequel, the Rees construction will only be appealed to in this formal context. Therefore, by abuse of notation, we will denote this formal stack once again by $\Rees(\Fil^\bullet A)$. 

In particular, we have $\Rees(\Fil^\bullet_{I,\pm}A) \simeq \Spf(A,I)$. 

\subsection{Witt vectors, $\delta$-rings and prisms}

We recall the notion of an animated $\delta$-ring from~\cite[App. A]{bhatt2022prismatization}. First, we note:

\begin{remark}
\label{rem:frobenius_mod_p}
Every $R\in \mathrm{CRing}_{\Field_p/}$ is equipped with a canonical Frobenius endomorphism $\varphi:R\to R$ obtained by animating the usual Frobenius endomorphism for polynomial algebras over $\Field_p$. In particular, for any $R\in \mathrm{CRing}$, we obtain a canonical map
\[
\overline{\varphi}:R\xrightarrow{\mathrm{can}}R/{}^{\mathbb{L}}p\xrightarrow{\varphi}R/{}^{\mathbb{L}}p,
\]
where $\mathrm{can}:R\to R/{}^{\mathbb{L}}p$ is the canonical surjection.
\end{remark}

\subsubsection{}
Now, one defines for any animated commutative $\Int_{(p)}$-algebra $R$ the $2$-truncated ($p$-typical) Witt ring $W_2(R)$ with underlying space $R^2$ such that the projection onto the first coordinate is a map of animated commutative rings $W_2(R)\twoheadrightarrow R$. This amounts to the observation that the functor $C\mapsto W_2(C)$ on discrete commutative rings is represented by a smooth ring scheme, and so extends canonically to an animated ring scheme equipped with a map to $\Ga$. There exists a canonical Cartesian square of animated commutative rings
\begin{align}\label{eqn:w2_cart_square}
\begin{diagram}
W_2(R)&\rTo&R\\
\dTo&&\dTo_{\mathrm{can}}\\
R&\rTo_{\overline{\varphi}}&R/{}^{\mathbb{L}}p.
\end{diagram} 
\end{align}

\begin{definition}
A \defnword{$\delta$-structure} on an animated $\Int_{(p)}$-algebra $R$ is a section of the natural map $W_2(R)\to R$.

An (animated) \defnword{$\delta$-ring} is an animated $\Int_{(p)}$-algebra $R$ equipped with a $\delta$-structure. We obtain an $\infty$-category $\mathrm{CRing}_\delta$ of animated $\delta$-rings in the usual fashion. 
\end{definition}

\begin{remark}
Via projection onto the second coordinate of $W_2(R)$, such a section yields an operator $\delta:R\to R$ satisfying certain properties, which, when $R$ is discrete, completely determines the $\delta$-structure.
\end{remark}

\begin{remark}
The Cartesian square~\eqref{eqn:w2_cart_square} shows that, giving a $\delta$ structure on $R$ is equivalent to giving a lift $\varphi:R\to R$ of the map $\overline{\varphi}:R\to R/{}^{\mathbb{L}}p$. In particular, when $R$ is flat over $\Int_{(p)}$, a $\delta$-structure is equivalent to giving an endomorphism $\varphi:R\to R$ lifting the $p$-power Frobenius endomorphism of $R/pR$.
\end{remark}

\begin{remark}
\label{rem:naive_frobenius_lift}
Note that, if $R$ is an $\Field_p$-algebra, then the natural map $R\to R/{}^{\mathbb{L}}p$ is \emph{not} a morphism in $\mathrm{CRing}_{\Field_p/}$, and is in fact \emph{not} compatible with the natural Frobenius endomorphisms on source and target. Therefore, the Frobenius endomorphism of $R$ does not underlie a prism structure on $R$.
\end{remark}

\begin{remark}
\label{rem:witt_postnikov_truncation}
If $R$ is $k$-truncated, so is $W_2(R)$. This shows that, for any $\delta$-ring $R$, the composition $R\to W_2(R)\to W_2(\tau_{\leq k}R)$ factors through $\tau_{\leq k}R$. In other words, $\tau_{\leq k}R$ inherits a canonical $\delta$-ring structure from $R$.
\end{remark}

\begin{remark}
The forgetful functor from $\mathrm{CRing}_\delta\to \mathrm{CRing}$ admits both a left and right adjoint: The former is the free $\delta$-ring functor and the latter is the Witt functor $A\mapsto W(A)$.
\end{remark}

\begin{lemma}
\label{lem:extracting_phi_roots}
 Let $\Int_{(p)}\{x\}$ be the free $\delta$-ring obtained by hitting $\Int_{(p)}[x]$ with the left adjoint. Then the associated Frobenius lift $\Int_{(p)}\{x\}\xrightarrow{x\mapsto \varphi(x)} \Int_{(p)}\{x\}$ is faithfully flat.
\end{lemma}
\begin{proof}
See~\cite[Lemma 2.11]{Bhatt2022-ee}.
\end{proof}

\begin{definition}
\label{defn:define_prisms}
Following~\cite[Def. 2.4]{bhatt2022prismatization}, we define a(n animated) \defnword{prism} to be an animated $\delta$-ring $A$ equipped with a generalized Cartier divisor $I\to A$ with quotient $\overline{A}$ such that the following conditions hold:
\begin{enumerate}
   \item $A$ is $(p,I)$-complete.
   \item Given a perfect field $k$ of characteristic $p$ and a map $A\to W(k)$ of $\delta$-rings, we have $W(k)\otimes_A\overline{A}\simeq k$.
\end{enumerate}
If we want to emphasize the Cartier divisor, we will sometimes also denote the prism by $(A,s:I\to A)$.
\end{definition}

\begin{definition}
A prism $(A,I)$ is \defnword{transversal} if $A$ is flat over $\Int_p$ and the map $I\to A$ is injective mod-$p$.
\end{definition}

\begin{definition}
A prism $(A,I)$ is \defnword{perfect} if the Frobenius lift $\varphi:A\to A$ is an isomorphism.
\end{definition}

\begin{remark}
If $(A,I)$ is a perfect prism, $A/{}^{\mathbb{L}}p$ is a perfect $\Field_p$-algebra and so is discrete; this implies that $A$ is also discrete and $p$-torsion free. It is a result of Bhatt-Scholze~\cite[Theorem 3.10]{Bhatt2022-ee} that the assignment $(A,I)\mapsto \overline{A}$ is an equivalence of categories between perfect prisms and perfectoid rings. The inverse carries a perfectoid ring $R$ to the perfect prism $(A_{\mathrm{inf}}(R),\ker \theta)$ where $A_{\mathrm{inf}}(R) = W(R^\flat)$ with $R^\flat$ the tilt of $R$ and $\theta:A_{\mathrm{inf}}(R)\to R$ is the usual map.
\end{remark}

\subsubsection{}
\label{subsubsec:abstract_bk_twist}
Associated with any prism $(A,I)$ is a canonical invertible module $A\{1\}$ over $A$ constructed in~\cite[Proposition 2.5.1]{bhatt2022absolute} and characterized by the following properties:
\begin{enumerate}
    \item (\cite[(2.2.11)]{bhatt2022absolute}) For any \emph{transversal} prism $(A,I)$, we have
   \[
     A\{1\} = \varprojlim_k I_k/I_{k+1},
   \]
   where $I_k = I\varphi_A^*(I)\cdots(\varphi_A^{k-1})^*(I)$, and the maps $I_k/I_{k+1}\to I_{k-1}/I_k$ are induced by dividing the natural maps by $p$.
   \item (\cite[(2.5.3)]{bhatt2022absolute}) For any map of prisms $(A,I)\to (B,J)$, there is a canonical isomorphism $B\otimes_AA\{1\}\xrightarrow{\simeq}B\{1\}$.
\end{enumerate}

By~\cite[Remark 2.5.9]{bhatt2022absolute}, we have a canonical isomorphism $I\otimes_A\varphi^*(A\{1\})\xrightarrow{\simeq}A\{1\}$.

For any $M\in \Mod{A}$ and $i\in\Int$, we will set $M\{i\} \defn M\otimes_A A\{1\}^{\otimes i}$.

\subsection{Animated higher frames}
\label{subsec:animated_frames}

\begin{definition}
A(n animated higher) \defnword{frame} $\underline{A}$ is a tuple $(\Fil^\bullet A, I\xrightarrow{s}A,\Phi,A\{1\})$, where:
\begin{enumerate}
   \item $\Fil^\bullet A$ is a non-negatively filtered $(p,I)$-complete animated commutative ring;
   \item $s:I\to A$ is a generalized Cartier divisor;
   \item $\Phi:\Fil^\bullet A\to \Fil^\bullet_I A$ is a map of filtered animated commutative rings such that the underlying endomorphism $\varphi:A\to A$ of animated commutative rings is a `na\"ive' Frobenius lift in the sense that the induced endomorphism of $\pi_0(A)/p\pi_0(A)$ is Frobenius;
   \item $A\{1\}$ is an invertible $A$-module equipped with an isomorphism\footnote{The main role of this `abstract' Breuil-Kisin twist is in the interpretation of the $(G,\mu)$-windows appearing in \S~\ref{subsec:G_mu_windows} below in terms of the general definitions of \S~\ref{subsec:abstract_def_theory} (see Remark~\ref{rem:B'Gmu_abstract}). As such, it can be ignored for now. In cases of interest, this twist will either be trivial or be determined by a prism structure on $(A,I)$.}
   \[
    I\otimes_A\varphi^*A\{1\}\xrightarrow{\simeq}A\{1\}.
   \]
\end{enumerate}
Frames organize into an $\infty$-category in a natural way. The morphisms 
\[
\underline{A}\to \underline{A}' =  (\Fil^\bullet A', I'\xrightarrow{s'}A',\Phi',A'\{1\})
\]
are maps $f:\Fil^\bullet A\to \Fil^\bullet B$ of filtered animated commutative rings equipped with isomorphisms $A'\otimes_A(I\xrightarrow{s}A)\simeq (I'\xrightarrow{s'}A')$ and $A'\otimes_AA\{1\}\xrightarrow{\simeq}A'\{1\}$, along with a commuting diagram
\[
\begin{diagram}
\Fil^\bullet A&\rTo^\Phi&\Fil^\bullet_IA\\
\dTo&&\dTo\\
\Fil^\bullet A'&\rTo_{\Phi'}&\Fil^\bullet_{I'}A'\simeq A'\otimes_A\Fil^\bullet_IA.
\end{diagram}
\]
\end{definition}

\subsubsection{}
Given a frame $\underline{A}$, we will write $R_A$ for the $p$-complete animated commutative ring $\gr^0A$ and $\overline{A}$ for the quotient of $I\xrightarrow{s}A$.

Let $\Fil^\bullet_{I,\pm}A$ be the two-sided $I$-adic filtration on $A$; then we obtain a map
\[
\Phi_{\pm}:\Fil^\bullet A \to \Fil^\bullet_{I,\pm} A,
\]
which restricts to $\Phi$ in non-negative degrees, and which in filtered degree $-i$ (for $i\in \Int_{>0}$) is given by $s^{-i}\circ\varphi$.

Write $\varphi_i:\Fil^iA \to I^{\otimes i}$ for the filtered degree-$i$ component of $\Phi$.

For any $M\in\Mod{A}$ and $i\in\Int$, set $M\{i\} = M\otimes_AA\{1\}^{\otimes i}$.

\begin{definition}
\label{defn:p-adic_frames}
If $(I \xrightarrow{s} A) = (A\xrightarrow{p} A)$ with $A\{1\}\simeq A$ we will say that $\underline{A}$ is a \defnword{$p$-adic frame}. 

Since $A\{1\}$ and $I\to A$ are superfluous here, we will denote a $p$-adic frame by a tuple $(A,\Fil^\bullet A,\Phi)$.
\end{definition}

\begin{remark}
\label{rem:relationship_with_lau}
Our definition of an animated frame is inspired by the definition of Lau in~\cite[\S 2]{MR4355476}, but allows for more general objects. In fact, our notion of a $p$-adic frame is quite close to being the natural animated generalization of Lau's definition. 

Indeed, suppose that we have a $p$-adic frame $\underline{A}$ such that $A$ is a discrete $\Int_p$-algebra, and such that $\Fil^i A$ is also discrete for all $i\ge 0$. Then we obtain the graded Rees ring
\[
 S(\underline{A}) \defn \mathrm{Rees}(\Fil^\bullet A) = \bigoplus_{i=0}\Fil^iA\cdot t^{-i}
\]
along with maps $\tau,\sigma:S(\underline{A}) \to A$. Here, $\tau$ is given by 
\[
\tau:S(\underline{A})\to S(\underline{A})/(t-1)\simeq A
\]
and $\sigma$ as the composition
\[
S(\underline{A}) \xrightarrow{\mathrm{Rees}(\Phi_{\pm})} \mathrm{Rees}(\Fil^\bullet_{p,\pm}A)\to \mathrm{Rees}(\Fil^\bullet_{p,\pm}A)/(t-p)\simeq A.
\]
The triple $(S(\underline{A}),\sigma,\tau)$ is---after a sign change in the graded degrees---a (higher) frame as defined by Lau~\cite[Definition 2.0.1]{MR4355476} 
\end{remark}

\begin{remark}
\label{rem:relationship_with_lau_reverse}
One can somewhat reverse this process: Suppose that $(S,\sigma,\tau)$ is a higher frame in the sense of Lau. Then one obtains an animated filtered commutative ring $\Fil^\bullet S_0$ with underlying ring $S_0$ and filtration given by $\Fil^i S_0 = S_i\xrightarrow{\tau_i} S_0$. Note that this is in general \emph{not} a filtered ring in the classical sense, since the $S_0$-modules $\Fil^i S_0$ are not necessarily ideals in $S_0$. This is the case for the truncated Witt frame from~\cite[Example 2.1.6]{MR4355476}.

As explained in~\cite[Remark 2.0.2]{MR4355476}, $\sigma_0:S_0\to S_0$ is a lift of the mod-$p$ Frobenius on $S/pS$, and $\sigma_i$ gives a map $\Fil^iS_0 \to S_0$ such that $p\sigma_{i+1} = \sigma_i\vert_{\Fil^{i+1}S_0}$ for all $i\ge 0$. This means precisely that these maps organize into a map $\Phi:\Fil^\bullet S_0\to \Fil^\bullet_pS_0$ of filtered animated commutative rings. 

If we now assume that each $S_i$ is derived $p$-complete, then we have recovered our notion of a $p$-adic frame. 
\end{remark}

\begin{remark}
\label{rem:reason_for_I}
As noted in the previous remarks, our definition is a closely related to that of Lau if we restrict to $p$-adic frames. To motivate our more general definition, we need to look ahead to Lemma~\ref{lem:nygaard_filtered_frame} and Theorem~\ref{thm:semiperf_crys} below. These results show that the Nygaard filtered prismatization and syntomification of a semiperfectoid ring can be described in terms of an animated higher frame---as defined here---obtained from its Nygaard filtered prismatic cohomology. Thus, we will be able to apply the theory from this section to study objects living over the stacks that will appear in Section~\ref{sec:bld_stacks}, and we do so in Section~\ref{sec:abstract}.
\end{remark}

\begin{definition}
\label{defn:prismatic_frames}
A frame $\underline{A}$ is \defnword{prismatic} if the following conditions hold:
\begin{itemize}
   \item The pair $(A,I)$ is a prism.
   \item The endomorphism $\varphi:A\to A$ is the one obtained from the underlying $\delta$-ring structure on $A$. In particular, it is a lift of the Frobenius endomorphism of $A/{}^{\mathbb{L}}p$.
   \item The invertible module $A\{1\}$ and the datum of the isomorphism $I\otimes_A\varphi^*A\{1\}\simeq A\{1\}$ are the canonical ones from~\eqref{subsubsec:abstract_bk_twist}.
\end{itemize}
Since the datum of $A\{1\}$ is superfluous here, we will denote a prismatic prism by the tuple $(A,I\to A,\Fil^\bullet A,\Phi)$. Note that any $p$-adic frame whose na\"ive Frobenius lift underlies a $\delta$-ring structure is automatically prismatic.
\end{definition}

\begin{example}
[The Witt frames]
\label{example:witt_frame}
Suppose that we have $R\in \mathrm{CRing}^{p\text{-comp}}$. If $R$ is discrete, since $W(R)$ is a derived $p$-complete $\delta$-ring, putting Remark~\ref{rem:relationship_with_lau_reverse} together with~\cite[Example 2.1.3]{MR4355476} gives us the \defnword{Witt frame} $\underline{W(R)}$ associated with $R$. More generally, by applying this to the $p$-completed ring of functions of the Witt scheme $W$, we see that $W(R)$ underlies a $p$-adic frame $\underline{W(R)}$ for any $p$-complete animated commutative ring $R$. It is a $p$-adic frame and the filtration $\Fil^\bullet_{\mathrm{Lau}}W(R)$ is given by 
\[
\Fil^i_{\mathrm{Lau}}W(R) = \begin{cases}
W(R)&\text{if $i\leq 0$};\\
F_*W(R)&\text{if $i \ge 1$},
\end{cases}
\]
with transition map $\Fil^1_{\mathrm{Lau}}W(R) =F_*W(R)\to W(R)$ given by the Verschiebung map, and that in higher degrees given by $F_*W(R)\xrightarrow{p}F_*W(R)$. The filtered Frobenius lift is obtained in filtered degree $0$ via the fact that we have $FV = p$, and is the identity in filtered degrees $i\ge 1$.

Note that this frame is prismatic.
\end{example}

\begin{example}
[The truncated Witt frames]
\label{example:truncated_witt_frames}
When $R$ is an $\Field_p$-algebra, in~\cite[Example 2.1.6]{MR4355476}, Lau also defines the \emph{truncated} Witt frames with underlying ring $W_n(R)$ for $n\ge 1$. In particular, when $n=1$, we obtain the \defnword{zip frame} from Example 2.1.7 of \emph{loc. cit.}. Similarly to the previous example, Remark~\ref{rem:relationship_with_lau_reverse} now also gives us $p$-adic frames $\underline{W_n(R)}$; however, they will no longer be prismatic. Explicitly, the underlying filtered animated commutative ring at level $n$ is $\Fil^\bullet_{\mathrm{Lau}}W_n(R)$ with 
\[
\Fil^i_{\mathrm{Lau}}W_n(R) = \begin{cases}
W_n(R)&\text{if $i\leq 0$};\\
F_*W_n(R)&\text{if $i \ge 1$},
\end{cases}
\]
and once again the transition maps are given by Verschiebung in degree $1$ and by multiplication by $p$ in higher degrees. The filtered Frobenius lift is given by $F$ in degree $0$ and the identity in higher degrees.
\end{example}

The Witt frames have some useful universal properties. To explain them we need to introduce the following notion.

\begin{definition}[Laminations]
\label{defn:laminated}
A \defnword{lamination} for a prismatic frame $\underline{A}$ is the provision of an isomorphism of generalized Cartier divisors\footnote{These are in fact Cartier-Witt divisors for $R_A$; see~\cite[Example 2.11]{bhatt2022prismatization} and also \S~\ref{subsec:cartier-witt} below.}
\[
 (I\otimes_AW(R_A)\to W(R_A))\xrightarrow{\simeq}(W(R_A)\xrightarrow{p}W(R_A))
\]
and a lift of the resulting isomorphism 
\[
\overline{A}\otimes_AW(R_A)=W(R_A)/{}^{\mathbb{L}}(I\otimes_AW(R_A))\xrightarrow{\simeq}W(R_A)/{}^{\mathbb{L}}p 
\]
to an isomorphism of $R_A$-algebras. Here, the $R_A$-algebra structure on the left is via the composition
\[
   R_A=\gr^0A \xrightarrow{\overline{\varphi} = \gr^0\Phi}\gr^0_IA = \overline{A}\to \overline{A}\otimes_AW(R_A)
\]
while that on the right is obtained from the map $R_A\to F_*\overline{W(R_A)}$ induced by the Frobenius lift $F:W(R_A)\to F_*W(R_A)$. 

A frame $\underline{A}$ is \defnword{laminated} when it is equipped with a lamination.
\end{definition}

\begin{lemma}
\label{lem:laminated_frames}
The following data are equivalent for a prismatic frame $\underline{A}$:
\begin{enumerate}
   \item A map of frames $\underline{A}\to \underline{W(R_A)}$ extending the canonical map $\lambda_A:A\to W(R_A)$.
   \item A lamination for $\underline{A}$.
\end{enumerate}
\end{lemma}
\begin{proof}
   The direction (1)$\Rightarrow$(2) is straightforward. The other direction uses the following description of the filtration underlying the Witt frame $\underline{W(R)}$: We have a Cartesian square of filtered animated commutative rings
 \[
  \begin{diagram}
  \Fil^\bullet_{\mathrm{Lau}}W(R)&\rTo&\Fil^\bullet_{\mathrm{triv}}R\\
  \dTo&&\dTo\\
  \Fil^\bullet_pF_*W(R)&\rTo&\Fil^\bullet_{\mathrm{triv}}F_*\overline{W(R)},
  \end{diagram}  
 \]
 where the right vertical arrow is obtained from the map $R\to F_*\overline{W(R)}$ induced by the Frobenius lift $F:W(R)\to F_*W(R)$. 

 If $\underline{A}$ is prismatic, then the $\delta$-ring structure on $A$ yields a map $\lambda_A:A\to W(R_A)$. A lamination for $\underline{A}$ extends this to a map of filtered animated commutative rings
 \[
    \Fil^\bullet_I\lambda_A: \Fil^\bullet_IA \to \Fil^\bullet_pW(R_A).
 \]
There is now a canonical map $\Fil^\bullet A\to \Fil^\bullet_{\mathrm{Lau}}W(R_A)$ induced by
 \[
   \Fil^\bullet A \xrightarrow{(\Phi,\mathrm{can})}\varphi_*\Fil^\bullet_IA \times_{\Fil^\bullet_{\mathrm{triv}}\varphi_*\overline{A}} \Fil^\bullet_{\mathrm{triv}}R_A \xrightarrow{(\varphi_*\Fil^\bullet_I\lambda_A,\mathrm{id})} \Fil^\bullet_pF_*W(R_A)\times_{\Fil^\bullet_{\mathrm{triv}}F_*\overline{W(R_A)}} \Fil^\bullet_{\mathrm{triv}}R_A.
 \]
 Note that for the definition of the second map we are crucially using the fact that $\varphi_*\overline{A}\to F_*\overline{W(R_A)}$ is a map of $R_A$-algebras. It is not difficult to see now that this underlies the map of frames whose existence is asserted in (1). 
\end{proof}

\begin{remark}
   An earlier version of this paper erroneously omitted the condition that the isomorphism of quotient rings be one of $R_A$-algebras. In particular, we had asserted the existence of a canonical map of frames $\underline{A}\to \underline{W(R_A)}$ for any $p$-adic prismatic frame $\underline{A}$. Here is an example communicated to us by Eike Lau that shows that this is incorrect: Let $R=\Field_p[\epsilon]$ be the ring of dual numbers over $\Field_p$, and let $h:R\to R$ be the automorphism given by $\epsilon\mapsto u\epsilon$ for $u\ne 1\in \Field_p^\times$. Then one can define a frame $\underline{A}$ with the same underlying prism as $\underline{W(R)}$ but where the map $\Phi:\Fil^\bullet_{\mathrm{Lau}}W(R)\to \Fil^\bullet_pF_*W(R)$ is replaced with $\Phi' = \Fil^\bullet_pW(h)\circ \Phi$. In this case, $R_A=R$ and the map $\lambda_A$ is simply the identity, but it cannot be extended to a map of frames $\underline{A}\to \underline{W(R)}$. Indeed, if it could be, then we would have $\Phi'=\Phi$. 

   The issue is that the natural identification $\varphi_*\overline{A} = F_*\overline{W(R)}$ is not an isomorphism of $R$-algebras. To see this, note that $\overline{W(R)}$ is a trivial square-zero extension of $R$ by $\Ga^\sharp(R)[1]$, and the map $R\to F_*\overline{W(R)}$ is classified (after shifting) by a map of complexes $ R\simeq \mathbb{L}_{R/\Field_p}[-1]\to \Ga^\sharp(R)\subset W(R)$ carrying $1$ to $[\epsilon]$, while the map $R\to \varphi_*\overline{A}$ corresponds to the map carrying $1$ to $[u\epsilon]$.
\end{remark}

There is however the following result.
\begin{lemma}
\label{lem:witt_frame_univ_property}
Suppose that $\underline{A}$ is a $p$-adic frame with $R_A$ an $\Field_p$-algebra and that $\varphi:A\to A$ is a lift of the Frobenius endomorphism of $R_A$\footnote{Note that this holds if either  $\underline{A}$ is prismatic, or if $A\to R_A$ is a map of $\Field_p$-algebras and $\varphi$ is the Frobenius endomorphism of $A$.}. Then there is a canonical map of frames $\underline{A} \to \underline{W_1(R_A)}$.
\end{lemma}
\begin{proof}
 When $R$ is an $\Field_p$-algebra, we have a Cartesian square of filtered animated commutative $R$-algebras
 \[
  \begin{diagram}
  \Fil^\bullet_{\mathrm{Lau}}W_1(R)&\rTo&\Fil^\bullet_{\mathrm{triv}}R\\
  \dTo&&\dTo\\
  \Fil^\bullet_p\varphi_*R&\rTo&\Fil^\bullet_{\mathrm{triv}}\varphi_*R,
  \end{diagram}  
 \]
 where the right vertical arrow is induced by the Frobenius endomorphism of $R$ and the bottom horizontal arrow is obtained via base-change from $\Field_p$ of the map of filtered animated commutative rings corresponding to the map of graded $\Field_p[t]$-algebras
 \[
   \Field_p[t,u]/(ut)\xrightarrow{u\mapsto 0}\Field_p[t].
 \] 

 Now, the map $\Fil^\bullet A\to \Fil^\bullet_{\mathrm{Lau}}W_1(R_A)$ underlying the map of frames $\underline{A}\to \underline{W_1(R_A)}$ is given by
 \[
   \Fil^\bullet A \xrightarrow{(\Phi,\mathrm{can})}\Fil^\bullet_p\varphi_*A \times_{\Fil^\bullet_{\mathrm{triv}}\varphi_*R_A} \Fil^\bullet_{\mathrm{triv}}R_A \xrightarrow{(\varphi_* \mathrm{can},\mathrm{id})} \Fil^\bullet_p\varphi_*R_A\times_{\Fil^\bullet_{\mathrm{triv}}\varphi_*R_A} \Fil^\bullet_{\mathrm{triv}}R_A.
 \]
\end{proof}

\begin{example}
[Breuil-Kisin frames]
\label{example:bk_frames}
Take $(A,I')$ to be a transversal prism, so that $I'\subset A$ is an ideal with $A/I'$ (and hence) $A$ $p$-completely flat, and let $I = \varphi(I')\subset A$. Taking $\Fil^iA = \Fil^i_{I'}A$ completes $(A,\varphi)$ to a prismatic frame $\underline{A}$. When $A = W(k)\pow{u}$ for a perfect field $k$ with $\varphi(u) = u^p$, and $I' = (E(u))$ is generated by an Eisenstein ideal, this appears in the classical Breuil-Kisin theory. 

The argument from~\cite[Proposition 3.6.6]{bhatt2022absolute} shows that $\underline{A}$ is canonically laminated: More precisely, $I'\otimes_AW(R_A)\to W(R_A)$ is a Cartier-Witt divisor in the Hodge-Tate locus of the Cartier-Witt stack (we will encounter this in this paper in \S~\ref{subsec:hodge-tate_locus}), and so its Frobenius twist $I\otimes_AW(R_A)\to W(R_A)$ is canonically isomorphic to $W(R_A)\xrightarrow{p}W(R_A)$. Moreover, since the map $A\to W(R_A)$ is a map of $\delta$-rings, one sees that the map $R_A\to \overline{W(R_A)}$ is the canonical one.

More generally, given any prismatic frame $\underline{A}$, we will say that it is of \defnword{Breuil-Kisin type} if the filtration $\Fil^\bullet A$ is of the form $\Fil^\bullet_{I'}A$, where $(A,I')$ is another prism structure on $A$ equipped with an isomorphism $\varphi^*I'\xrightarrow{\simeq}I$, and the map $\Phi$ is the associated map of filtered animated commutative rings. The argument from the previous paragraph shows that $\underline{A}$ also admits a lamination. See also Remark~\ref{rem:bk_filtered_prisms} and Lemma~\ref{lem:filtered_prisms_laminated} below.
\end{example}

\begin{example}
[The $A_{\mathrm{inf}}$ frame]
\label{example:a_inf_frame}
As a special case of the previous example, we have the case where $(A,I')$ is a perfect prism with $R\defn A/I'$ perfectoid. 

When $p=0\in R$, $A = A_{\mathrm{inf}}(R) = W(R)$ with $I = pW(R)$, and this gives a special case of Example~\ref{example:witt_frame}.
\end{example}

\begin{remark}[Operations of frames]
\label{rem:operations_on_frames}
For future reference, let us record some basic operations on frames:
\begin{enumerate}
   \item (Base-change) Suppose that $\underline{A}$ is a frame and that $\varphi_A:A\to A$ is the underlying na\"ive Frobenius lift. If $(A,\varphi_A)\to (B,\varphi_B)$ is a map of animated commutative rings equipped with na\"ive Frobenius lifts such that $B$ is also $(p,I)$-complete, then one sees that the tuple $(B\otimes_A\Fil^\bullet A,B\otimes_AI\to B,B\otimes_AA\{1\})$\footnote{The tensor products here are meant to be in the $(p,I)$-complete category.} underlies a canonical frame $B\otimes_A\underline{A}$. The filtered Frobenius lift is the one underlying the map
   \[
   \varphi^*_B\Fil^\bullet B \simeq B\otimes_A\varphi^*_A\Fil^\bullet B \xrightarrow{\mathrm{id}\otimes\Phi} B\otimes_A\Fil^\bullet_I A = \Fil^\bullet_{B\otimes_AI}B.
   \]
   Here, we are viewing $\Phi$ as a map $\varphi_A^*\Fil^\bullet A\to \Fil^\bullet_IA$ of filtered animated commutative $A$-algebras.
   \item (Reduction-mod-$p^n$) Applying this with $B=A/{}^{\mathbb{L}}p^n$ with its induced na\"ive Frobenius lift shows that we have a canonical frame $\underline{A}/{}^{\mathbb{L}}p^n$ obtained by reducing the original one mod $p^n$. Note that, if $A$ is not prismatic, then this induced lift will not agree with the canonical Frobenius endomorphism of $A/{}^{\mathbb{L}}p$ for $n=1$. Therefore, to avoid confusion, we will reserve this operation for prismatic frames in the sequel.
   \item (Postnikov tower) For $k\ge 0$, there is a canonical frame structure $\tau_{\leq k}\underline{A}$ obtained by taking the $k$-truncations of all the defining data. By Remark~\ref{rem:witt_postnikov_truncation}, this operation preserves the property of being prismatic.
\end{enumerate}
\end{remark}   

\begin{remark}[Mapping to zip frames]
\label{rem:mod_p_and_I_frames}
Suppose that $\underline{A} = (A,I\to A,\Fil^\bullet A,\Phi)$ is a prismatic frame. Then the map $A\to A/{}^{\mathbb{L}}(p,I)$ is compatible with Frobenius, and so by Remark~\ref{rem:operations_on_frames} we obtain via base-change a frame $\underline{A}^\modpI \defn \underline{A}/{}^{\mathbb{L}}(p,I)$. If further we can fix an \emph{orientation} for the prism\footnote{This is mostly for convenience. We could omit this condition at the cost of replacing the truncated Witt frame by a twisted version that incorporates the invertible module $I$.}---that is, a trivialization of $A$-modules $A\xrightarrow{\simeq}I$---then $\underline{A}^\modpI$ is now a $p$-adic frame, and so Lemma~\ref{lem:witt_frame_univ_property} gives us maps of frames 
\[
\underline{A}\to \underline{A}^\modpI\to \underline{W_1(R_{A^\modpI})}.
\]
If $\underline{A}$ is a $p$-adic prismatic frame, then the lemma already gives a map $\underline{A}\to \underline{W_1(R_A)}$ through which the above map factors.
\end{remark}

\begin{definition}[Stacks associated with frames]
Let $\Spf A \defn \Spf(A,I)$ be the $p$-adic formal scheme obtained from $A$ with its $(p,I)$-adic topology. If $\Rees(\Fil^\bullet A)\to \Aff^1/\Gm\times \Spf\Int_p$ is the associated formal Rees stack over $\Spf\Int_p$ as in~\S\ref{subsec:formal_rees}, we obtain two maps
\[
\tau,\sigma:\Spf A \hookrightarrow \Rees(\Fil^\bullet A)
\]
as follows:
\begin{itemize}
   \item $\tau$ is obtained by pulling back the open point
   \[
    \Gm/\Gm\times \Spf \Int_p\hookrightarrow \Aff^1/\Gm\times \Spf\Int_p.
   \]
   It is in particular, an open immersion.
   \item $\sigma$ is obtained as the composition
   \[
   \Spf A \xrightarrow{\simeq} \Rees(\Fil^\bullet_{I,\pm}A) \xrightarrow{\Rees(\Phi_{\pm})} \Rees(\Fil^\bullet A).
   \]
\end{itemize}
\end{definition}

\begin{remark}
\label{rem:abstract_structure_map}
We also have a canonical map $\pi:\Rees(\Fil^\bullet A)\to \mathbb{A}^1/\Gm\times \Spf A$ arising from the map of filtered rings $\Fil^\bullet_{\mathrm{triv}}A\to \Fil^\bullet A$. From it, we obtain a map $\Rees(\Fil^\bullet A)\to \Spf A$ whose composition with $\tau$ is the identity, while its composition with $\sigma$ is the endomorphism of $\Spf A$ obtained from the Frobenius lift $\varphi$.
\end{remark}

\begin{remark}[The abstract filtered de Rham point]
\label{rem:abstract_de_rham_point}
The restriction of the map $\tau$ to $\Spf{R_A}$ extends along the open immersion
\[
\Gm/\Gm\times\Spf {R_A}\hookrightarrow \Aff^1/\Gm\times \Spec {R_A}
\]
to a map $\Aff^1/\Gm\times\Spf {R_A}\to \Rees(\Fil^\bullet A)$. On the level of filtered rings this corresponds to the map $\Fil^\bullet A\to \Fil^\bullet_{\mathrm{triv}}{R_A}$. 

In particular, by restricting this last map to $B\Gm\times \Spf R_A$, we can view $\mathcal{R}(\Fil^\bullet A)$ as an $R_A$-pointed graded formal stack in the sense of Definition~\ref{defn:pointed_graded}.
\end{remark}

\begin{example}[The stacks in the oriented Breuil-Kisin case]
\label{ex:stacks_in_bk_case}
Suppose that $\underline{A}$ is of Breuil-Kisin type associated with a transversal prism $(A,I')$. Assume also that we have an orientation given by a generator $\xi'\in I'$. Then $\Fil^\bullet A = \Fil^\bullet_{\xi'}A$, and so we have
\[
\Rees(\Fil^\bullet A) = \Spf\left(\bigoplus_i(\xi')^{\min(0,i)}t^{-i}A\right)/\Gm \xleftarrow[\simeq]{\xi't^{-1}\mapsto u}\Spf\left(A[u,t]/(ut-\xi')\right)/\Gm
\]
The map $\tau$ now corresponds simply to the open immersion of the locus $t\neq 0$. The map $\sigma$ is a bit more interesting: On the level of maps of graded rings it corresponds to the composition
\[
A[u,t]/(ut-\xi')\to (\varphi_*A)[u,t]/(ut-\xi)\to (\varphi_*A)[u,u^{-1},t]/(ut-\xi)\xrightarrow{\simeq}(\varphi_*A)[u,u^{-1}].
\]
Here, the first map is induced by the Frobenius lift $\varphi$ acting on coefficients. Put more simply, $\sigma$ is obtained by inverting $u$ and then applying the Frobenius endomorphism of $\Spf A$. It is an open immersion if and only if $(A,I')$ is a perfect prism.
\end{example}

Before we state the next result, recall the following:
\begin{definition}
We will say that a map $B\to C$ in $\mathrm{CRing}$ is \defnword{Henselian} if $\pi_0(B)\to \pi_0(C)$ is surjective, and if $(\pi_0(B),\ker(\pi_0(B)\to\pi_0(C)))$ is a Henselian pair as defined for instance in~\cite[\href{https://stacks.math.columbia.edu/tag/09XD}{Tag 09XD}]{stacks-project}.
\end{definition}

\begin{proposition}
\label{prop:frame_props_hensel}
Suppose that $\underline{A}$ is a (prismatic) frame. Then $A\twoheadrightarrow {R_A}$ is Henselian. Moreover, every $p$-completely \'etale map ${R_A}\to {R_{A'}}$ lifts uniquely to a $(p,I)$-completely \'etale map $A\to A'$\footnote{By this, we mean that $A'$ is $(p,I)$-complete, and $A/{}^{\mathbb{L}}(p,I)\to A'/{}^{\mathbb{L}}(p,I)$ is \'etale.}, where $A'$ underlies a (prismatic) frame $\underline{A}'$ uniquely determined by the fact that $\Fil^\bullet A' = \Fil^\bullet A\otimes_AA'$.
\end{proposition} 
\begin{proof}
This is an animated variant of~\cite[Lemma 4.2.3]{MR4355476}.

Let us check that $A\twoheadrightarrow {R_A}$ is Henselian. We follow the argument from~\cite[Lemma 4.1.28]{MR4530092}. Since $A$ is $(p,I)$-adically complete, it is enough to check that $\pi_0(A/{}^{\mathbb{L}}(p,I))\to \pi_0({R_A}/^{\mathbb{L}}(p,I))$ is Henselian, which is true since its kernel is locally nilpotent; indeed, our hypotheses imply that it is annihilated by the $p$-power Frobenius.

In fact, this argument also proves the assertion on lifting $p$-completely \'etale maps to $(p,I)$-completely \'etale ones; see~\cite[\href{https://stacks.math.columbia.edu/tag/0ALI}{Tag 0ALI}]{stacks-project}. 

Similarly, we also have an endomorphism $\varphi':A'\to A'$ extending $\varphi:A\to A$ and lifting the Frobenius endomorphism of $\pi_0(A')/p\pi_0(A')$. The corresponding filtered map $\Phi':\Fil^\bullet A' \to \Fil^\bullet_I A'$ is now given by
\[
\Fil^\bullet A' \simeq A'\otimes_A\Fil^\bullet A' \xrightarrow{\varphi'\otimes \Phi}A'\otimes_A\Fil^\bullet_I A' \simeq \Fil^\bullet_I A'.
\]

If $\underline{A}$ is prismatic, we can interpret $\delta$-ring structures on $A'$ as sections $A'\to W_2(A')$, and the $(p,I)$-complete \'etaleness of $A'$ over $A$ guarantees that there exists a unique (up to homotopy) such section lifting the corresponding one for $A$.
\end{proof}

\subsection{$(G,\mu)$-windows over frames}
\label{subsec:G_mu_windows}
Now, suppose that we are in the situation of~\S\ref{subsec:cochar}, so that $G$ is a smooth group scheme over $\Int_p$, $\mathcal{O}$ is the ring of integers in a finite unramified extension of $\Int_p$ and $\mu:\Gmh{\mathcal{O}}\to G_{\mathcal{O}}$ is a cocharacter. The associated map $B\mu:B\Gmh{\mathcal{O}}\to BG_{\mathcal{O}}$ classifies a $G$-torsor $\mathcal{P}_\mu$ over $B\Gmh{\mathcal{O}}$.

\subsubsection{}
Let $k$ be the residue field of $\mathcal{O}$. If $\underline{A}$ is a frame such that ${R_A}$ lifts to $\mathrm{CRing}_{k/}$, then $A$ lifts canonically to $\mathrm{CRing}_{\mathcal{O}/}$: Lift the map $\mathcal{O}\to {R_A}/^{\mathbb{L}}(p,I)$ first to $\mathcal{O}\to A/{}^{\mathbb{L}}(p,I)$ using local nilpotence and the formal \'etaleness of $\mathcal{O}$, and then to $A$ by $(p,I)$-completeness.

We will view $\Rees(\Fil^\bullet A)$ as living over the stack $B\Gm$ via the line bundle associated with the filtered module 
\[
\Fil^\bullet A\{1\} \defn \Fil^\bullet A(-1)\otimes_{A}A\{1\}. 
\]
Note that the restriction of this line bundle along $\tau$ corresponds to the $A$-module $A\{1\}$, while that along $\sigma$ corresponds to $I\otimes_A\varphi^*A\{1\}\simeq A\{1\}$. Therefore, if we take the structure map $\Spf A \to B\Gm$ classifying the line bundle associated with $A\{1\}$, both $\sigma,\tau$ can be viewed as maps of stacks over $B\Gm$.

\begin{proposition}
\label{prop:fpqc_is_etale}
The following are equivalent for an fpqc $G\{\mu\}$-torsor $\mathcal{Q}^\mu$ over $\Rees(\Fil^\bullet A)\otimes\Int/p^n\Int$.
\begin{itemize}
   \item[(1)]$\mathcal{Q}^\mu$ is trivial \'etale-locally on $\Spf {R_A}$. That is, there exists a $p$-completely \'etale map ${R_A}\to {R_{A'}}$ such that the restriction of $\mathcal{Q}^\mu$ over $\Rees(\Fil^\bullet A')\otimes\Int/p^n\Int$ is trivial.
   \item[(2)]The restriction of $\mathcal{Q}^\mu$ to $\Rees(\Fil^\bullet A)_{(t=0)}\otimes\Int/p^n\Int$ is trivial \'etale-locally on $\Spf {R_A}$.
   \item[(3a)]The restriction of $\mathcal{Q}^\mu$ over $\Aff^1/\Gm\times \Spf R_A$ is trivial \'etale-locally on $\Spf {R_A}$.
   \item[(3b)]For every geometric point ${R_A}\to \kappa$ of $\Spf {R_A}$, the restriction of $\mathcal{Q}^\mu$ over $\mathbb{A}^1/\Gm\times\Spec \kappa$ is trivial.
   \item[(4a)]The restriction of $\mathcal{Q}^\mu$ over $B\Gm\times \Spf R_A$ is trivial \'etale-locally on $\Spf {R_A}$.
   \item[(4b)]For every geometric point ${R_A}\to \kappa$ of $\Spf {R_A}$, the restriction of $\mathcal{Q}^\mu$ over $B\Gm\times\Spec \kappa$ is trivial.
\end{itemize}
If $\Spf R_A$ is connected, then this is also equivalent to: For \emph{some} geometric point $R_A\to \kappa$ of $\Spf R_A$, the restriction of $\mathcal{Q}^\mu$ over $B\Gm\times \Spec \kappa$ is trivial.
\end{proposition} 
\begin{proof}
The equivalence of (3a), (3b), (4a) and (4b) follows from Lemma~\ref{lem:BGmu_trivial_locus} and Remark~\ref{rem:BGmu_simple}, as does the equivalence of these statements with the last unnumbered assertion.

We can finish by showing (3a)$\Rightarrow$(1). Since the stacks involved are $(p,I)$-complete, it is enough to know that every section of the smooth relative scheme $\mathcal{Q}^\mu$ over $\Aff^1/\Gm\times \Spec R_A/{}^{\mathbb{L}}(p,I)$ can be lifted to a section over $\Rees(\Fil^\bullet A/{}^{\mathbb{L}}(p,I))$.

Note that the kernel of the classical truncation of $A/{}^{\mathbb{L}}(p,I)\to R_A/{}^{\mathbb{L}}(p,I)$ is a locally nilpotent thickening, as observed in the proof of Proposition~\ref{prop:frame_props_hensel}. If $\mu$ is $1$-bounded, we can now conclude using Proposition~\ref{prop:1_bounded_cartesian}.

In general, we can use the argument from the proof of \emph{loc. cit.} to prove the following claim: Given a non-negatively filtered animated commutative ring $\Fil^\bullet B$ where the kernel of $\pi_0(B)\to \pi_0(\gr^0B)$ is locally nilpotent, and given a map $X\to \Rees(\Fil^\bullet B)$ fibered by smooth schemes, $X$ admits a section over $\Rees(\Fil^\bullet B)$ if and only if it admits a section over $\Aff^1/\Gm\times\Spec B$. Indeed, using the invariance of this property under filtered nilpotent thickenings of $\Fil^\bullet B$, we reduce successively to the case where $\Fil^\bullet B$ is discrete, then to the case where $\Fil^1B \to B$ is the zero map, and then using Proposition~\ref{prop:filtered_completeness} to the case where the Rees algebra of $\Fil^\bullet B$ is itself a nilpotent extension of $B[t]$.
\end{proof}

\begin{definition}
\label{defn:expl_desc_quotient}
$\Wind{\infty}{\underline{A}}({R_A})$ is the $\infty$-groupoid of $G$-torsors $\mathcal{Q}$ over $\Rees(\Fil^\bullet A)$ equipped with an isomorphism of $G$-torsors $\xi:\sigma^*\mathcal{Q}\xrightarrow{\simeq}\tau^*\mathcal{Q}$ over $\Spf R_A$, and satisfying the following equivalent conditions: 
\begin{enumerate}
   \item The associated $G\{\mu\}$-torsor $\mathcal{Q}^\mu$ is trivial \'etale locally on $\Spf {R_A}$.
   \item For every geometric point ${R_A}\to \kappa$, the restriction of $\mathcal{Q}^\mu$ over $B\Gm\times\Spec \kappa$ is trivial.
   \item For every geometric point $R_A\to \kappa$, the restriction of $\mathcal{Q}$ over $B\Gm\times\Spec \kappa$ is isomorphic to $\mathcal{P}_\mu$.
\end{enumerate}
We will refer to the objects of this $\infty$-groupoid as \defnword{$(G,\mu)$-windows} over $\underline{A}$.

When $\underline{A}$ is prismatic, we will define the $\infty$-groupoid of \defnword{$n$-truncated $(G,\mu)$-windows} over $\underline{A}$ by replacing $\Rees(\Fil^\bullet A)$ with $\Rees(\Fil^\bullet A)\otimes\Int/p^n\Int$ in the above definition, and asking for the isomorphism $\xi$ to be of $G$-torsors over $A/{}^{\mathbb{L}}p^n$.

We have
\[
\Wind{\infty}{\underline{A}} = \varprojlim_n \Wind{n}{\underline{A}}.
\]
\end{definition}

\begin{definition}
\label{defn:displays}
We will call $(G,\mu)$-windows over the Witt vector frame $\underline{W(R)}$ $(G,\mu)$\defnword{-displays over $R$}, and write $\Disp{\infty}(R)$ for the $\infty$-groupoid spanned by them. 

If $R$ is an $\Field_p$-algebra, then, for $n\ge 1$, we will call $(G,\mu)$-windows over the truncated Witt vector frame $\underline{W_n(R)}$ $n$-\defnword{truncated} $(G,\mu)$\defnword{-displays}.\footnote{This is unfortunately inconsistent with our use of the adjective `$n$-truncated' for windows, but it is compatible with the definition from~\cite{MR2983008}. For perfect $R$, the two notions of truncatedness agree.}

For classical $R$, these definitions recover the definition from~\cite{MR4120808} (for $n=\infty$) and~\cite[\S 5]{MR4355476}.

When $n=1$, one can show that, in the language of ~\cite{Pink2015-ye}, $\Disp{1}$ is the stack over $k$ parameterizing $G$-zips of type $\mu$.
\end{definition}

\begin{remark}
\label{rem:double_quotient_desc}
Here is a slightly different perspective on the definition, closer to the treatment in~\cites{MR4355476,MR4120808}. Suppose that we have a trivialization $A\xrightarrow{\simeq}A\{1\}$.

Let $\mathcal{Q}_0$ be the $G$-torsor over $\Rees(\Fil^\bullet A)\otimes\Int/p^n\Int$ obtained from the torsor $\mathcal{P}_\mu\to B\Gmh{\mathcal{O}}$: Its automorphisms are represented by the sheaf $L^+_{\underline{A}}G^{(n)}\{\mu\}$ given by
\[
L^+_{\underline{A}}G^{(n)}\{\mu\}({R_{A'}}) = \Map_{B\Gmh{\mathcal{O}}}(\Rees(\Fil^\bullet A')\otimes\Int/p^n\Int,G\{\mu\}).
\]

Because of our chosen trivialization of $A\{1\}$, the restriction of $\mathcal{Q}_0$ along both $\tau$ and $\sigma$ is also trivial, and so its automorphisms are represented on the $p$-completely \'etale site of $R_A$ by the sheaf $L^+_{\underline{A}}G^{(n)}$ given by
\[
L^+_{\underline{A}}G^{(n)}({R_{A'}}) = G^{(n)}(A') =  G(A'/^{\mathbb{L}}p^n).
\]

Pullback along $\sigma$ and $\tau$ gives two maps
\[
\sigma,\tau:L^+_{\underline{A}}G^{(n)}\{\mu\}\to L^+_{\underline{A}}G^{(n)}
\]

We now have
\[
 \Wind{n}{\underline{A}}= [L^+_{\underline{A}}G^{(n)}\dbqt{\sigma}{\tau} L^+_{\underline{A}}G^{(n)}\{\mu\}],
\]
where the right hand side indicates the quotient by the action given symbolically by
\[
L^+_{\underline{A}}G^{(n)}\times L^+_{\underline{A}}G^{(n)}\{\mu\}\xrightarrow{(h,g)\mapsto \tau(h)^{-1}g\sigma(h)}L^+_{\underline{A}}G^{(n)}.
\]
\end{remark}

\begin{remark}
\label{rem:infty_displays}
Under the same condition as in the previous remark, the $\infty$-groupoid of $(G,\mu)$-windows over $\underline{A}$ can be viewed as the quotient
\[
[L^+_{\underline{A}}G\dbqt{\sigma}{\tau} L^+_{\underline{A}}G\{\mu\}]
\]
with 
\[
L^+_{\underline{A}}G({R_{A'}}) = G(A')\;;\; L^+_{\underline{A}}G\{\mu\}({R_{A'}}) = \Map_{B\Gmh{\mathcal{O}}}(\Rees(\Fil^\bullet A'),G\{\mu\}).
\]

Note that, by Proposition~\ref{prop:1_bounded_cartesian}, if $\mu$ is $1$-bounded, we have
\[
L^+_{\underline{A}}G\{\mu\}({R_{A'}}) \simeq G(A')\times_{G(R_{A'})}P^-_\mu(R_{A'}).
\]
\end{remark}

\emph{Assume for the rest of this subsection that $\mu$ is $1$-bounded.}

\begin{remark}
\label{rem:witt_frame_windows}
If $R$ is a discrete $p$-complete $\mathcal{O}$-algebra and $\underline{W(R)}$ is the Witt frame from Example~\ref{example:witt_frame}, then the above description shows that our notion of a $(G,\mu)$-display in terms of $(G,\mu)$-windows over $\underline{W(R)}$, recovers the notion defined in~\cite{MR4120808}.
\end{remark}

\begin{remark}
\label{rem:B'Gmu_abstract}
Let $\mathfrak{S}(\underline{A})$ be the coequalizer in $p$-adic formal prestacks of the two maps
\[
\sigma,\tau:\Spf A \to \Rees(\Fil^\bullet A).
\]
Unwinding definitions, $\Wind{n}{\underline{A}}({R_A})$ is simply the space 
\[
\Map_{/B\Gm}(\mathfrak{S}(\underline{A})\otimes\Int/p^n\Int,\mathcal{B}(G,\mu)),
\]
where $\mathcal{B}(G,\mu)\to (B\Gm,\iota_{\mathcal{O}})$ is the $1$-bounded stack from Definition~\ref{defn:BGmu_1-bounded_stack} and we are viewing $\mathfrak{S}(\underline{A})$ as a pointed graded stack via Remark~\ref{rem:abstract_de_rham_point}.
\end{remark}

\begin{remark}
\label{rem:dilatations}
Suppose that $A$ is flat over $\Int_p$. Suppose also that we have $\Fil^1A = (E)$ for some non-zero divisor $E\in A$, so that $R_A = A/(E)$, $I = (\varphi(E))$ and $A\{1\}$ is trivial. In particular, $\underline{A}$ is prismatic.

Let $H_{\mu}\to G_{A}$ be the \emph{dilatation} of $G_{A}$ along the subgroup $P^-_\mu\otimes R_A \subset G\otimes R_A$ (see for instance~\cite[\S 3.2]{blr}). This is a smooth affine group scheme over $A$ characterized by the fact that, for any flat $A$-algebra $B$, we have $H_\mu(B) = G(B)\times_{G(B/(E))}P^-_\mu(B/(E))$. In particular, conjugation by $\mu(E)$ on the generic fiber of $G_{\mathcal{O}}$ restricts to a map $\mathrm{int}(\mu(E)):H_\mu\to G_{A}$

Then we have $L^+_{\underline{A}}G\{\mu\}(R_{A'}) = H_\mu(A')$, $\tau$ is just the natural map $H_\mu(A')\to G(A')$, while $\sigma$ is given by $\varphi\circ \mathrm{int}(\mu(E))$; see the argument in~\cite[Lemma 5.2.1]{MR4355476}.

Therefore, we see that $\Wind{\infty}{\underline{A}}$ is the \'etale sheafification of the functor
\[
R_{A'}\mapsto [G(A')\dbqt{\sigma*}{\tau^*}H_\mu(A')].
\]
In other words, a section over $A'$, amounts to giving an $H_\mu$-torsor $\mathcal{P}$ over $A$ for the $p$-completely \'etale topology, along with an isomorphism $\xi:\sigma^* \mathcal{P}\xrightarrow{\simeq} \tau^* \mathcal{P}$. 
\end{remark}

\begin{remark}
\label{rem:trivial_filtration_displays}
There is another situation in which the description of $\Wind{\infty}{\underline{A}}$ simplifies considerably: If $A$ is flat over $\Int_p$, $\Fil^iA = 0$ for $i\ge 1$, $R_A = A$ and $I = (p)$. This is of course not a laminated frame.

Then $L^+_{\underline{A}}G\{\mu\}(R_{A'}) = P^-_\mu(A')$, and $\tau^*:P^-_\mu(A')\to G(A')$ is the natural map, with $\sigma$ is given by $\varphi\circ \mathrm{int}(\mu(p))$. 

We find that an object of $\Wind{\infty}{\underline{A}}(A')$ is simply a $P^-_\mu$-torsor $\mathcal{P}'$ over $A'$ for the $p$-completely \'etale topology, along with an isomorphism $\sigma^* \mathcal{P}'\xrightarrow{\simeq}\tau^* \mathcal{P}'$
\end{remark}

\subsection{$\underline{A}$-gauges and the case of $\GL_h$}
\label{subsec:A_gauges}

Here, we look at what the above definitions specialize to for $G = \GL_h$ for some $h\ge 1$. First, a more general definition.

\begin{definition}
An \defnword{$\underline{A}$-gauge of level $n$} (resp. \defnword{$\underline{A}$-gauge}) is a quasicoherent sheaf $\mathcal{M}\in \mathrm{QCoh}(\Rees(\Fil^\bullet A)\otimes\Int/p^n\Int)$ (resp. $\mathcal{M}\in \mathrm{QCoh}(\Rees(\Fil^\bullet A))$) equipped with an isomorphism $\xi:\sigma^*\mathcal{M}\xrightarrow{\simeq}\tau^*\mathcal{M}$ in $\mathrm{QCoh}(\Spf A/{}^{\mathbb{L}}p^n)$ (resp. $\mathrm{QCoh}(\Spf A)$).

Note that by definition an $\underline{A}/{}^{\mathbb{L}}p^n$-gauge is the same as an $\underline{A}$-gauge of level $n$. 

As usual, one can append the adjectives `vector bundle', `perfect', `almost perfect', `connective' and `almost connective' to obtain objects in the corresponding full subcategories.
\end{definition}

\begin{remark}
We can also view $\underline{A}$-gauges (of level $n$) as quasicoherent sheaves over the prestack $\mathfrak{S}(\underline{A})$ from Remark~\ref{rem:B'Gmu_abstract}.
\end{remark}

\begin{definition}
\label{defn:A-gauge_HT_weights}
Every $\underline{A}$-gauge yields via pullback along  
\[
B\Gm\times \Spec \pi_0(R_A/{}^{\mathbb{L}}(p,I))\to \Rees(\Fil^\bullet A)
\]
a graded complex over $\pi_0(R_A/{}^{\mathbb{L}}(p,I))$. The \defnword{Hodge-Tate weights} of an $\underline{A}$-gauge are the integers $i\in \Int$ such that the associated graded complex is not nullhomotopic in degree $i$.
\end{definition}

\begin{remark}
\label{rem:sigma_star_expl}
Explicitly, given an $\underline{A}$-gauge $(\mathcal{M},\xi)$, we can view $\mathcal{M}$ as a derived $(p,I)$-complete filtered module $\Fil^\bullet\mathsf{M}$ over $\Fil^\bullet A$. Base-change along $\Phi_{\pm}:\Fil^\bullet A\to \Fil^\bullet_{I,\pm}A$ yields a filtered module $\Phi_{\pm}^*\Fil^\bullet \mathsf{M}$ over $\Fil^\bullet_{I,\pm}A$ whose degree-$0$ filtered piece is an $A$-module $\mathsf{M}_\sigma$ corresponding to $\sigma^*\mathcal{M}$. The isomorphism $\xi$ now corresponds to an isomorphism $\mathsf{M}_\sigma\xrightarrow{\simeq}\mathsf{M}$ in $\Mod{A}$. 
\end{remark}

\begin{remark}[Parasitic $\underline{A}$-gauges]
\label{rem:parasitic_gauges}
Suppose that we have a quasicoherent sheaf $\mathcal{Q}$ over $\Aff^1/\Gm\times\Spf R_A$ whose restriction over $\Gm/\Gm\times\Spf R_A$ is nullhomotopic: this corresponds to a filtered module $\Fil^\bullet Q$ over $R_A$ such that $\colim_n \Fil^nQ \simeq 0$. Via pushforward along the closed immersion 
\[
\iota:\Aff^1/\Gm\times\Spf R_A = \Rees(\Fil^\bullet_{\mathrm{triv}}R_A)\hookrightarrow \Rees(\Fil^\bullet A)
\]
we obtain a quasicoherent sheaf $\iota_* \mathcal{Q}$ over $\Rees(\Fil^\bullet A)$. One checks that we have
\[
\tau^*\iota_* \mathcal{Q} \simeq 0 \simeq \sigma^*\iota_* \mathcal{Q}.
\]
The first equivalence is from our hypothesis on $\mathcal{Q}$, while the second applies generally to any pushforward along $\iota$. In particular, $\iota_* \mathcal{Q}$ is trivially equipped with the structure of an $\underline{A}$-gauge. We will call this the \defnword{parasitic $\underline{A}$-gauge associated with $\mathcal{Q}$} (or, equivalently, with $\Fil^\bullet Q$).
\end{remark}  

\begin{example}
 \label{ex:parasitic_gauges}
In \S~\ref{subsec:abstract_def_theory}, we will need a particular instance of a parasitic $\underline{A}$-gauge arising in the following way: Given any $p$-complete $N\in \Mod{R_A}$ and an integer $m\in \Int$, there is a canonical $p$-complete filtered module $\Fil^\bullet Q(N,m)$ over $R_A$ such that, for any other $p$-complete filtered module $\Fil^\bullet M$, we have a canonical equivalence
\[
\mathrm{RHom}_{\Fil^\bullet_{\mathrm{triv}}R_A}(\Fil^\bullet M,\Fil^\bullet Q(N,m))\simeq \mathrm{RHom}_{R_A}(\Fil^m M,N).
\]
We have $\Fil^jQ(N,m) = N$ if $j\ge m$ and $0$ otherwise. The transition maps are the identity in filtered degrees $>m$ and $0$ in other degrees. 

Let $\mathcal{M}_{\mathrm{par}}(N,m)$ be the parasitic $\underline{A}$-gauge associated with $\Fil^\bullet Q(N,m)$. If $\Fil^\bullet \mathsf{M}_{\mathrm{par}}(N,m)$ is the associated filtered module over $\Fil^\bullet A$, then, for any other filtered module $\Fil^\bullet \mathsf{M}$ with filtered base-change $\Fil^\bullet_{\mathrm{Hdg}}M$ over $R_A$, we have
\[
\mathrm{RHom}_{\Fil^\bullet A}(\Fil^\bullet \mathsf{M},\Fil^\bullet \mathsf{M}_{\mathrm{par}}(N,m))\simeq \mathrm{RHom}_{\Fil^\bullet_{\mathrm{triv}}R_A}(\Fil^\bullet_{\mathrm{Hdg}}M,\Fil^\bullet Q(N,m))\simeq \mathrm{RHom}_{R_A}(\Fil^m_{\mathrm{Hdg}}M,N).
\]
In particular, if $\Fil^\bullet \mathsf{M}$ underlies an $\underline{A}$-gauge $(\mathcal{M},\xi)$, then giving a map of $\underline{A}$-gauges $\mathcal{M}\to \mathcal{M}_{\mathrm{par}}(N,m)$ is equivalent to giving a map $\Fil^m_{\mathrm{Hdg}}M\to N$ in $\Mod{R_A}$.
\end{example}

\begin{remark}
$\underline{A}$-gauges (of level $n$) organize into a symmetric monoidal stable $\infty$-category $\underline{A}\mathsf{-gauge}$ ($\underline{A}\mathsf{-gauge}_n$).

For any map $\underline{A}\to \underline{B}$ of frames, there is a natural base-change map 
\[
\underline{A}\mathsf{-gauge}\xrightarrow{(\mathcal{M},\xi)\mapsto \underline{B}\otimes_{\underline{A}}(\mathcal{M},\xi)}\underline{B}\mathsf{-gauge}
\]
that induces functors between gauges of level $n\ge 1$.
\end{remark}

The proof of the next result is as in~\cite[Example 5.3.5]{MR4355476}, and follows from Proposition~\ref{prop:fpqc_is_etale}.
\begin{proposition}
\label{prop:abstract_f-gauge_gl_n}
Suppose that $\mu:\Gm\to \GL_h$ is the cocharacter given by
\[
z\mapsto \mathrm{diag}(z^{m_1},z^{m_2},\ldots,z^{m_h})
\]
for $m_1,m_2,\ldots,m_n\in \Int$. Then $\Wind[\GL_h,\mu]{n}{\underline{A}}$ is equivalent to the $\infty$-groupoid of $\underline{A}$-gauges $\mathcal{M}$ of level $n$ such that the underlying filtered module $\Fil^\bullet\mathsf{M}$ satisfies the following condition: There exists a $p$-completely \'etale cover ${R_A}\to{R_{A'}}$ and an isomorphism
\[
\Fil^\bullet A'\otimes_{\Fil^\bullet A}\Fil^\bullet \mathsf{M} \xrightarrow{\simeq}\bigoplus_{i=1}^h \Fil^\bullet A'\{m_i\}/{}^{\mathbb{L}}p^n.
\]
Here, we have set
\[
\Fil^\bullet A'\{m_i\} \defn (\Fil^\bullet A')(-m_i)\otimes_AA\{m_i\}.
\]
\end{proposition}

\subsection{Divided Dieudonn\'e complexes}
\label{subsec:alb}
We will fix a prismatic frame $\underline{A}$. The purpose of this subsection is to connect the definitions here with those of Ansch\"utz-Le Bras~\cite{MR4530092}. More precisely, we show that perfect $\underline{A}$-gauges with Hodge-Tate weights in $\{0,1\}$ admit a more concrete description modeled after the notion of a filtered Dieudonn\'e module appearing in \emph{loc. cit.} Similar discussions---though in a more limited context---can be found in~\cite{guo2023frobenius} and~\cite{Mondal2024-cy}.

\begin{definition}
\label{defn:adm_dieudonne}
Set $R = R_A$. Then the filtered Frobenius lift carrying $\Fil^1A$ to $I$ induces a map $\overline{\varphi}:R\to \overline{A}$. A \defnword{divided Dieudonn\'e complex over $R$ with respect to $\underline{A}$} is a tuple $(\mathsf{M},\Fil^0M\to M,\mathsf{M}\xrightarrow{\Psi_{\mathsf{M}}}\varphi^*\mathsf{M},\xi)$ such that:
\begin{enumerate}
   \item $\mathsf{M}\in \mathrm{Perf}(A)$ is a perfect complex over $A$;
   \item $\Fil^0M \to M \defn R\otimes_{A}\mathsf{M}$ is a map of perfect complexes over $R$;
   \item $\Psi_{\mathsf{M}}:\mathsf{M}\to \varphi^*\mathsf{M}$ is a map of perfect complexes over $A$;
   \item $\xi$ is an isomorphism of perfect complexes over $R$ sitting in a diagram
   \[
    \begin{diagram}
    \varphi^*\mathsf{M}&\rTo&\hcoker(\Psi_{\mathsf{M}})\\
    &\rdTo&\dTo^{\simeq}_{\xi}\\
    &&\overline{A}\otimes_{\overline{\varphi},R}\gr^{-1}M.
    \end{diagram}
   \] 
\end{enumerate}
Here, $\gr^{-1}M = M/\Fil^0M$ and the diagonal map is obtained as the composition
\[
\varphi^*\mathsf{M}\to \overline{A}\otimes_A\varphi^*\mathsf{M}\xrightarrow{\simeq}\overline{A}\otimes_{\overline{\varphi},R}M\to \overline{A}\otimes_{\overline{\varphi},R}\gr^{-1}M.
\]

These can be organized into an $\infty$-category in a natural way, which we denote by $\mathrm{DDC}_{\underline{A}}(R)$.

A divided Dieudonn\'e complex has \defnword{Tor amplitude in $[a,b]$} if $\mathsf{M}$ is in $\mathrm{Perf}^{[a,b]}(A)$ and if $\Fil^0M,\gr^{-1}M$ are both in $\mathrm{Perf}^{[a,b]}(R)$. Write $\mathrm{DDC}^{[a,b]}_{\underline{A}}(R)$ for the subcategory spanned by the objects with Tor amplitude in $[a,b]$.
\end{definition}  

\begin{remark}
\label{rem:only_classical_frame_needed}
Observe that the definition of a divided Dieudonn\'e complex does not use the full frame structure. Indeed, we only need the prism $(A,I)$ and a commuting square of animated commutative rings
\[
\begin{diagram}
A&\rTo^\varphi&A\\
\dTo&&\dTo\\
R&\rTo_{\overline\varphi}&\overline{A}.
\end{diagram}
\]
With this setup, we can define categories $\mathrm{DDC}^{[a,b]}_{(A,I)}(R)$ just as before.
\end{remark}

\begin{proposition}
\label{prop:adm_dieudonne}
Let $\mathrm{Perf}_{\{0,1\}}^{[a,b]}(\underline{A}\mathsf{-gauge})$ be the full subcategory of $\underline{A}\mathsf{-gauge}$ spanned by the perfect $\underline{A}$-gauges with Hodge-Tate weights in $\{0,1\}$ and Tor amplitude in $[a,b]$. Then there is a canonical equivalence of $\infty$-categories
\[
\mathrm{Perf}_{\{0,1\}}^{[a,b]}(\underline{A}\mathsf{-gauge})\xrightarrow{\simeq}\mathrm{DDC}^{[a,b]}_{\underline{A}}(R).
\]
\end{proposition}

\begin{proof}
Let $\mathcal{P}_{\{0,1\}}\to B\Gm$ be the $1$-bounded stack from Example~\ref{ex:perfect_HT_wts_01_stack}, so that $\Map_{/B\Gm}(\Rees(\Fil^\bullet A),\mathcal{P}_{\{0,1\}})$ is canonically equivalent to the $\infty$-groupoid underlying the $\infty$-category $\mathrm{Perf}_{\{0,1\}}(\Rees(\Fil^\bullet A))$ of perfect complexes over $\Rees(\Fil^\bullet A)$ with Hodge-Tate weights in $\{0,1\}$. Then Proposition~\ref{prop:1_bounded_cartesian} shows that this $\infty$-groupoid is equivalent to that of tuples $(\mathsf{M},\Fil^\bullet M,\eta)$ where:
\begin{enumerate}
   \item $\mathsf{M}\in \mathrm{Perf}(A)$;
   \item $\Fil^\bullet M$ is a filtered perfect complex over $R$ with associated gradeds supported in degrees $0,-1$;
   \item $\eta:R\otimes_A\mathsf{M}\xrightarrow{\simeq}M$ is an isomorphism of perfect complexes over $R$.
\end{enumerate}
In fact, one can upgrade this equivalence to one between the $\infty$-category of such tuples and $\mathrm{Perf}_{\{0,1\}}(\Rees(\Fil^\bullet A))$. Indeed, morphisms between two objects $\mathcal{M},\mathcal{M}'$ in the latter category are parameterized by sections of the $1$-bounded stack $\mathbf{V}(\mathcal{M}\otimes \mathcal{M}^{',\vee})$ (see Example~\ref{ex:1_bounded_perfect_complexes_stack}), and one can now apply Proposition~\ref{prop:1_bounded_cartesian} once again to conclude.

Given a tuple $(\mathsf{M},\Fil^\bullet M,\eta)$ mapping to $\mathcal{M}$ in $\mathrm{Perf}_{\{0,1\}}(\Rees(\Fil^\bullet A))$, let $\Fil^\bullet \mathsf{M}$ be the filtered perfect complex over $\Fil^\bullet A$ corresponding to $\mathcal{M}$, and let $\Fil^\bullet\varphi^*\mathsf{M}$ be the filtered perfect complex over $\Fil^\bullet_IA$ obtained via base-change along the filtered Frobenius lift $\Phi$. By Lemma~\ref{lem:graded_weight_filtration}, we find:
\[
\gr^i\varphi^*\mathsf{M} \simeq \begin{cases}
0&\text{if $i<-1$}\\
\overline{A}\otimes_{\overline\varphi,R}\gr^{-1}M&\text{if $i=-1$}.
\end{cases}
\]
Using this, one finds a canonical identification $\mathsf{M}_\sigma\xrightarrow{\simeq}\Fil^0\varphi^*\mathsf{M}$, and so giving an isomorphism $\mathsf{M}_\sigma\xrightarrow{\simeq}\mathsf{M}$ is equivalent to giving a fiber sequence of perfect complexes over $A$ of the form
\[
\mathsf{M}\to \varphi^*\mathsf{M}\to \overline{A}\otimes_R\gr^{-1}M.
\]
The datum of the pair $(\Fil^\bullet M,\eta)$ is equivalent once again to giving the underlying map of perfect complexes $\Fil^0M \to R\otimes_A\mathsf{M}$ over $R$. With this, we have completed the proof of the proposition, except for the matching of the conditions on Tor amplitude. But this can be verified directly.
\end{proof}

\begin{corollary}
\label{cor:adm_dieudonne_vect}
Let $\mathrm{Vect}_{\{0,1\}}(\underline{A}\mathsf{-gauge})$ be the full subcategory of $\underline{A}\mathsf{-gauge}$ spanned by the vector bundle $\underline{A}$-gauges with Hodge-Tate weights in $\{0,1\}$. Then there is a canonical equivalence of $\infty$-categories
\[
\mathrm{Vect}_{\{0,1\}}(\underline{A}\mathsf{-gauge})\xrightarrow{\simeq}\mathrm{DDC}^{[0,0]}_{\underline{A}}(R).
\]
\end{corollary}

\subsubsection{}\label{subsubsec:vect_bundles_alb}
We can now complete the connection with~\cite{MR4530092}. Suppose that $\underline{A}$ is \emph{classical} in the following sense: $A$ is $p$-completely flat, $\Fil^\bullet A$ is a discrete filtered commutative ring filtered by $A$-submodules $\Fil^iA\subset A$, and also $I\subset A$ is a locally principal ideal. Suppose also that we have $\Fil^1A = \varphi^{-1}(I)\subset A$; equivalently, we are assuming that the map $\overline\varphi:R\to \overline{A}$ is \emph{injective}.

\begin{proposition}
\label{prop:adm_dieudonne_alb}
Under these conditions, there is a natural equivalence of categories between $\mathrm{Vect}_{\{0,1\}}(\underline{A}\mathsf{-gauge})$ and the category of pairs $(\mathsf{N},\varphi_{\mathsf{N}})$ where:
\begin{enumerate}
   \item $\mathsf{N}$ is a finite locally free $A$-module;
   \item $\varphi_{\mathsf{N}}:\mathsf{N}\to \mathsf{N}$ is a $\varphi$-linear map such that the cokernel of the linearization $\varphi^*\mathsf{N}\to \mathsf{N}$ is killed by $I$;
   \item The image of the composition $\mathsf{N}\xrightarrow{\varphi_{\mathsf{N}}}\mathsf{N}\to \mathsf{N}/I\mathsf{N}$ is a locally free $R$-module $F_{\mathsf{N}}$ such that the induced map $\overline{A}\otimes_{{R}}F_{\mathsf{N}}\to \mathsf{N}/I\mathsf{N}$ is injective.
\end{enumerate}
\end{proposition}
\begin{proof}
By Corollary~\ref{cor:adm_dieudonne_vect}, $\mathrm{Vect}_{\{0,1\}}(\underline{A}\mathsf{-gauge})$ is equivalent to the category of tuples $(\mathsf{M},\Psi_{\mathsf{M}},M\to \gr^{-1}M)$ where:
\begin{enumerate}
   \item $\mathsf{M}$ is a finite locally free $A$-module;
   \item $\Psi_{\mathsf{M}}:\mathsf{M}\to \varphi^*\mathsf{M}$ is an injective $A$-linear map;
   \item  $M = \mathsf{M}/(\Fil^1A)\mathsf{M}$ and $M\to \gr^{-1}M$ is a surjection onto a finite locally free $R$-module $\gr^{-1}M$
\end{enumerate}
such that $\Psi_{\mathsf{M}}(\mathsf{M})$ is the kernel of the composition
\[
\varphi^*\mathsf{M}\to \overline{A}\otimes_A\varphi^*\mathsf{M}\xrightarrow{\simeq}\overline{A}\otimes_{\overline{\varphi},R}M \to \overline{A}\otimes_{\overline{\varphi},R}\gr^{-1}M.
\]

Given such a tuple, we will identify $\mathsf{M}$ with a submodule of $\varphi^*\mathsf{M}$ via $\Psi_{\mathsf{M}}$. Note that the short exact sequence 
\[
0\to \mathsf{M}\to \varphi^*\mathsf{M}\to \overline{A}\otimes_{\overline\varphi,R}\gr^{-1}M\to 0
\]
implies in particular that we have $I\varphi^*\mathsf{M}\subset \mathsf{M}$. Now, set $\mathsf{N} = \mathsf{M}\{-1\}$, so that we have a map
\[
\varphi_{\mathsf{N}}:\mathsf{N}\xrightarrow{n\mapsto \varphi^*n}\varphi^*\mathsf{N} \simeq \varphi^* A\{-1\}\otimes_A\varphi^*\mathsf{M} \simeq A\{-1\}\otimes_AI\varphi^*\mathsf{M}\subset \mathsf{M}\{-1\} = \mathsf{N}.
\]
We claim the pair $(\mathsf{N},\varphi_{\mathsf{N}})$ satisfies the conditions in the proposition. Indeed, set $\Fil^0\mathsf{M} = \ker(\mathsf{M}\to \gr^{-1}M)$. Then one checks that we have $\varphi_{\mathsf{N}}^{-1}(I\mathsf{N}) = (\Fil^0\mathsf{M})\{-1\}$. Hence the composition of the projection onto $\mathsf{N}/I\mathsf{N}$ with $\varphi_{\mathsf{N}}$ has image $F_{\mathsf{N}}\simeq (\gr^{-1}M)\{-1\}$. Similarly, one finds that the base-change  along $\overline\varphi$ of $F_{\mathsf{N}}$ maps isomorphically onto the direct summand $(\overline{A}\otimes_R\gr^{-1}M)\{-1\}$ of $\mathsf{N}/I\mathsf{N}$.

Conversely, given a pair $(\mathsf{N},\varphi_{\mathsf{N}})$, we set 
\[
\mathsf{M} = \mathsf{N}\{1\}\;;\;\Fil^0\mathsf{M} =\varphi_{\mathsf{N}}^{-1}(I\mathsf{N})\{1\}\;;\; \gr^{-1}M = \mathsf{M}/\Fil^0\mathsf{M}. 
\]
Then condition (2) of the proposition tells us that $\gr^{-1}M$ is locally free over $R$. To obtain the map $\Psi_{\mathsf{M}}$, we first observe that we have
\[
I\mathsf{M}\{-1\} = (1\otimes\varphi_{\mathsf{N}})^{-1}(I\mathsf{N})\subset \varphi^*\mathsf{N}. 
\]
Indeed, since $I\mathsf{M}_{\sigma}\{-1\}=\ker(\varphi^*\mathsf{N}\to \overline{A}\otimes_{R}F_{\mathsf{N}})$, this is equivalent to condition (3), which asserts that $\overline{A}\otimes_{R}F_{\mathsf{N}}$ maps injectively into $\mathsf{N}/I\mathsf{N}$. Now, tensoring this inclusion with $I^{\otimes -1}\{1\}$ gives us $\Psi_{\mathsf{M}}$.
\end{proof}

\begin{remark}
\label{rem:comparison_with_alb}
As in Remark~\ref{rem:only_classical_frame_needed}, the definition of the category of pairs $(\mathsf{N},\varphi_{\mathsf{N}})$ appearing in the above proposition requires only the discrete prism $(A,I)$, since we have $R = A/\varphi^{-1}(I)$. When $(A,I) = (\Prism_R,I_R)$ is the initial prism associated with a qrsp ring $R$, the proof of Proposition~\ref{prop:adm_dieudonne_alb} shows that the category $\mathrm{DDC}^{[0,0]}_{(\Prism_R,I_R)}(R)$ is equivalent to the category of \emph{admissible prismatic Dieudonn\'e modules over} $R$ from~\cite[Definition 4.10]{MR4530092}. Later, in Section~\ref{sec:bld_stacks}, we will see that there is in fact a canonical frame $\underline{\Prism}_R$ extending the data of $(\Prism_R,I_R)$ and $\Fil^1\Prism_R = \ker(\Prism_R\to R)$. Therefore, these categories can be further identified with the category $\mathrm{Vect}_{\{0,1\}}(\underline{\Prism}_R\mathsf{-gauge})$, which in turn can be identified with the category of vector bundle $F$-gauges over $R$ with Hodge-Tate weights in $\{0,1\}$; see Proposition~\ref{prop:semiperfect_f-gauges}.
\end{remark}

\subsection{$\underline{W_1(R)}$-gauges and $F$-zips}
\label{subsec:Fzip_stack}

Suppose that $R$ is an $\Field_p$-algebra, so that we have the associated $1$-truncated Witt frame $\underline{W_1(R)}$. We will now find that $\underline{W_1(R)}$-gauges are the same as (derived) $F$-zips as appearing for instance in~\cite{epiga:10375}. In general, as explained in the introduction, the stack $R^{\Fzip}$ associated with this frame will play an important role in what follows, so we begin by describing it explicitly.

\subsubsection{}\label{subsubsec:fzip_stack_construction}
As explained in~\cite[Example 2.1.7]{MR4355476}, the Rees algebra $\mathrm{Rees}(\Fil^\bullet_{\mathrm{Lau}}W_1(R))$ admits the following description as a graded $R$-algebra:
\begin{align}\label{eqn:rees_W1_explicit}
\mathrm{Rees}(\Fil^\bullet_{\mathrm{Lau}}W_1(R)) = R[t]\times_{\varphi_*R}\varphi_*R[u].
\end{align}
Here, $t$ is in graded degree $1$ as usual, and $u$ is in graded degree $-1$. The map $R[t]\to \varphi_*R$ is the composition of the Frobenius map $R\to \varphi_*R$ with $t\mapsto 0$ and the map $\varphi_*R[u]\to \varphi_*R$ is given by $u\mapsto 0$.

Geometrically, this is telling us that $\Rees(\Fil^\bullet_{\mathrm{Lau}}W_1(R))$ is obtained as follows. Consider the two canonical closed immersions
\begin{align}\label{eqn:lambda_maps}
\chi_+:B\Gm\times\Spec \varphi_*R&\to \Aff^1_+/\Gm\times\Spec \varphi_*R \\
\chi_-:B\Gm\times\Spec R&\to \Aff^1/\Gm\times\Spec R
\end{align}
of stacks over $R$ and let ${}^\varphi\chi_-$ be the composition
\[
B\Gm\times\Spec \varphi_*R\xrightarrow{\mathrm{id}\times\varphi}B\Gm\times\Spec R\xrightarrow{\chi_-}\Aff^1/\Gm\times\Spec R.
\]

Then we obtain $\Rees(\Fil^\bullet_{\mathrm{Lau}}W_1(R))$ by gluing the two closed substacks $\Aff^1/\Gm\times\Spec R$ and $\Aff^1_+/\Gm\times\Spec \varphi_*R$ along $B\Gm\times \Spec \varphi_*R$ via the maps ${}^\varphi\chi_-$ and $\chi_+$.

\subsubsection{}
Let us now consider the two maps $\tau,\sigma:\Spec R \to \Rees(\Fil^\bullet_{\mathrm{Lau}}W_1(R))$. For this, note that we have a canonical map
\[
\Rees(\Fil^\bullet_{\mathrm{Lau}}W_1(R))\to \Aff^1/\Gm\times \Aff^1_+/\Gm
\]
of $B\Gm$-stacks. It suffices to construct this for $R = \Field_p$, where it arises from the identity $\Field_p[t]\times_{\Field_p}\Field_p[u] = \Field_p[t,u]/(ut)$.

It is now not difficult to check that the map $\tau$ (resp. $\sigma$) is the open immersion obtained as the pre-image of the open locus $t\neq 0$ (resp. $u\neq 0$) of $\Aff^1/\Gm\times \Aff^1_+/\Gm$.

\begin{definition}
The stack $R^{\Fzip}$ (the \defnword{$F$-zip stack for $R$}) is the stack $\mathfrak{S}(\underline{W_1(R)})$ obtained as the coequalizer of the open immersions
\[
\tau,\sigma:\Spec R \to \Rees(\Fil^\bullet_{\mathrm{Lau}}W_1(R))
\]
\end{definition}

\begin{remark}
\label{rem:fzip_stack_alternate}
We now obtain the following alternative construction of $R^{\Fzip}$: We first glue $\mathbb{A}^1/\Gm\times \Spec R$ with $\mathbb{A}^1_+/\Gm\times \Spec \varphi_*R$ along the open substack $\Spec R\simeq \Gm/\Gm\times \Spec R$. Denote the resulting stack by $Y_R$: Note that this is \emph{not} a stack over $\Spec R$, only over $\Spec \Field_p$. 

$R^{\Fzip}$ is now obtained as the coequalizer of the two maps 
\[
B\Gm\times\Spec \varphi_*R\xrightarrow{\chi_+}\Aff^1_+/\Gm\times\Spec \varphi_*R\hookrightarrow Y_R\;;\; B\Gm\times\Spec R\xrightarrow{{}^\varphi\chi_-}\Aff^1/\Gm\times\Spec R\hookrightarrow Y_R.
\]
\end{remark}

\begin{remark}
This stack is the same as the stack $\mathfrak{X}_S$ defined by Yaylali in~\cite[Appendix]{epiga:10375} with $S = \Spec R$.
\end{remark}   

\begin{definition}
An \defnword{$F$-zip over $R$} is an $\underline{W_1(R)}$-gauge. Equivalently, it is an object in $\mathrm{QCoh}(R^{\Fzip})$.
\end{definition}

\begin{remark}
\label{rem:fzip_explicit_desc}
Using Remark~\ref{rem:fzip_stack_alternate}, one sees that giving an $F$-zip $\bm{M}$ over $R$ is equivalent to specifying the following data:
\begin{itemize}
   \item A decreasingly filtered complex $\Fil^\bullet_{\mathrm{Hdg}}M^-$ over $R$ obtained via pullback along the map $\Aff^1/\Gm\times\Spec R \to R^{\Fzip}$;
   \item An increasingly filtered complex $\Fil_\bullet^{\mathrm{conj}}M^+$ over $R$ obtained via pullback along $\Aff^1_+/\Gm\times\Spec R \to R^{\Fzip}$;
   \item An isomorphism $\eta:M^+\xrightarrow{\simeq} M^-$ in $\Mod{R}$ identifying both with a common $R$-module $M$;
   \item An isomorphism of graded $R$-modules
   \[
     \alpha:\gr_\bullet^{\mathrm{conj}}M^+\xrightarrow{\simeq}\varphi^*\gr^{-\bullet}_{\mathrm{Hdg}}M^-.
   \]
\end{itemize}
In the sequel, we will use the identification $M^+\simeq M^-\simeq M$ to drop all superscripts.
\end{remark}

\begin{remark}
When $R$ is classical and $\bm{M}$ is a vector bundle over $R^{\Fzip}$, we essentially recover the definition of Moonen-Pink-Wedhorn-Ziegler from~\cite{Pink2015-ye}. The de Rham cohomology of any smooth projective scheme over $R$ with degenerating Hodge-to-de Rham spectral sequence, when equipped with its decreasing Hodge filtration and its increasing conjugate filtration, yields an example of such an $F$-zip. The more general definition given here corresponds to that of a \emph{derived} $F$-zip as appearing in the work of Yaylali~\cite[\S 3.6]{epiga:10375}.
\end{remark}

\begin{remark}
\label{rem:fzip_ht_weights}
Note that the Hodge-Tate weights of an $F$-zip $\bm{M}$ as defined in Definition~\ref{defn:A-gauge_HT_weights}) are precisely the integers $i$ such that $\gr^{-i}_{\mathrm{Hdg}}M$ is not nullhomotopic.
\end{remark}

\begin{construction}
\label{const:fzip_cohomology}
Given an $F$-zip $\bm{M}$ over $R$, we obtain $\Mod{\Field_p}$-valued prestacks over $R$:
\begin{align*}
R\Gamma^*_{\Fzip}(\bm{M}):C&\mapsto \mathrm{RHom}_{\mathrm{QCoh}(C^{\Fzip})}(\bm{M}\vert_{C^{\Fzip}},\Reg{C^{\Fzip}})\;;\\
R\Gamma_{\Fzip}(\bm{M}):C&\mapsto R\Gamma(C^{\Fzip},\bm{M}\vert_{C^{\Fzip}}).
\end{align*}
We can make this `$F$-zip cohomology' quite explicit using the description of the stack from Remark~\ref{rem:fzip_stack_alternate}. Let us do this for the first functor. We obtain two maps
\[
q_1,q_2:\mathrm{RHom}_{R}(M/\Fil^{\mathrm{conj}}_{-1}M,C)\times_{\mathrm{RHom}_R(M,C)}\mathrm{RHom}_{R}(M/\Fil^1_{\mathrm{Hdg}}M,C)\to \mathrm{RHom}_R(\gr^{\mathrm{conj}}_0M,C),
\]
where the first is via restriction to $\gr^{\mathrm{conj}}_0M$, and is $R$-linear, while the second is via 
\[
\mathrm{RHom}_{R}(M/\Fil_{\mathrm{Hdg}}^{1}M,C) \to\mathrm{RHom}_{R}(\gr^0_{\mathrm{Hdg}}M,C) \xrightarrow{\varphi^*\circ \alpha} \mathrm{RHom}_R(\gr^{\mathrm{conj}}_0M,C)
\]
and so is $\varphi$-semilinear. We now have:
\begin{align}
\label{eqn:section_fzips_perf_comp}
R\Gamma^*_{F\mathrm{Zip}}(\bm{M})(C) \simeq \hker\left(q_1-q_2\right).
\end{align}
For a full justification of this isomorphism, one can argue as in the proof of Lemma~\ref{lem:gamma_fzip_description} below.
\end{construction}

\begin{remark}
\label{rem:fzip_cohom-perfect_f-gauges}
If $\bm{M}$ is perfect with dual $F$-zip $\bm{M}^\vee$, then we have a canonical isomorphism
\[
R\Gamma^*_{\Fzip}(\bm{M})\xrightarrow{\simeq}R\Gamma_{\Fzip}(\bm{M}^\vee).
\]
\end{remark}

\begin{remark}
\label{rem:fzips_from_A_gauges}
Suppose that $\underline{A}$ is a prismatic frame and that we have fixed an orientation for $(A,I)$ (we will simply refer to this data as an \emph{oriented} prismatic frame from now on). Then Remark~\ref{rem:mod_p_and_I_frames} tells us that we have a map of prestacks $R_{A^\modpI}^{\Fzip}\to \mathfrak{S}(\underline{A})$, where 
\[
R_{A^\modpI} = R_A/{}^{\mathbb{L}}(I,p).
\]
In particular, there is a symmetric monoidal functor 
\[
\underline{A}\mathsf{-gauge}\xrightarrow{(\mathcal{M},\xi)\mapsto \bm{M}^\modpI} \mathrm{QCoh}(R_{A^\modpI}^{\Fzip}).
\]
If $\underline{A}$ is a $p$-adic prismatic frame, then we actually have a map $R_A^{\Fzip}\to \mathfrak{S}(\underline{A})$, and so we obtain a symmetric monoidal functor
\[
\underline{A}\mathsf{-gauge}\xrightarrow{(\mathcal{M},\xi)\mapsto \bm{M}} \mathrm{QCoh}(R_{A}^{\Fzip})
\]
lifting the previous one.
\end{remark}

\subsection{Abstract deformation theory for $1$-bounded stacks}
\label{subsec:abstract_def_theory} 
For any frame $\underline{B}$, we can view $\Rees(\Fil^\bullet B)$ as an  $R_B$-pointed graded formal stack using Remark~\ref{rem:abstract_de_rham_point}. More precisely, for every $n,m\ge 1$, we can view $\Rees(\Fil^\bullet B)\otimes_B B/{}^{\mathbb{L}}(p^n,I^m)$ as an $R_B/{}^{\mathbb{L}}(p^n,I^m)$-pointed graded stack. We will assume for simplicity that $\underline{B}$ is an oriented prismatic frame as in Remark~\ref{rem:fzips_from_A_gauges}. Given a map of prismatic frames $\underline{B}\to \underline{C}$, $\underline{C}$ again inherits the orientation.

\begin{construction}
\label{const:1_bounded_sections}
Suppose that $\mathcal{X} = (\mathcal{X}^\preoneb,X^0)\to \Rees(\Fil^\bullet B)$ is a strongly integrable $1$-bounded formal stack. By this, we mean that it is an inverse system of strongly integrable (in the sense of Definition~\ref{defn:strongly_integrable}) $1$-bounded stacks over $\Rees(\Fil^\bullet B)\otimes_BB/{}^{\mathbb{L}}(p^n,I^m)$. Suppose also that it is equipped with an isomorphism $\xi:\sigma^*\mathcal{X}^\preoneb\xrightarrow{\simeq}\tau^*\mathcal{X}^\preoneb$ of formal stacks over $\Spf B$. Then, for any $p$-completely \'etale ${R_B}$-algebra ${R_{B'}}$ we will set
\[
\begin{diagram}
\Gamma_{\underline{B}}(\mathcal{X},\xi)(R_{B'}) \defn \mathrm{eq}\biggl(\Map(\Rees(\Fil^\bullet B'),\mathcal{X})&\pile{\rTo^{\xi\circ\sigma^*}\\ \rTo_{\tau^*}}&\mathcal{X}^\preoneb(B')\biggr).
\end{diagram}
\]
All mapping spaces here are of formal prestacks over $\Rees(\Fil^\bullet B)$, and we have set $\mathcal{X}^\preoneb(B') = \Map(\Spf B',\mathcal{X}^\preoneb)$, where we are viewing $\Spf(B')$ as a formal scheme over $\Rees(\Fil^\bullet B)$ via the map $\tau$.

Equivalently, if $\mathfrak{S}(\underline{B})$ is as in Remark~\ref{rem:B'Gmu_abstract}, then $\xi$ gives a descent $\mathcal{X}_\xi\to \mathfrak{S}(\underline{B})$ for $\mathcal{X}$, and we have
\[
\Gamma_{\underline{B}}(\mathcal{X},\xi)(R_{B'})  = \Map(\mathfrak{S}(\underline{B}),\mathcal{X}_\xi)
\]
\end{construction}

\begin{remark}
\label{rem:weil_restriction_commuting_square}
Suppose that the structure map for $\mathcal{X}^\preoneb$ factors through $\Rees(\Fil^\bullet B)\otimes\Int/p^n\Int$ for some $n\ge 1$. For clarity, denote the corresponding stack by
\[
\mathcal{Y}^\preoneb\to \Rees(\Fil^\bullet B)\otimes\Int/p^n\Int.
\]
Since $\Rees(\Fil^\bullet B)\otimes\Int/p^n\Int$ is $R_B/{}^{\mathbb{L}}p^n$-pointed, the asociated fixed point stack $Y^{\preoneb,0}$ lives over $R_B/{}^{\mathbb{L}}p^n$, and the fixed point stack $X^{\preoneb,0}\to \Spf R_B$ for $\mathcal{X}^\preoneb$ is obtained via Weil restriction (see \S\ref{subsec:weil}):  $X^{\preoneb,0} = \Res_{(R_B/{}^{\mathbb{L}}p^n)/R_B}Y^{\preoneb,0}$. The open substack $X^0\subset X^{\preoneb,0}$ of $1$-bounded points determines and is determined by a canonical open substack $Y^0\subset Y^{\preoneb,0}$ such that $X^0$ is the Weil restriction of $Y^0$ along $R_B\to R_B/{}^{\mathbb{L}}p^n$. The attractor stack $X^-$ is similarly the Weil restriction of $Y^- = Y^{\preoneb,-}\times_{Y^{\preoneb,0}}Y^0$.

In particular, we find that $\mathcal{X}$ arises from a $1$-bounded stack $\mathcal{Y} = (\mathcal{Y}^\preoneb,Y^0)$ over $\Rees(\Fil^\bullet B)\otimes\Int/p^n\Int$. Suppose now that the isomorphism $\xi$ arises from an isomorphism $\xi_n:\sigma^* \mathcal{Y}^\preoneb \xrightarrow{\simeq} \tau^* \mathcal{Y}^\preoneb$ of stacks over $B/{}^{\mathbb{L}}p^n$. Then we can define
\[
\begin{diagram}
\Gamma_{\underline{B}}(\mathcal{Y},\xi_n)(R_{B'}) \defn \mathrm{eq}\biggl(\Map(\Rees(\Fil^\bullet B')\otimes\Int/p^n\Int,\mathcal{Y})&\pile{\rTo^{\xi_n\circ\sigma^*}\\ \rTo_{\tau^*}}&\mathcal{Y}^\preoneb(B'/{}^{\mathbb{L}}p^n)\biggr).
\end{diagram}
\]
We now have a canonical isomorphism of prestacks $\Gamma_{\underline{B}}(\mathcal{Y},\xi_n)\xrightarrow{\simeq}\Gamma_{\underline{B}}(\mathcal{X},\xi)$
\end{remark}

 \begin{remark}
\label{rem:1_bounded_sections_dependence}
By the filtered integrability of $\mathcal{X}$, we have
\begin{align*}
\Map(\Rees(\Fil^\bullet B'),\mathcal{X})&\simeq \Map(\Aff^1/\Gm\times \Spf R_{B'},\mathcal{X})\times_{\mathcal{X}^\preoneb(R_{B'})}\mathcal{X}^\preoneb(B').
\end{align*}
Indeed, it is enough to know that we have
\[
\Map(\Rees(\Fil^\bullet B'/{}^{\mathbb{L}}(p^n,I^m)),\mathcal{X})\simeq \Map(\Aff^1/\Gm\times \Spec R_{B'}/{}^{\mathbb{L}}(p^n,I^m),\mathcal{X})\times_{\mathcal{X}^\preoneb(R_{B'}/{}^{\mathbb{L}}(p^n,I^m))}\mathcal{X}^\preoneb(B'/{}^{\mathbb{L}}(p^n,I^m))
\]
for each $n,m\ge 1$, and this follows from the filtered integrability assumption, because the kernel of $\pi_0(B'/{}^{\mathbb{L}}(p,I))\to \pi_0(R_{B'}{}^{\mathbb{L}}(p,I))$ is locally nilpotent, as observed in the proof of Proposition~\ref{prop:frame_props_hensel}.
\end{remark}

\begin{construction}
\label{const:maps_between_frames}
Let $q:\underline{B}\to \underline{A}$ be a map of frames such that the associated map ${R_B}\to {R_A}$ is a locally nilpotent thickening---that is, $\pi_0({R_B})\to \pi_0({R_A})$ is a surjection with locally nilpotent kernel. We now have a canonical equivalence of small $p$-completely \'etale sites $(R_B)_{\et}\xrightarrow[\simeq]{B'\mapsto A'}(R_A)_{\et}$. Write $\Gamma_{\underline{A}}(\mathcal{X},\xi)$ for the sheaf on $(R_B)_{\et}\xrightarrow{\simeq}(R_A)_{\et}$ obtained from the pullback of $(\mathcal{X},\xi)$ over $\Rees(\Fil^\bullet A)$. We then obtain a canonical map of sheaves
\[
q_*:\Gamma_{\underline{B}}(\mathcal{X},\xi)\to \Gamma_{\underline{A}}(\mathcal{X},\xi).
\]
\end{construction}

\begin{definition}
Let $X^{-}$ be the (formal) attractor on $p$-complete $R_B$-algebras given by
\[
R\mapsto \Map(\Aff^1/\Gm\times \Spf R,\mathcal{X}),
\]
and let $X$ be the stack $R\mapsto \mathcal{X}^\preoneb(R) = \Map(\Spf R,\mathcal{X}^\preoneb)$. For any map of frames $\underline{B}\to \underline{A}$ as in Construction~\ref{const:maps_between_frames}, write $X^-_A$ (resp. $X_A$) for the sheaf $X^-_A:R_{B'}\mapsto X^-(R_{A'})$ (resp. $X_A:R_{B'}\mapsto X(R_{A'})$) on $(R_B)_{\et}$.
\end{definition}

\begin{construction}
\label{const:commuting_square_frames}
Suppose that $B\twoheadrightarrow A$ is surjective, and set $K = \hker(B\twoheadrightarrow A)$. Suppose in addition that the map $B\to {R_B}$ factors through a map $\pi:A\to {R_B}$ lifting $A\to {R_A}$; equivalently, $K\to B$ factors through a map $\Fil^1B\to B$. Note that this is trivially the case if either $B\xrightarrow{\simeq}A$ or ${R_B}\xrightarrow{\simeq}{R_A}$. Then we have a commuting diagram
\[
\Square{\Gamma_{\underline{B}}(\mathcal{X},\xi)}{}{X^{-}}{q_*}{}{\Gamma_{\underline{A}}(\mathcal{X},\xi)}{}{X^{-}_A\times_{X_A}X}.
\]
of sheaves on $(R_B)_{\et}$. Here, the top arrow arises from pullback along the map from Remark~\ref{rem:abstract_de_rham_point} (for $\underline{B}$), while that on the bottom is from this map (for $\underline{A}$) in the first co-ordinate combined with pullback along $\tau\circ \Spf(\pi):\Spf R_B\to \Rees(\Fil^\bullet A)$ in the second co-ordinate.
\end{construction}

There is a trivial situation in which the square from Construction~\ref{const:commuting_square_frames} is Cartesian.
\begin{proposition}
\label{prop:def_theory_frames_trivial}
Suppose that $q:B\xrightarrow{\simeq}A$, and the lift $A\xrightarrow{\simeq}B\to {R_B}$ is the obvious one. Then the square in Construction~\ref{const:commuting_square_frames} obtained via the map $A\xrightarrow{\simeq}B\to {R_B}$ is Cartesian.
\end{proposition}
\begin{proof}
For any $p$-completely \'etale map ${R_B}\to {R_{B'}}$ reducing to ${R_A}\to {R_{A'}}$, first note that our hypothesis says that $X(B')\xrightarrow{\simeq}X(A')$.

By Remark~\ref{rem:1_bounded_sections_dependence}, we have
\begin{align*}
\Map(\Rees(\Fil^\bullet B'),\mathcal{X})&\simeq X^{-}(R_{B'})\times_{X(R_{B'})}\mathcal{X}^\preoneb(B');\\
\Map(\Rees(\Fil^\bullet A'),\mathcal{X})&\simeq X^{-}(R_{A'})\times_{X(R_{A'})}\mathcal{X}^\preoneb(A').
\end{align*}

The proposition follows quite formally now from the two previous paragraphs.
\end{proof}

\begin{remark}
\label{rem:phi1_lifts}
With the hypotheses of Construction~\ref{const:commuting_square_frames}, set $\widetilde{\Fil}^1A = \hker(\pi: A\to {R_B})$, which factors through $\Fil^1A\to A$ via a map $v:\widetilde{\Fil}^1A \to \Fil^1A$. This gives us a canonical $\varphi$-semilinear map $\dot{\varphi}_1:K\to K$ such that we have a commuting diagram with exact rows\footnote{Here and elsewhere we are using the orientation on $(B,I)$ to identify the target of the divided Frobenius lift with $B$.}:
\[
\begin{diagram}
K&\rTo&\Fil^1B&\rTo&\widetilde{\Fil}^1A\\
\dTo^{\dot{\varphi}_1}&&\dTo^{\varphi_1}&&\dTo_{(\varphi_1)\circ v}\\
K&\rTo&B&\rTo&A.
\end{diagram}
\]
\end{remark}

\begin{proposition}
\label{prop:def_theory_frames}
With the hypotheses and notation of Remark~\ref{rem:phi1_lifts}, suppose that the following additional conditions hold:
\begin{enumerate}
   \item $\underline{B}$ and $\underline{A}$ are prismatic.
   \item $\pi_0(\Fil^1 B)\twoheadrightarrow \pi_0(\Fil^1 A)$ is surjective; equivalently,
   \[
\Fil^1K = \hker(\Fil^1B\to \Fil^1A)
\]
is connective.

   \item The map $K\to K$ induced by $\dot{\varphi}_1$ is topologically locally nilpotent; equivalently, the endomorphism induced on $K/{}^{\mathbb{L}}(p,I)$ is locally nilpotent.\footnote{That is, an endomorphism that is a filtered colimit of nilpotent endomorphisms.}
\end{enumerate}
Then the commuting square in Construction~\ref{const:commuting_square_frames} is Cartesian.
\end{proposition}

\begin{corollary}
\label{cor:def_theory_frames}
With the hypotheses as in Proposition~\ref{prop:def_theory_frames}, if ${R_B}\xrightarrow{\simeq}{R_A}$, then $q_*$ is an equivalence.
\end{corollary}

\begin{remark}
The proof of Proposition~\ref{prop:def_theory_frames}, which is the main result of this subsection, will be given below, and is inspired by the arguments of Lau~\cite{MR4355476}, who, in our language here, considered the case of the $1$-bounded stack yielding $(G,\mu)$-windows. Similar arguments, building on the work of Zink~\cite{MR1827031} for nilpotent displays, also appear in ~\cite{MR4120808},~\cite{bartling2022mathcalgmudisplayslocalshtuka} and~\cite{hedayatzadeh_partofard}. All of these arguments involve some kind of nilpotence criterion. Our result here is closest in spirit to that of Lau, since the locus of nilpotence is on the fiber $K$, while in the other works, it is manifested in the display or window that is being deformed. 
\end{remark}

\begin{remark}
\label{rem:witt_surjective_def_theory_application}
Suppose that $\underline{B}$ is a laminated prismatic frame. Then Lemma~\ref{lem:laminated_frames} gives us a canonical map of frames $q:\underline{B}\to \underline{A}\defn \underline{W(R_B)}$. The map $B\to W(R_B)$ is surjective if and only if $B$ surjects onto $W(C)/pW(C)$ where we have set $C = \pi_0(R_B)/p\pi_0(R_B)$. If $C$ is semiperfect, then we have $W(C)/pW(C) = C$, and so this surjectivity is immediate. Therefore, if we knew the topological local nilpotence of the endomorphism $\dot{\varphi}_1$, then it would follow from Corollary~\ref{cor:def_theory_frames} that $q_*$ is an equivalence, and so we can compute $\Gamma_{\underline{B}}(\mathcal{X},\xi)$ using the Witt frame. In practice, this kind of nilpotence is seldom true. However, one can salvage the situation by transferring the nilpotence condition to the pair $(\mathcal{X},\xi)$ instead. See Remark~\ref{rem:nilpotent_locus_witt} below.
\end{remark}

We now begin our preparations for the proof.
\begin{remark}
\label{rem:frames_surjective}
By $(p,I)$-completeness, to see that $B\to A$ is surjective, it is enough to know that $B/{}^{\mathbb{L}}(p,I)\to A/{}^{\mathbb{L}}(p,I)$ is so. Moreover, in the situation of Construction~\ref{const:commuting_square_frames}, we have a commuting diagram with exact rows
\[
\begin{diagram}
K&\rTo&\Fil^1B&\rTo&\widetilde{\Fil}^1A\\
\dTo&&\dTo&&\dTo\\
\Fil^1K&\rTo&\Fil^1B&\rTo&\Fil^1A\\
\dTo&&\dTo&&\dTo\\
K&\rTo&B&\rTo&A
\end{diagram}
\]
where the composition of the vertical arrows on the left is isomorphic to the identity on $K$. In other words, we have a section $K\to \Fil^1K$ splitting the fiber sequence
\[
R_K[-1]\to \Fil^1K\to K,
\]
where $R_K = \hker(R_B\twoheadrightarrow R_A)$. Therefore, we have $\Fil^1K \xrightarrow{\simeq}K\oplus R_K[-1]$. This shows in particular that condition (2) of Proposition~\ref{prop:def_theory_frames} holds if and only if $R_K$ is $1$-connective.
\end{remark}

\begin{construction}
\label{const:K_B_gauge}
Set $\Fil^\bullet K \defn \hker(\Fil^\bullet B\to \Fil^\bullet A)$: The associated quasicoherent sheaf $\mathcal{K}$ over $\Rees(\Fil^\bullet B)$ naturally underlies an $\underline{B}$-gauge $(\mathcal{K},\xi)$, where $\xi$ is determined by the fact that both $\sigma^* \mathcal{K}$ and $\tau^* \mathcal{K}$ can be canonically identified with the quasicoherent sheaf over $\Spf B$ associated with $K$. 
\end{construction}

\begin{lemma}
\label{lem:factors_through_hodge}
There is a canonical fiber sequence of $\underline{B}$-gauges
\[
\mathcal{K}'\to \mathcal{K}\to \mathcal{M}_{\mathrm{par}}(R_K[-1],1)
\]
where $\mathcal{K}'$ has Hodge-Tate weights $\le -1$. Here, $\mathcal{M}_{\mathrm{par}}(R_K[-1],1)$ is the parasitic $\underline{B}$-gauge from Example~\ref{ex:parasitic_gauges}. 
\end{lemma}
\begin{proof}
Let $\Fil^\bullet_{\mathrm{Hdg}}(R_B\otimes_BK)$ be the filtered base-change of $\Fil^\bullet K$ along $\Fil^\bullet B\to \Fil^\bullet_{\mathrm{triv}}R_B$. Then  above factors though $\Fil^1_{\mathrm{Hdg}}(R_B\otimes_BK)$. As observed in Example~\ref{ex:parasitic_gauges}, giving the second map in the purported fiber sequence amounts to giving a map $\Fil^1_{\mathrm{Hdg}}(R_B\otimes_BK)\to R_K[-1]$. To see this, it is enough to know that the map $\Fil^1K\to R_K[-1]$ induced by the splitting from Remark~\ref{rem:frames_surjective} factors through $\Fil^1_{\mathrm{Hdg}}(R_B\otimes_BK)\to R_K[-1]$. Now, there is a natural fiber sequence
\[
\Fil^1B\otimes_BK \to \Fil^1K \to \Fil^1_{\mathrm{Hdg}}(R_B\otimes_BK),
\]
so we want to show that the induced map $\Fil^1B\otimes_BK\to R_K[-1]$ is nullhomotopic. Equivalently, we want to know that it factors through the map $K\to \Fil^1K$ defined in Remark~\ref{rem:frames_surjective}. Unraveling definitions, this amounts to knowing that $\Fil^1B\otimes_BA\to A$ factors through $\widetilde{\Fil}^1A = \hker(A\to R_B)$, which is clear.

To finish, we need to check that $\mathcal{K}' = \hker(\mathcal{K}\to \mathcal{M}_{\mathrm{par}}(R_K[-1],1))$ has Hodge-Tate weights $\leq -1$. This amounts to knowing that, if $\Fil^\bullet K'$ is the associated filtered module, then $\gr^0K' \simeq 0$. Equivalently, we want to know that the map $\gr^0K \to \gr^0\mathsf{M}_{\mathrm{par}}(R_K[-1],1)$ is an isomorphism; but the source and target of this map are both canonically identified with $R_K$ with the map being identified with the identity.
\end{proof}

The key tool for the proof of Proposition~\ref{prop:def_theory_frames} is deformation theory in the form of the next lemma.
\begin{lemma}
\label{lem:def_theory_for_prismatic_frames}
Suppose that we have a commuting diagram of frames
\[
\Square{\underline{B}}{}{\underline{A}}{}{}{\underline{D}}{}{\underline{C}}
\]
satisfying the following properties:
\begin{enumerate}
   \item The horizontal arrows satisfy the hypotheses of Proposition~\ref{prop:def_theory_frames}. 
   \item The lift $C\to R_D$ is compatible with the lift $A\to R_B$ in that we have a commuting square
   \[
   \Square{A}{}{C}{}{}{R_B}{}{R_D}.
   \]
   \item The vertical arrows are square-zero extensions. More precisely, the underlying maps $\Fil^\bullet B \to \Fil^\bullet D$ and $\Fil^\bullet A \to \Fil^\bullet C$ are square-zero extensions of filtered animated commutative rings.

   \item The commuting square 
   \[
\Square{\Gamma_{\underline{D}}(\mathcal{X},\xi)}{}{X^{-}_D}{}{}{\Gamma_{\underline{C}}(\mathcal{X},\xi)}{}{X^{-}_C\times_{X_C}X_D}.
\]
from Construction~\ref{const:commuting_square_frames} applied to $\underline{D}\to \underline{C}$ is Cartesian.
\end{enumerate}
Then the corresponding commuting square for $\underline{B}\to \underline{A}$ is also Cartesian.
\end{lemma}

The bulk of the rest of this subsection will be devoted to the proof of Lemma~\ref{lem:def_theory_for_prismatic_frames}. For now, let us see that the lemma implies the proposition.
\begin{proof}
[Proof of Proposition~\ref{prop:def_theory_frames} assuming Lemma~\ref{lem:def_theory_for_prismatic_frames}]
Applying the lemma to the squares
\[
\Square{\tau_{\leq (k+1)}\underline{B}}{}{\tau_{\leq (k+1)}\underline{A}}{}{}{\tau_{\leq k}\underline{B}}{}{\tau_{\leq k}\underline{A}}
\]
for $k\ge 0$ and using nilcompleteness for the stacks involved, we reduce to the case where $\mathcal{R}(\Fil^\bullet B)$ and $\mathcal{R}(\Fil^\bullet A)$ are classical $(p,I)$-complete formal stacks (with bounded $p$-power and $I$-power torsion).

Now, the images of $K$ in $B/(p,I)B$ and $\pi_0(R_B)/p\pi_0(R_B)$ are locally nilpotent. Therefore, using the strong integrability hypothesis, we find that if we set\footnote{We are implicitly using the following fact: If $I,J\subset B$ are $\delta$-ideals in the classical $\delta$-ring $B$ (that is, we have $\delta(I)\subset I,\delta(J)\subset J$, or equivalently the $\delta$-structure descends to both $B/I$ and $B/J$), then $IJ$ is once again a $\delta$-ideal. This follows easily from the defining identites $\delta(x+y) = \delta(x)+\delta(y) + \sum_{i=1}^{p-1}\frac{1}{p}\binom{p}{i}x^iy^{p-i}$ and $\delta(xy) = x^p\delta(y)+y^p\delta(x)+p\delta(x)\delta(y)$.}
\[
\underline{B}_m = \pi_0(B/K^m\otimes_B\underline{B}),
\]
then we have
\begin{align*}
\Gamma_{\underline{B}}(\mathcal{X},\xi) = \varprojlim_{m}\Gamma_{B_m}(\mathcal{X},\xi)\;;\;X^- = \varprojlim_m X^-_{B_m}\;;\; X = \varprojlim_m X_{B_m}.
\end{align*}
Therefore, by applying the lemma to the squares
\[
\begin{diagram}
\underline{B}_{m+1}&\rTo&\underline{A}\\
\dTo&&\dEquals\\
\underline{B}_m&\rTo&\underline{A},
\end{diagram}
\]
we reduce to the case $B=B_1=A$. Here, the result follows from Proposition~\ref{prop:def_theory_frames_trivial}.
\end{proof}

\begin{construction}
\label{const:A_gauges_RMap}
Suppose that we have $\underline{A}$-gauges $(\mathcal{M},\xi)$, $(\mathcal{M}',\xi')$ corresponding to $(p,I)$-complete filtered modules $\Fil^\bullet \mathsf{M}$, $\Fil^\bullet\mathsf{M}'$ over $\Fil^\bullet A$, with underlying $A$-modules $\mathsf{M}$ and $\mathsf{M}'$, and isomorphisms
\[
\xi:\mathsf{M}_\sigma\xrightarrow{\simeq}\mathsf{M}\;;\;\xi':\mathsf{M}'_\sigma\xrightarrow{\simeq}\mathsf{M}'
\]
of $A$-modules. Here, $\mathsf{M}_\sigma$ and $\mathsf{M}'_\sigma$ are the $A$-modules underlying the quasicoherent sheaves obtained from $\mathcal{M}$ and $\mathcal{M}'$ via restriction along $\sigma$. Set
\[
RH_{\underline{A}}(\mathcal{M},\mathcal{M}')({R_A})\defn \hker(\mathrm{RHom}_{\Fil^\bullet A}(\Fil^\bullet \mathsf{M},\Fil^\bullet \mathsf{M}')\xrightarrow{\xi'\circ\sigma^*\circ \xi^{-1}-\tau^*}\mathrm{RHom}_{A}(\mathsf{M},\mathsf{M}')).
\]
Put more succinctly, we have
\[
RH_{\underline{A}}(\mathcal{M},\mathcal{M}')({R_A}) = \mathrm{RHom}_{\underline{A}\mathsf{-gauge}}((\mathcal{M},\xi),(\mathcal{M}',\xi')).
\]
\end{construction}

\begin{construction}
\label{const:operator_needs_to_be_nilpotent}
Suppose that $\mathcal{M}$ (resp. $\mathcal{M}'$) has Hodge-Tate weights $\ge -1$ (resp. $\leq -1$). Let $\bm{M}^\modpI$ and $\bm{M}^{',\modpI}$ be the associated $F$-zips over $R_{A^\modpI}$ (see Remark~\ref{rem:fzips_from_A_gauges}). They correspond to tuples
\[
(\Fil^\bullet_{\mathrm{Hdg}}M^\modpI,\Fil^{\mathrm{conj}}M^\modpI,\varphi^*\gr^\bullet_{\mathrm{Hdg}}M^\modpI\xrightarrow{\simeq}\gr^{\mathrm{conj}}_\bullet M^\modpI)\;;\; (\Fil^\bullet_{\mathrm{Hdg}}M^{',\modpI},\Fil^{\mathrm{conj}}M^{',\modpI},\varphi^*\gr^\bullet_{\mathrm{Hdg}}M^{',\modpI}\xrightarrow{\simeq}\gr^{\mathrm{conj}}_\bullet M^{',\modpI}).
\]
The hypotheses on the weights tell us that 
\begin{align*}
\gr^i_{\mathrm{Hdg}}M^\modpI \simeq \gr^{\mathrm{conj}}_iM^\modpI \simeq 0\text{ for $i\le 1$};\\
\gr^i_{\mathrm{Hdg}}M^{',\modpI} \simeq \gr^{\mathrm{conj}}_iM^{',\modpI} \simeq 0\text{ for $i\ge 1$};
\end{align*}
Therefore, we obtain maps
\begin{align*}
\psi_{\bm{M}^\modpI}:\Fil^1_{\mathrm{Hdg}}M^\modpI\to M^\modpI \to \gr^{\mathrm{conj}}_1M^\modpI\xrightarrow{\simeq}\varphi^*\Fil^1_{\mathrm{Hdg}}M^\modpI\\
\psi^*_{\bm{M}^{',\modpI}}:\varphi^*\gr^1_{\mathrm{Hdg}}M^{',\modpI}\xrightarrow{\simeq}\Fil^{\mathrm{conj}}_1M^{',\modpI}\to M^{',\modpI} \to \gr^1_{\mathrm{Hdg}}M^{',\modpI}.
\end{align*}
This yields an operator 
\[
\psi^\modpI_{\bm{M},\bm{M}'}: \mathrm{RHom}_{R_{A^\modpI}}(\Fil^1_{\mathrm{Hdg}}M^\modpI,\gr^{1}_{\mathrm{Hdg}}M^{',\modpI})\to \mathrm{RHom}_{R_{A^\modpI}}(\Fil^1_{\mathrm{Hdg}}M^\modpI,\gr^{1}_{\mathrm{Hdg}}M^{',\modpI})
\]
given as the composition
\begin{align*}
\mathrm{RHom}_{R_{A^\modpI}}(\Fil^1_{\mathrm{Hdg}}M^\modpI,\gr^{1}_{\mathrm{Hdg}}M^{',\modpI})&\xrightarrow{\varphi^*}\mathrm{RHom}_{R_{A^\modpI}}(\varphi^*\Fil^1_{\mathrm{Hdg}}M^\modpI,\varphi^*\gr^{1}_{\mathrm{Hdg}}M^{',\modpI})\\
&\simeq \mathrm{RHom}_{R_{A^\modpI}}(\gr^{\mathrm{conj}}_{1}M^\modpI,\Fil^{\mathrm{conj}}_{1}M^{',\modpI})\\
&\to  \mathrm{RHom}_{R_{A^\modpI}}(M^\modpI,M^{',\modpI})\\
&\to \mathrm{RHom}_{R_{A^\modpI}}(\Fil^1_{\mathrm{Hdg}}M^\modpI,\gr^{1}_{\mathrm{Hdg}}M^{',\modpI}).
\end{align*}
Here, the penultimate (resp. last) arrow is obtained via precomposition along $M^\modpI\to \gr^{\mathrm{conj}}_{1}M^\modpI$ (resp. $\Fil^1_{\mathrm{Hdg}}M^\modpI\to M^\modpI$) and postcomposition along $\Fil^{\mathrm{conj}}_1M^{',\modpI}\to M^{',\modpI}$ (resp. $M^{',\modpI}\to \gr^1_{\mathrm{Hdg}}M^{',\modpI}$).
\end{construction}

\begin{remark}
\label{rem:operator_needs_to_be_nilpotent}
If $\underline{A}$ is $p$-adic and $R_A$ is an $\Field_p$-algebra, then we obtain $F$-zips $\bm{M},\bm{M}'$ over $R_A$ from which $\bm{M}^\modpI,\bm{M}^{',\modpI}$ are produced via base-change along the natural map $R_A\to R_{A^\modpI}$. Moreover, the operator $\psi^\modpI_{\bm{M},\bm{M}'}$ arises from an operator
\[
\psi_{\bm{M},\bm{M}'}:\mathrm{RHom}_{R_{A}}(\Fil^1_{\mathrm{Hdg}}M,\gr^{1}_{\mathrm{Hdg}}M^{'})\to \mathrm{RHom}_{R_{A}}(\Fil^1_{\mathrm{Hdg}}M,\gr^{1}_{\mathrm{Hdg}}M^{'}).
\]
Since $R_{A^\modpI} = R/{}^{\mathbb{L}}(p,I)\simeq R/{}^{\mathbb{L}}(0,0)$ is an iterated square-zero extension of $R_A$, $\psi^\modpI_{\bm{M},\bm{M}'}$ is locally nilpotent if and only if $\psi_{\bm{M},\bm{M}^{'}}$ is so.
\end{remark}

\begin{lemma}
\label{lem:A_gauge_nilpotent}
Suppose that we have $\underline{A}$-gauges $(\mathcal{M},\xi)$ and $(\mathcal{M}',\xi)$ with the following properties:
\begin{enumerate}
   \item $\mathcal{M}$ is almost perfect with Hodge-Tate weights $\ge -1$ and $\mathcal{M}'$ has Hodge-Tate weights $\leq -1$;
   \item The operator $\psi^\modpI_{\bm{M},\bm{M}'}$ is locally nilpotent.
\end{enumerate}
Then $RH_{\underline{A}}(\mathcal{M},\mathcal{M}')({R_A})\simeq 0$.
\end{lemma} 
\begin{proof}
With the notation as in Construction~\ref{const:A_gauges_RMap}, we want to show that the map
\[
\xi'\circ \sigma^*\circ \xi^{-1} - \tau^*:\mathrm{RHom}_{\Fil^\bullet A}(\Fil^\bullet \mathsf{M},\Fil^\bullet \mathsf{M}'')\to \mathrm{RHom}_{A}(\mathsf{M},\mathsf{M}')
\]
is an equivalence.

Let $\Fil^\bullet_{\mathrm{Hdg}}M'$ be the filtered $R_A$-module obtained via pullback along $\Aff^1/\Gm\times\Spf R_A \to \Rees(\Fil^\bullet A)$. Our hypothesis on the Hodge-Tate weights of $\mathcal{M}'$ implies
\[
\gr^{1}\mathsf{M}'\xrightarrow{\simeq} \gr^{1}_{\mathrm{Hdg}}M'\;;\; \gr^i\mathsf{M}'\simeq 0\text{ for $i<1$}.
\]
This can be seen for instance from Lemma~\ref{lem:graded_weight_filtration} and its proof. Arguing as in the proof of Proposition~\ref{prop:1_bounded_cartesian}, and using the condition on the Hodge-Tate weights of $\mathcal{M}$, one now finds that the natural maps
\[
\mathrm{RHom}_{\Fil^\bullet A}(\Fil^\bullet \mathsf{M},\Fil^\bullet \mathsf{M}')\to \mathrm{RHom}_{\Fil^\bullet A}(\Fil^\bullet \mathsf{M},\Fil^\bullet \mathsf{M}'(-j))
\]
are equivalences for all $j\ge 1$, and this shows that the map
\[
\tau^*:\mathrm{RHom}_{\Fil^\bullet A}(\Fil^\bullet \mathsf{M},\Fil^\bullet \mathsf{M}')\to \mathrm{RHom}_{A}(\mathsf{M},\mathsf{M}')
\]
is an equivalence.

To finish, therefore, it is enough to know that the composition
\[
\mathrm{RHom}_{A}(\mathsf{M},\mathsf{M}')\xrightarrow{(\tau^*)^{-1}}\mathrm{RHom}_{\Fil^\bullet A}(\Fil^\bullet \mathsf{M},\Fil^\bullet \mathsf{M}')\xrightarrow{\xi'\circ \sigma^*\circ \xi^{-1}}\mathrm{RHom}_{A}(\mathsf{M},\mathsf{M}')
\]
is a topologically locally nilpotent endomorphism. We will in fact see that, with $\mathcal{M}^{',\modpI} = \mathcal{M}'/{}^{\mathbb{L}}(p,I)$, the induced endomorphism of $\mathrm{RHom}_{A}(\mathsf{M},\mathsf{M}^{',\modpI})$ is locally nilpotent. 

Let $\Fil^\bullet \varphi^*\mathsf{M}$ and $\Fil^\bullet\varphi^*\mathsf{M}'$ be the filtered complexes over $\Fil^\bullet_IA$ obtained via base-change along $\Phi:\Fil^\bullet A\to \Fil^\bullet_IA$. Then the condition on the Hodge-Tate weights of $\mathcal{M}$ tells us that we have
\[
\mathsf{M}_\sigma \xrightarrow{\simeq}I^{-1}\otimes_A\Fil^1\varphi^*\mathsf{M}\simeq \Fil^1\varphi^*\mathsf{M}, 
\]
where we have used the chosen orientation for $I$ for the last isomorphism. From this, we obtain a map
\[
\overline{\zeta}:\mathsf{M}\xrightarrow[\simeq]{\xi^{-1}}\mathsf{M}_\sigma\xrightarrow{\simeq}\Fil^1\varphi^*\mathsf{M}\to \overline{A}\otimes_{\overline{\varphi},R_A}\Fil^1_{\mathrm{Hdg}}M.
\]
On the other hand, consider the composition
\[
\varphi^*\mathsf{M}'= \Fil^1\varphi^*\mathsf{M}'\xrightarrow{\simeq}I^{-1}\otimes_A\Fil^1\varphi^*\mathsf{M}'\to \mathsf{M}'_\sigma\xrightarrow[\simeq]{\xi'}\mathsf{M}'.
\]
This factors through a map
\[
\overline{\xi}:\overline{A}\otimes_{\overline\varphi,R_A}\gr^{1}M^{',\modpI}\xrightarrow{\simeq}\gr^1\varphi^*\mathsf{M}^{',\modpI}\to \mathsf{M}^{',\modpI}.
\]

One now checks that the relevant endomorphism of $\mathrm{RHom}_{A}(\mathsf{M},\mathsf{M}^{',\modpI})$ is given as the composition
\begin{align*}
   \mathrm{RHom}_{A}(\mathsf{M},\mathsf{M}^{',\modpI})&\to \mathrm{RHom}_{R_{A^\modpI}}(M^\modpI,M^{',\modpI})\\
   &\to \mathrm{RHom}_{R_{A^\modpI}}(\Fil^1_{\mathrm{Hdg}}M^\modpI,\gr^1_{\mathrm{Hdg}}M^{',\modpI})\\
   &\xrightarrow{\overline{A}\otimes_{\overline\varphi,R_A}(\cdot)}\mathrm{RHom}_{\overline{A}^\modpI}(\overline{A}\otimes_{\overline\varphi,R_A}\Fil^1_{\mathrm{Hdg}}M^\modpI,\overline{A}\otimes_{\overline\varphi,R_A}\gr^1_{\mathrm{Hdg}}M^{',\modpI})\\
   &\xrightarrow{\overline\xi\circ(\cdot)\circ\overline{\zeta}} \mathrm{RHom}_{A}(\mathsf{M},\mathsf{M}^{',\modpI}).
\end{align*}
This shows that it is enough to know the local nilpotence of the endomorphism obtained as the composition
\begin{align*}
 \mathrm{RHom}_{R_{A^\modpI}}(\Fil^1_{\mathrm{Hdg}}M^\modpI,\gr^1_{\mathrm{Hdg}}M^{',\modpI})&\xrightarrow{\overline{A}\otimes_{\overline\varphi,R_A}(\cdot)}\mathrm{RHom}_{\overline{A}^\modpI}(\overline{A}\otimes_{\overline\varphi,R_A}\Fil^1_{\mathrm{Hdg}}M^\modpI,\overline{A}\otimes_{\overline\varphi,R_A}\gr^1_{\mathrm{Hdg}}M^{',\modpI})\\
   &\xrightarrow{\overline\xi\circ(\cdot)\circ\overline{\zeta}} \mathrm{RHom}_{A}(\mathsf{M},\mathsf{M}^{',\modpI})\\
   &\to \mathrm{RHom}_{R_{A^\modpI}}(M^\modpI,M^{',\modpI})\\
    &\to \mathrm{RHom}_{R_{A^\modpI}}(\Fil^1_{\mathrm{Hdg}}M^\modpI,\gr^1_{\mathrm{Hdg}}M^{',\modpI})
\end{align*}
But this is nothing but $\psi^\modpI_{\bm{M},\bm{M}'}$!
\end{proof}

\begin{remark}
\label{rem:nilpotence_conditions}
The local nilpotence of $\psi^\modpI_{\bm{M},\bm{M}'}$ can be deduced from two finer conditions:
\begin{enumerate}
   \item A condition on $\psi_{\bm{M}^\modpI}$: Successive pullback and composition yields a map
   \[
     \Phi_n:\Map_{R_{A^\modpI}}(\Fil^1_{\mathrm{Hdg}}M^\modpI,\varphi^*\Fil^1_{\mathrm{Hdg}}M^\modpI)\xrightarrow{f\mapsto (\varphi^n)^*f \circ (\varphi^{n-1})^*f \circ \cdots\circ \varphi^*f \circ f}\Map_{R_{A^\modpI}}(\Fil^1_{\mathrm{Hdg}}M^\modpI,(\varphi^{n+1})^*\Fil^1_{\mathrm{Hdg}}M^\modpI)
   \]
   If we have $\Phi_n(\psi_{\bm{M}^\modpI})\simeq 0$ for some $n$, then $\psi^\modpI_{\bm{M},\bm{M}'}$ will be nilpotent.

   \item A condition on $\psi^*_{\bm{M}^{',\modpI}}$: We can view $\psi^*_{\bm{M}^{',\modpI}}$ as an endomorphism in the classical derived category of the connective complex of $\Field_p$-vector spaces underlying $\gr^1_{\mathrm{Hdg}}M^{',\modpI}$; that is, as an element of $\End_{D^{\leq 0}(\Field_p)}(\gr^1_{\mathrm{Hdg}}M^{',\modpI})$. Within this ring, we have the ideal of locally nilpotent endomorphisms
   \[
   \mathrm{End}^{\mathrm{nilp}}_{D^{\leq 0}(\Field_p)}(\gr^1_{\mathrm{Hdg}}M^{',\modpI}),
   \]
   and we can ask that $\psi^*_{\bm{M}^{',\modpI}}$ belong to this ideal. This will imply that $\psi^\modpI_{\bm{M},\bm{M}'}$ is locally nilpotent.
\end{enumerate}
\end{remark}

\begin{construction}
\label{const:f_zip_ideal_psi_star}
Let us put ourselves in the setup of Lemma~\ref{lem:def_theory_for_prismatic_frames}. Set
\begin{align*}
\Fil^\bullet J = \hker(\Fil^\bullet B\to \Fil^\bullet D)\;&;\; \Fil^\bullet I = \hker(\Fil^\bullet A\to \Fil^\bullet C);\\
\Fil^\bullet P = \hker(\Fil^\bullet D\to \Fil^\bullet C)\;&;\; \Fil^\bullet N = \hker(\Fil^\bullet J\to \Fil^\bullet I)\simeq \hker(\Fil^\bullet K\to \Fil^\bullet P).
\end{align*}
Since they are fibers of square zero extensions, we can view $\Fil^\bullet J$ (resp. $\Fil^\bullet I$) as filtered modules over $\Fil^\bullet D$ (resp. $\Fil^\bullet C$). As in Construction~\ref{const:K_B_gauge}, we then obtain a $\underline{D}$-gauge $\mathcal{J}$ (resp. $\underline{C}$-gauge $\mathcal{I}$) with underlying filtered module $\Fil^\bullet J$ (resp. $\Fil^\bullet I$). Similarly, associated with $\Fil^\bullet P$ and $\Fil^\bullet N$, we have $\underline{D}$-gauges $\mathcal{P}$ and $\mathcal{N}$. Note that we have fiber sequences of $\underline{B}$-gauges\footnote{Given a map of frames $\underline{S}\to \underline{T}$, we can via restriction view any $\underline{T}$-gauge also an $\underline{S}$-gauge. We will do so freely in what follows.}
\[
\mathcal{N}\to \mathcal{K}\to \mathcal{P}\;;\; \mathcal{N}\to \mathcal{J}\to \mathcal{I}.  
\]
Moreover, by Lemma~\ref{lem:factors_through_hodge}, we have a commutative diagram of $\underline{B}$-gauges
\[
\begin{diagram}
\mathcal{N}'&\rTo&\mathcal{N}&\rTo&\mathcal{M}_{\mathrm{par}}(R_N[-1],1)\\
\dTo&&\dTo&&\dTo\\
\mathcal{K}'&\rTo&\mathcal{K}&\rTo&\mathcal{M}_{\mathrm{par}}(R_K[-1],1)\\
\dTo&&\dTo&&\dTo\\
\mathcal{P}'&\rTo&\mathcal{P}&\rTo&\mathcal{M}_{\mathrm{par}}(R_P[-1],1)
\end{diagram}
\]
where the rows and columns are all fiber sequences and where the left vertical row only involves $\underline{B}$-gauges with Hodge-Tate weights $\leq -1$. Here, we have set $R_N = \hker(R_K\to R_P)$.
\end{construction}

\begin{remark}
\label{rem:K_P_N_nilpotence}
Let $\bm{K}^{',\modpI}$ be the $F$-zip over $R_{B^\modpI}$ associated with $\mathcal{K}'$. Then, as in Construction~\ref{const:operator_needs_to_be_nilpotent}, we obtain an operator
\[
\psi^*_{\bm{K}^{',\modpI}}:\varphi^*\gr^1_{\mathrm{Hdg}}K^{',\modpI} \to \gr^1_{\mathrm{Hdg}}K^{',\modpI}.
\]
Similarly, we have operators
\[
\psi^*_{\bm{N}^{',\modpI}}:\varphi^*\gr^1_{\mathrm{Hdg}}N^{',\modpI} \to \gr^1_{\mathrm{Hdg}}N^{',\modpI}\;;\; \psi^*_{\bm{P}^{',\modpI}}:\varphi^*\gr^1_{\mathrm{Hdg}}P^{',\modpI} \to \gr^1_{\mathrm{Hdg}}P^{',\modpI}
\]
obtained from $\mathcal{N}'$ and $\mathcal{P}'$, respectively. Now, we have 
\[
\gr^1_{\mathrm{Hdg}}K^{',\modpI}\simeq \gr^1K/{}^{\mathbb{L}}(p,I),
\]
and the underlying $\varphi$-semilinear map for $\psi^*_{\bm{K}^{',\modpI}}$ is the composition
\[
\gr^1K/{}^{\mathbb{L}}(p,I) \to K/{}^{\mathbb{L}}(p,I)\to \Fil^1K/{}^{\mathbb{L}}(p,I)\to \gr^1K/{}^{\mathbb{L}}(p,I)
\]
where the first map is induced from the divided Frobenius lift $\varphi_1:\Fil^1K\to K$, and the second is obtained from the section $K\to \Fil^1K$ seen in Remark~\ref{rem:frames_surjective}. In particular, we find that the topological local nilpotence of $\dot\varphi_1$ implies the local nilpotence of $\psi^*_{\bm{K}^{',\modpI}}$ in the sense of condition (2) of Remark~\ref{rem:nilpotence_conditions}. A similar conclusion holds for $\psi^*_{\bm{P}^{',\modpI}}$, and hence for $\psi^*_{\bm{N}^{',\modpI}}$ as well.
\end{remark}

\begin{construction}
\label{const:various_xs_cotangent}
For each $y\in X(R)$ and $y^-\in X^-(R)$, the relative cotangent complex of $\mathcal{X}^\preoneb$ yields via pullback an almost perfect module $\mathbb{L}_{\mathcal{X}^\preoneb,y}\in \Mod{R}$ and a filtered almost perfect module $\Fil^\bullet\mathbb{L}_{\mathcal{X}^\preoneb,y^-}\in \mathrm{FilMod}_R$. 

For $S = B,A,C,D$, and $Q\in \Mod{S}$, define maps of sheaves
\[
\underline{\Map}(\mathbb{L}_X,Q) \to X_S\;;\; \underline{\Map}(\Fil^1\mathbb{L}^-_X,Q) \to X^-_S\;;\; \underline{\Map}(\mathbb{L}^-_X/\Fil^1\mathbb{L}^-_X,Q) \to X^-_S
\]
as follows: The fiber of $\underline{\Map}(\mathbb{L}_X,Q)$ (resp. $\underline{\Map}(\Fil^1\mathbb{L}^-_X,Q)$, $\underline{\Map}(\mathbb{L}^-_X/\Fil^1\mathbb{L}^-_X,Q)$) over $y\in X_S(R_{B'}) = X(R_{S'})$ (resp. $y^-\in X^-_S(R_{B'}) = X^-(R_{S'})$) is the sheaf on $(R_{B'})_{\et}$ given by
\begin{align*}
B''&\mapsto \Map_{R_{S'}}(\mathbb{L}_{X,y},R_{S''}\otimes_{R_{S}}N)\\
\text{(resp. }B''&\mapsto \Map_{R_{S'}}(\Fil^1\mathbb{L}^-_{X,y^-},R_{S''}\otimes_{R_{S}}N))\\
\text{(resp. }B''&\mapsto \Map_{R_{S'}}(\mathbb{L}^-_{X,y^-}/\Fil^1\mathbb{L}^-_{X,y^-},R_{S''}\otimes_{R_{S}}N))
\end{align*}
\end{construction}

\begin{lemma}
\label{lem:various_xs_cartesian}
There is a canonical Cartesian square of sheaves on $(R_B)_{\et}$
\[
\begin{diagram}
X^-&\rTo&X^-_D\\
\dTo&&\dTo_0\\
(X\times_{X_A}X^-_A)\times_{X_D\times_{X_C}X^-_C}X^-_D&\rTo&\underline{\Map}(\Fil^1\mathbb{L}^-_X,R_N)
\end{diagram}  
\]
where the two arrows out of $X^-$ are the natural ones.
\end{lemma}
\begin{proof}
This is an exercise in deformation theory. An application of Lemma~\ref{lem:attractor_cotangent_complex} gives us canonical Cartesian squares
\[
\begin{diagram}
X^-&\rTo&X^-_D\\
\dTo&&\dTo_0\\
X^-_D&\rTo&\underline{\Map}(\mathbb{L}^-_X/\Fil^1\mathbb{L}^-_X,R_J[1]);
\end{diagram}
\]
and
\[
\begin{diagram}
X\times_{X_A}X^-_A&\rTo&X_D\times_{X_C}X^-_C\\
\dTo&&\dTo_0\\
X_D\times_{X_C}X^-_C&\rTo&\underline{\Map}(\mathbb{L}_X,R_J[1])\times_{\underline{\Map}(\mathbb{L}_X,R_I[1])}\underline{\Map}(\mathbb{L}^-_X/\Fil^1\mathbb{L}^-_X,R_I[1]).
\end{diagram}
\]
The lemma can be shown by combining these diagrams appropriately.
\end{proof}

\begin{proof}
[Proof of Lemma~\ref{lem:def_theory_for_prismatic_frames}]
For simplicity, we will omit $\xi$ from the notation in what follows. It suffices to show that the diagram
\[
\begin{diagram}
\Gamma_{\underline{B}}(\mathcal{X})&\rTo&X^-\\
\dTo&&\dTo\\
\Gamma_{\underline{A}}(\mathcal{X})\times_{\Gamma_{\underline{C}}(\mathcal{X})}\Gamma_{\underline{D}}(\mathcal{X})&\rTo&(X\times_{X_A}X^-_A)\times_{X_D\times_{X_C}X^-_C}X^-_D
\end{diagram}  
\]
is Cartesian. Using the Cartesian square from Lemma~\ref{lem:various_xs_cartesian}, we reduce to showing that the resulting diagram
\[
\begin{diagram}
\Gamma_{\underline{B}}(\mathcal{X})&\rTo&X^-_D\\
\dTo&&\dTo_0\\
\Gamma_{\underline{A}}(\mathcal{X})\times_{\Gamma_{\underline{C}}(\mathcal{X})}\Gamma_{\underline{D}}(\mathcal{X})&\rTo&\underline{\Map}(\Fil^1\mathbb{L}^-_X,R_N)
\end{diagram} 
\]
is Cartesian.

Given any $x\in\Gamma_{\underline{C}}(\mathcal{X},\xi)(R_{C'})$, pulling back the cotangent complex of $\mathcal{X}^\preoneb$ along $x$ yields an almost perfect $\underline{C}'$-gauge $\mathcal{L}(\mathcal{X})_x$. We now obtain a map of sheaves $\underline{H}(\mathcal{L}(\mathcal{X}),\mathcal{I}[1])\to \Gamma_{\underline{C}}(\mathcal{X})$ whose fiber over $x$ is the sheaf of animated abelian groups on $(R_{C'})_{\et}$ given by
\[
H_{\underline{C}'}(\mathcal{L}(\mathcal{X})_x,\underline{C}'\otimes_{\underline{C}}\mathcal{I}[1]) \defn \tau^{\leq 0}RH_{\underline{C}'}(\mathcal{L}(\mathcal{X})_x,\underline{C}'\otimes_{\underline{C}}\mathcal{I}[1]).
\]
Here, we have denoted by $\underline{C}'\otimes_{\underline{C}}\mathcal{I}$ the $\underline{C}'$-gauge obtained from $\mathcal{I}$ via the obvious base-change operation. The same construction applied with $\mathcal{I}$ replaced by $\mathcal{J}$ and $\mathcal{N}$ gives maps of sheaves 
\[
\underline{H}(\mathcal{L}(\mathcal{X}),\mathcal{J}[1])\to \Gamma_{\underline{D}}(\mathcal{X})\;;\; \underline{H}(\mathcal{L}(\mathcal{X}),\mathcal{N}[1])\to \Gamma_{\underline{D}}(\mathcal{X}).
\]

Standard arguments using filtered deformation theory (see the proof of Proposition~\ref{prop:bootstrapping_coeffs} below for instance) tell us that we have canonical isomorphisms of sheaves
\begin{align}\label{eqn:GammaB_GammaD_Gamma_A_Gamma_C}
\Gamma_{\underline{A}}(\mathcal{X}) \xrightarrow{\simeq}\Gamma_{\underline{C}}(\mathcal{X})\times_{\mathrm{ob}_C,\underline{H}(\mathcal{L}(\mathcal{X}),\mathcal{I}[1]),0}\Gamma_{\underline{C}}(\mathcal{X});\\
\Gamma_{\underline{B}}(\mathcal{X}) \xrightarrow{\simeq}\Gamma_{\underline{D}}(\mathcal{X})\times_{\mathrm{ob}_D,\underline{H}(\mathcal{L}(\mathcal{X}),\mathcal{J}[1]),0}\Gamma_{\underline{D}}(\mathcal{X}).\nonumber
\end{align}
for certain sections
\begin{align*}
\mathrm{ob}_C:\Gamma_{\underline{C}}(\mathcal{X})\to \underline{H}(\mathcal{L}(\mathcal{X}),\mathcal{I}[1])\;;&\;
\mathrm{ob}_D:\Gamma_{\underline{D}}(\mathcal{X})\to \underline{H}(\mathcal{L}(\mathcal{X}),\mathcal{J}[1]).
\end{align*}

These isomorphisms imply that the restriction of $\mathrm{ob}_D$ along the projection to the second coordinate lifts to a map
\[
\mathrm{ob}:\Gamma_{\underline{A}}(\mathcal{X})\times_{\Gamma_{\underline{C}}(\mathcal{X})}\Gamma_{\underline{D}}(\mathcal{X})\to \underline{H}(\mathcal{L}(\mathcal{X}),\mathcal{N}[1])
\]
of sheaves over $\Gamma_{\underline{D}}(\mathcal{X})$, and furthermore that we have a canonical Cartesian diagram
\[
\begin{diagram}
\Gamma_{\underline{B}}(\mathcal{X})&\rTo&\Gamma_{\underline{D}}(\mathcal{X})\\
\dTo&&\dTo_0\\
\Gamma_{\underline{A}}(\mathcal{X})\times_{\Gamma_{\underline{C}}(\mathcal{X})}\Gamma_{\underline{D}}(\mathcal{X})&\rTo_{\mathrm{ob}}&\underline{H}(\mathcal{L}(\mathcal{X}),\mathcal{N}[1]).
\end{diagram} 
\]
Therefore, the proof of the lemma is completed by:
\begin{sublemma}
\label{lem:application_of_parasitism}
There is a canonical isomorphism
\[
\underline{H}(\mathcal{L}(\mathcal{X}),\mathcal{N}[1])\xrightarrow{\simeq}\Gamma_{\underline{D}}(\mathcal{X})\times_{X^-_D}\underline{\Map}(\Fil^1\mathbb{L}^-_X,R_N).
\]
\end{sublemma}
\begin{proof}
Given $x\in \Gamma_{\underline{D}}(\mathcal{X})(R_{B'})$, the associated $\underline{D}'$-gauge $\mathcal{L}(\mathcal{X})_x$ has Hodge-Tate weights $\ge -1$ by our $1$-boundedness hypothesis. Explicitly, the underlying quasicoherent sheaf over $\Rees(\Fil^\bullet D')$ is obtained via pullback of the relative cotangent complex of $\mathcal{X}^\preoneb$ along the section $\Rees(\Fil^\bullet D')\to \mathcal{X}^\preoneb$ underlying $x$. The associated filtered module $\Fil^\bullet_{\mathrm{Hdg}}L(\mathcal{X})_x$ over $R_{D'}$ admits the following description: Let $\eta^-_{D'}(x)\in X^-_D(R_{B'}) = X^-(R_{D'})$ be the image of $x$ under the map $ \Gamma_{\underline{D}}(\mathcal{X})\to X^-_D$. Then we have a canonical isomorphism
\[
\Fil^\bullet_{\mathrm{Hdg}}L(\mathcal{X})_x \simeq \Fil^\bullet\mathbb{L}^-_{X,\eta^-_{D'}(x)} \in \mathrm{FilMod}_{R_{D'}}.
\] 

From Construction~\ref{const:f_zip_ideal_psi_star}, we obtain a fiber sequence of $\underline{D}$-gauges
\[
\mathcal{N}'[1]\to \mathcal{N}[1]\to \mathcal{M}_{\mathrm{par}}(R_N,1)
\]
with $\mathcal{N}'[1]$ having Hodge-Tate weights $\le -1$. Now, as explained in Example~\ref{ex:parasitic_gauges}, and from the discussion in the first paragraph of the proof, we obtain a canonical isomorphism
\[
\underline{H}(\mathcal{L}(\mathcal{X}), \mathcal{M}_{\mathrm{par}}(R_N,1))\xrightarrow{\simeq}\Gamma_{\underline{D}}(\mathcal{X})\times_{X^-_D}\underline{\Map}(\Fil^1\mathbb{L}^-_X,R_N)
\]

Therefore, to prove the lemma, it is now enough to know that $\underline{H}(\mathcal{L}(\mathcal{X}),\mathcal{N}'[m]) \simeq 0$ for all $m\in \Int$. This follows from Lemma~\ref{lem:A_gauge_nilpotent}, criterion (2) from Remark~\ref{rem:nilpotence_conditions}, and Remark~\ref{rem:K_P_N_nilpotence}.
\end{proof}
\end{proof}

\begin{construction}
 \label{const:f_zip_sections_psi_x}
Since we have a map of frames $\underline{B}\to \underline{W_1(R_{B^\modpI})}$, we can consider the associated sheaf $\Gamma_{\underline{W_1(R_{B^\modpI})}}(\mathcal{X},\xi)$ on $(R_B)_{\et}$ equipped with a natural map  
\[
\Gamma_{\underline{A}}(\mathcal{X},\xi)\to \Gamma_{\underline{W_1(R_{B^\modpI})}}(\mathcal{X},\xi)
\]
Given $x\in \Gamma_{\underline{W_1(R_{B^\modpI})}}(\mathcal{X},\xi)(R_B)$, by pulling back the cotangent complex of $\mathcal{X}^\preoneb$ along the associated section of $\mathcal{X}$, we obtain an $F$-zip $\bm{L}^\modpI(\mathcal{X})_x$ over $R_{B^\modpI}$ with Hodge-Tate weights $\ge -1$. In particular, via the process explained in Construction~\ref{const:operator_needs_to_be_nilpotent}, we obtain an operator
\[
\psi^\modpI_x:\Fil^1L^\modpI(\mathcal{X})_x \to \varphi^*\Fil^1L^\modpI(\mathcal{X})_x.
\]
As $x$ varies, this data organizes into an $F$-zip $\bm{L}^\modpI(\mathcal{X})$ over $\Gamma_{\underline{W_1(R_{B^\modpI})}}(\mathcal{X},\xi)$ along with an operator $\psi^\modpI:\Fil^1L^\modpI(\mathcal{X}) \to \varphi^*\Fil^1L^\modpI(\mathcal{X})$.
\end{construction}

\begin{remark}
\label{rem:nilpotent_locus}
An analysis of the proof of Lemma~\ref{lem:def_theory_for_prismatic_frames}, and in particular that of Sublemma~\ref{lem:application_of_parasitism}, shows that it goes through if one replaces the topological nilpotence hypothesis on $\dot\varphi_1$ with the following condition: The operator $\psi^\modpI:\Fil^1_{\mathrm{Hdg}}L^\modpI(\mathcal{X})\to \varphi^*\Fil^1_{\mathrm{Hdg}}L^\modpI(\mathcal{X})$ satisfies $\Phi_n(\psi^\modpI)\simeq 0$ for $n$ sufficiently large. Indeed, this amounts to replacing the use of nilpotence criterion (2) of Remark~\ref{rem:nilpotence_conditions} with that of criterion (1). 

If $\bm{L}^\modpI(\mathcal{X})$ is perfect with dual $F$-zip $\bm{T}^\modpI(\mathcal{X})$ obtained from the relative tangent complex of $\mathcal{X}$, then this is also equivalent to requiring that the dual map 
\[
\psi^{\modpI,*} = (\psi^\modpI)^\vee:\varphi^*\gr^{-1}_{\mathrm{Hdg}}T^\modpI(\mathcal{X})\to \gr^{-1}_{\mathrm{Hdg}}T^\modpI(\mathcal{X}),
\]
when viewed as a $\varphi$-semilinear endomorphism of $\gr^{-1}_{\mathrm{Hdg}}T^\modpI(\mathcal{X})$, is nilpotent in the usual sense. 
\end{remark}

\begin{remark}
\label{rem:nilpotent_locus_witt}
Taking Remark~\ref{rem:witt_surjective_def_theory_application} into account, we see that $\Gamma_{\underline{B}}(\mathcal{X},\xi)$ can be computed using the Witt frame $\underline{W(R_B)}$ when $\pi_0(R_B)/p\pi_0(R_B)$ is semiperfect and when the nilpotence condition of Remark~\ref{rem:nilpotent_locus} holds.
\end{remark}

\begin{remark}
\label{rem:nilpotent_locus_p-adic}
Suppose that $\underline{B}$ is a $p$-adic prismatic frame and that $\mathcal{X}$ arises from a pair $(\mathcal{Y},\xi_1)$ over $\Rees(\Fil^\bullet B)\otimes\Field_p$ as in Remark~\ref{rem:weil_restriction_commuting_square}. Then by the same process from which $\bm{L}^\modpI(\mathcal{X})$ was obtained, we get an $F$-zip $\bm{L}(\mathcal{Y})$ over $\Gamma_{\underline{W_1(R_B)}}(\mathcal{Y},\xi_1)$, and by Remark~\ref{rem:operator_needs_to_be_nilpotent}, it is enough to check that the corresponding operator 
\[
\psi:\Fil^1_{\mathrm{Hdg}}L(\mathcal{Y})\to \varphi^*\Fil^1_{\mathrm{Hdg}}L(\mathcal{Y})
\]
satisfies $\Phi_n(\psi) \simeq 0$ for $n$ sufficiently large.
\end{remark}

\subsubsection{}
There is a somewhat different (and much simpler) situation in which one also obtains a Cartesian diagram as above.  This will prove useful to us in Section~\ref{sec:explicit}. Instead of the hypotheses from Proposition~\ref{prop:def_theory_frames}, suppose instead that the map $\sigma:\Spf B\to \Rees(\Fil^\bullet B)$ admits a factoring
\[
\Spf B\xrightarrow{\overline{\sigma}}\Rees(\Fil^\bullet A)\to \Rees(\Fil^\bullet B)
\]
along with a factoring of $\sigma:\Spf A\to \Rees(\Fil^\bullet A)$ as
\[
\Spf A \to \Spf B \xrightarrow{\overline{\sigma}}\Rees(\Fil^\bullet A).
\]
Here, the first map is the one arising from the map of frames.

\begin{proposition}
\label{prop:hokey_def_theory}
Under these hypotheses, there is a canonical Cartesian diagram
\[
\Square{\Gamma_{\underline{B}}(\mathcal{X},\xi)({R_B})}{}{X^{-}(R_B)}{q_*}{}{\Gamma_{\underline{A}}(\mathcal{X},\xi)({R_A})}{}{X^{-}(R_A)\times_{X(R_A)}X(R_B)}.
\]
\end{proposition}
\begin{proof}
We have Cartesian squares
\[
\begin{matrix}
\Square{\Gamma_{\underline{A}}(\mathcal{X},\xi)({R_A})}{}{\Map(\Rees(\Fil^\bullet A),\mathcal{X})}{}{(\tau^*,\overline{\sigma}^*)}{\mathcal{X}^\preoneb(B)}{(q_*, \mathrm{id})}{\mathcal{X}^\preoneb(A)\times \mathcal{X}^\preoneb(B)}&&&;&&&
\Square{\Gamma_{\underline{B}}(\mathcal{X},\xi)({R_B})}{}{\Map(\Rees(\Fil^\bullet B),\mathcal{X})}{}{(\tau^*,\sigma^*)}{\mathcal{X}^\preoneb(B)}{\Delta}{\mathcal{X}^\preoneb(B)\times \mathcal{X}^\preoneb(B)}
\end{matrix},
\]
which together yield another Cartesian square
\[
\Square{\Gamma_{\underline{B}}(\mathcal{X},\xi)({R_B})}{}{\Map(\Rees(\Fil^\bullet B),\mathcal{X})}{}{(q_*,\sigma^*)}{\Gamma_{\underline{A}}(\mathcal{X},\xi)({R_A})}{}{\Map(\Rees(\Fil^\bullet A),\mathcal{X})\times_{\tau^*,\mathcal{X}^\preoneb(A)}\mathcal{X}^\preoneb(B)}
\]

The proof of the proposition is now completed using the filtered integrability of $\mathcal{X}$.
\end{proof}

\subsection{A pseudo-torsor structure associated with $1$-bounded stacks}
\label{subsec:torsor_struct}

Suppose that we have a $p$-adic prismatic frame $\underline{A}$ as in Definitions~\ref{defn:p-adic_frames} and~\ref{defn:prismatic_frames}, and suppose also that the following conditions hold:
\begin{itemize}
   \item $\underline{A}$ is laminated, so that we have a map of frames $\underline{A}\to \underline{W(R_A)}$ (Lemma~\ref{lem:laminated_frames}).
   \item $R_A$ is a \emph{semiperfect} $\Field_p$-algebra.
\end{itemize}
 Let us take a $1$-bounded stack $\mathcal{Y}\to \Rees(\Fil^\bullet A)\otimes\Field_p$ that is strongly integrable, and is equipped with an isomorphism $\xi_1:\sigma^* \mathcal{Y}^\preoneb\to \tau^* \mathcal{Y}^\preoneb$ over $\Spec \overline{A}$. We will view this data as a $1$-bounded stack $\mathcal{Y}$ over $\mathfrak{S}(\underline{A})$. Via Remark~\ref{rem:weil_restriction_commuting_square}, we obtain the sheaf 
\[
\Gamma_{\underline{A}}(\mathcal{Y}) \defn \Gamma_{\underline{A}}(\mathcal{Y},\xi_1)
\]
over $(R_A)_{\et}$. Further, as in Remark~\ref{rem:nilpotent_locus_p-adic}, we can consider the sheaf $\Gamma_{\Fzip}(\mathcal{Y})\defn \Gamma_{\underline{W_1(R_A)}}(\mathcal{Y},\xi_1)$, and we obtain a canonical map
\[
\Gamma_{\underline{A}}(\mathcal{Y})\to \Gamma_{\Fzip}(\mathcal{Y})
\] 
of sheaves on $(R_A)_{\et}$. The goal of this section is to prove Proposition~\ref{prop:abstract_devissage_to_fzips}, which gives a somewhat explicit picture of this map of stacks. Later, in Section~\ref{sec:abstract}, we will apply it to prove a general version of Theorem~\ref{introthm:map_to_f-zips}.

\begin{remark}
We will see that $\Gamma_{\underline{A}}(\mathcal{Y})$ is a pseudotorsor over $\Gamma_{\Fzip}(\mathcal{Y})$ for a particular group stack constructed from $\underline{A}$ and the relative cotangent complex of $\mathcal{Y}^\preoneb$. This group stack is an abstraction of the complexes of Artin-Milne computing the fppf cohomology of height $1$ group schemes~\cite{artin_milne}. 

The basic strategy behind the study of the map involves the following steps:
\begin{enumerate} 
   \item When $\underline{A}$ is the Witt frame $\underline{W(R_A)}$, $\mathfrak{S}(\underline{W(R_A)})\otimes\Field_p$ is a square-zero thickening of $R_A^{\Fzip} = \mathfrak{S}(\underline{W_1(R_A)})$, and so we can use deformation theory to describe the map $\Gamma_{\underline{W(R_A)}}(\mathcal{Y})\to \Gamma_{\Fzip}(\mathcal{Y})$.
   \item In general, we have a map of frames $\underline{A}\to \underline{W(R_A)}$ lifting the mod-$p$ map from above, and we can use it, along with the $1$-boundedness of $\mathcal{Y}$, to lift the description in the first step from $\underline{W(R_A)}$ to $\underline{A}$. This step also uses deformation theory and the strong integrability of $\mathcal{Y}$.
\end{enumerate}
To carry this out, we break down the stacks $\Rees(\Fil^\bullet A)\otimes\Field_p$ and $\Rees(\Fil^\bullet_{\mathrm{Lau}}W_1(R_A))$ via compatible stratifications by certain somewhat explicit substacks, and we use a stratum-by-stratum analysis to get a grip on the picture.
\end{remark}

\begin{remark}
\label{rem:p-adic_structure_map}
Our hypotheses imply that the filtration $\Fil^\bullet A$ is $p$-adic, in the sense that $\Rees(\Fil^\bullet A)$ is a $p$-adic formal stack over 
\[
\Rees(\Fil^\bullet_p\Int_p)\simeq \Spf\Int_p[u,t]/(ut-p)/\Gm, 
\]
where $u$ has degree $-1$ and $t$ has degree $1$. We can interpret $u$ as yielding, for every $i\in \Int$, a commuting diagram
\[
\begin{diagram}
\Fil^{i-1}A\cdot t^{-i+1}&\rTo^u&\Fil^{i}A\cdot t^{-i}\\
&\rdTo_p&\dTo_t\\
&&\Fil^{i-1}A\cdot t^{-i+1}.
\end{diagram}
\]
This gives us a diagram of closed immersions
\[
\begin{diagram}
&&\Rees(\Fil^\bullet A)_{(u=0)}\\
&\ruTo^{\lambda_-}&&\rdTo\\
\Rees(\Fil^\bullet A)_{(t=u=0)}&&&&&\Rees(\Fil^\bullet A)\otimes\Field_p\\
&\rdTo_{\lambda_+}&&\ruTo\\
&&\Rees(\Fil^\bullet A)_{(t=0)}
\end{diagram}
\]
obtained via base-change from the diagram
\[
\begin{diagram}
&&\Aff^1/\Gm\times\Spec\Field_p\simeq\Rees(\Fil^\bullet_p\Int_p)_{(u=0)}\\
&\ruTo&&\rdTo\\
B\Gm\times\Spec\Field_p\simeq \Rees(\Fil^\bullet_p\Int_p)_{(t=u=0)}&&&&&\Rees(\Fil^\bullet_p \Int_p)\otimes\Field_p\\
&\rdTo&&\ruTo\\
&&\Aff^1_+/\Gm\times\Spec\Field_p\simeq\Rees(\Fil^\bullet_p\Int_p)_{(t=0)}.
\end{diagram}
\]
\end{remark}

\begin{remark}
\label{rem:structure_fzip_to_syn}
We have canonical maps of stacks over $\Spec \overline{A}$
\begin{align*}
\alpha_-:\Aff^1/\Gm\times\Spec R_{A}\simeq \Rees(\Fil^\bullet_{\mathrm{Lau}}W_1(R_A))_{(u=0)} &\to \Rees(\Fil^\bullet A)_{(u=0)}\hookrightarrow \Rees(\Fil^\bullet A)\otimes\Field_p;\\
\alpha_+:\Aff^1_+/\Gm\times\Spec \varphi_*R_{A}\simeq \Rees(\Fil^\bullet_{\mathrm{Lau}}W_1(R_A))_{(t=0)}  &\to \Rees(\Fil^\bullet A)_{(t=0)}\hookrightarrow \Rees(\Fil^\bullet A)\otimes\Field_p.
\end{align*}

In particular, $\alpha_-$ gives a map $\alpha_0:B\Gm\times\Spec R_{A}\to \Rees(\Fil^\bullet A)_{(t=u=0)}$, and we can use it to view $\Rees(\Fil^\bullet A)_{(u=0)}$, $\Rees(\Fil^\bullet A)_{(t=0)}$ and $\Rees(\Fil^\bullet A)_{(t=u=0)}$ as $R_{A}$-pointed graded stacks.
\end{remark}

\begin{definition}
\label{defn:fp_stacks_from_fzip}
Let $Y^-\to \Spec R_A$ (resp. $Y^0\to \Spec R_A$, $Y^+\to \Spec R_A$) be the attractor stack (resp. fixed point stack, repeller stack) over $R_A$ for $\mathcal{Y}^\preoneb\to \Rees(\Fil^\bullet A)\otimes\Field_p$ associated with the maps $\alpha_-$ (resp. $\alpha_0$, $\alpha_+$). 
\end{definition}

\begin{remark}
We have canonical maps $Y^-\to Y^0$ and $Y^+\to {}^\varphi Y^0$, where ${}^\varphi Y^0 = Y^0\times_{\Spec R_A,\varphi}\Spec R_A$, induced via restriction along the projections $\Aff^1/\Gm\to B\Gm$ and $\Aff^1_+/\Gm\to B\Gm$. 

We also have maps $\chi^*_-:Y^0\to Y^-$ and $\chi^*_+:Y^+\to {}^{\varphi}Y^0$ induced by restriction along the closed immersions $\chi_-:B\Gm\to \Aff^1/\Gm$ and $\chi_+:B\Gm\to \Aff^1_+/\Gm$. 

There are also maps $j_-^*:Y^-\to Y$ and $j_+^*:Y^+\to {}^\varphi Y = Y\times_{\Spec R_A,\varphi}\Spec R_A$ obtained via restriction along the open points of $\Aff^1/\Gm$ and $\Aff^1_+/\Gm$. The isomorphism $\xi:\sigma^*\mathcal{Y}^\preoneb \to \tau^* \mathcal{Y}^\preoneb$ induces an isomorphism ${}^\varphi Y \xrightarrow{\simeq}Y$ of stacks over $R_A$, which we also denote by $\xi$. In particular, composing with this isomorphism gives a map $Y^+\to Y$.

Finally, we have the relative Frobenius map $Y^0\xrightarrow{y\mapsto {}^\varphi y} {}^{\varphi}Y^0$ over $\Spec R_A$.
\end{remark}

\begin{lemma}
\label{lem:gamma_fzip_description}
There is a natural isomorphism of \'etale sheaves
\[
\begin{diagram}
\Gamma_{\Fzip}(\mathcal{Y}) &\xrightarrow{\simeq} \mathrm{eq}\biggl[Y^0\times_{{}^\varphi Y^0}Y^+\times_{Y}Y^-&\pile{\rTo^{\mathrm{pr}_1}\\ \rTo_{\chi^*_-\circ \mathrm{pr}_3}}&Y^0\biggr].
\end{diagram}
\]
\end{lemma}
\begin{proof}
By definition, we have
\[
\begin{diagram}
\Gamma_{\Fzip}(\mathcal{Y})(R_{A'}) &\xrightarrow{\simeq} \mathrm{eq}\biggl[\Map(\mathcal{R}(\Fil^\bullet_{\mathrm{Lau}}W_1(R_{A'})),\mathcal{Y})&\pile{\rTo^{\xi\circ \sigma^*}\\ \rTo_{\tau^*}}&Y(R_{A'})\biggr].
\end{diagram}
\]

The compositions
\[
\Map(\Rees(\Fil^\bullet_{\mathrm{Lau}}W_1(R_A)),\mathcal{Y})\to Y^-\xrightarrow{j_-^*} Y\;;\; \Map(\Rees(\Fil^\bullet_{\mathrm{Lau}}W_1(R_A)),\mathcal{Y})\to Y^+\xrightarrow{j_-^*} {}^\varphi Y
\]
arising from restrictions along the $u=0$ and $t=0$ loci, respectively, can be identified with $\tau^*$ and $\sigma^*$.

Now, quite formally, we see that the right hand side in the isomorphism asserted by the lemma is equivalent to
\[
\begin{diagram}
 \mathrm{eq}\biggl(Y^+\times_{Y}Y^-&\pile{\rTo^{\chi_+^*\circ\mathrm{pr}_1}\\ \rTo_{{}^{\varphi}(\chi^*_-\circ \mathrm{pr}_2)}}&{}^\varphi Y^0\biggr).
\end{diagram}
\]

Therefore, to finish, it is enough to know that we have
\[
\begin{diagram}
\Map(\Rees(\Fil^\bullet_{\mathrm{Lau}}W_1(R_A)),\mathcal{Y}) &\xrightarrow{\simeq} \mathrm{eq}\biggl((Y^+\times Y^-)(R_{A'})&\pile{\rTo^{\chi_+^*\circ \mathrm{pr}_1}\\ \rTo_{{}^\varphi(\chi_-^*\circ \mathrm{pr}_2})}&{}^\varphi Y^0(R_{A'})\biggr).
\end{diagram}
\]

Let $T$ be the Rees algebra $\Rees(\Fil^\bullet_{\mathrm{Lau}}W_1(R_A))$. Via smooth descent for $\mathcal{Y}^\preoneb$, we are reduced to knowing that, for any $R_A$-algebra $S$, we have 
\[
\begin{diagram}
\mathcal{Y}^\preoneb(S\otimes_{R_A}T)&\xrightarrow{\simeq} \mathrm{eq}\biggl[\mathcal{Y}^\preoneb(S\otimes_{R_A}\varphi_*R_A[u])\times \mathcal{Y}^\preoneb(S\otimes_{R_A}R_A[t])&\pile{\rTo^{u\mapsto 0}\\ \rTo_{\varphi\circ (t\mapsto 0)}}&{}^\varphi \mathcal{Y}^\preoneb(S\otimes_{R_A}\varphi_*R_A)\biggr].
\end{diagram}
\]

The claim now follows from the cohesiveness of $\mathcal{Y}\times_{\Rees(\Fil^\bullet A)\otimes\Field_p}\Spec S\to \Spec S$ (property (4) of~\cite[Theorem 7.5.1]{lurie_thesis}), and the fact that the commutative diagram
\[
\begin{diagram}
S\otimes_{R_A}T&\rTo&S\otimes_{R_A}R_A[t]\\
\dTo&&\dTo\\
S\otimes_{R_A}\varphi_*R_A[u]&\rTo&S\otimes_{R_A}\varphi_*R_A
\end{diagram}
\]
is a Cartesian square of surjective maps of animated commutative rings. Indeed, it suffices to check this for $S = R_A$, where it was seen in the identity~\eqref{eqn:rees_W1_explicit}.
\end{proof} 

We would like a similar description of $\Gamma_{\underline{A}}(\mathcal{Y})$, and we begin by looking more closely at the substacks from Remark~\ref{rem:p-adic_structure_map}.

\begin{construction}[The abstract Hodge filtration]
\label{const:hodge_filtration_abstract}
The relatively affine map
\[
\Rees(\Fil^\bullet A)_{(u=0)}\to \Rees(\Fil^\bullet_p\Int_p)_{(u=0)}\simeq \Aff^1/\Gm\times\Spec \Field_p
\]
corresponds to a lift $\Fil^\bullet_{\mathrm{Hdg}}\overline{A}$ of $\overline{A}$ to an animated filtered commutative ring: this is the \defnword{Hodge filtration} on $\overline{A}$. By construction, we have
\[
\Fil^i_{\mathrm{Hdg}}\overline{A} = \hcoker(u:\Fil^{i-1}A\cdot t^{-i+1}\to \Fil^{i}A\cdot t^{-i}).
\]
We have an open immmersion
\[
j_-:\Spec \overline{A} \simeq \Rees(\Fil^\bullet_{\mathrm{Hdg}}\overline{A})_{(t\neq 0)}\hookrightarrow \Rees(\Fil^\bullet_{\mathrm{Hdg}} \overline{A})
\]
through which the map $\tau:\Spec\overline{A}\to \Rees(\Fil^\bullet A)\otimes\Field_p$ factors.
\end{construction}

\begin{example}
\label{ex:u=0witt}
Let us take the Witt frame from Example~\ref{example:witt_frame} associated with $R\in \mathrm{CRing}_{\Field_p/}$ semiperfect. The filtration in this case is $p$-adic: The map $u$ is an isomorphism in filtered degree $i\ge 1$, is the map $F:W(R)\to F_*W(R)$ in degree $0$ and is multiplication by $p$ in negative degrees. Thus we have
\[
\Fil^i_{\mathrm{Hdg}}\overline{W(R)}\xrightarrow{\simeq} \begin{cases}
\hcoker(F:W(R)\to F_*W(R))\simeq (W(R)[F])[1]\simeq \Ga^\sharp(R)[1]&\text{if $i=1$}\\
\overline{W(R)}&\text{if $i\leq 0$}\\
0&\text{otherwise}.
\end{cases}
\]
Here, in the case $i=1$, we have used the surjectivity of $F:W(R)\to F_*W(R)$---which is a consequence of the semiperfectness of $R$---and~\cite[Variant 3.4.12]{bhatt2022absolute}. Concretely, for $R$ discrete, the isomorphism $W(R)[F]\to \mathbb{G}_a^\sharp$ sends an element $x$ with Witt coordinates $(x_0,x_1,\ldots)$ to the element $x_0$ equipped with the unique collection of divided powers $\gamma_m(x_0)$ determined by the fact that $\gamma_p^r(x_0) = x_r$ for all $r\ge 0$.

The map 
\[
\Ga^\sharp(R)[1]\simeq \hcoker(F:W(R)\to F_*W(R)) \simeq \Fil^1_{\mathrm{Hdg}}\overline{W(R)}\to \overline{W(R)}
\]
is induced from $V:F_*W(R)\to W(R)$. 

In particular, we find that $\Fil^\bullet_{\mathrm{Hdg}}\overline{W(R)}$ is a square-zero extension of $\Fil^\bullet_{\mathrm{triv}}R$, and we have
\begin{align*}
\hker(\Fil^\bullet_{\mathrm{Hdg}}\overline{W(R)}\to\Fil^\bullet_{\mathrm{triv}}R)&\simeq\Ga^\sharp(R)[1]\otimes_R\Fil^\bullet_{\mathrm{triv}}R(1),
\end{align*}
where $\Fil^\bullet_{\mathrm{triv}}R(1)$ is the free filtered module of rank $1$ over $R$ with associated graded supported in degree $-1$.\footnote{Recall that according to our convention the $i$-th associated graded piece for a decreasing filtration is in graded degree $-i$.}

Note also that the associated graded algebra $\gr^\bullet_{\mathrm{Hdg}}\overline{W(R)}$ is supported in degrees $-1,0$, and is isomorphic as a graded $R$-algebra to $\Ga^\sharp(R)[1]\oplus R$.
\end{example}

\begin{construction}
[The abstract conjugate filtration]\label{const:conj_filt_abstract}
The relatively affine map
\[
\Rees(\Fil^\bullet A)_{(t=0)}\to \Rees(\Fil^\bullet_p\Int_p)_{(t=0)}\simeq \Aff^1_+/\Gm\times\Spec \Field_p
\]
corresponds to an increasingly filtered animated commutative $\Field_p$-algebra: The underlying animated commutative $\Field_p$-algebra corresponds to the derived affine scheme $\Rees(\Fil^\bullet A)_{(t=0,u\neq 0)}$, and the degree $-i$ component of the associated Rees algebra is isomorphic to $\gr^iA\cdot u^i$.

Now, $\sigma$ factors through a map 
\[
\Spf\overline{A}\simeq\Rees(\Fil^\bullet_pA)_{(t=0,u\neq 0)}\to \Rees(\Fil^\bullet A)_{(t=0,u\neq 0)}. 
\]
In our applications, $\underline{A}$ will be the frame associated with the Nygaard filtered prismatic cohomology of a semiperfect $\Field_p$-algebra $R$, and this map will be an isomorphism: This last fact is equivalent to the known assertion that the conjugate filtration on the Hodge-Tate cohomology of $R$ is exhaustive.

Therefore, we will write $\Fil^{\mathrm{conj}}_\bullet \overline{A}$ for the increasingly filtered animated commutative ring associated with $\Rees(\Fil^\bullet A)_{(t=0)}$, and call it the \defnword{conjugate filtration} on $\overline{A}$. In particular, we have a map
\[
j_+:\Spec \overline{A} \to \Rees(\Fil^{\mathrm{conj}}_\bullet\overline{A})_{(u\neq 0)}\hookrightarrow\Rees(\Fil^{\mathrm{conj}}_\bullet \overline{A})
\]
through which $\sigma$ factors.
\end{construction}

\begin{remark}
The pullbacks of both $\Rees(\Fil^\bullet A)_{(t=0)}$ and $\Rees(\Fil^\bullet A)_{(u=0)}$ over $B\Gm\times\Spec\Field_p$ are isomorphic to $\Rees(\Fil^\bullet A)_{(t=u=0)}$: this identification corresponds to an isomorphism of graded animated commutative rings
\[
\gr^{\mathrm{conj}}_\bullet \overline{A}\xrightarrow{\simeq}\gr^{\bullet}_{\mathrm{Hdg}}\overline{A}.
\]
\end{remark}

\begin{example}
\label{ex:t=0witt}
Let us return to the Witt frame from Example~\ref{ex:u=0witt}. Here, we have 
\begin{align*}
\Fil^{\mathrm{conj}}_i\overline{W(R)} = \gr^i_{\mathrm{Lau}}W(R)&\simeq \begin{cases}
R&\text{if $i=0$}\\
F_*\overline{W(R)}&\text{if $i\ge 1$}\\
0&\text{otherwise}.
\end{cases}
\end{align*}

The transition maps $\Fil^{\mathrm{conj}}_i\overline{W(R)}\to \Fil^{\mathrm{conj}}_{i+1}\overline{W(R)}$ are the identity when $i\ge 1$, while the map
\[
R = \Fil^{\mathrm{conj}}_0\overline{W(R)} \to \Fil^{\mathrm{conj}}_1\overline{W(R)} = F_*\overline{W(R)}
\]
is induced from the commuting diagram
\[
\begin{diagram}
W(R)&\rTo^F&F_*W(R)\\
\dTo&&\dTo\\
R&\rTo&F_*\overline{W(R)}.
\end{diagram}
\]

The map $\Fil^{\mathrm{conj}}_i\overline{W(R)}\to \overline{W(R)}$ for $i\ge 1$ is the identification $F_*\overline{W(R)}\simeq \overline{W(R)}$.
\end{example}

\begin{definition}
Define sheaves on $(R_A)_{\et}$:
\begin{align*}
\Gamma_{(t=0)}(\mathcal{Y}):R_{A'}&\mapsto \Map(\Rees(\Fil^\bullet A')_{(t=0)},\mathcal{Y});\\
\Gamma_{(u=0)}(\mathcal{Y}):R_{A'}&\mapsto \Map(\Rees(\Fil^\bullet A')_{(u=0)},\mathcal{Y});\\
\Gamma_{(t=u=0)}(\mathcal{Y}):R_{A'}&\mapsto \Map(\Rees(\Fil^\bullet A')_{(t=u=0)},\mathcal{Y});\\
\Gamma_{\overline{A}}(\mathcal{Y}):R_{A'}&\mapsto \mathcal{Y}(\overline{A}').
\end{align*}
All spaces of maps are over $\Rees(\Fil^\bullet A)\otimes\Field_p$, and we are viewing $\Spec \overline{A}'$ as a scheme over it via $\tau$.
\end{definition} 

\begin{lemma}
\label{lem:abstract_gamma_syn_alternate}
We have an isomorphism of \'etale sheaves
\[
\begin{diagram}
\Gamma_{\underline{A}}(\mathcal{Y})&\rTo^{\simeq}&\mathrm{eq}\biggl(\Gamma_{(t=0)}(\mathcal{Y})\times_{j_+^*,\Gamma_{\overline{A}}(\mathcal{Y}),j_-^*}\Gamma_{(u=0)}(\mathcal{Y})&\pile{\rTo^{\lambda_+^*\circ \mathrm{pr}_1}\\ \rTo_{\lambda_-^*\circ \mathrm{pr}_2}}&\Gamma_{(t=u=0)}(\mathcal{Y})\biggr).
\end{diagram}
\]
\end{lemma} 
\begin{proof}
As in the proof of Lemma~\ref{lem:gamma_fzip_description}, this reduces via smooth descent to the claim that, for any $p$-adic formal affine scheme $\Spf T\to \Rees(\Fil^\bullet A)$ we have
\begin{align*}
\begin{diagram}
\Map(\Spec T/{}^{\mathbb{L}}p,\mathcal{Y})&\rTo^{\simeq}& \mathrm{eq}\biggl[\Map(\Spec T/{}^{\mathbb{L}}u,\mathcal{Y})\times \Map(\Spec T/{}^{\mathbb{L}}t,\mathcal{Y})&\pile{\rTo\\ \rTo}&\Map(\Spec T/{}^{\mathbb{L}}(u,t),\mathcal{Y})\biggr].
\end{diagram}
\end{align*}
Here, we have set
\begin{align*}
T/{}^{\mathbb{L}}u \defn T\otimes_{\Int_p[u,t],u\mapsto 0}\Int_p[t]\;;&\; T/{}^{\mathbb{L}}t \defn T\otimes_{\Int_p[u,t],t\mapsto 0}\Int_p[u]\\
T/{}^{\mathbb{L}}(u,t) &\defn T\otimes_{\Int_p[u,t],u\mapsto 0,t\mapsto 0}\Int_p
\end{align*}
and all the mapping spaces are over $\Rees(\Fil^\bullet A)\otimes\Field_p$.

Just like in the proof of Lemma~\ref{lem:gamma_fzip_description}, the claim now follows from the fact that the commutative diagram
\[
\begin{diagram}
T/{}^{\mathbb{L}}p&\rTo&T/{}^{\mathbb{L}}u\\
\dTo&&\dTo\\
T/{}^{\mathbb{L}}t&\rTo&T/{}^{\mathbb{L}}(u,t)
\end{diagram}
\]
of surjective maps of animated commutative rings is Cartesian. This last assertion only needs to be checked when $T = \Int_p[u,t]/(ut-p)$, where it is clear.
\end{proof}

\begin{construction}
\label{const:various_sheaves_on_ys}
As we saw in Construction~\ref{const:various_xs_cotangent}, there are certain quasicoherent sheaves over the stacks $Y,Y^{\pm},Y^0,\Gamma_{\Fzip}(\mathcal{Y})$ that can be produced from the relative cotangent complex $\mathbb{L}_{\mathcal{Y}^\preoneb}$ of $\mathcal{Y}^\preoneb$ over $\Rees(\Fil^\bullet A)\otimes\Field_p$. 

To begin, we can simply restrict the cotangent complex over $Y$ to obtain an almost perfect complex $L(\mathcal{Y})\in \mathrm{APerf}(Y)$. Its pullback over $Y^-$ (resp. $Y^+$) underlies a decreasingly filtered (resp. increasingly filtered) almost perfect complex $\Fil^\bullet_{\mathrm{Hdg}}L(\mathcal{Y})$ (resp. $\Fil^{\mathrm{conj}}L(\mathcal{Y})$). The associated graded complex $\mathrm{gr}^\bullet_{\mathrm{Hdg}}L(\mathcal{Y})$ (resp. $\mathrm{gr}^{\mathrm{conj}}_\bullet L(\mathcal{Y})$) is pulled back from a graded almost perfect complex over $Y^0$ (resp. over ${}^\varphi Y^0$) along $Y^-\to Y^0$ (resp. $Y^+\to {}^\varphi Y^0$), which we denote by the same symbol. 

Finally, over $\Gamma_{\Fzip}(\mathcal{Y})$, we have a canonical isomorphism between the pullbacks of $\varphi^*\gr^\bullet_{\mathrm{Hdg}}L(\mathcal{Y})$ and $\gr^{\mathrm{conj}}_\bullet L(\mathcal{Y})$ along the canonical map $\Gamma_{\Fzip}(\mathcal{Y})\to Y^0$ obtained via projection onto the target of the equalizer diagram in Lemma~\ref{lem:gamma_fzip_description}. This yields an $F$-zip $\bm{L}(\mathcal{Y})$ over $\Gamma_{\Fzip}(\mathcal{Y})$, which is precisely the one described in Remark~\ref{rem:nilpotent_locus_p-adic}.
\end{construction}

\begin{construction}
Let $\tilde{Q}$ be a quasicoherent sheaf over $\Spec R_A$ associated with an $R_A$-module $Q$. For any sheaf $Z$ over $(R_A)_{\et}$ and any almost perfect complex $M$ over $Z$, as in Construction~\ref{const:various_xs_cotangent}, we can define a map of sheaves
\[
\underline{\Map}(M,\tilde{Q})\to Z
\]
whose fiber over $y\in Z(R_{A'})$ is the sheaf on $(R_{A'})_{\et}$ given by
\[
A''\mapsto \Map_{R_{A'}}(M_y,R_{A''}\otimes_{R_A}Q).
\]
\end{construction}

\begin{construction}
Let $Z^1_{\underline{A}}$ be the quasicoherent \'etale sheaf\footnote{Recall that all sheaves are with respect to the \emph{small} \'etale site.} over $R_A$ given by
\[
Z^1_{\underline{A}}(R_{A'}) = \Fil_1^{\mathrm{conj}}\overline{A}'\times_{\overline{A}'}\Fil^1_{\mathrm{Hdg}}\overline{A}' \simeq \hker(\Fil_1^{\mathrm{conj}}\overline{A}'\to \varphi_*R_{A'}).
\]
The $R_{A'}$-module structure is given via the isomorphism $R_{A'}=\gr^0A'\xrightarrow{\simeq}\Fil^{\mathrm{conj}}_0\overline{A}'$.

Let $H^1_{\underline{A}}$ be the quasicoherent \'etale sheaf
\[
H^1_{\underline{A}}:R_{A'}\mapsto \gr^{\mathrm{conj}}_1\overline{A}'.
\]

Note that we have two maps
\[
q_1,q_2:Z^1_{\underline{A}}\to H^1_{\underline{A}}.
\]
The first of these is obtained via the natural map $\Fil_1^{\mathrm{conj}}\overline{A}'\to \gr^{\mathrm{conj}}_1\overline{A}'$, and is a linear map of quasicoherent sheaves, while the second is obtained from the composition
\[
Z^1_{\underline{A}}(R_{A'})\to \Fil^{1}_{\mathrm{Hdg}}\overline{A}'\to \gr^1_{\mathrm{Hdg}}\overline{A}'\xrightarrow{\simeq}\gr^{\mathrm{conj}}_1\overline{A}'.
\]
This one is $\varphi$-semilinear, and so corresponds to a map $1\otimes q_2:\varphi^*Z^1_{\underline{A}}\to H^1_{\underline{A}}$.
\end{construction}

\begin{example}
\label{ex:z1h1_witt}
Let us return to the example of the Witt frame. Here, using Examples~\ref{ex:u=0witt} and~\ref{ex:t=0witt}, we see that 
\begin{align*}
H^1_{\underline{W(R)}}(R)&\simeq \hcoker(R\to F_*\overline{W(R)})\simeq \Ga^\sharp(R)[1];\\
Z^1_{\underline{W(R)}}(R)&\simeq\hker(F_*\overline{W(R)}\to \varphi_*R)\simeq \varphi_*\Ga^{\sharp}[1]
\end{align*}
Here, in the first isomorphism, we have used the inverse of the composition of the isomorphisms
\[
\Ga^\sharp(R)[1]\simeq \hcoker(F:W(R)\to F_*W(R))\xrightarrow{\simeq}\hcoker(R\to F_*\overline{W(R)}).
\]

Via these identifications, $q_2$ is the identity on the underlying $\Field_p$-modules, corresponding to the $R$-linear counit $\varphi^*\varphi_*\Ga^\sharp(R)[1]\to \Ga^\sharp[1]$, while  $q_1$ arises after a shift from the map
\[
\varphi_*\Ga^\sharp(R)\simeq (F_*W(R))[F]\xrightarrow{V} W(R)[F]\simeq \Ga^\sharp(R). 
\] 
\end{example}

We can now describe the relative situation between the consituents of the equalizer diagrams involved in Lemmas~\ref{lem:gamma_fzip_description} and~\ref{lem:abstract_gamma_syn_alternate}.
\begin{lemma}
 \label{lem:abstract_map_to_fzips}
\begin{enumerate}
   \item There is a canonical equivalence
   \[
    \Gamma_{\overline{A}}(\mathcal{Y})\times_{Y}Y^-\xrightarrow{\simeq}\Gamma_{(u=0)}(\mathcal{Y}).
   \]
   \item There is a canonical Cartesian square of prestacks over $R_A$
   \[
    \begin{diagram}
    \Gamma_{(t=0)}(\mathcal{Y})&\rTo&Y^0\\
    \dTo&&\dTo_0\\
    Y^+\times_{{}^\varphi Y^0}Y^0&\rTo_{d^+}&\underline{\Map}(\Fil^1_{\mathrm{Hdg}}L(\mathcal{Y}),Z^1_{\underline{A}}[1])
    \end{diagram}
   \]

   \item There is a canonical map
   \[
    \Gamma_{(t=u=0)}(\mathcal{Y})\to Y^0
   \]
   presenting the source as a trivial torsor under $\underline{\Map}(\Fil^1_{\mathrm{Hdg}}L(\mathcal{Y}),H^1_{\underline{A}})$. In particular, there is a Cartesian square
   \[
     \begin{diagram}
    \Gamma_{(t=u=0)}(\mathcal{Y})&\rTo&Y^0\\
    \dTo&&\dTo_0\\
    Y^0&\rTo_{d^0\;\;\;\;}&\underline{\Map}(\Fil^1_{\mathrm{Hdg}}L(\mathcal{Y}),H^1_{\underline{A}}[1])
    \end{diagram}
   \]
\end{enumerate}
\end{lemma} 
\begin{proof}
The first assertion is immediate from filtered integrability and Proposition~\ref{prop:1_bounded_cartesian}. The other two will use graded integrability and Proposition~\ref{prop:1bounded_graded_deformation}. This tells us that we have
\begin{align*}
\Gamma_{(t=0)}(\mathcal{Y})(R_{A'}) &\simeq \Map\left(\Rees_+(\Fil^{\mathrm{conj}}_\bullet \overline{A}'),\mathcal{Y}\right)\\
&\simeq \Map\left(\Spec(\Fil^{\mathrm{conj}}_1\overline{A}'\cdot u \oplus \Fil^{\mathrm{conj}}_0\overline{A}')/\Gm,\mathcal{Y}\right),
\end{align*}
and
\begin{align*}
Y^+(R_{A'})&\simeq \Map(\Aff^1_+/\Gm\times\Spec R_{A'},\mathcal{Y})\\
&\simeq\Map\left(\Spec(\varphi_*(R_{A'}u\oplus R_{A'}))/\Gm,\mathcal{Y}\right).
\end{align*}

This description, combined with graded deformation theory, shows that $\Gamma_{(t=0)}(\mathcal{Y})$ is a trivial torsor over $Y^0$ under the sheaf $\underline{\Map}(\Fil^1_{\mathrm{Hdg}}L(\mathcal{Y}),\Fil^{\mathrm{conj}}_1)$. Here, $\Fil^{\mathrm{conj}}_1$ is the \'etale sheaf over $R_A$ given by $R_{A'}\mapsto \Fil^{\mathrm{conj}}_1\overline{A}'$.

Similarly, $Y^+$ is a trivial torsor over ${}^\varphi Y^0$---and therefore $Y^+\times_{{}^\varphi Y^0}Y^0$ is a trivial torsor over $Y^0$---under the sheaf $\underline{\Map}(\Fil^1_{\mathrm{Hdg}}L(\mathcal{Y}),\varphi_* \mathcal{O})$, where $\mathcal{O}$ is the structure sheaf on $(R_A)_{\et}$.

Assertion (2) follows from this and the definition of $Z^1_{\underline{A}}$. Assertion (3) is shown in similar fashion.
\end{proof}

\begin{notation}
Let us simplify our notation in the following way: Set $Y^{\Sigma} = Y^0\times_{{}^\varphi Y^0}Y^+$, $\Gamma_{\Fzip} = \Gamma_{\Fzip}(\mathcal{Y})$ and
\begin{align*}
\Gamma_{(t=0)} = \Gamma_{(t=0)}(\mathcal{Y})\times_{Y^\Sigma}\Gamma_{\Fzip}\;;&\; \Gamma_{(u=0)} = \Gamma_{(u=0)}(\mathcal{Y})\times_{Y^-}\Gamma_{\Fzip}\simeq \Gamma_{\overline{A}} = \Gamma_{\overline{A}}(\mathcal{Y})\times_Y\Gamma_{\Fzip};\\
\Gamma_{(t=u=0)} &= \Gamma_{(t=u=0)}(\mathcal{Y})\times_{Y^0}\Gamma_{\Fzip}.
\end{align*}
\end{notation}

\begin{remark}
\label{rem:zeta1zeta2maps}
Note that we have two maps $\zeta_1,\zeta_2:\Gamma_{(t=0)}\to \Gamma_{(t=u=0)}$ obtained as follows: The first arises from the natural map $\Gamma_{(t=0)}(\mathcal{Y})\to \Gamma_{(t=u=0)}(\mathcal{Y})$, while the second is defined as the composition
\begin{align}\label{eqn:t=0tou=t=0_second}
\Gamma_{(t=0)} &= \Gamma_{(t=0)}(\mathcal{Y})\times_{Y^\Sigma}\Gamma_{\Fzip}\\
&\to \Gamma_{\underline{A}}(\mathcal{Y})\times_{Y}\Gamma_{\Fzip}\nonumber\\
&\xleftarrow{\simeq} \Gamma_{(u=0)}(\mathcal{Y})\times_{Y^-}\Gamma_{\Fzip}\nonumber\\
&\to \Gamma_{(t=u=0)}(\mathcal{Y})\times_{Y^0}\Gamma_{\Fzip} = \Gamma_{(t=u=0)}.\nonumber
\end{align}
Unwinding definitions now shows that we have
\[
\begin{diagram}
\Gamma_{\underline{A}}(\mathcal{Y})&\rTo^{\simeq}&\mathrm{eq}\bigl(\Gamma_{(t=0)}&\pile{\rTo^{\zeta_1}\\ \rTo_{\zeta_2}}&\Gamma_{(t=u=0)}\bigr).
\end{diagram}
\]
\end{remark}   

\begin{remark}
\label{rem:gamma_u_t_pulled_back_to_gamma_zip}
Write 
\[
\gamma^0:\Gamma_{\Fzip}\to Y^0\;;\; \gamma^+:\Gamma_{\Fzip}\to Y^\Sigma
\]
for the tautological maps. Define maps
\begin{align*}
\delta^0:\Gamma_{\Fzip}&\xrightarrow{d^0\circ \gamma^0}\underline{\Map}(\Fil^1_{\mathrm{Hdg}}L(\mathcal{Y}),H^1_{\underline{A}}[1])\\
\delta^+:\Gamma_{\Fzip}&\xrightarrow{d^+\circ \gamma^+} \underline{\Map}(\Fil^1_{\mathrm{Hdg}}L(\mathcal{Y}),Z^1_{\underline{A}}[1])
\end{align*}
Lemma~\ref{lem:abstract_map_to_fzips} now gives us canonical isomorphisms
\begin{align}
\label{eqn:gamma_u_t_plus_pulled_back}
\Gamma_{(t=0)}&\simeq\Gamma_{\Fzip}\times_{0,\underline{\Map}(\Fil^1_{\mathrm{Hdg}}L(\mathcal{Y}),Z^1_{\underline{A}}[1]),\delta^+}\Gamma_{\Fzip};
\end{align}
\begin{align}
\label{eqn:gamma_u_t_0_pulled_back}
\Gamma_{(t=u=0)}&\simeq\Gamma_{\Fzip}\times_{0,\underline{\Map}(\Fil^1_{\mathrm{Hdg}}L(\mathcal{Y}),H^1_{\underline{A}}[1]),\delta^0}\Gamma_{\Fzip}.
\end{align}

\end{remark}   

\begin{construction}
\label{const:defining_q1psi_q2psi}
As explained in Remark~\ref{rem:nilpotent_locus_p-adic}, the $F$-zip $\bm{L}(\mathcal{Y})$ over $\Gamma_{\Fzip}$ with Hodge-Tate weights $\ge -1$ gives us a map 
\[
\psi:\Fil^1_{\mathrm{Hdg}}L(\mathcal{Y})\to \varphi^*\Fil^1_{\mathrm{Hdg}}L(\mathcal{Y})
\]
of almost perfect complexes over $\Gamma_{\Fzip}$.

Combining the maps $q_1,q_2:Z^1_{\underline{A}}\to H^1_{\underline{A}}$ with $\psi$ yields, for every $i\in \Int$, two further maps of \'etale sheaves
\[
q_{1,\psi}[i],q_{2,\psi}[i]:\underline{\Map}(\Fil^1_{\mathrm{Hdg}}L(\mathcal{Y}),Z^1_{\underline{A}}[i])\to \underline{\Map}(\Fil^1_{\mathrm{Hdg}}L(\mathcal{Y}),H^1_{\underline{A}}[i]).
\]
The map $q_{1,\psi}[i]$ is simply given by postcomposition with $q_1[i]$, and is independent of $\psi$, while the map $q_{2,\psi}[i]$ is given by the composition 
\[
\underline{\Map}(\Fil^1_{\mathrm{Hdg}}L(\mathcal{Y}),Z^1_{\underline{A}}[i])\xrightarrow{\varphi^*}\underline{\Map}(\varphi^*\Fil^1_{\mathrm{Hdg}}L(\mathcal{Y}),\varphi^*Z^1_{\underline{A}}[i])\xrightarrow{(1\otimes q_2)\circ ()\circ \psi}\underline{\Map}(\Fil^1_{\mathrm{Hdg}}L(\mathcal{Y}),H^1_{\underline{A}}[i]).
\]
We now set for any $i\in\Int$\footnote{The geometric meaning of this definition---when specialized to the case where $\underline{A} = \underline{\Prism}_R$ is the frame associated with the prismatic cohomology of a semiperfect $\Field_p$-algebra $R$---will be explained in Section~\ref{sec:technical_repbility}.}
\[
\Gamma_{\underline{A},\psi}(\Fil^{1}_{\mathrm{Hdg}}L(\mathcal{Y})[-i]) \defn \tau^{\leq 0}\hker(q_{1,\psi}[i]-q_{2,\psi}[i]).
\]
This is a $\Mod[\mathrm{cn}]{\Field_p}$-valued prestack over $\Gamma_{\Fzip}$.

We also obtain two further maps
\[
q_{1,\psi}[1]\circ \delta^+,q_{2,\psi}[1]\circ \delta^+:\Gamma_{\Fzip}\to \underline{\Map}(\Fil^1_{\mathrm{Hdg}}L(\mathcal{Y}),H^1_{\underline{A}}[1]).
\]
\end{construction}  

Here is the main result of this subsection:
\begin{proposition}
\label{prop:abstract_devissage_to_fzips}
\begin{enumerate}
   \item The maps $\delta^0,q_{1,\psi}[1]\circ \delta^+,q_{2,\psi}[1]\circ \delta^+$ are all canonically isomorphic. In particular, for $i=1,2$, $q_{i,\psi}[1]$ yields a map $\eta_i$ obtained as the composition of
   \begin{align*}
    \Gamma_{(t=0)}&\xrightarrow[\simeq]{\eqref{eqn:gamma_u_t_plus_pulled_back}}\Gamma_{\Fzip}\times_{0,\underline{\Map}(\Fil^1_{\mathrm{Hdg}}L(\mathcal{Y}),Z^1_{\underline{A}}[1]),\delta^+}\Gamma_{\Fzip}\\
    &\to\Gamma_{\Fzip}\times_{0,\underline{\Map}(\Fil^1_{\mathrm{Hdg}}L(\mathcal{Y}),H^1_{\underline{A}}[1]),q_{i,\psi}[1]\circ\delta^+}\Gamma_{\Fzip}\\
    &\simeq \Gamma_{\Fzip}\times_{0,\underline{\Map}(\Fil^1_{\mathrm{Hdg}}L(\mathcal{Y}),H^1_{\underline{A}}[1]),\delta^0}\Gamma_{\Fzip}\\
    &\xrightarrow[\simeq]{\eqref{eqn:gamma_u_t_0_pulled_back}}\Gamma_{(t=u=0)}.
    \end{align*}

    \item There are natural isomorphisms of maps $\eta_i\simeq \zeta_i$ for $i=1,2$. In particular, we have
  \[
\begin{diagram}
\Gamma_{\underline{A}}(\mathcal{Y})&\rTo^{\simeq}&\mathrm{eq}\bigl(\Gamma_{(t=0)}&\pile{\rTo^{\eta_1}\\ \rTo_{\eta_2}}&\Gamma_{(t=u=0)}\bigr).
\end{diagram}
\]

    \item There is a canonical Cartesian diagram
\[
\begin{diagram}
\Gamma_{\underline{A}}(\mathcal{Y})&\rTo&\Gamma_{\Fzip}\\
\dTo&&\dTo_0\\
\Gamma_{\Fzip}&\rTo&\Gamma_{\underline{A},\psi}(\Fil^1_{\mathrm{Hdg}}L(\mathcal{Y})[-1]).
\end{diagram}
\]
\end{enumerate}
\end{proposition}

\begin{remark}
\label{rem:note_on_proof_of_def_theory_prop}
Assertion (3) is simply a reinterpretation of the description of $\Gamma_{\underline{A}}(\mathcal{Y})$ given to us by the second part of assertion (2), which in turn is immediate from the first part and Remark~\ref{rem:zeta1zeta2maps}. Moreover, it is straightforward to see from the deformation theory that assertion (1) and the first part of assertion (2) both hold for $i=1$. The non-trivial part of the proposition therefore is seeing that they are also valid for $i=2$. 
\end{remark} 

\begin{remark}
\label{rem:connection_with_drinfeld_shimurian}
Though this is not obvious from the formulation here, the validity of assertion (1) for $i=2$ contains as a special case the key identity (5.6) of~\cite[Lemma 5.3.3]{drinfeld2023shimurian}. 
\end{remark}

 Just as in the proof of Proposition~\ref{prop:def_theory_frames}, the main input to the proof of Proposition~\ref{prop:abstract_devissage_to_fzips} is deformation theory, which we encapsulate in the Lemmas~\ref{lem:abstract_devissage_def_theory} and~\ref{lem:abstract_devissage_witt_vectors} below. For now, we make some preliminary remarks.

\begin{remark}
\label{rem:increasingly_filtered_modules}
Suppose that $\Fil_\bullet S$ is a non-negatively and increasingly filtered animated commutative ring equipped with a map $\pi:S\to \degzero{S}\defn \Fil_0S$ of animated commutative rings, and write $f$ for the composition $\degzero{S} = \Fil_0S\to S \to \degzero{S}$. The map $\pi$ induces a filtered morphism $\Fil_\bullet \pi:\Fil_\bullet S\to \Fil_\bullet^{\mathrm{triv}}\degzero{S}$. If $\gr_\bullet S$ is the associated graded object, then we have a projection $\varpi:\gr_\bullet S \to \gr_0S = \degzero{S}$ onto the degree $0$ part whose composition with the endomorphism $f$ agrees with $\gr_\bullet \pi$ if we view $\degzero{S}$ as a trivially graded animated commutative ring. Given a module $\Fil_\bullet M$ over $\Fil_\bullet S$, base-change along $\Fil_\bullet\pi$ gives an increasingly filtered module $\Fil_\bullet\degzero{M}$ over $\degzero{S}$. Taking the associated graded and then graded base-change along $\varpi$ gives us a graded module $\degzero{(\gr_\bullet M)}$ over $\degzero{S}$. We have a canonical isomorphism
\[
f^*\degzero{(\gr_\bullet M)}\xrightarrow{\simeq}\gr_\bullet \degzero{M}.
\]
Note in particular that for all $i\in \Int$, we have an $\degzero{S}$-linear map $\Fil_iM\to f_*\Fil_i\degzero{M}$.
\end{remark}

\begin{lemma}
\label{lem:increasingly_filtered_modules}
With the setup from Remark~\ref{rem:increasingly_filtered_modules}, let $\Fil_\bullet M$ and $\Fil_\bullet N$ be modules over $\Fil_\bullet S$ with the following properties:
\begin{enumerate}
   \item $\degzero{(\gr_i M)}\simeq 0$ for $i>1$;
   \item $\Fil_i N\simeq 0$ for $i<1$.
\end{enumerate}
Then there is a canonical isomorphism
\[
\zeta:\Map_{\Fil_\bullet S}(\Fil_\bullet M,\Fil_\bullet N)\xrightarrow{\simeq}\Map_{\degzero{S}}(\degzero{(\gr_1 M)},\Fil_1N) 
\]
such that we have a commuting diagram
\[
\begin{diagram}
\Map_{\degzero{S}}(\degzero{(\gr_1 M)},\Fil_1N)&\lTo_{\simeq}^\zeta&\Map_{\Fil_\bullet S}(\Fil_\bullet M,\Fil_\bullet N)&\rTo&\Map_S(M,N)\\
\dTo&&\dTo&&\dTo\\
\Map_{\degzero{S}}(\gr_1\degzero{M},\Fil_1\degzero{N})&\lTo_{\simeq}&\Map_{\Fil^{\mathrm{triv}}_\bullet\degzero{S}}(\Fil_\bullet\degzero{M},\Fil_\bullet\degzero{N})&\rTo&\Map_{\degzero{S}}(\degzero{M},\degzero{N})
\end{diagram}
\]
Here, the vertical arrow on the left is the composition
\[
\Map_{\degzero{S}}(\degzero{(\gr_1 M)},\Fil_1N)\to \Map_{\degzero{S}}(\degzero{(\gr_1 M)},f_*\Fil_1\degzero{N})\xrightarrow{\simeq}\Map_{\degzero{S}}(f^*\degzero{(\gr_1 M)},\Fil_1\degzero{N})\xrightarrow{\simeq}\Map_{\degzero{S}}(\gr_1\degzero{M},\Fil_1\degzero{N}).
\]

\end{lemma} 
\begin{proof}
Taking the associated graded and then base-change along $\varpi$ gives a canonical map
\[
\Map_{\Fil_\bullet S}(\Fil_\bullet M,\Fil_\bullet N)\to \Map_{\gr^{\mathrm{triv}}_\bullet\degzero{S}}(\degzero{(\gr_\bullet M)},\degzero{(\gr_\bullet N)}).
\]
Using our hypotheses, one can now use Lemma~\ref{lem:graded_weight_filtration} to see that this map is an isomorphism and also that the right hand side can be identified with $\Map_{\degzero{S}}(\degzero{(\gr_1M)},\Fil_1N)$. This gives us the isomorphism $\zeta$. The map on the bottom left is obtained from restriction to $\Fil_1\degzero{M}$: The fact that any such restriction must factor through $\gr_1\degzero{M}$ and that the resulting map of mapping spaces is an isomorphism can also be deduced using the weight filtration from Lemma~\ref{lem:graded_weight_filtration}.
\end{proof}

\begin{lemma}
\label{lem:abstract_devissage_def_theory}
Suppose that $\underline{A}\to \underline{A}_\defbase$ is a square-zero extension of prismatic, $p$-adic frames equipped with $p$-adic filtration, and suppose that assertions (1) and (2) of Proposition~\ref{prop:abstract_devissage_to_fzips} hold with $\underline{A}$ replaced with $\underline{A}_\defbase$. Then they hold for $\underline{A}$.
\end{lemma}

As usual, we will need to set up some prepwork before we embark on the proof.
\begin{notation}
Set $R = R_A$ and $R_\defbase = R_{A_{\defbase}}$. We have a canonical equivalence $R_{\et} \simeq R_{\defbase,\et}$. There are the counterparts of the various sheaves on $R_{\et}$ considered above, but obtained from $\underline{A}_\defbase$ instead of $\underline{A}$. We will distinguish these counterparts with a $\defbase$ index: $Y^\Sigma_{\defbase}$, $\Gamma_{(t=0),\defbase}(\mathcal{Y})$, $\Gamma_{\Fzip,\defbase}$, $L_{\defbase}(\mathcal{Y})$, etc. We will also simplify $L_{\defbase}(\mathcal{Y})$ to $L_{\defbase}$ in what follows.

Set
\[
\Fil^\bullet_{\mathrm{Hdg}}\overline{K} \defn \hker(\Fil^\bullet_{\mathrm{Hdg}}\overline{A}\to \Fil^\bullet_{\mathrm{Hdg}}\overline{A}_\defbase)\;;\;\Fil_\bullet^{\mathrm{conj}}\overline{K} \defn \hker(\Fil_\bullet^{\mathrm{conj}}\overline{A}\to \Fil_\bullet^{\mathrm{conj}}\overline{A}_\defbase).
\]
Also, set $I = \hker(R\to R_\defbase)$.
\end{notation}

\begin{construction}
Write $\Fil^{\mathrm{conj}}_{\bullet,\defbase}$ for the sheaf of increasingly filtered animated commutative rings given by $R'\mapsto \Fil^{\mathrm{conj}}_{\bullet}\overline{A}'_{\defbase}$. Over $\Gamma_{(t=0),\defbase}(\mathcal{Y})$, we obtain a sheaf of increasingly filtered modules $\Fil^{\mathrm{conj}}_{\bullet}\mathcal{L}_\defbase$ over $\Fil^{\mathrm{conj}}_{\bullet,\defbase}$, whose fiber over a point $x\in \Gamma_{(t=0),\defbase}(\mathcal{Y})(R')$ with underlying map
\[
y:\mathcal{R}_+(\Fil^{\mathrm{conj}}_\bullet A'_\defbase) \to \mathcal{Y}^\preoneb
\]
is the pullback $\Fil^{\mathrm{conj}}_\bullet\mathcal{L}_{\defbase,y}$ along $y$ of the relative cotangent complex of $\mathcal{Y}^\preoneb$ over $\Rees(\Fil^\bullet A)\otimes\Field_p$. Filtered base-change to the structure sheaf $\mathcal{O}_\defbase:R'\mapsto R'_{\defbase}$ with its trivial increasing filtration\footnote{The map in question is obtained from $\Fil^{\mathrm{conj}}_\bullet\overline{A}_\defbase \to R_\defbase$.} yields the pullback along the map $\Gamma_{(t=0),\defbase}(\mathcal{Y})\to Y^+_{\defbase}$ of the filtered module $\Fil^{\mathrm{conj}}_\bullet L_\defbase$. Similarly, taking the associated graded and then graded base-change to the structure sheaf $\mathcal{O}_\defbase$ with its trivial grading\footnote{Here, we are viewing $\mathcal{O}_\defbase$ as the zeroth graded piece of $\Fil^\bullet_{\mathrm{conj}}$.} gives the pullback along $\Gamma_{(t=0),\defbase}(\mathcal{Y})\to Y^0_{\defbase}$ of $\gr^\bullet_{\mathrm{Hdg}}L_\defbase$. 
\end{construction}

\begin{construction}
Write $\Fil^\bullet_{\mathrm{Hdg},\defbase}$ for the sheaf of decreasingly filtered animated commutative rings given by $R'\mapsto \Fil^\bullet_{\mathrm{Hdg}}\overline{A}'_{\defbase}$. Over $\Gamma_{(u=0),\defbase}(\mathcal{Y})$, we obtain from the cotangent complex of $\mathcal{Y}^\preoneb$ a sheaf of decreasingly filtered modules $\Fil^\bullet_{\mathrm{Hdg}}\mathcal{L}_\defbase$ over $\Fil^\bullet_{\mathrm{Hdg},\defbase}$. Filtered base-change to the structure sheaf $\mathcal{O}_\defbase:R'\mapsto R'_{\defbase}$ with its trivial decreasing filtration yields the pullback along the map $\Gamma_{(u=0),\defbase}(\mathcal{Y})\to Y^-_{\defbase}$ of the filtered module $\Fil^\bullet_{\mathrm{Hdg}} L_\defbase$. Similarly, taking the associated graded and then graded base-change to the structure sheaf $\mathcal{O}_\defbase$ with its trivial grading gives the pullback along $\Gamma_{(u=0),\defbase}(\mathcal{Y})\to Y^0_{\defbase}$ of $\gr^\bullet_{\mathrm{Hdg}}L_\defbase$. 
\end{construction}

\begin{construction}
Write $\mathrm{gr}^\bullet_\defbase$ for the sheaf of graded animated commutative rings given by
\[
R'\mapsto \gr^\bullet_{\mathrm{Hdg}}\overline{A}'_\defbase\simeq \gr^{\mathrm{conj}}_{\bullet}\overline{A}'_\defbase.
\]
Over $\Gamma_{(t=u=0),\defbase}(\mathcal{Y})$, entirely analogously to the previous constructions, we obtain a sheaf of graded $\mathrm{gr}^\bullet_\defbase$-modules $\mathcal{L}_{\bullet,\defbase}$ over $\Gamma_{(t=u=0),\defbase}(\mathcal{Y})$. Graded base-change to the structure sheaf gives the pullback along $\Gamma_{(t=u=0),\defbase}(\mathcal{Y})\to Y^0_{\defbase}$ of $\gr^\bullet_{\mathrm{Hdg}}L_\defbase$. Moreover, the pullback of this graded module to $\Gamma_{(t=0),\defbase}(\mathcal{Y})$ (resp. $\Gamma_{(u=0),\defbase}(\mathcal{Y})$) is canonically isomorphic to the associated graded for $\Fil_\bullet^{\mathrm{conj}}\mathcal{L}_\defbase$ (resp. $\Fil_{\mathrm{Hdg}}^{\bullet} \mathcal{L}_\defbase$).
\end{construction}  

\begin{remark}
The $\overline{A}_\defbase$-module $\overline{J} \defn \Fil^1_{\mathrm{Hdg}}\overline{K}$ can be equipped with the structure of a filtered module $\Fil_\bullet\overline{J}$ over $\Fil^{\mathrm{conj}}_\bullet\overline{A}_\defbase$ with
\[
\Fil_i\overline{J} = \begin{cases}
\Fil_i\overline{K}\times_{\overline{K}}\overline{J}&\text{if $i>0$};\\
0&\text{otherwise}.
\end{cases}
\]
Note in particular that we have
\begin{align*}
\Fil_1\overline{J} \xrightarrow{\simeq} \Fil_1^{\mathrm{conj}}\overline{K}\times_{\overline{K}}\Fil^1_{\mathrm{Hdg}}\overline{K}\xrightarrow{\simeq}Z^1_{\overline{K}} \defn \hker(Z^1_{\underline{A}}\to Z^1_{\underline{A}_\defbase}).
\end{align*}
Moreover, the map $\gr_\bullet\overline{J}\to \gr^{\mathrm{conj}}_\bullet\overline{K}$ factors through 
\[
\gr^{\mathrm{conj}}_{\ge 1}\overline{K} \defn \hker(\gr^{\mathrm{conj}}_\bullet\overline{K}\to \gr^{\mathrm{triv}}_\bullet I).
\]
Note also the composition of maps
\begin{align}\label{eqn:barJ_fil1_map}
\overline{J} = \Fil^1_{\mathrm{Hdg}}\overline{K} \to \gr^1_{\mathrm{Hdg}}\overline{K}\xrightarrow{\simeq}\gr_1^{\mathrm{conj}}\overline{K}.
\end{align}
\end{remark} 

\begin{proof}
[Proof of Lemma~\ref{lem:abstract_devissage_def_theory}]
Consider the map $\Gamma_{(t=0)}\to \Gamma_{(t=0),\defbase}\times_{\Gamma_{\Fzip,\defbase}}\Gamma_{\Fzip}$: An unwinding of the definitions shows that this can be rewritten as
\begin{align*}
\Gamma_{(t=0)}(\mathcal{Y})\times_{Y^\Sigma}\Gamma_{\Fzip}\to \left[\Gamma_{(t=0),\defbase}(\mathcal{Y})\times_{Y^\Sigma_\defbase}Y^\Sigma\right]\times_{Y^\Sigma}\Gamma_{\Fzip}.
\end{align*}

Via this and deformation theory we obtain a Cartesian square
\[
\Square{\Gamma_{(t=0)}}{}{\Gamma_{(t=0),\defbase}\times_{\Gamma_{\Fzip,\defbase}}\Gamma_{\Fzip}}{}{0}{\Gamma_{(t=0),\defbase}\times_{\Gamma_{\Fzip,\defbase}}\Gamma_{\Fzip}}{}{\underline{\Map}(\Fil^{\mathrm{conj}}_\bullet \mathcal{L}_\defbase,\Fil_\bullet \overline{J}[1]).}
\]

We also obtain similar Cartesian squares
\begin{align*}
\Square{\Gamma_{(u=0)}}{}{\Gamma_{(u=0),\defbase}\times_{\Gamma_{\Fzip,\defbase}}\Gamma_{\Fzip}}{}{0}{\Gamma_{(u=0),\defbase}\times_{\Gamma_{\Fzip,\defbase}}\Gamma_{\Fzip}}{}{\underline{\Map}(\mathcal{L}_\defbase,\overline{J}[1])}\;;\qquad&\Square{\Gamma_{(t=u=0)}}{}{\Gamma_{(t=u=0),\defbase}\times_{\Gamma_{\Fzip,\defbase}}\Gamma_{\Fzip}}{}{0}{\Gamma_{(t=u=0),\defbase}\times_{\Gamma_{\Fzip,\defbase}}\Gamma_{\Fzip}}{}{\underline{\Map}(\mathcal{L}_{\bullet,\defbase},\gr_{\ge 1}^{\mathrm{conj}}\overline{K}[1]).}
\end{align*}

Lemma~\ref{lem:increasingly_filtered_modules} gives us canonical isomorphisms
\begin{align}\label{eqn:t=0deftheory}
\underline{\Map}(\Fil^{\mathrm{conj}}_\bullet \mathcal{L}_\defbase,\Fil_\bullet \overline{J}[1])\xrightarrow{\simeq}\underline{\Map}(\Fil^1_{\mathrm{Hdg}}L_\defbase,\Fil_1\overline{J}[1])\xrightarrow{\simeq}\underline{\Map}(\Fil^1_{\mathrm{Hdg}}L_\defbase,Z^1_{\underline{K}}[1]).
\end{align}
Further, using Lemma~\ref{lem:graded_weight_filtration}, one deduces that there are canonical isomorphisms
\begin{align}\label{eqn:t=u=0deftheory}
\underline{\Map}(\mathcal{L}_{\bullet,\defbase},\gr_{\ge 1}^{\mathrm{conj}}\overline{K}[1])\xrightarrow{\simeq}\underline{\Map}(\Fil^1_{\mathrm{Hdg}}L_\defbase,\gr^{\mathrm{conj}}_1\overline{K}[1])\xrightarrow{\simeq}\underline{\Map}(\Fil^1_{\mathrm{Hdg}}L_\defbase,H^1_{\underline{K}}[1]),
\end{align}
where $H^1_{\underline{K}} \defn \hker(H^1_{\underline{A}}\to H^1_{\underline{A}_\defbase})$.

The map $\Gamma_{(t=0)}\to \Gamma_{(t=u=0)}$ is compatible via these isomorphisms and the Cartesian squares above with the composition
\begin{align*}
\underline{\Map}(\Fil^1_{\mathrm{Hdg}}L_\defbase,Z^1_{\underline{K}}[1])\to \underline{\Map}(\Fil^1_{\mathrm{Hdg}}L_\defbase,H^1_{\underline{K}}[1])
\end{align*}
induced by $q_1:Z^1_{\underline{K}}\to H^1_{\underline{K}}$.

The map $\Gamma_{(u=0)}\to \Gamma_{(t=u=0)}$ is compatible via the map
\begin{align}\label{eqn:u=0tot=u=0}
\underline{\Map}(\mathcal{L}_\defbase,\overline{J}[1])\to \underline{\Map}(L_\defbase,\gr_1^{\mathrm{conj}}\overline{K}[1])\to  \underline{\Map}(\Fil^1_{\mathrm{Hdg}}L_\defbase,H^1_{\underline{K}}[1]),
\end{align}
where the first map is obtained from post-composition with the map~\eqref{eqn:barJ_fil1_map}, and the second via restriction to $\Fil^1_{\mathrm{Hdg}}L_\defbase$. Note that the first mapping space is of morphisms linear over the sheaf $R'\mapsto \overline{A}'$ while the other two spaces are of maps linear over the structure sheaf.

The map from~\eqref{eqn:t=0tou=t=0_second} is in turn compatible with the composition
\begin{align*}
\underline{\Map}(\Fil^1_{\mathrm{Hdg}}L_\defbase,Z^1_{\underline{K}}[1])&\xrightarrow[\simeq]{\eqref{eqn:t=0deftheory}}\underline{\Map}(\Fil^{\mathrm{conj}}_\bullet \mathcal{L}_\defbase,\Fil_\bullet \overline{J}[1])\\
&\to \underline{\Map}(\mathcal{L}_\defbase, \overline{J}[1])\\
&\xrightarrow{\eqref{eqn:u=0tot=u=0}}\underline{\Map}(\Fil^1_{\mathrm{Hdg}}L_\defbase,H^1_{\underline{K}}[1]).
\end{align*}
Lemma~\ref{lem:increasingly_filtered_modules} tells us that the resulting map is canonically isomorphic to the composition
\begin{align*}
\underline{\Map}(\Fil^1_{\mathrm{Hdg}}L_\defbase,Z^1_{\underline{K}}[1])&\to \underline{\Map}(\Fil^1_{\mathrm{Hdg}}L_\defbase,\varphi_*H^1_{\underline{K}}[1])\\
&\xrightarrow{\simeq} \underline{\Map}(\varphi^*\Fil^1_{\mathrm{Hdg}}L_\defbase,H^1_{\underline{K}}[1])\\
&\xrightarrow{\simeq}\underline{\Map}(\gr^{\mathrm{conj}}_1L_\defbase,H^1_{\underline{K}}[1])\\
&\to\underline{\Map}(L_\defbase,H^1_{\underline{K}}[1])\\
&\to\underline{\Map}(\Fil^1_{\mathrm{Hdg}}L_\defbase,H^1_{\underline{K}}[1]).
\end{align*}

The three preceding paragraphs complete the verification of Lemma~\ref{lem:abstract_devissage_def_theory}.
\end{proof}

\begin{lemma}
\label{lem:abstract_devissage_witt_vectors}
Assertions (1) and (2) of Proposition~\ref{prop:abstract_devissage_to_fzips} hold for $\underline{A} = \underline{W(R)}$ the Witt vector frame associated with a semiperfect $\Field_p$-algebra $R$.
\end{lemma}
\begin{proof}
We'll follow the notation simplifications from Lemma~\ref{lem:abstract_devissage_def_theory}. To begin, Lemma~\ref{lem:abstract_map_to_fzips}, combined with Example~\ref{ex:z1h1_witt}, tells us that we have Cartesian squares 
\[
\Square{\Gamma_{(t=0)}\times_{\Gamma_{\overline{W(R)}}}\Gamma_{(u=0)}}{}{\Gamma_{\Fzip}}{}{0}{\Gamma_{\Fzip}}{\delta^+\;\qquad\;}{\underline{\Map}(\Fil^1_{\mathrm{Hdg}}L(\mathcal{Y}),\varphi_*\Ga^\sharp[2])}
\]
and
\begin{align*}
\Square{\Gamma_{(t=u=0)}}{}{\Gamma_{\Fzip}}{}{0}{\Gamma_{\Fzip}}{\delta^0\;\;}{\underline{\Map}(\Fil^1_{\mathrm{Hdg}}L(\mathcal{Y}),\Ga^\sharp[2]).}
\end{align*}

Since $\overline{W(R)}$ is a square-zero extension of $R$ with fiber $\Ga^\sharp(R)[1]$, standard deformation theory also gives us a Cartesian square
\[
\Square{\Gamma_{(u=0)}}{}{\Gamma_{\Fzip}}{}{0}{\Gamma_{\Fzip}}{\delta^-}{\underline{\Map}(L(\mathcal{Y}),\Ga^\sharp[2]).}
\]
Furthermore, the composition
\[
\Gamma_{\Fzip}\xrightarrow{\delta^-} \underline{\Map}(L(\mathcal{Y}),\Ga^\sharp[2])\to \underline{\Map}(\Fil^1_{\mathrm{Hdg}}L(\mathcal{Y}),\Ga^\sharp[2])
\]
is canonically isomorphic to $\delta^0$.

An unwinding of the deformation theory used above now shows that $\delta^-$ agrees with the composition
\[
\Gamma_{\Fzip}\xrightarrow{\delta^+}\underline{\Map}(\Fil^1_{\mathrm{Hdg}}L(\mathcal{Y}),\varphi_*\Ga^\sharp[2])\xrightarrow{\simeq}\underline{\Map}(\gr_1^{\mathrm{conj}}L(\mathcal{Y}),\Ga^\sharp[2])\to \underline{\Map}(L(\mathcal{Y}),\Ga^\sharp[2])
\]
Here, in the isomorphism in the middle, we have used the isomorphism $\varphi^*\Fil^1_{\mathrm{Hdg}}L(\mathcal{Y})\xrightarrow{\simeq}\gr^{\mathrm{conj}}_1L(\mathcal{Y})$ and adjunction for $\varphi$. 

Combining the last two paragraphs with the definition of $q_{2,\psi}$ and the explicit description of $q_2$ in Example~\ref{ex:z1h1_witt} now completes the proof.
\end{proof}

\begin{proof}
[Proof of Proposition~\ref{prop:abstract_devissage_to_fzips}]
To begin, note that the last assertion of the proposition follows immediately from the definitions and the first two assertions. 

The verification of the other two assertions follows the same format as the proof of Proposition~\ref{prop:def_theory_frames}, so we will be terse. First, we can use Lemma~\ref{lem:abstract_devissage_def_theory} and the nilcompleteness of the various prestacks involved to reduce to the case where the Rees algebra $\mathrm{Rees}(\Fil^\bullet A)$ is a discrete graded ring.

Now, Lemma~\ref{lem:laminated_frames} gives us a canonical map of frames $\underline{A}\to \underline{W(R)}$, where we have set $R = \pi_0(R_A)$. Since $R_A$ is semiperfect, this map is surjective on the underlying filtered animated commutative rings: First, the map $A\to W(R)$ is surjective. Indeed, by $p$-completeness, this assertion is equivalent to saying that $\pi_0(\overline{A})\to W(R)/pW(R)$ is surjective. But the semiperfectness of $R_A$ tells us that $W(R)/pW(R)\simeq R$. Next, the map $\Fil^i A \to \Fil^i_{\mathrm{Lau}}W(\pi_0(R_A))$ is surjective for each $i\ge 1$. By the definition of the Lau filtration on $W(R)$ and the fact that $\Fil^\bullet A$ is $p$-adic, it suffices to check this for $i=1$, where it follows from the fact that $\Fil^1_{\mathrm{Lau}}W(R) = F_*W(R)$ is a submodule of $W(R)$ via the map $V$.

As in Remark~\ref{rem:operations_on_frames}, classical base-change along $A\to W(R)$ (viewed as the derived base-change, followed by taking the $0$-truncation), yields a non-negatively filtered animated commutative ring $\Fil^\bullet W(R)$ lifting $W(R)$, underlying a frame ${}_A\underline{W(R)}$, and admitting surjective maps of frames
\[
\underline{A}\to {}_A\underline{W(R)}\to \underline{W(R)}.
\]
The kernel of the first map is locally nilpotent mod-$p$, and so---using strong integrability and Lemma~\ref{lem:abstract_devissage_def_theory} once again---it is enough to know that the proposition holds for the frame ${}_A\underline{W(R)}$. 

If $\Fil^\bullet K$ is the kernel of $\Fil^\bullet W(R)\to \Fil^\bullet_{\mathrm{Lau}}W(R)$, then the frame structure on $\Fil^\bullet W(R)$ descends to the filtered quotient $\Fil^\bullet_{(m)}W(R)$, where
\[
\Fil^i_{(m)}W(R) = \Fil^iW(R)/\bigl(\sum_{j_1+\ldots+j_m = i}\mathrm{im}(\Fil^{j_1}K\otimes_{W(R)}\otimes{\cdots}\otimes_{W(R)}\Fil^{j_m}K)\bigr)
\]
Now, we have $\Fil^\bullet_{(1)}W(R) = \Fil^\bullet_{\mathrm{Lau}}W(R)$ and the map $\Fil^\bullet_{(m+1)}W(R)\to\Fil^\bullet_{(m)}W(R)$ is a square-zero extension, for each $m\ge 1$. Therefore, we can use filtered integrability and the deformation argument above to reduced to the case of the Witt frame $\underline{W(R)}$, which is taken care of by Lemma~\ref{lem:abstract_devissage_witt_vectors}.
\end{proof}

Here is a useful corollary to Proposition~\ref{prop:abstract_devissage_to_fzips}.
\begin{corollary}
\label{cor:abstract_fzip_smooth_torsor}
Suppose that the following additional conditions hold:
\begin{itemize}
   \item $\mathcal{Y}^\preoneb$ is \emph{smooth} over $\Rees(\Fil^\bullet A)\otimes\Field_p$;
   \item Its relative tangent complex $\mathbb{T}_{\mathcal{Y}^\preoneb}$ is $1$-connective;
\end{itemize}
Then $\Gamma_{\underline{A}}(\mathcal{Y})\to \Gamma_{\Fzip}(\mathcal{Y})$ is a torsor under $\Gamma_{\underline{A},\psi}(\Fil^1_{\mathrm{Hdg}}L(\mathcal{Y}))$.
\end{corollary}   
\begin{proof}
It is enough to know that, in assertion (2) of Lemma~\ref{lem:abstract_map_to_fzips}, we have the stronger assertion that
\[
\Gamma_{(t=0)}(\mathcal{Y})\to Y^+\times_{{}^\varphi Y^0}Y^0
\]
is a torsor under $\underline{\Map}(\Fil^1_{\mathrm{Hdg}}L(\mathcal{Y}),Z^1_{\underline{A}})$.

Looking at the proof of that assertion, we find that we need to know that the map
\[
(\Fil^1_{\mathrm{Hdg}}L(\mathcal{Y}))^\vee\otimes_{\mathcal{O}}\Fil^{\mathrm{conj}}_1\to (\Fil^1_{\mathrm{Hdg}}L(\mathcal{Y}))^\vee\otimes_{\mathcal{O}}\varphi_*\mathcal{O}
\]
is surjective on connective covers. But in fact our hypothesis on the tangent complex shows that the source and target of this map are already connective. Therefore, it is enough to know that the map
\[
R_{A'}\simeq \Fil^{\mathrm{conj}}_0A'\to \varphi_*R_{A'}
\]
is surjective for any \'etale map $R_A\to R_{A'}$. But this is guaranteed to us by the semiperfectness of $R_A$, and the fact that any \'etale algebra over a semiperfect ring is also semiperfect; see Lemma~\ref{lem:f-semiperf_etale} below.
\end{proof}

\section{The stacks of Drinfeld and Bhatt-Lurie}
\label{sec:bld_stacks}

\subsection{Transmutation}
\label{subsec:transmutation}

Suppose that we have a map $\pi:Z\to Y$ of $p$-adic formal prestacks such that $Z$ is a \defnword{relative ring prestack} over $Y$: For us, this will mean that we have specified a lift of the associated functor
\[
\mathrm{CRing}^{p\text{-nilp}}_{/Y}\xrightarrow{(C,y)\mapsto Z((C,y))}\mathrm{Spc}
\]
to a presheaf valued in $\mathrm{CRing}$, which we will denote by the same symbol.

Here, $\mathrm{CRing}^{p\text{-nilp}}_{/Y}$ is the $\infty$-category of pairs $(C,y)$ with $C\in \mathrm{CRing}^{p\text{-nilp}}$ and $y\in Y(C)$, and $Z((C,y))$ is the fiber of $Z(C)$ over $y$. Then, for any $R\in \mathrm{CRing}^{p\text{-nilp}}$, its \defnword{transmutation with respect to $\pi$} is the $p$-adic formal prestack over $Y$ given by
\begin{align*}
\mathfrak{Tr}_\pi(R):\mathrm{CRing}^{p\text{-nilp}}_{/Y}&\to \mathrm{Spc}\\
(C,y)&\mapsto \Map_{\mathrm{CRing}}(R,Z((C,y)))
\end{align*}

This gives us a limit preserving functor
\[
\mathrm{CRing}^{p\text{-nilp},\op}\xrightarrow{\Spec R \mapsto \mathfrak{Tr}_\pi(R)}\mathrm{PStk}_{/Y}.
\]

\subsection{Cartier-Witt divisors and prismatizations}
\label{subsec:cartier-witt}

Here, we quickly recall the story of (derived) absolute prismatizations from~\cite[\S 8]{bhatt2022prismatization}. 

\subsubsection{}
To begin, we have the $p$-adic formal prestack $\Int_p^\Prism$ (the notation is from~\cite{bhatt_lectures}) of Cartier-Witt divisors, denoted $\mathrm{WCart}$ in~\cite{bhatt2022absolute}. For $R\in \mathrm{CRing}^{p\text{-nilp}}$, $\Int_p^\Prism(R)$ parameterizes surjective maps $\pi:W(R)\twoheadrightarrow \overline{W(R)}$ of animated rings such that two properties hold:
\begin{itemize}
   \item $I = \hker(\pi)$ is a locally free $W(R)$-module of rank $1$;
   \item The map $\pi_0(I) \simeq I\otimes_{W(R)}W(\pi_0(R))\to W(\pi_0(R))$ is a Cartier-Witt divisor in the sense of~\cite[\S 3.1.1]{bhatt2022absolute}. 
\end{itemize}

The second condition means that, Zariski-locally on $\Spec R$, we have a $W(\pi_0(R))$-linear isomorphism $\pi_0(I)\simeq W(\pi_0(R))$ such that the composition $W(\pi_0(R))\simeq \pi_0(I)\to W(\pi_0(R))$ is given by multiplication by a \defnword{distinguished element} $d\in W_{\mathrm{dist}}(\pi_0(R))$, given in Witt coordinates by $(d_0,d_1,\ldots)$ with $d_0\in \pi_0(R)$ nilpotent mod-$p$ and with $d_1\in \pi_0(R)^\times$.

In this situation, we will call the map $I\to W(R)$ a \defnword{Cartier-Witt divisor} over $R$.

This description shows (see~\cite[Proposition 8.4]{bhatt2022prismatization}):
\begin{proposition}
\label{prop:ZpDelta_presentation}
We have $\Int_p^\Prism \simeq W_{\mathrm{dist}}/W^\times$, where $W_{\mathrm{dist}}$ is represented by the formal spectrum of $\Int_p[x_0,x_1^{\pm 1},x_2,\ldots]_{(x_0,p)}^{\wedge}$, and inherits the $W^\times$-action from the natural one on $W$, where $W$ is represented by the formal spectrum of $\Int_p[x_0,x_1,\ldots]_p^{\wedge}$. In particular, $\Int_p^\Prism$ is \emph{classical}.
\end{proposition}

\subsubsection{}
Over $\Int_p^\Prism$, we have the tautological relative ring prestack $\mathbb{G}_a^\Prism$ given by $(W(R)\xrightarrow{\pi}\overline{W(R)})\mapsto \overline{W(R)}$. Transmutation with respect to this (see~\S\ref{subsec:transmutation}) now gives a functorial assignment $R\to R^{\Prism}$ from $\mathrm{CRing}^{p\text{-nilp}}$ to $\mathrm{PStk}_{/\Int_p^\Prism}$. Concretely, $R^\Prism$ associates to any $(W(C)\twoheadrightarrow\overline{W(C)})\in \Int_p^\Prism(C)$ the space $\Map_{\mathrm{CRing}}(R,\overline{W(C)})$. We call $R^\Prism$ the \defnword{prismatization} of $R$.

\begin{remark}
\label{rem:prismatization_of_Fp}
We have a canonical equivalence $\Spf \Int_p\xrightarrow{\simeq}\Field_p^\Prism$ induced by the Cartier-Witt divisor $W(\Field_p) = \Int_p\xrightarrow{p} \Int_p=W(\Field_p)$.
\end{remark}

\begin{remark}
\label{rem:frob_endomorphism}
There is a canonical `Frobenius lift' $\varphi:\Int_p^\Prism\to \Int_p^\Prism$ arising from the map $F:W\to W$, which carries a Cartier-Witt divisor $I\to W(R)$ to $F^*I \to W(R)$
\end{remark}

\begin{remark}
\label{rem:prisms_and_prismatization}
If $(A,I)$ is an animated prism as in~\cite[\S 2]{bhatt2022prismatization}, then there is a canonical map $\iota_{(A,I)}:\Spf A \to \Int_p^{\Prism}$ associating to each $p$-nilpotent $A$-algebra $C$, the Cartier-Witt divisor $I\otimes_AW(C) \to W(C)$. Here, $A\to W(C)$ is the canonical lift of $A\to C$ afforded by the $\delta$-ring structure on $A$. If we have a map $R\to A/I$---that is, if $(A,I)$ lifts to an object in the (animated) prismatic site of $R$---then $\iota_{(A,I)}$ admits a lift to a map to $R^{\Prism}$. 
\end{remark}

\subsection{The Hodge-Tate locus}
\label{subsec:hodge-tate_locus}

\subsubsection{}
Let $\hat{\Aff}^1$ be the $p$-adic formal completion of $\Aff^1$, equipped with the inverse of the usual action of $\Gm$ (as in~\S\ref{subsec:rees_construction}). Then, $\hat{\Aff}^1/\Gm$ parameterizes line bundles $\mathcal{L}$ equipped with a cosection $t:\mathcal{L}\to \Rg$ that is nilpotent mod-$p$, meaning that for some (hence any) local trivialization of $\mathcal{L}$, $t$ is given by a section of $\Rg$ that is nilpotent mod-$p$.

Then the natural map $W_{\mathrm{dist}}\to \hat{\Aff}^1$ descends to a map $\Int_p^\Prism\to \hat{\Aff}^1/\Gm$, and the \defnword{Hodge-Tate locus} $\Int_p^{\mathrm{HT}}$ is defined to be the fiber product (derived or classical: both are the same in this case)
\[
\Int_p^{\mathrm{HT}} = \Int_p^{\Prism}\times_{\hat{\Aff}^1/\Gm}B\Gm.
\]
This is a closed substack of $\Int_p^{\Prism}$ with locally invertible ideal sheaf, which Bhatt-Lurie make a very detailed study of in~\cite[\S 3.4]{bhatt2022absolute}. They show that $\Int_p^{\mathrm{HT}}$ has a somewhat concrete description.

To explain this, first note that we have a canonical map $\Spf \Int_p\to \Int_p^{\mathrm{HT}}$ corresponding to the Cartier-Witt divisor $W(\Int_p)\xrightarrow{V(1)}W(\Int_p)$.
\begin{proposition}
\label{prop:hodge_tate_locus}
The above map  presents $\Int_p^{\mathrm{HT}}$ as the formal classifying stack $B\mathbb{G}_m^\sharp$ over $\Spf \Int_p$. In particular, it is a flat surjection. Moreover, there is a natural equivalence
\[
\Spf \Int_p\times_{\Int_p^{\Prism}}\Field_p^{\Prism}\xrightarrow{\simeq}\mathbb{G}_m^\sharp\times \Spec \Field_p.
\]

\end{proposition}
\begin{proof}
The first assertion is~\cite[Theorem 3.4.13]{bhatt2022absolute}; see also~\cite[Lemma 4.5.2]{drinfeld2022prismatization}. 

For the second, note that the left hand side is a canonically trivial $\Gm^\sharp$-torsor over 
\[
\Field_p^{\mathrm{HT}}\defn \Int_p^{\mathrm{HT}}\times_{\Int_p^{\Prism}}\Field_p^{\Prism},
\]
but this base is canonically identified with the closed subscheme $\Spec\Field_p\subset\Spf \Int_p$ via the equivalence of Remark~\ref{rem:prismatization_of_Fp}.
\end{proof}

\subsection{The Nygaard filtered prismatization}
\label{subsec:nygaard}

The underived story of the Nygaard filtered prismatization is explained in Bhatt's notes~\cite[\S 5]{bhatt_lectures} following Drinfeld in~\cite[\S 5]{drinfeld2022prismatization}. We will now explain how to make sense of these constructions in the context of animated rings. This explanation was helped greatly by conversations with Juan Esteban Rodr\'iguez Camargo.

\subsubsection{}
View $W$ as a ring stack over $\Spf \Int_p$, associating with each $p$-nilpotent $R$ the ring of Witt vectors $W(R)$. We will consider modules for the ring stack $W$ in the $\infty$-category of fppf sheaves over $R$, and refer to them simply as \emph{$W$-modules over $R$}. Here are some examples of $W$-modules: 
\begin{itemize}
\item Any quasicoherent sheaf can be viewed as a $\Ga$-module, and is hence equipped with the structure of a $W$-module via the map $W\to \Ga$.
\item Given an invertible $R$-module $L$, we can take the divided power envelope of the identity section within the vector group scheme $\mathbf{V}(L^\vee)$\footnote{Recall that we are using Grothendieck's convention, so that this is the total space of $L$.} to obtain the $W$-module $\mathbf{V}(L^\vee)^\sharp$. Here, the $W$-action is via the map $W\to \mathbb{G}_a$. 
   \item Given any $W$-module $M$, we obtain a new $W$-module $F_*M$ via restriction along the map $F:W\to W$. If $M=W$, then $F_*W$ is in fact an animated $W$-algebra.
   \item Given any $W(R)$-module $M$, we obtain an associated $W$-module $I\otimes_{W(R)}W$ that assigns to every $R$-algebra $S$ the $W(S)$-module $I\otimes_{W(R)}W(S)$.
\end{itemize}

We now have a diagram of $W$-modules over $\Spf \Int_p$ where all the rows and columns are fiber sequences of $W$-modules.
\begin{align}\label{eqn:GadR_two_presentations}
\begin{diagram}
&&F_*W&\rEquals&F_*W\\
&&\dTo_V&&\dTo_p\\
\Ga^\sharp&\rTo&W&\rTo^F& F_*W\\
\dEquals&&\dTo&&\dTo\\
\Ga^\sharp&\rTo&\Ga&\rTo &\Ga^{\dR}\simeq F_*W/{}^{\mathbb{L}}p.
\end{diagram}
\end{align}
See~\cite[Corollary 2.6.8]{bhatt_lectures}.

\subsubsection{}
Consider the (derived) $p$-adic formal stack $\Int_p^{\Prism}\times \Aff^1/\Gm$: over any $p$-nilpotent $R$, this parameterizes pairs $(I\xrightarrow{d}W(R),L\to R)$, consisting of a Cartier-Witt divisor and a generalized Cartier divisor. We will view the map of $W$-modules
\[
F_*M'\defn F_*(I\otimes_{W(R)}W)\xrightarrow{d} F_*W
\]
as a generalized Cartier divisor for the $W$-algebra $F_*W$. In turn this induces a generalized Cartier divisor
\[
F_*M'\otimes_{F_*W}\Ga^{\dR} \to \Ga^{\dR}
\]
for the $W$-algebra $\Ga^{\dR}$. On the other hand, over this formal prestack, we have another generalized Cartier divisor $L\otimes_{R}\Ga^{\dR}\to \Ga^{\dR}$ for the same $W$-algebra.

Suppose that we are given a map 
\[
\begin{diagram}
F_*M'\otimes_{F_*W}\Ga^{\dR}&\rTo^\alpha&L\otimes_{R}\Ga^{\dR}\\
&\rdTo(1,2)\ldTo(1,2)\\
&\Ga^{\dR}
\end{diagram}
\]
of generalized Cartier divisors. Then pulling back the fiber sequence
\[
\mathbf{V}(L^\vee)^\sharp\to L\otimes_R\Ga\to L\otimes_R\Ga^{\dR}
\]
of $W$-modules along the composition 
\[
F_*M'\to F_*M'\otimes_{F_*W}\Ga^{\dR}\xrightarrow{\alpha}L\otimes_R\Ga^{\dR}
\]
now gives us a sequence of $W$-modules
\[
\mathbf{V}(L^\vee)^\sharp \to M(\alpha)\to F_*M'.
\]

\begin{lemma}
\label{lem:filtered_cw_divisor}
There is a canonical map of $W$-modules $M(\alpha)\xrightarrow{d_\alpha} W$ whose cofiber has a canonical structure of a $W$-algebra.
\end{lemma}
\begin{proof}
We can lift $F_*W$ to a filtered $W$-algebra with filtration $\Fil^\bullet_{F_*M'} F_*W$ obtained from the generalized Cartier divisor $F_*M'\to F_*W$. Similarly, we can lift $\Ga^{\dR}$ (resp. $\Ga$) to a filtered $W$-algebra with filtration $\Fil^\bullet_L \Ga^{\dR}$ (resp. $\Fil^\bullet_L \Ga$).

The map $\alpha$ can be viewed now as yielding a map of non-negatively filtered animated $W$-algebras
\[
F_*\Fil^\bullet_{M'}W\to \Fil^\bullet_L \Ga^{\dR},
\]
and we can upgrade $W$ to a non-negatively filtered animated $W$-algebra $\Fil^\bullet_\alpha W$ such that we have a Cartesian square
\[
\begin{diagram}
\Fil^\bullet_\alpha W &\rTo & F_*\Fil^\bullet_{M'}W\\
\dTo&&\dTo\\
\Fil^\bullet_L \Ga&\rTo&\Fil^\bullet_L\Ga^{\dR}.
\end{diagram}
\]
This has the feature that $\Fil^\bullet_\alpha W = W$ and $\Fil^1_\alpha W = M(\alpha)$. In particular, we obtain a canonical map $M(\alpha)\to W$, and we also see that its cofiber is the $W$-algebra $\gr^0_\alpha W$. 
\end{proof}

\begin{definition}
\label{defn:ZpN}
Let $\Int_p^{\mathcal{N}}$ be the $p$-adic formal prestack whose values on $R\in \mathrm{CRing}^{p\text{-nilp}}$ are given by triples $(I\to W(R),L\to R,\alpha)$, where $(I\to W(R),L\to R)$ is a section of $\Int_p^\Prism\times \Aff^1/\Gm$, and $\alpha$ is a map of generalized Cartier divisors on ring stacks
\[
\alpha:(F_*M'\otimes_{F_*W}\Ga^{\dR}\to \Ga^{\dR})\to (L\otimes_R\Ga^{\dR}\to \Ga^{\dR}).
\]

By abuse of notation we will refer to this tuple by the symbol $M\xrightarrow{d} W$, where $M \defn M(\alpha)$ and $d = d_\alpha$, and refer to it as a \defnword{filtered Cartier-Witt divisor} over $R$.\footnote{We will see in Proposition~\ref{prop:nyg_classical} below this agrees with the definitions from~\cite{bhatt_lectures} for discrete inputs. The notation and terminology have been adapted from Bhatt's notes.} We will also write $W/_dM$ instead of $\gr^0_\alpha W$. We will also write $L_M\to R$ for the underlying generalized Cartier divisor and $M'\xrightarrow{d'} W$ for the map of $W$-modules $I\otimes_{W(R)}W\to W$ arising from the Cartier-Witt divisor $I\to W(R)$, so that we have a diagram of fiber sequences of $W$-modules
\[
\begin{diagram}
\mathbf{V}(L^\vee_M)^\sharp&\rTo&M&\rTo&F_*M'\\
\dTo&&\dTo_d&&\dTo_{d'}\\
\Ga^\sharp&\rTo&W&\rTo&F_*W
\end{diagram}
\]
\end{definition}

\begin{definition}
Over $\Int_p^{\mathcal{N}}$, we have the $p$-adic formal ring stack $\Ga^{\mathcal{N}}$ assigning to every filtered Cartier-Witt divisor $W\xrightarrow{d}M$ over $R$, the $W(R)$-algebra $(W/_dM)(R)$. 

For any $R\in\mathrm{CRing}^{p\text{-nilp}}$, we now define the \defnword{Nygaard filtered prismatization} $R^\mathcal{N}$ to be the transmutation of $\Spec R$ with respect to  $\mathbb{G}_a^{\mathcal{N}}$. That is, it is the $p$-adic formal prestack over $\Int_p^{\mathcal{N}}$ that associates with each $C\in \mathrm{CRing}^{p\text{-nilp}}$ and $(M\xrightarrow{d}W)\in\Int_p^{\mathcal{N}}(C)$ the space $\Map_{\mathrm{CRing}}(R,(W/_dM)(C))$.
\end{definition}

\begin{proposition}
\label{prop:nyg_classical}
The $p$-adic formal prestacks $\Int_p^{\mathcal{N}}$ and $\mathbb{G}_a^{\mathcal{N}}$ are classical. Moreover, their classical truncations agree with the descriptions in~\cite[\S 5]{bhatt_lectures}.
\end{proposition}
\begin{proof}
By construction $\Int_p^{\mathcal{N}}$ lives over the classical formal prestack $\Int_p^\Prism\times \Aff^1/\Gm$. So it is enough to know that it is classical after flat base-change over the latter prestack.

Consider the pullback $X$ of $\Int_p^{\mathcal{N}}$ along the flat cover
\[
W_{\mathrm{dist}}\times \Aff^1\to \Int_p^\Prism\times \Aff^1/\Gm.
\]
Then the fiber of $X$ over a section $(d',t)\in  W_{\mathrm{dist}}(R)\times \Aff^1(R)$ is the stack of commuting squares of maps of $F_*W$-modules
\[
\begin{diagram}
F_*W&\rTo&\Ga^{\dR}\\
\dTo^{d'}&&\dTo_{t^{\mathrm{dR}}}\\
F_*W&\rTo&\Ga^{\dR}
\end{diagram}
\]
where the bottom arrow is the natural one, and where $t^{\mathrm{dR}}$ is the map induced by $t$ and the $\Ga$-algebra structure on $\Ga^{\dR}$. Another way of looking at $X$ is as the stack over $W_{\mathrm{dist}}\times \Aff^1\times \Ga^{\dR}$ sitting in a Cartesian diagram
\[
\begin{diagram}
X&\rTo&W_{\mathrm{dist}}\times \Aff^1\times \Ga^{\dR}\\
\dTo&&\dTo_{(d',t,\alpha)\mapsto (d^{',\dR},t^{\dR}\circ \alpha)}\\
\Ga^{\dR}&\rTo_{\Delta}&\Ga^{\dR}\times\Ga^{\dR}
\end{diagram}
\]
Here, $d^{',\dR}$ is the section of $\Ga^{\dR}$ associated with $d'$. 

Pulling back along the flat cover 
\[
W_{\mathrm{dist}}\times \Aff^1\times \Ga \to W_{\mathrm{dist}}\times \Aff^1\times \Ga^{\dR}
\]
gives us a flat cover $Y\to X$ sitting in a Cartesian diagram
\[
\begin{diagram}
Y&\rTo&W_{\mathrm{dist}}\times \Aff^1\times \Ga\\
\dTo&&\dTo_{(d',t,u)\mapsto (d^{',\dR},t^{\dR}\circ u^{\dR})}\\
\Ga^{\dR}&\rTo^{\Delta}&\Ga^{\dR}\times\Ga^{\dR}
\end{diagram}
\]

By invoking~\eqref{eqn:GadR_two_presentations}, we can write this as a composition of two Cartesian squares
\[
\begin{diagram}
Y&\rTo&W_{\mathrm{dist}}\times \Aff^1\times \Ga\\
\dTo&&\dTo_{(d',t,u)\mapsto (d',t\circ u)}\\
W&\rTo_{d\mapsto (F(d),\pi(d))}&W\times \Ga\\
\dTo&&\dTo_{(d',a)\mapsto (d^{',\dR},a^{\dR})}\\
\Ga^{\dR}&\rTo^{\Delta}&\Ga^{\dR}\times\Ga^{\dR}
\end{diagram}
\]

After all is said and done, the classicality of $\Int_p^{\mathcal{N}}$ is now reduced to the classicality of $Y$, which comes down to the assertion that the maps of formal schemes
\[
W\xrightarrow{(F,\pi)}W\times \mathbb{G}_a\;;\; W_{\mathrm{dist}}\times\Aff^1\times\mathbb{G}_a\xrightarrow{(\mathrm{id},m)}W\times \mathbb{G}_a
\]
are $p$-completely Tor-independent. Here, $m:\Aff^1\times\mathbb{G}_a\xrightarrow{(t,u)\mapsto tu}\mathbb{G}_a$ is the multiplication map. 

But these are actually maps of \emph{affine} formal schemes corresponding to maps of $p$-complete rings given by:
\begin{align*}
\alpha:\Int_p[T,x_0,x_1,\ldots]_p^\wedge &\xrightarrow{T\mapsto x_0,x_i\mapsto F^*x_i}\Int_p[x_0,x_1,\ldots]_p^\wedge\\
\beta:\Int_p[T,x_0,x_1,\ldots]_p^\wedge&\xrightarrow{T\mapsto ut,x_i\mapsto x_i}\Int_p[u,t,x_0,x_1^{\pm 1},\ldots]_{(x_0,p)}^{\wedge}.
\end{align*}
The first map is the composition of the map $\alpha':T\mapsto x_0,x_i\mapsto x_i$ with the \emph{flat} map $F^*$. Therefore, it is enough to check that $\alpha'$ and $\beta$ are $p$-completely Tor-independent. This comes down to the concrete (and easy) assertion that the element
\[
ut-x_0\in \Field_p[u,t,x_0,x_1^{\pm 1},\ldots]_{(x_0)}^{\wedge}
\]
is a non-zero divisor.

Let us now look at $\Ga^{\mathcal{N}}$. It suffices to check that it is classical after base-change to $Y$. Here, the above description shows that we have a commuting diagram of fiber sequences of $W$-modules
\begin{align}\label{eqn:fil_cartier_witt_local_description}
\begin{diagram}
\Ga^\sharp&\rTo&W&\rTo&F_*W\\
\dTo^{u^\sharp}&&\dTo&&\dEquals\\
\Ga^\sharp&\rTo&M&\rTo&F_*W\\
\dTo^{t^\sharp}&&\dTo&&\dTo_{d'}\\
\Ga^\sharp&\rTo&W&\rTo&F_*W
\end{diagram}
\end{align}
where the composition of the middle vertical arrows is multiplication by a section $d$ of $W$ with $F(d) = d'$, and where $M\to W$ is the universal filtered Cartier-Witt divisor over $Y$. This description shows that the $W$-module scheme $M$ over $Y$ is classical: It is a colimit of the classical schemes $W$ and $\Ga^\sharp$, and so the restriction of $\Ga^{\mathcal{N}}$ to $Y$ is also classical. 

The last assertion about the values on discrete inputs can now be deduced from~\cite[Lemma 5.2.8]{bhatt_lectures}.
\end{proof} 

\begin{remark}
\label{rem:fil_cw_pushout}
The proof shows that every filtered Cartier-Witt divisor can be obtained flat locally as one sitting in the diagram~\eqref{eqn:fil_cartier_witt_local_description} for sections $u,t$ of $\Ga$ and $d$ of $W$ with $d$ restricting to multiplication by $(tu)^\sharp$ on $\Ga^\sharp$. 

More generally, for any invertible module $L$ equipped with a section $R\xrightarrow{u}L$ and a cosection $L\xrightarrow{t}R$, and a section $d$ of $W$ restricting to $(t\circ u)^\sharp $ on $\Ga^\sharp$, we obtain a similar diagram with the middle sequence now an extension of $F_*W$ by $\mathbf{V}(L^\vee)^\sharp$. For future reference, we will denote a filtered Cartier-Witt divisor obtained in this way by $M(L,d,t,u)\to W$.
\end{remark}

\begin{remark}
\label{rem:juan_comment}
As pointed out to us by Camargo, the definition of $\Int_p^{\mathcal{N}}$ can also be formulated as follows. Let $(\Aff^1/\Gm)^{\mathrm{dR}}$ be the transmutation of $\Aff^1/\Gm$ with respect to the ring stack $\Ga^{\dR}$: It is a formal prestack over $\Spf \Int_p$ parameterizing generalized Cartier divisors $P\to \Ga^{\dR}$. There is a multiplication map 
\begin{align*}
\mu:\Aff^1/\Gm\times (\Aff^1/\Gm)^{\mathrm{dR}}&\to (\Aff^1/\Gm)^{\mathrm{dR}}\\
((L\xrightarrow{t} \Ga),(Q\xrightarrow{t'} \Ga^{\dR}))&\mapsto (L\otimes Q \xrightarrow{t'\circ (t\otimes \mathrm{id}_Q)}\Ga^{\dR}).
\end{align*}
We now have a Cartesian diagram:
\begin{align}\label{eqn:juan_square}
\begin{diagram}
\Int_p^{\mathcal{N}}&\rTo&\Aff^1/\Gm\times(\Aff^1/\Gm)^{\dR}\\
\dTo&&\dTo_\mu\\
\Int_p^{\Prism}&\rTo_{\qquad(M'\to W)\mapsto (F_*M'\otimes_{F_*W}\Ga^{\dR}\to \Ga^{\dR})}&(\Aff^1/\Gm)^{\dR}
\end{diagram}  
\end{align}
\end{remark}

\begin{remark}
\label{rem:nyg_filt_prismatization_of_Fp}
The prestack $\Field_p^\mathcal{N}$ admits a quite explicit description explained in~\cite[Prop. 5.4.2]{bhatt_lectures}. Suppose that we are given a filtered Cartier-Witt divisor $(M\xrightarrow{d}W)\in \Int_p^{\mathcal{N}}(R)$. Then it is not difficult to see that the existence of a map $\Field_p\to (W/_dM)(R)$ implies that $M\to W$ is isomorphic canonically to a filtered Cartier-Witt divisor of the form $M(L,p,t,u)\to W$. Therefore we have
\[
\Field_p^{\mathcal{N}}\simeq Z(ut-p)/\Gm,
\]
where $Z(ut-p)\subset \Aff^1\times\Aff^1_+= \Spf \Int_p[t,u]$ is the closed formal $\Gm$-equivariant subscheme defined by the equation $ut-p$.
\end{remark}

\subsection{The de Rham and Hodge-Tate embeddings and syntomification}
\label{subsec:syntomification}

We will now define two canonical open immersions $j_{\mathrm{dR}},j_{\mathrm{HT}}:\Int_p^\Prism\to \Int_p^{\mathcal{N}}$, called the \defnword{de Rham} and \defnword{Hodge-Tate} embeddings. These are described in~\cite[\S 5.3,5.6]{drinfeld2022prismatization}, and in~\cite[\S 5.3]{bhatt_lectures}. Let $M'\to W$ be the map of $W$-modules over $\Int_p^\Prism$ associated with the universal Cartier-Witt divisor over $\Int_p^\Prism$.

\begin{definition}
The open immersion $j_{\mathrm{dR}}$ is simply the open locus $\Int_p^\Prism\times_{\Aff^1/\Gm}(\Gm/\Gm)$ parameterizing filtered Cartier-Witt divisors $(M\xrightarrow{d} W)\in \Int_p^\Prism(R)$, where the underlying map $L_M\xrightarrow{\simeq} R$ is an isomorphism. Over this locus, the classifying map $(F_*M'\to F_*W) \to (\Ga^{\dR}\xrightarrow{\mathrm{id}}\Ga^{\dR})$ yields the filtered Cartier-Witt divisor given by
\[
\begin{diagram}
\Ga^\sharp&\rTo&M&\rTo&F_*M'\\
\dEquals&&\dTo&&\dTo_{d'}\\
\Ga^\sharp&\rTo&W&\rTo&F_*W
\end{diagram}
\]
where the right square is Cartesian.

The pullback of $j_{\dR}$ to the stack $Y$ from the proof of Proposition~\ref{prop:nyg_classical} is the open locus where $t\neq 0$. Equivalently, it is the pullback of the open locus 
\[
\Gm/\Gm \times (\Aff^1/\Gm)^{\dR} \hookrightarrow \Aff^1/\Gm\times (\Aff^1/\Gm)^{\dR}
\]
via the top horizontal map in~\eqref{eqn:juan_square}.
\end{definition}

\begin{definition}
The open immersion $j_{\mathrm{HT}}$ is given by the natural map
\[
(F_*F^*M'\xrightarrow{F^*(d')} F_*W)\to (F_*F^*M'\otimes_{F_*W}\Ga^{\dR}\to \Ga^{\dR})
\]
which classifies the filtered Cartier-Witt divisor given by
\[
\begin{diagram}
\mathbf{V}(L^\vee)^\sharp&\rTo&M'&\rTo&F_*F^*M'\\
\dTo&&\dTo&&\dTo\\
\Ga^\sharp&\rTo&W&\rTo&F_*W
\end{diagram}
\]
Here $L\to \Ga$ is the generalized Cartier divisor for $\Ga$ obtained via base-change from $M'\to W$. The filtered Cartier-Witt divisors obtained in this fashion will be called \defnword{invertible}.

The pullback of $j_{\mathrm{HT}}$ to the stack $Y$ from the proof of Proposition~\ref{prop:nyg_classical} is the open locus where $u\neq 0$. Equivalently, it is the pullback of the open locus 
\[
\Aff^1/\Gm \times (\Gm/\Gm)^{\dR} \hookrightarrow \Aff^1/\Gm\times (\Aff^1/\Gm)^{\dR}
\]
via the top horizontal map in~\eqref{eqn:juan_square}.
\end{definition}  

\begin{definition}
We now define the prestack $\Int_p^{\mathrm{syn}}$ to be the coequalizer of the immersions $j_{\mathrm{dR}},j_{\mathrm{HT}}$. Practically, what this means is that we have
\[
\begin{diagram}
\mathrm{QCoh}(\Int_p^{\mathrm{syn}})&\rTo^{\simeq}&\mathrm{eq}\biggl(\mathrm{QCoh}(\Int_p^{\mathcal{N}})&\pile{\rTo^{j_{\mathrm{dR}}^*}\\ \rTo_{j_{\mathrm{HT}}^*}}&\mathrm{QCoh}(\Int_p^{\Prism}) \biggr).
\end{diagram}
\]
\end{definition}

\begin{definition}
The process of transmutation now yields for every $R\in \mathrm{CRing}^{p\text{-nilp}}$ open immersions $j_{\mathrm{dR}},j_{\mathrm{HT}}:R^{\Prism}\to R^{\mathcal{N}}$, and we now define the \defnword{syntomification} of $R$, $R^{\mathrm{syn}}$, to be their coequalizer. By construction we have a canonical structure map $R^{\mathrm{syn}}\to \Int_p^{\mathrm{syn}}$. Equivalently, $R^{\mathrm{syn}}$ is the transmutation of $R$ with respect to the ring stack $\Ga^{\mathrm{syn}}\to \Int_p^{\mathrm{syn}}$.
\end{definition}

\subsection{The Breuil-Kisin twist}
\label{subsec:breuil-kisin}

Here we will describe a canonical line bundle $\mathcal{O}^{\mathrm{syn}}\{1\}$ on $\Int_p^{\mathrm{syn}}$ called the \defnword{Breuil-Kisin twist}. 

\subsubsection{}
Specifying such a line bundle is equivalent to specifying a line bundle $\mathcal{O}^{\mathcal{N}}\{1\}$ on $\Int_p^{\mathcal{N}}$ equipped with an isomorphism
\[
j_{\mathrm{dR}}^*\mathcal{O}^{\mathcal{N}}\{1\}\xrightarrow{\simeq}j_{\mathrm{HT}}^*\mathcal{O}^{\mathcal{N}}\{1\}
\]
of line bundles on $\Int_p^{\Prism}$. 

We begin by considering the pullback of the tautological line bundle on $B\Gm$: this gives a line bundle $\mathcal{O}^{B\Gm}\{1\}$ over $\Int_p^{\mathcal{N}}$. Note that we have a canonical trivialization
\begin{align}\label{eqn:de_rham_triv}
j_{\mathrm{dR}}^*\mathcal{O}^{B\Gm}\{1\}\xrightarrow{\simeq}\Reg{\Int_p^\Prism}.
\end{align}
Via the Hodge-Tate embedding, we obtain a canonical equivalence
\begin{align}\label{eqn:hodge_tate-divisor_pullback}
j_{\mathrm{HT}}^*\mathcal{O}^{B\Gm}\{1\}\xrightarrow{\simeq}\mathscr{I}_{\Int_p^{\mathrm{HT}}}
\end{align}
where the right hand side is the ideal sheaf of the Hodge-Tate divisor.

\subsubsection{}
Next, we consider a canonical line bundle $\mathcal{O}^{\Prism}\{1\}$ on $\Int_p^{\Prism}$ that is characterized up to isomorphism by any of the following properties:
\begin{itemize}
   \item (\cite[(2.2.11)]{bhatt2022absolute}) For any \emph{transversal} prism $(A,I)$ the pullback of $\mathcal{O}^{\Prism}\{1\}$ to $\Spf A$ under the map $\iota_{(A,I)}$ from Remark~\ref{rem:prisms_and_prismatization} is canonically isomorphic to $A\{1\}$.

   \item (\cite[\S 4.8,4.9]{drinfeld2022prismatization}) We have the line bundle $\Rg(\Int_p^{\mathrm{HT}})$ on $\Int_p^\Prism$ corresponding to the Hodge-Tate divisor. The endomorphism $\varphi^*-\mathrm{id}$ of the Picard group $\Pic(\Int_p^\Prism)$ is an equivalence, and we take $\mathcal{O}^{\Prism}\{1\}$ to be the preimage of $\Rg(\Int_p^{\mathrm{HT}})$.

   \item (\cite[(9.1.6)]{bhatt2022absolute}) Let $f:(\mathbb{P}^1_{\Int_p})^{\Prism}\to \Int_p^\Prism$ be the map of prismatizations arising from the structure morphism $\mathbb{P}^1_{\Int_p}\to \Spf \Int_p$. Then, we have $\mathcal{O}^{\Prism}\{-1\}\simeq R^2f_*\Rg$.
\end{itemize}

\subsubsection{}
There is a canonical isomorphism
\begin{align}\label{eqn:bk_delta_isom}
\varphi^*\mathcal{O}^{\Prism}\{1\}\xrightarrow{\simeq}\Rg(\Int_p^{\mathrm{HT}})\otimes\mathcal{O}^{\Prism}\{1\}
\end{align}
of line bundles over $\Int_p^\Prism$.

We now set
\[
\mathcal{O}^{\mathcal{N}}\{1\} = \mathcal{O}^{B\Gm}\{1\}\otimes\pi^*\mathcal{O}^{\Prism}\{1\}.
\]
Since $\pi\circ j_{\mathrm{dR}} = \mathrm{id}$ and $\pi\circ j_{HT} = \varphi$, combining~\eqref{eqn:bk_delta_isom},~\eqref{eqn:de_rham_triv} and~\eqref{eqn:hodge_tate-divisor_pullback} shows that this line bundle does admit a canonical descent to a line bundle $\mathcal{O}^{\mathrm{syn}}\{1\}$ over $\Int_p^{\mathrm{syn}}$.

\subsection{The canonical sections $x_{\mathrm{dR}}$ and $x_{\mathrm{dR}}^{\mathcal{N}}$}
\label{subsec:can_sections}

\subsubsection{}
For any $R\in \mathrm{CRing}^{p\text{-nilp}}$, we have canonical maps
\[
x_{\mathrm{dR}}:\Spec R \to R^{\Prism}\;;\; x_{\mathrm{dR}}^{\mathcal{N}}:\Aff^1/\Gm\times\Spec R \to R^{\mathcal{N}}
\]
with the following properties:
\begin{itemize}
   \item The composition
   \[
    \Aff^1/\Gm \times \Spec R\xrightarrow{x_{\mathrm{dR}}^{\mathcal{N}}} R^{\mathcal{N}}\xrightarrow{t}\Aff^1/\Gm
   \]
   is the canonical structure map.
   \item The restriction of $x_{dR}^{\mathcal{N}}$ to the open point $\Spec R$ is isomorphic to $x_{dR}$.

   \item We have a commuting square
   \[
    \Square{\Aff^1/\Gm\times \Spec R}{x^{\mathcal{N}}_{dR}}{R^{\mathcal{N}}}{}{\pi}{\Spec R}{x_{dR}}{R^\Prism}.
   \]
\end{itemize}

\subsubsection{}
The map $x_{dR}$ corresponds to the Cartier-Witt divisor $(W(R)\xrightarrow{p}W(R))$ with the map $R\to W(R)/{}^{\mathbb{L}}p$ the one through which the Frobenius endomorphism of $W(R)/{}^{\mathbb{L}}p$ canonically factors. Alternatively, there are equivalences of abstract animated rings 
\[
W(R)/{}^{\mathbb{L}}p \xrightarrow{\simeq}F_*\overline{W(R)} \xrightarrow{\simeq}\mathbb{G}_a^{dR}(R),
\]
and the desired map is the evaluation of the natural map $\mathbb{G}_a\to \mathbb{G}_a^{dR}$ on $R$.

\subsubsection{}
The map $x_{dR}^{\mathcal{N}}$ associates with every cosection $t:L\to C$ of a line bundle $L$ over an $R$-algebra $C$ the filtered Cartier-Witt divisor associated with the zero map $\Ga^{\dR}\to L\otimes_R\Ga^{\dR}$. Explicitly, it is given by the map
\[
M(t) = F_*W\oplus \mathbf{V}(L^\vee)^\sharp\xrightarrow{d=(V,t^\sharp)}W,
\]
where the quotient $W/_d M(t)$ is also a quotient of $W/VW \simeq \mathbb{G}_a$, giving us the map
\[
R \to C = \mathbb{G}_a(C)\to (W/_d M(t))(C).
\]

\subsection{Relationship with divided powers}
\label{subsec:divided_powers_lifting}

The next observation will be used later to formulate abstract Grothendieck-Messing type statements.

\begin{lemma}
\label{lem:lift_divided_powers}
Suppose that we have a divided power thickening $R'\twoheadrightarrow R$ in $\mathrm{CRing}^{p\text{-nilp}}$. Then the canonical point $x_{\dR,R'}$ from~\S\ref{subsec:can_sections} factors canonically through a point $\tilde{x}_{\dR,R'}:\Spec R'\to R^{\Prism}$.
\end{lemma}
\begin{proof}
The claim here is that the canonical map $R'\to \mathbb{G}^{\dR}_a(R')$ admits a factoring through $R$ that depends only on the divided powers on the fiber $J$ of $R'\twoheadrightarrow R$. Since the constructions and the conclusion are compatible with sifted colimits, we reduce to the case where $R'\twoheadrightarrow R$ is a classical divided power thickening of $p$-completely flat $\Int_p$-algebras.

Now, it suffices to construct a canonical lift of $R'$-modules $J\to \mathbb{G}_a^\sharp(R')$: this follows from Remark~\ref{rem:divided_powers_div_algebra}.
\end{proof}

\begin{remark}
\label{rem:crystalline_site}
The lemma can be interpreted as follows: For any $A\in \mathrm{CRing}^{p\text{-comp}}$, we can consider the \defnword{big crystalline site}, whose underlying $\infty$-category is the opposite to that of tuples $(A\to R,R'\twoheadrightarrow R,\gamma)$ of divided power extensions of objects in $\mathrm{CRing}^{p\text{-nilp}}_{A/}$. Associated with such an object we obtain a canonical map
\[
\Spf R'\to R^{\Prism}\to A^{\Prism}.
\]

If one wants to be more judicious, one could also restrict to the \emph{small} version where $A\to R$ is required to be \'etale and finitely presented. This will not make any difference in the sequel.

Note that the sites we are considering here are the \emph{animated} versions of the usual notions. The reader can consult~\cite[\S 4]{Mao2021-jt} for a comparison of the two notions, derived and classical.
\end{remark}  

\subsection{Prismatic cohomology}
\label{subsec:cohomology}

We will need the relationship between the stacks defined above and relative (Nygaard filtered) prismatic cohomology as constructed in~\cite{Bhatt2022-ee} and~\cite{bhatt2022absolute} for semiperfectoid rings.

\begin{definition}[Relative prismatic cohomology]
Suppose that we have a (classical) bounded prism $(A,I)$ with $\overline{A} = A/I$. Then, with any $R\in \mathrm{CRing}_{\overline{A}/}$ we can associate its \defnword{relative prismatic cohomology} $\Prism_{R/A}$. 

This can be obtained---see~\cite[Construction 4.1.3]{bhatt2022absolute}---as the left Kan extension of its restriction to polynomial $\overline{A}$-algebras $R$, where $\Prism_{R/A}$ is defined to be an inverse limit in the $\infty$-category of $p$-complete \emph{derived rings} over $A$ (see also~\cite[\S 7]{bhatt2022prismatization}):
\[
\Prism_{R/A}\xrightarrow{\simeq}\varprojlim_{(B,v)\in (R/A)_{\Prism}}B,
\]
where $(R/A)_{\Prism}$ is the prismatic site for $R$ relative to $(A,I)$.
\end{definition}

\begin{remark}[Conjugate filtered Hodge-Tate cohomology]
\label{rem:hodge-tate_cohomology}
The base-change $\overline{\Prism}_{R/A} \defn \overline{A}\otimes_A\Prism_{R/A}$ over $\Field_p$ is the \defnword{relative Hodge-Tate cohomology}, which is equipped with a canonical exhaustive increasing filtration $\Fil_\bullet^{\mathrm{conj}}\overline{\Prism}_{R/\Int_p}$ supported in non-positive degrees, and a canonical equivalence\footnote{The cotangent complex appearing here is the $p$\emph{-completed} version.}
\[
\gr_{i}^{\mathrm{conj}}\overline{\Prism}_{R/A}\{i\}\xrightarrow{\simeq}\wedge^i\mathbb{L}_{R/\overline{A}}[-i]
\]
for each $i\ge 0$; see~\cite[Remark 4.1.7]{bhatt2022absolute}. This is characterized by the property that it respects sifted colimits in $\overline{A}$-algebras and is isomorphic to the Postnikov filtration on $\overline{\Prism}_{R/A}$ for polynomial algebras $R$ over $\overline{A}$.
\end{remark}

\begin{definition}
\label{defn:semiperfectoid}
A $p$-complete animated commutative ring $R$ is \defnword{semiperfectoid} if there exists a \emph{perfectoid} ring $R_0$ and a surjective map $R_0\twoheadrightarrow R$. 

\end{definition}  

In~\cite[\S 4.4]{bhatt2022absolute} one finds the construction of the \defnword{absolute prismatic cohomology} $\Prism_R$ of any $p$-complete animated commutative ring $R$; see in particular Construction 4.4.10 of \emph{loc. cit.} We will need to know the following result due to Bhatt-Lurie and Holeman~\cite{holeman2023derived}:

\begin{theorem}[Affineness of the prismatization]
\label{thm:absolute_prismatic_cohom}
Suppose that $R_0$ is a perfectoid ring with corresponding perfect prism $(A,I) = (A_{\mathrm{inf}}(R_0),\ker\theta)$. Then:
\begin{enumerate}
   \item There are canonical isomorphisms of commutative rings
   \[
    \Prism_{R_0}\xrightarrow{\simeq}\Prism_{R_0/A}\xrightarrow{\simeq}A.
   \]
   \item In fact, for any $R_0$-algebra $R$, there is an isomorphism
   \[
    \Prism_R \xrightarrow{\simeq}\Prism_{R/A}.
   \]
   \item If $R$ is a semiperfectoid $R_0$-algebra, then $\Prism_{R}\simeq\Prism_{R/A}$ is a $(p,I)$-complete animated commutative ring, and there is a canonical isomorphism (independent of the choice of $R_0$) of $p$-adic formal stacks
   \[
   \Spf(\Prism_{R})\xrightarrow{\simeq}R^{\Prism}.
   \]
\end{enumerate}
\end{theorem}
\begin{proof}
The second statement follows from~\cite[Proposition 4.4.12]{bhatt2022absolute}, while the second isomorphism in assertion (1) is from the fact that $(A,I)$ is an initial prism with a map $R_0\to \overline{A}$. 

The last assertion is~\cite[Corollary 7.18]{bhatt2022prismatization}. The only additional remark to be made is that $(\Prism_R,\Prism_R\otimes_{A}I)$ is an initial prism equipped with a map $R\to \overline{\Prism}_R$, and so is independent of the choice of the perfectoid ring $R_0$: this follows for instance from~\cite[Theorem 3.3.7]{holeman2023derived}. The map $\Spf(\Prism_R)\to R^{\Prism}$ is canonical and arises from the general construction in Remark~\ref{rem:prisms_and_prismatization}.
\end{proof}

\begin{remark}
\label{rem:derived_crystalline_semiperfect}
Suppose that $R$ is a semiperfect $\Field_p$-algebra, then we can take $R_0 = R^\flat$ to be its perfection. In this case, we obtain isomorphisms
\begin{align}\label{eqn:semiperfect_prismatization}
\Prism_{R}\xrightarrow{\simeq}\Prism_{R/\Int_p}\xrightarrow[\simeq]{\gamma_{R/\Int_p}^{\mathrm{crys}}}\Prism^{\mathrm{crys}}_{R/\Int_p}\xrightarrow{\simeq}\Prism^{\mathrm{crys}}_{R/W(R^\flat)}\simeq A_{\mathrm{crys}}(R).
\end{align}

Here, $\Prism^{\mathrm{crys}}_{R/W(R^\flat)}$ (resp. $\Prism^{\mathrm{crys}}_{R/\Int_p}$) is the relative derived crystalline cohomology of $R$ over $W(R^\flat)$ (resp. over $\Int_p$) and $A_{\mathrm{crys}}(R)\twoheadrightarrow R$ is the $p$-completed \emph{animated} divided power envelope of the map $W(R^\flat)\twoheadrightarrow R$.\footnote{Note that this will not agree with the classical $p$-complete divided power envelope unless $R$ is quasisyntomic.} The second isomorphism is explained in~\cite[Remark 4.6.5]{bhatt2022absolute}, the third is a consequence of the fact that $W(R^\flat)$ is formally \'etale over $\Int_p$, and the last isomorphism is an animated enhancement of a result of Illusie~\cite[Prop. 4.64]{Mao2021-jt}. The composition of these isomorphisms is $\varphi$-semilinear over $W(R^\flat)$ via the isomorphism $\Prism_{R^\flat}\xrightarrow{\simeq}W(R^\flat)$. 

In fact, Remark 4.6.5 of~\cite{bhatt2022absolute} shows that we have canonical isomorphism
\[
\gamma_{R/\Int_p}^{\mathrm{crys}}:\Prism_{R/\Int_p}\xrightarrow{\simeq}\Prism^{\mathrm{crys}}_{R/\Int_p}
\]
for \emph{any} $R\in \mathrm{CRing}_{\Field_p/}$. In particular, for any such $R$ we have canonical isomorphisms
\begin{equation}
\label{eqn:semiperfect_mod_p_prismatization}
\overline{\Prism}_R\xrightarrow{\simeq} \Prism_R/{}^{\mathbb{L}}p\xrightarrow{\simeq}\left(\Prism^{\mathrm{crys}}_{R/\Int_p}\right)/{}^{\mathbb{L}}p\xrightarrow{\simeq}\dR_{R/\Field_p},
\end{equation}
where on the right hand side we now have the derived de Rham cohomology of $R$ over $\Field_p$.
\end{remark}

\begin{remark}
\label{rem:crys_cohomology_de_rham}
The isomorphism $\Prism_R\xrightarrow{\simeq}A_{\crys}(R)$ is compatible with the relationship between the big crystalline site and the prismatization explained in Remark~\ref{rem:crystalline_site}. Namely, if $R'\twoheadrightarrow R$ is a divided power thickening, then Lemma~\ref{lem:lift_divided_powers} gives a lift $\Prism_R\to R'$ of $\Prism_R\to R$, and the universal property of $A_{\crys}(R)$ as a $p$-completed divided power envelope also yields a lift $\Prism_R\xrightarrow{\simeq}A_{\crys}(R)\to R'$. These two are canonically isomorphic.
\end{remark}

\subsection{Filtered Cartier-Witt divisors associated with filtered prisms}
\label{subsec:frames_cartier-witt}

Our purpose here is to extend the construction from Remark~\ref{rem:prisms_and_prismatization} to the filtered setting. 

\begin{definition}
Let $\Fil^\bullet \Ga^{\dR}$ be the filtered $W$-algebra over $\Aff^1/\Gm$ assigning to every section $(L\xrightarrow{t}C)\in(\Aff^1/\Gm)(C)$ the filtered animated commutative ring $\Fil^\bullet_L\Ga^{\dR}$ (see notation in Lemma~\ref{lem:filtered_cw_divisor}). A \defnword{filtered prism} is a prismatic frame $\underline{A}$ equipped with a commuting diagram of filtered derived commutative $A$-algebras
\[
 \begin{diagram}
    \Fil^\bullet A &\rTo^\Phi& \varphi_*\Fil^\bullet_I A\\
    &\rdTo&\dTo_\zeta\\
    &&R\Gamma(\Rees(\Fil^\bullet A),\Fil^\bullet\Ga^{\dR}).
   \end{diagram}
\]
Moreover, we will ask for an isomorphism of the composition
\[
  \varphi_*\Fil^\bullet_I A\to  R\Gamma(\Rees(\Fil^\bullet A),\Fil^\bullet\Ga^{\dR})\to \varphi_*\Fil^\bullet_I\Ga^{\dR}(A)
\]
with the canonical map of filtered animated commutative $A$-algebras
\[
   \varphi_*\Fil^\bullet_I A\to \varphi_*\Fil^\bullet_I\Ga^{\dR}(A).
\]
We will denote such a structure by the pair $(\underline{A},\zeta)$. 

Explicitly, for every map $\eta:\Fil^\bullet A \to \Fil^\bullet_L C$ of animated commutative rings, $\zeta$ yields a commuting diagram of filtered animated commutative $A$-algebras
\[
 \begin{diagram}
    \Fil^\bullet A &\rTo^\Phi& \varphi_*\Fil^\bullet_I A\\
    \dTo^\eta&&\dTo_{\zeta(\eta)}\\
    \Fil^\bullet_LC&\rTo&\Fil^\bullet_L\Ga^{\dR}(C).
   \end{diagram}
\]
These diagrams are functorial in $\eta$, and agree with the obvious one when $(L\to C)=\varphi_*(I\to A)$.
\end{definition}

\begin{remark}
[Breuil-Kisin filtered prisms]
\label{rem:bk_filtered_prisms}
Consider the situation where $\underline{A}$ is associated with a prism $(A,I')$ as in Example~\ref{example:bk_frames}. Then we have a canonical filtered prism structure arising from the composition
\[
\zeta:\varphi_*\Fil^\bullet_IA \xrightarrow{\simeq}\Fil^\bullet_{I'}A\otimes_A\varphi_*A\to \Fil^\bullet_{I'}A\otimes_AF_*W(A)\to \Fil^\bullet_{I'}A\otimes_A\Ga^{\dR}(A)
\]
\end{remark}

\begin{proposition}
\label{prop:frames_to_nygaard}
Suppose that we are given a filtered prism $(\underline{A},\zeta)$. Then we have a map
\[
\iota_{(\underline{A},\zeta)}:\Rees(\Fil^\bullet A)\to R_A^{\mathcal{N}}
\]
of $p$-adic formal stacks over $\Aff^1/\Gm$ with the following properties:
\begin{enumerate}
   \item There is a canonical Cartesian square
\begin{align*}
\begin{diagram}
\Spf(A)&\rTo^{\tau}&\Rees(\Fil^\bullet A)\\
\dTo^{\iota_{(A,I)}}&&\dTo_{\iota_{(\underline{A},\zeta)}}\\
R_A^\Prism&\rTo_{j_{\dR}}&R^{\mathcal{N}}_A
\end{diagram}
\end{align*}

\item 
There is a canonical commutative diagram
\begin{align*}
\begin{diagram}
\Spf(A)&\rTo^{\sigma}&\Rees(\Fil^\bullet A)\\
\dTo^{\iota_{(A,I)}}&&\dTo_{\iota_{(\underline{A},\zeta)}}\\
R_A^\Prism&\rTo_{j_{\mathrm{HT}}}&R^{\mathcal{N}}_A.
\end{diagram}
\end{align*}

\item The composition of $\iota_{(\underline{A},\zeta)}$ with the map 
\[
\Aff^1/\Gm\times\Spf R_A \to \Rees(\Fil^\bullet A)
\]
from Remark~\ref{rem:abstract_de_rham_point} is isomorphic to the filtered de Rham section $x^{\mathcal{N}}_{\dR}:\Aff^1/\Gm\times\Spf R_A \to R_A^{\mathcal{N}}$
\end{enumerate}

\end{proposition}

\begin{proof}
Write $(M'\to W) = (I\otimes_AW\to W)$ for the canonical Cartier-Witt divisor over $\Spf A$. Given a generalized Cartier divisor $(L\to C)$ and a map $\eta:\Fil^\bullet A\to \Fil^\bullet_L C$ parameterizing a $C$-valued point of $\Rees(\Fil^\bullet A)$, we obtain a commuting diagram of filtered animated commutative $W$-algebras over $C$:
\begin{align}\label{eqn:commuting_diagram_zeta_eta}
    \begin{diagram}
    \Fil^\bullet A\otimes_AW&\rTo&\Fil^\bullet_{\alpha(\eta)}W &\rTo&F_*\Fil^\bullet_{M'}W\\
    &&\dTo&&\dTo\\
    &&\Fil^\bullet_LC\otimes_C\Ga&\rTo&\Fil^\bullet_LC\otimes_C\Ga^{\dR}.
 \end{diagram}
\end{align}
Here, the right vertical arrow is the factoring via $F_*\Fil^\bullet_{M'}W\simeq (\varphi_*\Fil^\bullet_IA)\otimes_{\varphi_*A}F_*W$ of $\zeta(\eta)$. The square on the right is Cartesian and $\Fil^\bullet_{\alpha(\eta)}W$ is the filtered animated $W$-algebra associated with the filtered Cartier-Witt divisor over $C$ parameterized by the right vertical map as in Lemma~\ref{lem:filtered_cw_divisor}. In particular, note that we obtain maps of animated commutative rings
\[
  R_A = \gr^0A \to \gr^0A\otimes_AW(C) \to (\gr^0_{\alpha(\eta)}W)(C).
\]
Therefore, we have produced from $\eta$ a point in $R_A^{\mathcal{N}}(C)$. This gives us the map $\iota_{(A,\zeta)}$.

If $(L\xrightarrow{\simeq}C)$ is an isomorphism, then it's clear that $\alpha(\eta)$ lies in the de Rham locus, and moreover that the top right arrow in~\eqref{eqn:commuting_diagram_zeta_eta} is an isomorphism on associated graded algebras. In particular, the induced map $(\gr^0_{\alpha(\eta)}W)(C)\to \overline{W(C)}$ is an isomorphism of $R_A$-algebras. This verifies assertion (1).

Let us look at assertion (2). When $\eta$ factors as
\[
   \Fil^\bullet A \to \varphi_*\Fil^\bullet_{I,\pm}A\to \Fil^\bullet_LC,
\]
the map $A\to C$ factors as $A\xrightarrow{\varphi}\varphi_*A \to C$. Then $\Fil^\bullet_LC$ is the base-change of $\Fil^\bullet_IA$ along $\varphi_*A\to C$ and the diagram of $W$-modules over $C$ from~\eqref{eqn:commuting_diagram_zeta_eta} is of the form
\[
\begin{diagram}
 \Fil^\bullet A\otimes_AW&\rTo&\varphi_*(\Fil^\bullet_IA)\otimes_{\varphi_*A}W &\rTo&F_*(\Fil^\bullet_IA\otimes_AW)\simeq F_*(\varphi_*(\Fil^\bullet_{\varphi^*I}A)\otimes_{\varphi_*A}W)\\
    &&\dTo&&\dTo\\
    &&\Fil^\bullet_LC\otimes_C\Ga&\rTo&\Fil^\bullet_LC\otimes_C\Ga^{\dR}
\end{diagram}
\]
where the top left horizontal arrow is obtained from the composition
\[
\Fil^\bullet A\otimes_AW\simeq \varphi_*\varphi^*\Fil^\bullet A\otimes_{\varphi_*A}W\xrightarrow{\Phi\otimes 1}\varphi_*(\Fil^\bullet_IA)\otimes_{\varphi_*A}W.
\]
This implies that $\alpha(\eta)$ lies in the Hodge-Tate locus.

For assertion (3), we begin with the following observation:
\begin{lemma}
   [Filtered prisms are laminated]
\label{lem:filtered_prisms_laminated}
The filtered prism structure canonically endows $\underline{A}$ with a lamination.
\end{lemma}
\begin{proof}
   Using the map $\Fil^\bullet A\to \Fil^\bullet_{\mathrm{triv}}R_A$ and the filtered prism structure, we get a commuting square of filtered animated commutative rings
   \[
     \begin{diagram}
    \Fil^\bullet A &\rTo&F_*\Fil^\bullet_{M'}W(R_A)\\
    \dTo&&\dTo\\
    \Fil^\bullet_{\mathrm{triv}}R_A&\rTo&\Fil^\bullet_{\mathrm{triv}}\Ga^{\dR}(R_A)
 \end{diagram}
   \]
   Unwinding definitions, the existence of the right vertical arrow means that $I\otimes_RW(R_A)\to W(R_A)/{}^{\mathbb{L}}p$ is equipped with a nullhomotopy, and hence that $I\otimes_RW(R_A)\to W(R_A)$ admits a factoring through $W(R_A)\xrightarrow{p}W(R_A)$. Since these are both Cartier-Witt divisors, we conclude that there is a canonical isomorphism between them. Moreover, the bottom horizontal arrow shows that the isomorphism between the associated quotient algebras in fact lifts to one of $R_A$-algebras.
\end{proof}
Therefore, for any $\eta$ factoring through a point of $\Aff^1/\Gm\times R_A$ associated with a generalized Cartier divisor $t:L\to C$, we see that the top right corner in the diagram~\eqref{eqn:commuting_diagram_zeta_eta} is $F_*\Fil^\bullet_pW$, and the right vertical arrow is the zero map on all positive filtered degrees. It now follows that we have
\[
   (\Fil^1_{\alpha(\eta)}W\to W)\xrightarrow{\simeq}(F_*W\oplus(L\otimes\Ga^\sharp)\xrightarrow{(V,t^\sharp)}W)
\]
and the map $\Fil^1A \to \Fil^1_{\alpha(\eta)}W \simeq F_*W\oplus (L\otimes\Ga^\sharp)$ is the divided Frobenius in the first coordinate and the zero map in the second coordinate. In particular, $\gr^0_{\alpha(\eta)}W$ is a quotient of $\Ga\simeq W/F_*W$, and the map $R_A\to (\gr^0_{\alpha(\eta)}W)(R_A)$ is just the evaluation at $R_A$ of this quotient map. This verifies (3) and completes the proof of the proposition.
\end{proof}

\begin{example}[The Breuil-Kisin case]
\label{ex:filtered_cw_transversal_prism}
Let us return to the situation of Remark~\ref{rem:bk_filtered_prisms}. A map $\eta:\Fil^\bullet A \to \Fil^\bullet_L C$ for an $A$-algebra $C$ corresponds to a factoring $C\otimes_AI'\xrightarrow{u}L\xrightarrow{t}C$ of the map $C\otimes_AI'\to C$. Associated with this is an arrow between maps of $W$-modules over $C$
\[
(F_*M'\to F_*W) \to (I'\otimes_A\Ga^{\dR}\to\Ga^{\dR})\to (L\otimes_C\Ga^{\dR}\xrightarrow{t^{\dR}}\Ga^{\dR})
\]
where the second map is given by $u^{\dR}$ on the source. The associated filtered Cartier-Witt divisor is obtained from $M'$ via pushforward along $u^{\dR}$. 
\end{example}

\begin{remark}[The perfectoid case]
\label{rem:perfectoid_case}
In the previous example, when $\underline{A} = \underline{A_{\mathrm{inf}}(R_0)} = \underline{\Prism}_{R_0}$ is the frame associated with a perfectoid ring $R_0$ (Example~\ref{example:a_inf_frame}), and $I' = (\varphi^{-1}(\xi))$ for an orientation $I = (\xi)$, then, via Example~\ref{ex:stacks_in_bk_case}, we see that we have a canonical filtered prism structure, which in this case is actually the \emph{unique} possible one, giving us a canonical map
\[
\Spf\left(\Prism_{R_0}[u,t]/(ut-\varphi^{-1}(\xi))\right)/\Gm \to R_0^{\mathcal{N}}.
\]
This map is actually an \emph{isomorphism}; see~\cite[Proposition 5.5.8]{bhatt_lectures}.
\end{remark}

\begin{example}
   [The Witt filtered prism]
\label{ex:witt_filtered_prism}
The Witt frame $\underline{W(R)}$ can be endowed with a canonical filtered prism structure, and so Proposition~\ref{prop:frames_to_nygaard} gives us a canonical map $\Rees(\Fil^\bullet_{\mathrm{Lau}}W(R))\to R^{\mathcal{N}}$. To see this structure, note that, quite generally, if $\underline{A}$ is a $p$-adically filtered frame, then we have a map $\Fil^\bullet_pA \to \Fil^\bullet A$, and base-change along $\varphi$ yields a map
\[
\tilde{\zeta}_{\underline{A}}:\varphi_*\Fil^\bullet_pA \simeq \Fil^\bullet_pA\otimes_A\varphi_*A\to \Fil^\bullet A\otimes_A\varphi_*A.
\]
When $A = W(R)$, we claim that the resulting diagram
\[
\begin{diagram}
   \Fil^\bullet_{\mathrm{Lau}}W(R)&\rTo&F_*\Fil^\bullet_pW(R)\\
   &\rdTo_{1\otimes F}&\dTo_{\tilde{\zeta}_{\underline{W(R)}}}\\
   &&\Fil^\bullet_{\mathrm{Lau}}W(R)\otimes_{W(R)}F_*W(R)
\end{diagram}
\]
is commutative. This follows from the fact that $\Fil^\bullet_{\mathrm{Lau}}W(R)\to F_*\Fil^\bullet_pW(R)$ is an isomorphism in filtered degrees $\ge 1$, while the right vertical arrow is an isomorphism in filtered degrees $\le 0$. In particular, mapping the vertical arrow further down to $R\Gamma(\Rees(\Fil^\bullet_{\mathrm{Lau}}W(R)),\Fil^\bullet \Ga^{\dR})$ now gives us a canonical filtered prism structure on $\underline{W(R)}$.
\end{example}

\begin{corollary}
\label{cor:frames_to_nygaard}
Suppose that we have a map of frames $\gamma:\underline{A}\to \underline{B}$ with $\underline{A}$ underlying a filtered prism $(\underline{A},\zeta)$. Let $\mathfrak{S}(\underline{B})$ be as in Remark~\ref{rem:B'Gmu_abstract}. Then we have an associated map
\[
\iota_{\gamma}:\mathfrak{S}(\underline{B})\to R_A^{\mathrm{syn}}
\]
of $p$-adic formal prestacks.
\end{corollary}
\begin{proof}
Proposition~\ref{prop:frames_to_nygaard} gives a map $\mathfrak{S}(\underline{A})\to R_A^{\mathrm{syn}}$, which we can compose with the map $\mathfrak{S}(\underline{B})\to \mathfrak{S}(\underline{A})$ obtained from $\gamma$.
\end{proof}

\begin{remark}
\label{rem:iota_gamma_mod-pn}
If $B$ is a $\Int/p^n\Int$-algebra, then $\iota_\gamma$ will factor through $R_A^{\mathrm{syn}}\otimes\Int/p^n\Int$.
\end{remark} 

\begin{example}
   \label{ex:witt_frame_to_syntomification}
Suppose that $R$ is an $\Field_p$-algebra. Example~\ref{ex:witt_filtered_prism} and Corollary~\ref{cor:frames_to_nygaard} applied to the map of frames $\underline{W(R)}\to \underline{W_1(R)}$ now give us a canonical map
\[
   R^{\Fzip} = \mathfrak{S}(\underline{W_1(R)})\to R^{\mathrm{syn}}\otimes\Field_p.
\]
See \S~\ref{subsec:fzip_to_syn} for an explicit description of this map.
\end{example}

\subsection{Nygaard filtered prismatic cohomology}
\label{subsec:nygaard_filtered_affineness}

We now review the story of the Nygaard filtration on prismatic cohomology.

\begin{definition}
A map $R\to S$ of $p$-complete animated commutative rings is \defnword{$p$-quasisyntomic} (or simply \defnword{quasisyntomic}) if it is $p$-completely flat (that is, $S/{}^{\mathbb{L}}p$ is flat over $R/{}^{\mathbb{L}}p$), and if $\mathbb{L}_{S/R}$ has $p$-complete Tor amplitude $[-1,0]$: that is, $\mathbb{L}_{S/R}\otimes\Field_p$ has Tor amplitude $[-1,0]$ over $S/{}^{\mathbb{L}}p$.

We will say that $R\to S$ is a \defnword{quasisyntomic cover} if it is quasisyntomic and $S/{}^{\mathbb{L}}p$ is faithfully flat over $R/{}^{\mathbb{L}}p$.

These properties are invariant under base-change via maps $R\to R'$ of $p$-complete animated commutative rings.
\end{definition}

\begin{example}
The key example of a quasisyntomic map is
\[
\Int_p[T]^{\wedge}\to \Int_p[T^{1/p^\infty}]^{\wedge},
\]
where the superscript ${}^\wedge$ denotes $p$-adic completion. 
\end{example}

\begin{remark}[Relative Nygaard filtration]
\label{rem:relative_nygaard}
We begin with the story relative to a bounded prism $(A,I)$. For what we need, there is no harm in restricting even to perfect prisms, and we will do so.

Given $R\in \mathrm{CRing}^{p\text{-comp}}_{\overline{A}/}$,  Write $\varphi^*\Prism_{R/A}$ for the base-change $A\otimes_{\varphi,A}\Prism_{R/A}$. It is shown in~\cite[\S 5.1]{bhatt2022absolute} that there is now a canonical lift $\Fil^\bullet_{\mathcal{N}}\varphi^*\Prism_{R/A}$ of $\Prism_{R/A}$ to $\mathrm{FilMod}_{A}$ characterized by the following properties:
\begin{enumerate}
   \item It respects sifted colimits in $R$.
   \item It satisfies $p$-quasisyntomic descent with respect to $R$.
   \item If $R$ is a $p$-quasisyntomic over $\overline{A}$ such that the quotient ring $R/pR$ is generated by the images of $\overline{A}$ and $(R/pR)^{\flat}$, then $\Fil^i_{\mathcal{N}}\Prism_{R/A}$ is discrete and we have
   \[
    \Fil^i_{\mathcal{N}}\varphi^*\Prism_{R/A} = \{x\in \varphi^*\Prism_{R/A}:\; (1\otimes\varphi)(x)\in I^i\Prism_{R/A}\}.
   \]
\end{enumerate}
\end{remark}

\begin{remark}
\label{rem:filtered_frob_rel_nygaard}
The map $1\otimes\varphi:\varphi^*\Prism_{R/A}\to \Prism_{R/A}$ can be canonically lifted to a filtered map
   \[
    \Fil^\bullet_{\mathcal{N}}\varphi^*\Prism_{R/A}\to \Fil^\bullet_I\Prism_{R/A}.
   \]
inducing for every $i\in\Int$ an equivalence
   \[
   \gr^i_{\mathcal{N}}\varphi^*\Prism_{R/A}\xrightarrow{\simeq}I^i/I^{i+1}\otimes_{\overline{A}}\Fil^{\mathrm{conj}}_i\overline{\Prism}_{R/A}\simeq \Fil^{\mathrm{conj}}_i\overline{\Prism}_{R/A}\{i\}.
\]
See~\cite[Remark 5.1.2]{bhatt2022absolute}.
\end{remark}

\begin{remark}[Absolute Nygaard filtration]
\label{subsubsec:absolute_nygaard}
In~\cite[\S 5.5]{bhatt2022absolute}, we find the construction of an \emph{absolute} Nygaard filtration $\Fil^\bullet_{\mathcal{N}}\Prism_R$ on absolute prismatic cohomology. It can be characterized by two properties: First, it satisfies quasisyntomic descent. Second, if $(A,I)$ is a perfect prism and $R$ is an $R_0=\overline{A}$-algebra, then, by~\cite[Theorem 5.6.2]{bhatt2022absolute}, the isomorphism
\[
\Prism_{R}\xrightarrow[\simeq]{\text{Theorem}~\ref{thm:absolute_prismatic_cohom}(2)}\Prism_{R/A}\xrightarrow{\simeq}\varphi^*\Prism_{R/A}
\]
lifts to an isomorphism of filtered objects
\begin{align}\label{eqn:absolute_to_relative_nygaard}
\Fil^\bullet_{\mathcal{N}}\Prism_R\xrightarrow{\simeq}\Fil^\bullet_{\mathcal{N}}\varphi^*\Prism_{R/A}.
\end{align}

In fact, if $R$ is qrsp and $(\Prism_R,I_R)$ is the canonical initial prism equipped with a map $R\to \overline{\Prism}_R$, then we have
\[
\Fil^i_{\mathcal{N}}\Prism_R = \{x\in \Prism_R:\;\varphi(x)\in \Fil^i_{I_R}\Prism_R\}.
\]

More generally, if $R$ is semiperfectoid and $(\Prism_R,I_R)$ is the associated initial prism, then we have a canonical filtered map
\[
\Phi:\Fil^\bullet_{\mathcal{N}}\Prism_R\to \Fil^\bullet_{I_R}\Prism_R
\]
lifting the Frobenius endomorphism of $\Prism_R$, which agrees with the relative counterpart from Remark~\ref{rem:filtered_frob_rel_nygaard} via the isomorphism~\eqref{eqn:absolute_to_relative_nygaard}. This follows from the construction explained in~\cite[Notation 5.7.5]{bhatt2022absolute}.
\end{remark}

\begin{lemma}
\label{lem:nygaard_filtered_frame}
Suppose that $R$ is semiperfectoid. Then, in the notation of~\S\ref{subsec:animated_frames}, the tuple $(\Fil^\bullet_{\mathcal{N}}\Prism_{R},I_R\to \Prism_R,\Phi)$ underlies a \emph{filtered prism} $(\underline{\Prism}_{R},\zeta_R)$ with $\gr^0_{\mathcal{N}}\Prism_R\simeq R$. Moreover, when $R$ is perfectoid, this structure agrees with the Breuil--Kisin one explained in Remark~\ref{rem:perfectoid_case}.
\end{lemma}
\begin{proof}
We would like to first check that  $\Fil^\bullet_{\mathcal{N}}\Prism_R$ is a filtered animated commutative ring and that the map $\Phi$ is a map of filtered animated commutative rings. Via mod-$p$ Tot descent~\cite[Propositions 4.4.15, 5.5.24]{bhatt2022absolute}, we can reduce to the case where $R$ is a semiperfect $\Field_p$-algebra.  When $R$ is qrsp, then the assertion is clear, and to know it in general, it is easiest to note that the construction is via right Kan extension from $\Field_p$-algebras $R$ as in (3) of Remark~\ref{rem:relative_nygaard} followed by left Kan extension from polynomial $\Field_p$-algebras. This endows $\Fil^\bullet_{\mathcal{N}}\Prism_R$ with the structure of a filtered \emph{derived} commutative ring in general, and also shows that $\Phi$ is a map of filtered derived commutative rings. This assertion specializes when $R$ is semiperfect to the structure desired.

The other thing to be verified is the existence of a commuting diagram witnessing the structure of a filtered prism. Once again, via mod-$p$ Tot descent and the right Kan extension followed by left Kan extension procedure, we can reduce to the case where $R$ is a qrsp $\Field_p$-algebra. In fact, to make things more concrete for the general case, we will consider another possibility as well: One where $R$ is qrsp and $p$-torsionfree. In both these cases, we can make the following observations:
\begin{enumerate}
   \item $\Prism_R$ is classical and $p$-torsionfree, and $\Fil^\bullet_{\mathcal{N}}\Prism_R=\varphi^{-1}(\Fil^\bullet_IR)$ is a filtration by submodules.
   \item $I =\Prism_R$ is a principal ideal generated by the image $\xi$ of a distinguished element of $W(R_0^\flat)\simeq \Prism_{R_0}$. Moreover, $\xi = \varphi(\xi')$, where $\xi'$ is the image of another distinguished element. We have $\xi' = \xi = p$ when $R$ is an $\Field_p$-algebra. 
   \item By working with the affine cover $\Spf \mathrm{Rees}(\Fil^\bullet A)\to \Rees(\Fil^\bullet A)$, it is enough to check that we have a \emph{unique} factoring
   \[
      \begin{diagram}
          \Fil^\bullet_{\mathcal{N}}\Prism_R&\rTo&\varphi_*\Fil^\bullet_I\Prism_R\\
   &\rdTo&\dTo_{\zeta(\eta)}\\
   &&\Fil^\bullet_{L}\Ga^{\dR}(C)
      \end{diagram}
   \]
   under the hypothesis that $C$ is $p$-torsionfree and that $L = (t)$ for some non-zero divisor $t\in C$ such that $\eta(\xi')\in (t)$. In this case, we have $\Ga^{\dR}(C)\simeq W(C)/pW(C)$.
\end{enumerate}

We now consider the two cases separately:
\begin{itemize}
   \item ($p$-torsionfree) In this case, we can also assume that $(p,\xi)$ maps to a regular sequence in $W(C)$ under the natural maps $\Prism_R\to W(\Prism_R)\to W(C)$. In turn, this implies that $t$ maps to a non-zero divisor in $\Ga^{\dR}(C)$. In filtered degree $1$, the map $I\to \Fil^1_L\Ga^{\dR}(C) = t\Ga^{\dR}(C)$ underlying $\zeta(\eta)$ is now the unique one sending $\xi$ to the image of $\eta(\xi')\in tC$. The factoring in higher filtered degrees is now immediate.

   \item ($\Field_p$-algebra) Here, by Remark~\ref{rem:derived_crystalline_semiperfect}, we have an isomorphism
   \[
      A_{\crys}(R)\xrightarrow{\simeq}\Prism_R,
   \]
   where $A_{\crys}(R)$ is the $p$-completed (classical) divided power envelope of the map $W(R^\flat)\to R$. Let $J = \ker(R^\flat\to R)$; then, by~\cite[Proposition 5.3.6]{bhatt2022absolute}, we can further identify $\Fil^n_{\mathcal{N}}\Prism_R$ with the ideal in $A_{\crys}(R)$ generated by the elements of the form
   \begin{equation}\label{eqn:element_of_fil_m}
      p^{m_0}\gamma_{m_1}([x_1])\cdots\gamma_{m_k}([x_k])
   \end{equation}
   with $x_1,\ldots,x_k\in J$ and $m_0 + m_1 + \cdots + m_k\geq n$. In particular, we see that there exists $u\in C$ such that $p = ut$, and also, for all $x\in J$ and $m\ge 1$, there exists $s_m(x)\in C$ such that $\eta(\gamma_m(x)) = s_m(x)t^m$. Since $C$ has no $t$-torsion, we deduce that $s_1(x)$ belongs to $\Ga^\sharp(C)$ with $m$-th divided power given by $s_m(x)$. In turn, this implies that $s_m(x)$ is in the kernel of the map $C\to \Ga^{\dR}(C)$, and one sees that the map $\Fil^m\Prism_R\to \Fil^m_L\Ga^{\dR}(C)$ kills all elements of the form~\eqref{eqn:element_of_fil_m} with $m_i\ge 1$ for $i\ge 1$. This means that we have a commuting diagram
   \[
   \begin{diagram}
      \Fil^m_{\mathcal{N}}\Prism_R&\rTo&\varphi_* p^m\Prism_R\\
      \dTo^{\eta^{\dR}}&\ldTo_{\zeta(\eta)}^{p^m y\mapsto u^m\overline{y}}&\dTo\\
      \Fil^m_L\mathbb{G}_a^{\dR}(C) =\Ga^{\dR}(C)&\rTo_{t^m}&\Ga^{\dR}(C).
   \end{diagram}
   \]
   Here, $\eta^{\dR}$ is the composition of $\eta$ with the natural map $\Fil^\bullet_LC \to \Fil^\bullet_L\Ga^{\dR}(C)$, and $y\mapsto \overline{y}$ is the composition
   \[
      \Prism_R\to W(\Prism_R)\to W(C) \to W(C)/pW(C)\simeq \Ga^{\dR}(C).
   \]
   It's now clear that $\zeta$ as defined gives us the unique filtered prism structure on $\underline{\Prism}_R$.
\end{itemize}
\end{proof}

The following description of the Nygaard filtered prismatization of semiperfectoid rings due to Bhatt-Lurie will be important for us.
\begin{theorem}
\label{thm:semiperf_crys}
Suppose that $R$ is semiperfectoid. Then there is a canonical isomorphism
\[
\iota_{(\underline{\Prism}_R,\zeta_R)}:\Rees(\Fil^\bullet_{\mathcal{N}}\Prism_R)\xrightarrow{\simeq} R^{\mathcal{N}}
\]
of $p$-adic formal stacks.
\end{theorem}
\begin{proof}
The existence of the map $\iota_{(\underline{\Prism}_R,\zeta_R)}$ follows from Lemma~\ref{lem:nygaard_filtered_frame} and Proposition~\ref{prop:frames_to_nygaard}. 

To complete the proof, we follow the format of~\cite[Theorem 5.5.10]{bhatt_lectures}. Choose a map $R_0\to R$ with $R_0$ perfectoid, and consider the map $\iota \defn \iota_{(\underline{\Prism}_R,\zeta_R)}$ as a map of stacks over 
\[
[\Spf(\Prism_{R_0}[u,t]/(ut-\varphi^{-1}(\xi)))/\Gm]\xrightarrow{\simeq}\Rees(\Fil^\bullet_{\mathcal{N}}\Prism_{R_0})\xrightarrow[\simeq]{\iota_{(\underline{\Prism}_{R_0},\zeta_{R_0})}}R_0^{\mathcal{N}}.
\]
See Remark~\ref{rem:perfectoid_case}. Theorem~\ref{thm:absolute_prismatic_cohom} now shows that $\iota$ is an isomorphism over the $t\neq 0$ locus. That it is also an isomorphism over the $u\neq 0$ locus follows from the same theorem and the fact that the conjugate filtration on the Hodge-Tate cohomology $\overline{\Prism}_{R/R_0}$ is exhaustive.

To see that $\iota$ is an isomorphism, it is now enough to know that it is also an isomorphism over the $t=u=0$ locus. The restriction of $\mathbb{G}^{\mathcal{N}}_{a}$ over 
\[
R^{\mathcal{N}}_{0,(t=u=0)}\simeq R_0^{\mathrm{HT}}\times B\Gm\simeq  \Spf R_0\times B\Gm
\]
 is canonically isomorphic to the trivial square-zero extension $\Ga\oplus (\mathcal{O}(1)\{-1\}\otimes\Ga^\sharp)[1]$. Therefore, the fiber of the map $R^{\mathcal{N}}_{(t=u=0)}\to R^{\mathcal{N}}_{0,(t=u=0)}$ over a point of $R_0^{\mathrm{HT}}\times B\Gm$--given by a point of $R_0^{\mathrm{HT}}(C)$---equivalent by our choice of orientation $\xi$ to giving an $R_0$-algebra structure on $C$ and a line bundle $L$ over $C$---is isomorphic to the space $\Map_{\mathrm{CRing}_{/R_0}}(R,C\oplus B(L\{-1\}\otimes_C\Ga^\sharp)(C))$. This space in turn parameterizes $R$-algebra structures on the $R_0$-algebra $C$ along with a section of
\[
\Map_{R}(\mathbb{L}_{R/R_0},B(L\{-1\}\otimes_C\Ga^\sharp)(C))\simeq\Map_R(\mathbb{L}_{R/R_0}\otimes_RL^\vee[-1]\{-1\},\Ga^\sharp(C))\simeq \Map_R(\Gamma_R(\mathbb{L}_{R/R_0}\otimes_RL^\vee[-1]\{-1\}),C).
\]
Here, we have used Lemma~\ref{lem:divided_powers_univ}. 

To summarize, $R^{\mathcal{N}}_{(t=u=0)}$ is represented over $R_0^{\mathrm{HT}}\times B\Gm$ by the relatively affine formal scheme represented by the formal spectrum of the graded $p$-complete animated commutative ring
\[
\Gamma_R(\mathbb{L}_{R/R_0}[-1]\{-1\}) \simeq \bigoplus_{i=0}^\infty \wedge^i\mathbb{L}_{R/R_0}[-i]\{-i\}.
\]
By the Hodge-Tate decomposition (see Remark~\ref{rem:hodge-tate_cohomology}), this is also the case for $\Rees(\Fil^\bullet_{\mathcal{N}}\Prism_R)_{(u=t=0)}$. 

Therefore, the map
\begin{align}\label{eqn:t=u=0_map}
\Rees(\Fil^\bullet_{\mathcal{N}}\Prism_R)_{(u=t=0)}\to R^{\mathcal{N}}_{(t=u=0)}
\end{align}
corresponds to an endomorphism of the graded $R_0$-algebra $\bigoplus_{i=0}^\infty \wedge^i\mathbb{L}_{R/R_0}[-i]\{-i\}$ that is functorial in the semiperfectoid $R_0$-algebra $R$, is compatible with tensor products and is the identity in degree $0$. We claim that this is in fact the identity. To see this, one follows the argument from~\cite[Theorem 5.5.10]{bhatt_lectures} to find that this endomorphism is obtained via left Kan extension from an endomorphism of the same functor but now defined on $p$-completely smooth $R_0$-algebras and valued in derived $R_0$-algebras that are $p$-complete. In fact, one can argue as in the proof of~\cite[Theorem 7.17]{bhatt2022prismatization} to the consideration just of the $p$-completed polynomial algebra $R_0[x]^{\wedge}$ over $R_0$. Here, the graded algebra is just $R_0[x]^{\wedge}\oplus R_0[x]^{\wedge}dx\{-1\}[-1]$, and one sees that the only possible functorial endomorphisms with the required properties are the identity and the map killing $dx$. 

We claim that the latter is not possible. Indeed, it is enough to find one semiperfectoid $R_0$-algebra $R$ such that the map~\eqref{eqn:t=u=0_map} does not factor through $R^{\mathrm{HT}}\times B\Gm$. If $R_0$ is $p$-torsion free, then we choose $R = R_0/pR_0$; otherwise, we choose $R = R_0/xR_0$ for some regular element $x\in R_0$. In fact, in the first case, we can replace $R_0$ with its tilt, and reduce to the consideration just of the second case, where the claim is easily verified using Remark~\ref{rem:derived_crystalline_semiperfect}.
\end{proof}

\subsection{Descent}
\label{subsec:descent_stacks}
The stacks we are concerned with here carry $p$-quasisyntomic covers to covers in the fpqc topology. This is well-known to experts, but we include proofs here, since it is important for what follows.

\begin{proposition}
   \label{prop:etale_goes_to_etale}
Suppose that we have a $p$-completely \'etale map $g:R\to S$ in $\mathrm{CRing}^{p\text{-comp}}$. Then the associated maps $g^{\mathcal{N}}:S^{\mathcal{N}}\to R^{\mathcal{N}}$ and $g^{\Prism}:S^{\Prism}\to R^{\Prism}$ are $(p,I)$-completely \'etale.
\end{proposition}
\begin{proof}
  The assertion for $g^\Prism$ is shown in~\cite[Remark 3.9]{bhatt2022prismatization}. The main point is that, for any $C\in \mathrm{CRing}^{p\text{-nilp}}$ and a Cartier-Witt divisor $I\to W(C)$ with quotient $\overline{W(C)}$, the small $p$-completely \'etale site of $\overline{W(C)}$ is equivalent to that of $W(C)$, and hence to that of $C$. This shows that, given $R\to \overline{W(C)}$, we have
\[
S\otimes_R\overline{W(C)} \simeq \overline{W(C')}
\]
for a canonical \'etale cover $C\to C'$.

Let us now look at $g^{\mathcal{N}}$: given $(M\xrightarrow{d} W)$ in $\Int_p^{\mathcal{N}}(C)$, $W/_dM$ is an extension of $\overline{W}$ by $\hcoker(t^\sharp)$, where $t^\sharp:\mathbf{V}(L^\vee)^\sharp\to \mathbb{G}_a^\sharp$ is the map associated with a cosection of a line bundle over $\Spec C$. Therefore, the map
\[
(W/_dM)(C)\to \overline{W(C)}
\]
is surjective as long as $C$ is $\mathbb{G}_a^\sharp$-acyclic. For such $C$, the same argument as in the previous paragraph shows that 
\[
S\otimes_R(W/_dM)(C)\simeq (W/_dM)(C')
\]
for a canonical \'etale map $C\to C'$. In general, we can choose a faithfully flat map $C\to D$ such that $D$ is $\mathbb{G}_a^\sharp$-acyclic, and so the conclusion from the previous sentence follows for $C$ by fpqc descent for \'etale $C$-algebras.
 \end{proof}

\begin{proposition}
\label{prop:flat_surjections}
Suppose that we have $g:R\to S$ in $\mathrm{CRing}^{p\text{-comp}}$ with $S$ the $p$-completion of $\Int_p[T^{1/p^\infty}]\otimes_{\Int_p[T]}R$ for some map $\Int_p[T]\to R$.  Then the associated maps $g^{\mathcal{N}}:S^{\mathcal{N}}\to R^{\mathcal{N}}$ and $g^{\Prism}:S^{\Prism}\to R^{\Prism}$ are surjective in the $p$-completely flat topology.
\end{proposition}
\begin{proof}
By the limit-preserving properties of the prismatization and Nygaard filtered prismatization functors, we are reduced to the situation where $R = \Int_p[T]^\wedge$ and $S = \Int_p[T^{1/p^\infty}]^\wedge$.

Here, we will use the following consequence of Lemma~\ref{lem:extracting_phi_roots}: For any $C\in \mathrm{CRing}^{p\text{-nilp}}$ and any map $\beta:\Int_p[T]\to W(C)$, we can find a faithfully flat map $C\to C'$, and an extension $\Int_p[T^{1/p^\infty}]\to W(C')$ of $\beta$. Now, given any diagram $W(C)\twoheadrightarrow \overline{W(C)}\leftarrow \Int_p[T]$ corresponding to a $C$-valued point of $R^\Prism$, we can first lift the second map to a map $\Int_p[T]\to W(C)$, and then find $C\to C'$ as in the previous sentence, so that we have a lift to $C'$-valued point of $S^\Prism$.

This completes the proof in the case of the prismatizations. For the Nygaard filtered prismatization, suppose that we have a point in $R^{\mathcal{N}}(C)$. By replacing $C$ with a flat cover if necessary we can assume that it is of the form $M(C,d,t,u)\to W$ (see Remark~\ref{rem:fil_cw_pushout}). We then have a surjection $W/{}^{\mathbb{L}}d\to W/M$ with homotopy kernel $\hcoker(u^\sharp)$. This means that, over $\mathbb{G}_a^\sharp$-acyclic algebras, any map $\Int_p[T]\to (W/M)(C)$ can be lifted \'etale locally to a map $\Int_p[T]\to W(C)/{}^{\mathbb{L}}d$, and thence to $W(C)$, and we can run the argument used for the prismatization again.
\end{proof}

\begin{corollary}
\label{cor:semiperf_qsynt}
For any $R\in \mathrm{CRing}^{p\text{-comp}}$, there exists a quasisyntomic cover $R\to R_\infty$ with $R_\infty^{\otimes_R m}$ \emph{semiperfectoid} for all $m\ge 1$, and also satisfying the following properties:
\begin{enumerate}
   \item The maps 
\[
R_\infty^{\Prism}\to R^{\Prism}\;;\; R_{\infty}^{\mathcal{N}}\to R^{\mathcal{N}}
\]
are surjective in the $p$-completely flat topology;
   \item $R_{\infty}^{\Prism}\simeq \Spf\Prism_R$ is a derived formal affine scheme and $R_{\infty}^{\mathcal{N}}\simeq \Rees(\Fil^\bullet_{\mathcal{N}}\Prism_R)$ is a derived $p$-adic formal Artin $1$-stack; 
   \item The natural maps
   \[
    (R_{\infty}\hat{\otimes}_RR_{\infty})^{\mathcal{N}} \to R_{\infty}^{\mathcal{N}}
   \]
   are $(p,I)$-completely faithfully flat, where $I = I_{R_{\infty}}$.
\end{enumerate}
\end{corollary}
\begin{proof}
Choose a set of generators $\{x_i:i\in I\}$ for $\pi_0(R)/p\pi_0(R)$, and a map $\Int_p[T_i:i\in I]^{\wedge}_p\to R$ carrying $T_i$ onto $x_i\in \pi_0(R)/p\pi_0(R)$. The $R$-algebra
\[
R_\infty \defn \Int_p[\zeta_p^{1/p^\infty}][T_i^{1/p^\infty}:i\in I]^{\wedge}_p\hat{\otimes}_{\Int_p[T_i:i\in I]^{\wedge}_p}R
\]
now does the job. By construction $R_\infty^{\otimes_R m}$ is semiperfectoid for all $m\ge 1$, and it follows from Proposition~\ref{prop:flat_surjections} that $R_\infty^\Prism\to R^\Prism$ and $R_\infty^\mathcal{N}\to R^\mathcal{N}$ are surjective in the $p$-completely flat topology: Indeed, the Proposition tells us that these are filtered limits of such covers, so one is reduced to the fact that a filtered colimit of faithfully flat maps of rings is once again faithfully flat.

Assertion (2) is immediate from Theorem~\ref{thm:semiperf_crys}.

The last assertion follows from Lemma~\ref{lem:faithfully_flat_nygaard} below.
\end{proof}

\begin{lemma}
\label{lem:faithfully_flat_nygaard}
Suppose that $R\to S$ is a quasisyntomic cover of semiperfectoid rings. Then the map
\[
S^{\mathcal{N}}\otimes\Int/p^n\Int\to R^{\mathcal{N}}\otimes\Int/p^n\Int
\]
is a faithfully flat map of derived $p$-adic formal Artin $1$-stacks.
\end{lemma}
\begin{proof}
Given Theorem~\ref{thm:semiperf_crys}, the proof is identical to that of assertion (1) in the proof of~\cite[Proposition 2.29]{guo2023frobenius}.
\end{proof}

Combining Corollary~\ref{cor:semiperf_qsynt} with Lemma~\ref{lem:faithfully_flat_nygaard} also yields:
\begin{corollary}
\label{cor:quasisyntomic_descent_general}
Suppose that $R\to S$ is a quasisyntomic cover in $\mathrm{CRing}^{p\text{-comp}}$. Then the map
\[
S^{\mathcal{N}}\to R^{\mathcal{N}}
\]
is a surjection in the $p$-completely flat topology.
\end{corollary}

\subsection{A nilpotence result}
\label{subsec:nilpotence}

We will need a certain nilpotence result for applications to the abstract Grothendieck-Messing theory formulated in~\S\ref{subsec:1-bounded_stacks_def_theory}. Suppose that $(R'\twoheadrightarrow R,\eta)$ is a divided power thickening of semiperfectoid rings in $\mathrm{CRing}^{p\text{-nilp}}$. Then by Lemma~\ref{lem:lift_divided_powers} and Theorem~\ref{thm:absolute_prismatic_cohom}, the natural map $\Prism_{R'}\to R'$ factors canonically through $\Prism_R$.

Set
\[
K_{R'\twoheadrightarrow R} \defn \hker\left(\Prism_{R'}\to \Prism_R\right)
\]
Just as in Remark~\ref{rem:phi1_lifts}, the maps
\[
\varphi_1: \Fil^1_{\mathcal{N}}\Prism_{R'}\to I_{R'}\;;\; \varphi_1:\Fil^1_{\mathcal{N}}\Prism_R\to I_R
\]
now give rise to a $\varphi$-semilinear map
\[
K_{R'\twoheadrightarrow R} \to I_{R'}\otimes_{\Prism_{R'}}K_{R'\twoheadrightarrow R}.
\]
Using the defining properties of the Breuil-Kisin twist, this map yields a $\varphi$-semilinear endomorphism\footnote{If $M\to I_{R'}\otimes_{\Prism_{R'}}M$ is a $\varphi$-semilinear map for a $\Prism_{R'}$-module $M$, then it factors through a linear map $\varphi^*M \to I_{R'}\otimes_{\Prism_{R'}}M$. In turn this gives a linear map $\varphi^*(M\{1\}) \simeq I_{R'}^\vee\otimes_{\Prism_{R'}}(\varphi^*M)\{1\}\to (I_{R'}^\vee\otimes_{\Prism_{R'}}I_{R'})\otimes_{\Prism_{R'}}M\{1\}\simeq M\{1\}$ corresponding to a $\varphi$-semilinear endomorphism of $M\{1\}$.}
\[
\dot{\varphi}_1: K_{R'\twoheadrightarrow R}\{1\}\to K_{R'\twoheadrightarrow R}\{1\}.
\]

\begin{proposition}
\label{prop:dot_varphi_1_nilpotent}
Suppose that the divided power structure is nilpotent of nilpotence degree $m$ as in Definition~\ref{defn:nilpotent_divided_powers}. Then the composition 
\[
K_{R'\twoheadrightarrow R}\{1\} \xrightarrow{\dot{\varphi}_1^{2m+1}} K_{R'\twoheadrightarrow R}\{1\}\to K_{R'\twoheadrightarrow R}\{1\}/{}^{\mathbb{L}}(p,I)
\]
is nullhomotopic.
\end{proposition}

We will need a little preparation. 
\begin{assumption}
Until further notice, we will assume that $R'$ is an $\Field_p$-algebra. In particular, the Breuil-Kisin twist is trivial in this case, and so will be ignored.
\end{assumption}

\begin{notation}
Set $K = K_{R'\twoheadrightarrow R}$, $\Fil^\bullet A = \Fil^\bullet_{\mathcal{N}}\Prism_R$, $\Fil^\bullet A' = \Fil^\bullet_{\mathcal{N}}\Prism_{R'}$, and $I = \hker(R'\twoheadrightarrow R)$. We will use an overline $\overline{(\cdot)}$ to denote mod-$p$ reduction. Set 
\[
\Fil^iK = \hker(\Fil^iA'\to \Fil^iA)\;;\; \Fil^{\mathrm{conj}}_i\overline{K} = \hker(\Fil^{\mathrm{conj}}_i\overline{A}'\to \Fil^{\mathrm{conj}}_i\overline{A}).
\]
\end{notation}

\begin{remark}
\label{rem:reduction_to_map_u}
The composition $\Fil^1K\to K \to \overline{K}$ induced by $\varphi_1$ factors through $\Fil^{\mathrm{conj}}_1\overline{K}$. Moreover, we have a canonical isomorphism
\[
I=\hker(R'\twoheadrightarrow R)\xrightarrow{\simeq}\hker(\Fil^{\mathrm{conj}}_0\overline{A}'\to \Fil^{\mathrm{conj}}_0\overline{A})=\Fil^{\mathrm{conj}}_0\overline{K}.
\]
The factoring $A'\to A\to R'$ induces a splitting (see the argument in Remark~\ref{rem:frames_surjective}):
\begin{align}\label{eqn:fil1k_splitting}
\Fil^1K \xrightarrow{\simeq}K\oplus I[-1].
\end{align}
The endomorphism of $\overline{K}$ arising from $\dot{\varphi}_1$ is now given by the composition
\begin{align*}
\overline{K} \to \overline{\Fil^1K} \to \Fil^{\mathrm{conj}}_1\overline{K} \to \overline{K},
\end{align*}
where the first map is obtained from the splitting~\eqref{eqn:fil1k_splitting}. It will be useful therefore to consider the composition
\begin{align}\label{eqn:factoring_through_u}
u:\Fil^{\mathrm{conj}}_1\overline{K}\to \overline{K}\to \overline{\Fil^1K}\to \Fil^{\mathrm{conj}}_1\overline{K}
\end{align}

\end{remark}

\begin{remark}[The not-necessarily-nilpotent qrsp case]
\label{rem:nnn_qrsp_case}
Let us jettison the nilpotence hypothesis on the divided power structure, and observe that all the constructions involved here make sense for arbitrary divided power extensions $(R'\twoheadrightarrow R,\gamma)$ of $\Field_p$-algebras. 

Let us assume instead for the moment that $R'$ and $R$ are qrsp $\Field_p$-algebras and that we have $\varphi^n(I) = 0$ for $n$ sufficiently large; for instance, this is the case if the divided powers are nilpotent.  In this situation, we can understand $K$ and the section $K\to \Fil^1K$ in reasonably explicit fashion. Let $R^\flat$ be the inverse limit perfection of $R$, and hence also of $R'$. Then $A'$ (resp. $A$) is the $p$-completed divided power envelope of $W(R^\flat)\to R'$ (resp. $W(R^\flat)\to R$). The lift $A\to R'$ is obtained from the divided powers on $I$. The map $\Fil^\bullet A'\to \Fil^\bullet A$ is an injective map of classical filtered commutative rings, and so we have 
\[
K\simeq (A/A')[-1]\;;\;\Fil^1K = (\Fil^1A/\Fil^1A')[-1].
\]
Set $J' = \ker(R^\flat\to R')$ and $J = \ker(R^\flat\to R)$. Write $\pi:A\to A/A'$ for the quotient map. Then the $A'$-module $A/A'$ is topologically generated by elements of the form
\[
\pi(\gamma_{m_1}([\tilde{x}_1])\cdots\gamma_{m_r}([\tilde{x}_r])), 
\]
for $\tilde{x}_i\in J$ and $m_i\ge 2$. The divided powers here are those in $\Fil^1A$ given by the construction of $A$ as a divided power envelope. The section $K\to \Fil^1K$ is obtained after a cohomological shift from the map
\begin{align}\label{eqn:divided_power_section_explicit}
A/A' \to \Fil^1A/\Fil^1A'
\end{align}
sending $\pi(\gamma_m([\tilde{x}]))$ to the image of $\gamma_m([\tilde{x}]) - [\widetilde{\eta_m(x)}]$, where $x\in I$ is the image of $\tilde{x}$ and $\widetilde{\eta_m(x)}\in J$ is a lift of the divided power $\eta_m(x)\in I$. Products of such elements are sent compatibly to the products of their images.
\end{remark}

\begin{remark}
\label{rem:fil_1_conj_qrsp}
Let us now explicate the structure of $\Fil^{\mathrm{conj}}_1\overline{K}$ in the qrsp context. We begin by noting that the map 
\[
R\xrightarrow{\simeq}\Fil^{\mathrm{conj}}_0\overline{A}\subset \overline{A}
\]
carries $x\in R$ to $\tilde{x}^p$, for any $\tilde{x}\in \overline{A}$ lifting $x$; see~\cite[Remark F.7]{bhatt2022absolute}. Moreover, by Construction F.8 of \emph{loc. cit}, we see that 
$\Fil^{\mathrm{conj}}_1\overline{A}$ is generated as a module over $\Fil^{\mathrm{conj}}_0\overline{A}$ by elements of the form $\gamma_p(\tilde{x})$ for $\tilde{x}\in J$. Finally, by Proposition F.9 of \emph{loc. cit.}, there is a canonical isomorphism $J/J^2\xrightarrow{\simeq}\gr^{\mathrm{conj}}_1\overline{A}$ sending $\tilde{x}+J^2$ to the image of $\gamma_p(\tilde{x})$. Since a similar description holds for $\Fil^{\mathrm{conj}}_1\overline{A}'$, we obtain compatible short exact sequences
\[
0\to R'\to \Fil^{\mathrm{conj}}_1\overline{A}\to J'/(J')^2\to 0\;;\;0\to R\to \Fil^{\mathrm{conj}}_1\overline{A}\to J'/(J')^2\to 0.
\]
Together they yield a fiber sequence
\[
I \to \Fil^{\mathrm{conj}}_1\overline{K}\to \gr^{\mathrm{conj}}_1\overline{K}\simeq\hker(J'/(J')^2\to J/J^2).
\]
\end{remark}

\begin{lemma}
\label{lem:qrsp_case_u}
With the setup of Remark~\ref{rem:fil_1_conj_qrsp}, consider the commuting diagram 
\[
\begin{diagram}
H^0(\Fil^{\mathrm{conj}}_1\overline{K})[0]&\rTo &\Fil^{\mathrm{conj}}_1\overline{K}&\rTo& H^1(\Fil^{\mathrm{conj}}_1\overline{K})[-1]\\
\dTo^{H^0(u)[0]}&&\dTo^u&&\dTo^{H^1(u)[-1]}\\
H^0(\Fil^{\mathrm{conj}}_1\overline{K})[0]&\rTo &\Fil^{\mathrm{conj}}_1\overline{K}&\rTo& H^1(\Fil^{\mathrm{conj}}_1\overline{K})[-1]
\end{diagram}
\]
where the rows are the tautological fiber sequences. Set $u_0 = H^0(u)$ and $u_1 = H^1(u)$. Then:
\begin{enumerate}
   \item We have $u_0(I) = u_0(H^0(\Fil^{\mathrm{conj}}_1\overline{K}))\subset H^0(\Fil^{\mathrm{conj}}_1\overline{K})$.
   \item The composition
   \[
     I\xrightarrow{u_0}H^0(\Fil^{\mathrm{conj}}_1\overline{K})\to \Fil^{\mathrm{conj}}_1\overline{A}'
   \]
   carries $x\in I$ to $\eta_p(x) - \gamma'_p(\tilde{x}^p)$ where $\tilde{x}\in J$ is any lift of $x$. Here, we are viewing $\eta_p(x)$ as an element of $\overline{A}'$ via the identification $R'=\Fil_0^{\mathrm{conj}}\overline{A}'$.
   \item Via the isomorphisms $H^1(\Fil^{\mathrm{conj}}_1\overline{K}) \simeq J/(J^2+J')\simeq I/I^2$, $u_1$ corresponds to the endomorphism
   \[
      I/I^2\xrightarrow{x+I^2 \mapsto \eta_p(x)+I^2}I/I^2.
   \]
\end{enumerate}
In particular, if the divided power structure is nilpotent of nilpotence degree $m$, then $u^{2m}$ is nullhomotopic.
\end{lemma} 
\begin{proof}
Assertion (1) amounts to knowing that, for all $\tilde{x}\in J'\cap J^2$, $u$ kills $\gamma'_p(\tilde{x})$. Via the isomorphism $\zeta:H^0(\overline{K}) \simeq (A/A')[p]$, $\gamma'_p(\tilde{x})$ is mapped to the image of $\frac{\gamma_p([\tilde{x}])}{p}\in A$. If $\tilde{x} = \sum_i\tilde{y}_i\tilde{z}_i$ for $\tilde{y}_i,\tilde{z}_i\in J$, then we find that
\[
\zeta(\tilde{x}) = (p-1)!\sum_{i}\pi(\gamma_p([\tilde{y}_i])\gamma_p([\tilde{z}_i]))\overset{\eqref{eqn:divided_power_section_explicit}}{\mapsto}(p-1)!\sum_i(\gamma_p[\tilde{y}_i]-[\widetilde{\eta_p(y_i)}])(\gamma_p[\tilde{z}_i]-[\widetilde{\eta_p(z_i)}])\in (\Fil^1A/\Fil^1A')[p].
\]
It is clear now that this element dies under the divided Frobenius: it lies in the image of $\Fil^2A$.

Let us now proceed to assertion (2). The map $I\to H^0(\Fil^{\mathrm{conj}}_1\overline{K})\subset H^0(\overline{K}) \simeq (A/A')[p]$ now admits the following explicit description: It associates with each $x\in I$ the element $(p-1)!\pi(\gamma_p([\tilde{x}]))$, where $\tilde{x}\in J$ is any lift of $x$. Via the section~\eqref{eqn:divided_power_section_explicit} we obtain maps
\[
I\to (A/A')[p]\to (\Fil^1A/\Fil^1A')[p]\to \overline{\Fil^1A'}
\]
whose composition now takes $x$ to the image in $\overline{\Fil^1A'}$ of the element
\[
[\tilde{x}^p] - p!\cdot[\widetilde{\eta_p(x)}]\in \Fil^1A'.
\]
Hitting this element with the divided Frobenius and then taking the image in $\overline{A}'$ gives
\[
u(x) = (p-1)!\left(\gamma'_p(\tilde{x}^p) - \eta_p(x)\right) = \eta_p(x) - \gamma'_p(\tilde{x}^p)\in  H^0(\Fil^{\mathrm{conj}}_1\overline{K}).
\]

Now for assertion (3): Unwinding definitions, we see that the composition of this endomorphism with the inclusion
\[
J/(J^2+J')\xrightarrow{\tilde{x} \mapsto \gamma_p(\tilde{x})}\overline{A}/\mathrm{im}(\overline{A}')
\]
can be described as follows: Given $\tilde{x}\in J$ lifting an element of $J/(J^2+J')$, we consider $\gamma_p([\tilde{x}])\in A$, map it to the element $\gamma_p([\tilde{x}]) - [\widetilde{\eta_p(x)}]\in \Fil^1A$, hit it with the divided Frobenius to get the element $\frac{\gamma_p([\tilde{x}]^p)}{p} - (p-1)!\gamma_p([\widetilde{\eta_p(x)}])\in A$ and take its image in $\overline{A}/\mathrm{im}(\overline{A}')$. Since $\gamma_p([\tilde{x}^p])$ is divisible by $p^2$, this image is exactly that of $\widetilde{\eta_p(x)}\in J$.

Finally, note that (1), (2) and (3) together imply that, when $\eta_p^m = 0$, $u_0^m$ and $u_1^m$ are both the zero endomorphism. Therefore, $u^m$ factors through a map $H^1(\Fil^{\mathrm{conj}}_1\overline{K})[-1]\to H^0(\Fil^{\mathrm{conj}}_1\overline{K})$ and so its square will be nullhomotopic.
\end{proof}

\begin{proof}
[Proof of Proposition~\ref{prop:dot_varphi_1_nilpotent}]
Suppose first that $R'$ and $R$ are quasisyntomic $\Field_p$-algebras, and that the divided power structure has nilpotence degree $m$. Consider the endomorphism $u$ from~\eqref{eqn:factoring_through_u}. Then, using quasisyntomic descent and Lemma~\ref{lem:qrsp_case_u}, we find that $u^{2m} = 0$. Since $\dot{\varphi}_1^{2m+1}$ factors through $u^{2m}$ mod $p$, this shows that we have $\dot{\varphi}_1^{2m+1} = 0$ mod $p$.

The case where $R'$ is a general $\Field_p$-algebra can be deduced from this and the fact that the $\infty$-category of divided power extensions $(R'\twoheadrightarrow R,\eta)$ of $\Field_p$-algebras of nilpotence degree $m$ is generated under sifted colimits by extensions where $R'$ and $R$ are quasisyntomic. This boils down to the observation that the universal quotient of the divided power algebra $\Field_p\langle X\rangle$ in one variable in which the divided power has nilpotence degree $m$ is isomorphic to $\Field_p[X_0,\ldots,X_{m-1}]/(X_0^p,\ldots,X_{m-1}^p)$.

For the general case, using mod-$p$ Tot descent for prismatic cohomology~\cite[Proposition 5.5.24]{bhatt2022absolute}, one now sees that, even when $R'\in \mathrm{CRing}^{p\text{-comp}}$ is arbitrary, the divided power structure on $R'\twoheadrightarrow R$ being nilpotent of nilpotence degree $m$ implies that $\dot{\varphi}_1^{2m+1}$ acts nullhomotopically on $K\{1\}/{}^{\mathbb{L}}(p,I)$.
\end{proof}

\section{A result of Bragg-Olsson}
\label{sec:technical_repbility}

The goal of this section is to state and prove Theorem~\ref{thm:repbility_height_one}, which is due to Bragg-Olsson in the case of perfect complexes, and is an important ingredient in the general representability theorem that will appear in the next section.

\subsection{Formulation of the result}
\label{subsec:height_1}

\subsubsection{}
For $R\in \mathrm{CRing}_{\Field_p/}$, set
\[
Z^1_{\Prism}(R) = \Fil^{\mathrm{conj}}_1\overline{\Prism}_R\times_{\overline{\Prism}_R}\Fil^1_{\mathrm{Hdg}}\overline{\Prism}_R\simeq\hker(\Fil_1^{\mathrm{conj}}(R/^{\mathbb{L}}p)\to R)
\]
Then we have two maps
\[
q_1,q_2:Z^1_{\Prism}(R) \to \gr_1^{\mathrm{conj}}\overline{\Prism}_R\simeq \mathbb{L}_{R/\Field_p}[-1],
\]
where $q_1$ arises from the natural map $\Fil_1^{\mathrm{conj}}\overline{\Prism}_R\to \gr_1^{\mathrm{conj}}\overline{\Prism}_R$, and $q_2$ arises from the natural map $ \Fil_{\mathrm{Hdg}}^1\overline{\Prism}_R\to \gr^1_{\mathrm{Hdg}}\overline{\Prism}_R$ composed with the Cartier isomorphism.

$Z^1_{\Prism}(R)$ inherits the structure of an object in $\Mod{R}$ from $\Fil_1^{\mathrm{conj}}\overline{\Prism}_R$. For this structure, $q_1$ is $R$-linear, while $q_2$ is $\varphi$-semilinear, and so corresponds to an $R$-linear map $1\otimes q_2:\varphi^*Z^1_{\Prism}(R)\to \gr_1^{\mathrm{conj}}\overline{\Prism}_R$. 

\subsubsection{}
Let $\Mod[\varphi,*]{R}$ be the $\infty$-category of pairs $(N,\psi)$, where $N\in \Mod{R}$, and $\psi:N\to \varphi^* N$ is a map in $\Mod{R}$. For each such pair, we obtain two maps
\[
q_{1,\psi},q_{2,\psi}:\mathrm{RHom}_R(N,Z^1_{\Prism}(R))\to \mathrm{RHom}_R(N,\gr_1^{\mathrm{conj}}\overline{\Prism}_R).
\]
The map $q_{1,\psi}$ is simply postcomposition with $q_1$, and is independent of $\psi$, while the second is the composition
\[
\mathrm{RHom}_R(N,Z^1_{\Prism}(R))\xrightarrow{\varphi^*} \mathrm{RHom}_R(\varphi^*N,\varphi^* Z^1_{\Prism}(R))\xrightarrow{(1\otimes q_2)\circ(\_)\circ \psi}\mathrm{RHom}_R(N,\gr_1^{\mathrm{conj}}\overline{\Prism}_R).
\]

\begin{remark}
\label{rem:zprism_expl}
Suppose that $R$ is a smooth $\Field_p$-algebra. Then we have 
\[
\Fil_i^{\mathrm{conj}}\overline{\Prism}_R\simeq \tau^{\le i}\Omega^\bullet_{R/\Field_p},
\]
and therefore $Z^1_{\Prism}(R) \simeq \Omega^{1,\mathrm{cl}}_{R/\Field_p}[-1]$ is the shifted module of closed differential forms, while $\gr^{\mathrm{conj}}_1\overline{\Prism}_R\simeq H^1(\Omega^\bullet_{R/\Field_p})[-1]$.

The map $q_1$ now is the natural surjection, while $q_2$ is the composition
\[
\Omega^{1,\mathrm{cl}}_{R/\Field_p}[-1]\to \Omega^1_{R/\Field_p}[-1]\xrightarrow{\simeq}H^1(\Omega^\bullet_{R/\Field_p})[-1],
\]
where the first map is the inclusion, and the second is the Cartier isomorphism.

When $N$ is a vector bundle over $R$, there is an associated finite flat commutative group scheme $G(N,\psi)$ of height one (see~\S\ref{subsec:height_one_group_schemes} below), and the complex $R\Gamma_\varphi(R,(N,\psi))$ is (up to shift) precisely that of Artin-Milne appearing in~\cite[Prop. 2.4]{artin_milne}, which computes the flat cohomology of $G(N,\psi)$.
\end{remark}

\subsubsection{}
We can now associate with each $(N,\psi)\in \Mod[\varphi,*]{R}$ the space
\[
R\Gamma_\varphi\left(R,(N,\psi)\right) = \hker\left(\mathrm{RHom}_R(N,Z^1_{\Prism}(R))\xrightarrow{q_{1,\psi}-q_{2,\psi}}\mathrm{RHom}_R(N,\gr_1^{\mathrm{conj}}\overline{\Prism}_R)\right).
\]
This yields a prestack
\begin{align*}
\mathsf{S}_{(N,\psi)}:\mathrm{CRing}_{R/}&\to \Mod[\mathrm{cn}]{\Field_p}\\
C&\mapsto \tau^{\leq 0}R\Gamma_\varphi\left(C,(C\otimes_RN, \mathrm{id}\otimes \psi)\right)
\end{align*}

\begin{theorem}
\label{thm:repbility_height_one}
Suppose that $N$ is almost perfect and $(-m)$-connective; then $\mathsf{S}_{(N,\psi)}$ is represented by an almost finitely presented derived Artin $m$-stack over $R$. If $N$ is perfect, then $\mathsf{S}_{(N,\psi)}$ is finitely presented. If, further, $N$ has Tor amplitude in $[1,\infty)$, then $\mathsf{S}_{(N,\psi)}$ is smooth over $R$.
\end{theorem}

\begin{remark}
When $R$ is classical and $N$ is finite locally free, we will see in Corollary~\ref{cor:fppf_cohomology} below that $\mathsf{S}_{(N,\psi)}$ is in fact a finite flat group scheme over $R$. The stacks associated with shifts $(N,\psi)[-r]$ can then be described in terms of this group scheme. When $R$ is Noetherian and $N$ is a finitely generated $R$-module, the classical truncation of $\mathsf{S}_{(N,\psi)}$ can be described non-canonically as follows: Choose a presentation $(F_1,\psi_1)\to (F_2,\psi_2)\to (N,\psi)\to 0$ with $F_1,F_2$ finite locally free over $R$, and take the kernel of the resulting map of finite flat group schemes $\mathsf{S}_{(F_2,\psi_2)}\to \mathsf{S}_{(F_1,\psi_2)}$.
\end{remark}

\begin{remark}
\label{rem:cartesian_SNpsi}
Suppose that we have a cofiber sequence
\[
(N',\psi')\to (N,\psi)\to (N'',\psi'')
\]
in $\Mod[\varphi,*]{R}$. Then we obtain a Cartesian diagram of prestacks
\[
\begin{diagram}
\mathsf{S}_{(N,\psi)}&\rTo&\mathsf{S}_{(N',\psi')}\\
\dTo&&\dTo^0\\
\mathsf{S}_{(N',\psi')}&\rTo&\mathsf{S}_{(N'',\psi'')[-1]}.
\end{diagram}
\]
\end{remark}

\subsection{Flat cohomology of height one group schemes}
\label{subsec:height_one_group_schemes}

We now present the proof of Bragg-Olsson with appropriate modifications to accommodate almost perfect complexes.

\begin{construction}
\label{const:height_one_group_schemes}
Suppose that $R$ is discrete and that $N$ is locally free of finite rank over $R$. Then the pair $(N,\psi)$ gives rise to a finite flat commutative $p$-group scheme $G(N,\psi)$ over $R$ of height $1$, meaning that its Frobenius endomorphism is trivial; see~\cite[Exp. VIIA]{MR2867621},~\cite[\S 2]{MR1235021}. Explicitly, the Cartier dual $G(N,\psi)^*$ is given as the kernel of the map
\[
\mathbf{V}(N^\vee) \xrightarrow{\mathbf{V}(\psi^\vee) - F}\mathbf{V}(\varphi^*N^\vee)\simeq R\otimes_{\varphi,R}\mathbf{V}(N^\vee),
\]
where $F$ is the relative Frobenius map for $\mathbf{V}(\varphi^*N^\vee)$ with respect to $R$. 
\end{construction}

We now have:
\begin{theorem}
\label{thm:fppf_cohomology}
For any $C\in \mathrm{CRing}_{R/}$, we have a canonical isomorphism
\[
R\Gamma_\varphi\left(C,(C\otimes_RN,\mathrm{id}\otimes\psi)\right)\xrightarrow{\simeq}R\Gamma_{\mathrm{fppf}}(\Spec C,G(N,\psi)).
\]
\end{theorem}

\begin{proof}
By considering the universal situation for pairs $(N,\psi)$, one reduces to the case where $R$ is a smooth $\Field_p$-algebra. The arrow in question is obtained via left Kan extension from an isomorphism of functors for smooth $R$-algebras $C$, where it is a classical theorem of Artin-Milne~\cite[Proposition 2.4]{artin_milne}; see Remark~\ref{rem:zprism_expl}. 

That this map is an isomorphism is a theorem of Bragg-Olsson~\cite[Theorem 4.8]{bragg2021representability}, who attribute the proof to Bhatt-Lurie.
\end{proof}

\begin{corollary}
\label{cor:fppf_cohomology}
\begin{enumerate}
   \item The stack $\mathsf{S}_{(N,\psi)}$ is canonically isomorphic to the group scheme $G(N,\psi)$.
   \item If $r\ge 0$, the stack $\mathsf{S}_{(N,\psi)[-r]}$ is isomorphic to the iterated classifying stack $B^rG(N,\psi)$.
   \item If $r\leq 0$, the $\mathsf{S}_{(N,\psi)[-r]}$ is derived affine over $R$ and is of the form $\Spec B_{-r}$, where the map $R\to B_{-r}$ has $(-r)$-connective cofiber.
\end{enumerate}

\end{corollary}
\begin{proof}
The first two assertions are clear from Theorem~\ref{thm:fppf_cohomology}. For the last assertion, note that, for each $r\leq 1$, applying Remark~\ref{rem:cartesian_SNpsi} to the cofiber sequence
\[
(N,\psi)[-r]\xrightarrow{\mathrm{id}} (N,\psi)[-r]\to (0,\mathrm{id})
\]
and using the obvious isomorphism $\mathsf{S}((0,\mathrm{id}))\simeq \Spec R$ gives us a Cartesian square of prestacks
\[
\begin{diagram}
\mathsf{S}_{(N,\psi)[-r]}&\rTo&\Spec R\\
\dTo&&\dTo_0\\
\Spec R&\rTo_0&\mathsf{S}_{(N,\psi)[-r+1]}.
\end{diagram}
\]
Therefore, induction on $r$ shows that $\mathsf{S}_{(N,\psi)[-r]}$ is derived affine and is the spectrum of $B_{-r}\defn R\otimes_{B_{-r+1}}R$. Furthermore, we can assume that the augmentation map $B_{-r+1}\to R$ has $(-r+1)$-connective fiber $I_{-r+1}$. Consideration of the fiber sequence
\[
I_{-r+1}\otimes_{B_{-r+1}}R\to R \to  R\otimes_{B_{-r+1}}R = B_{-r}
\]
now shows that $R\to B_{-r}$ has $(-r)$-connective cofiber, as desired.
\end{proof}

\begin{proof}
[Proof of Theorem~\ref{thm:repbility_height_one}]
Suppose that $N$ is $(-m)$-connective. By~\cite[Proposition 7.2.4.11(5)]{Lurie2017-oh} and its proof, we can assume that $(N,\psi)$ is given as a filtered colimit 
\[
(N,\psi)\simeq \colim_{j\in\Int_{\ge 1}} (N_j,\psi_j)
\]
where, for each $j\ge 1$, $Q_j = \hcoker(N_{j-1}\to N_j)[m-j]$ is finite locally free over $R$. Here, we set $N_0 \simeq 0$. If $N$ is perfect, we can assume that $(N_j,\psi_j)\simeq (N,\psi)$ for $j$ sufficiently large. We now have
\[
\mathsf{S}_{(N,\psi)}\simeq \varprojlim_{j}\mathsf{S}_{(N_j,\psi_j)}.
\]

By Remark~\ref{rem:cartesian_SNpsi}, for each $j\ge 1$, we have a Cartesian diagram
\[
\begin{diagram}
\mathsf{S}_{(N_j,\psi_j)}&\rTo&\mathsf{S}_{(N_{j-1},\psi_{j-1})}\\
\dTo&&\dTo^0\\
\mathsf{S}_{(N_{j-1},\psi_{j-1})}&\rTo&\mathsf{S}_{(Q_j,\psi'_j)[j-m-1]},
\end{diagram}
\]
where $\psi'_j:Q_j\to \varphi^*Q_j$ is the map induced from $\psi_j[m-j]$ and $\psi_{j-1}[m-j]$. Combined with Corollary~\ref{cor:fppf_cohomology}, this shows:
\begin{itemize}
   \item $\mathsf{S}_{(N_j,\psi_j)}$ is a finitely presented derived Artin $m$-stack for every $j\ge 1$;
   \item For $j\ge m$ and every map $\Spec C\to \mathsf{S}_{(N_{j-1},\psi_{j-1})}$ over $\Spec R$, the map 
    \[
   \mathsf{S}_{(N_j,\psi_j)}\times_{\mathsf{S}_{(N_{j-1},\psi_{j-1})}}\Spec C\to \Spec C
    \]
   is represented by a finitely presented derived affine scheme $\Spec B_j$ such that $C\to B_j$ has $(j-m)$-connective cofiber.
\end{itemize}
By the next lemma, this implies that $\mathsf{S}_{(N,\psi)}\to \mathsf{S}_{(N_{m-1},\psi_{m-1})}$ is relatively represented by an almost finitely presented derived scheme, and is hence itself an almost finitely presented derived Artin $m$-stack over $R$.

\begin{lemma}
Suppose that $B\in \mathrm{CRing}_{C/}$ is presented as a filtered colimit
\[
B\simeq \colim_{j\ge m}B_j
\]
in $\mathrm{CRing}_{C/}$ such that the following conditions hold for all $j\ge m$:
\begin{enumerate}
   \item $B_j$ is finitely presented over $C$;
   \item $\hcoker(B_j\to B_{j+1})$ is $(j-m)$-connective.
\end{enumerate}
Then $B$ is almost finitely presented over $C$.
\end{lemma}
\begin{proof}
For each $j\ge m$, the map $\tau_{\leq (j-m)}B_j \to \tau_{\leq (j-m)}B$ is an isomorphism, and the source is clearly finitely presented in the full subcategory of $(j-m)$-truncated objects in $\mathrm{CRing}_{C/}$.
\end{proof}

To finish, it remains to show that $\mathsf{S}_{(N,\psi)}$ is smooth if $N$ has Tor amplitude in $[1,\infty)$: But the above proof tells us that it is enough to know that $\mathsf{S}_{(N,\psi)[-r]}$ is smooth when $N$ is locally free and $r\ge 1$, and this is also clear from Corollary~\ref{cor:fppf_cohomology}.
\end{proof}

\section{Representability theorems for $1$-bounded stacks}
\label{sec:abstract}

In this section, we prove a representability theorem under somewhat general hypotheses. We also record some applications to stacks obtained from $F$-gauges.

\subsection{The map from the $F$-zip stack to the mod-$p$ syntomification}
\label{subsec:fzip_to_syn}

\subsubsection{}
We have a canonical map $\eta:R^{\Fzip}\to R^{\mathrm{syn}}\otimes\Field_p$ obtained from Example~\ref{ex:witt_frame_to_syntomification}. We can give a much more explicit description of $\eta$. Consider first the map
\begin{align}\label{eqn:u=0_piece_map}
\Aff^1/\Gm\times \Spec R\xrightarrow{x^{\mathcal{N}}_{\dR}}R^{\mathcal{N}}\otimes\Field_p\to R^{\mathrm{syn}}\otimes\Field_p.
\end{align}

Next, we will construct another map
\begin{align}\label{eqn:t=0_piece_map}
\Aff^1_+/\Gm\times \Spec R\xrightarrow{x^{\mathcal{N}}_{\mathrm{HT}}} R^{\mathcal{N}}\otimes\Field_p\to R^{\mathrm{syn}}\otimes\Field_p
\end{align}
whose restriction to the open substack $\Gm/\Gm\times \Spec R$ agrees with that of~\eqref{eqn:u=0_piece_map}, and thus yields a map $\tilde{\eta}:Y\times \Spec R\to R^{\mathrm{syn}}\otimes\Field_p$.

This is constructed as follows: Given a point $(L,u:C\to L)$ of $\Aff^1_+/\Gm\times \Spec R$ over $C\in \mathrm{CRing}_{R/}$, we have the Cartier-Witt divisor $W\xrightarrow{p}W$ along with the map of trivial generalized Cartier divisors $(\Ga^{\dR}\xrightarrow{0}\Ga^{\dR})\to (\Ga^{\dR}\otimes_CL\xrightarrow{0}\Ga^{\dR})$ induced by $u$. In the notation of Remark~\ref{rem:fil_cw_pushout}, the associated filtered Cartier-Witt divisor is simply $M(L,p,0,u)\to W$; the quotient $W/M(L,p,0,u)$ is also a quotient of $\Ga$, and this yields the structure map from $R$. The restriction to the open point corresponds to the natural structure map $R\to C\to W(C)/{}^{\mathbb{L}}p$, as desired.

\subsubsection{}
Consider now the two maps
\begin{align*}
B\Gm\times \Spec R\to \Aff^1/\Gm\times \Spec R\xrightarrow{x^{\mathcal{N}}_{\dR}}R^\mathcal{N}\otimes\Field_p;\\
B\Gm\times \Spec R\to \Aff^1_+/\Gm\times \Spec R\xrightarrow{x^{\mathcal{N}}_{\mathrm{HT}}}R^\mathcal{N}\otimes\Field_p.
\end{align*}
The composition with the map to $\Field_p^{\mathcal{N}}$ is the same for both of these: It attaches to every line bundle $L$ the filtered Cartier-Witt divisor $M(L,p,0,0)\to W$; however the structure map $R\to (W/M(L,p,0,0))(C)$ associated with the second composition differs from that associated with the first via pre-composition with $\varphi:R\to R$. This shows that the map $\tilde{\eta}$ descends to the desired map $\eta:R^{\Fzip}\to R^{\mathrm{syn}}\otimes\Field_p$.

\subsection{$F$-gauges and $F$-zips}
\label{subsec:fgauges_fzips}

\subsubsection{}
Suppose that $R\in \mathrm{CRing}_{(\Int/p^m\Int)/}$. An \defnword{$F$-gauge over $R$ of level $n$} is a quasi-coherent sheaf $\mathcal{M}$ over $R^{\mathrm{syn}}\otimes\Int/p^n\Int$. Via the map $x^{\mathcal{N}}_{\dR}$ of~\S\ref{subsec:can_sections}, we obtain a symmetric monoidal functor of symmetric monoidal $\infty$-categories
\[
\mathrm{QCoh}(R^{\mathrm{syn}}\otimes\Int/p^n\Int)\to \mathrm{QCoh}(R^{\mathcal{N}}\otimes\Int/p^n\Int)\xrightarrow{(x^{\mathcal{N}}_{\dR})^*}\mathrm{QCoh}(\Aff^1/\Gm\times \Spec(R/{}^{\mathbb{L}}p^n)),
\]
which can be viewed as a functor $\mathcal{M}\to \Fil^\bullet_{\mathrm{Hdg}}M_n$ from $F$-gauges of level $n$ to filtered modules over $R/{}^{\mathbb{L}}p^n$. 

The \defnword{Hodge-Tate weights} of an $F$-gauge $\mathcal{M}$ are the integers $i$ such that $\gr^{-i}_{\mathrm{Hdg}}M_n\neq 0$. We will say that $\mathcal{M}$ is \defnword{$1$-bounded} if its Hodge-Tate weights are bounded above by $1$.

As an immediate consequence of Theorem~\ref{thm:semiperf_crys}, we find:
\begin{proposition}
\label{prop:semiperfect_f-gauges}
Suppose that $R$ is semiperfectoid. Then there is a canonical symmetric monoidal equivalence of stable $\infty$-categories
\[
\mathrm{QCoh}(R^{\mathrm{syn}}\otimes\Int/p^n\Int)\xrightarrow{\simeq}\underline{\Prism}_R\mathsf{-gauge}_n,
\] 
where the right hand side is the $\infty$-category of $\underline{\Prism}_R$-gauges of level $n$ from~\S\ref{subsec:A_gauges}. This is compatible with the definition of Hodge-Tate weights on both sides.\footnote{See Definition~\ref{defn:A-gauge_HT_weights} for the one on the right hand side.}
\end{proposition}

From this we obtain:
\begin{proposition}
\label{prop:f_gauges_descent}
Suppose that $R\to R_\infty$ is as in Corollary~\ref{cor:semiperf_qsynt}. Then we have an equivalence of symmetric monoidal stable $\infty$-categories compatible with Hodge-Tate weights
\[
\mathrm{QCoh}(R^{\mathrm{syn}}\otimes\Int/p^n\Int)\xrightarrow{\simeq}\mathrm{Tot}\left(\underline{\Prism}_{R_\infty^{\otimes(\bullet +1)}}\mathsf{-gauge}_n\right).
\]
\end{proposition}

\begin{construction}
\label{const:syn_cohom_f-gauge}
Associated with $\mathcal{M}$ are the $\Mod{\Int/p^n\Int}$-valued prestacks over $R$
\begin{align*}
R\Gamma^*_{\mathrm{syn}}(\mathcal{M}):\mathrm{CRing}_{R/}&\to \Mod{\Int/p^n\Int}\\
C&\mapsto \mathrm{RHom}_{\mathrm{QCoh}(C^{\mathrm{syn}}\otimes\Int/p^n\Int)}(\mathcal{M}\vert_{C^\mathrm{syn}},\Reg{C^{\mathrm{syn}}\otimes\Int/p^n\Int}).
\end{align*}
\begin{align*}
R\Gamma_{\mathrm{syn}}(\mathcal{M}):\mathrm{CRing}_{R/}&\to \Mod{\Int/p^n\Int}\\
C&\mapsto R\Gamma_{\mathrm{syn}}(\Spec C,\mathcal{M}\vert_{C^\mathrm{syn}}) \defn R\Gamma(C^{\mathrm{syn}},\mathcal{M}\vert_{C^\mathrm{syn}})
\end{align*}

We will set 
\[
\Gamma_{\mathrm{syn}}^*(\mathcal{M}) = \tau^{\leq 0}R\Gamma^*_{\mathrm{syn}}(\mathcal{M})\;;\; \Gamma_{\mathrm{syn}}(\mathcal{M}) = \tau^{\leq 0}R\Gamma_{\mathrm{syn}}(\mathcal{M}).
\]

Suppose now that $R$ is an $\Field_p$-algebra. Then pullback along $\eta:R^{\Fzip}\to R^{\mathrm{syn}}\otimes\Field_p$. gives a symmetric monoidal functor
\[
\eta^*:\mathrm{QCoh}(R^{\mathrm{syn}}\otimes\Field_p)\to \mathrm{QCoh}(R^{\Fzip}).
\]

Note that there are natural maps
\[
R\Gamma^*_{\mathrm{syn}}(\mathcal{M})\to R\Gamma^*_{\Fzip}(\eta^* \mathcal{M})\;;\;R\Gamma_{\mathrm{syn}}(\mathcal{M})\to R\Gamma_{\Fzip}(\eta^* \mathcal{M})
\]
for any $F$-gauge $\mathcal{M}$ over $R$ of level $1$, where the right hand side of each map is as in Construction~\ref{const:fzip_cohomology}.
\end{construction}

\begin{remark}
\label{rem:syn_cohom-perfect_f-gauges}
If $\mathcal{M}$ is perfect with dual $F$-gauge $\mathcal{M}^\vee$, then we have a canonical isomorphism
\[
R\Gamma^*_{\mathrm{syn}}(\mathcal{M})\xrightarrow{\simeq}R\Gamma_{\mathrm{syn}}(\mathcal{M}^\vee).
\]
\end{remark}

\subsection{Prestacks of sections}
\label{subsec:prestacks_of_sections}
We will now use the terminology and results of~\S\ref{subsec:1-bounded_fixed_points}.

\subsubsection{}
Suppose that $R\in \mathrm{CRing}_{(\Int/p^m\Int)/}$. Note that $R^{\mathrm{syn}}\otimes\Int/p^n\Int$ is a pointed graded stack equipped with a canonical map of graded stacks
\[
B\Gm\times \Spec(R/{}^{\mathbb{L}}p^n)\to R^{\mathrm{syn}}\otimes\Int/p^n\Int.
\]
arising via the map  
\[
x^{\mathcal{N}}_{\dR,R}\otimes\Int/p^n\Int:\Aff^1/\Gm\times  \Spec(R/{}^{\mathbb{L}}p^n)\to R^{\mathcal{N}}\otimes\Int/p^n\Int.
\]

Suppose that we have a $1$-bounded stack $\mathcal{X}= (\mathcal{X}^\preoneb,X^0)\to R^{\mathrm{syn}}\otimes\Int/p^n\Int$.  Define a prestack $\Gamma_{\mathrm{syn}}(\mathcal{X})$ that associates with each $C\in \mathrm{CRing}_{R/}$ the space
\[
\Gamma_{\mathrm{syn}}(\mathcal{X})(C) = \Map_{R^{\mathrm{syn}}\otimes\Int/p^n\Int}(C^{\mathrm{syn}}\otimes\Int/p^n\Int,\mathcal{X}).
\]
Unpacking definitions, one sees that we have
\[
\Gamma_{\mathrm{syn}}(\mathcal{X})(C) = \Map_{R^{\mathrm{syn}}\otimes\Int/p^n\Int}(C^{\mathrm{syn}}\otimes\Int/p^n\Int,\mathcal{X}^\preoneb)\times_{X^{\preoneb,0}(C/{}^{\mathbb{L}}p^n)}X^0(C/{}^{\mathbb{L}}p^n),
\]
where $X^{\preoneb,0}\to \Spec R/{}^{\mathbb{L}}p^n$ is the fixed point locus of $\mathcal{X}^\preoneb$.

\subsubsection{}
This has a slightly more explicit description. First, define prestacks $\Gamma_{\mathcal{N}}(\mathcal{X})$ and $\Gamma_{\Prism}(\mathcal{X})$ over $R$ by
\begin{align*}
\Gamma_{\mathcal{N}}(\mathcal{X})(C) &= \Map_{/R^{\mathrm{syn}}\otimes\Int/p^n\Int}(C^{\mathcal{N}}\otimes\Int/p^n\Int,\mathcal{X});\\
\Gamma_{\Prism}(\mathcal{X})(C) &= \Map_{/R^{\mathrm{syn}}\otimes\Int/p^n\Int}(C^{\Prism}\otimes\Int/p^n\Int,\mathcal{X}^\preoneb).
\end{align*}

Restriction along the de Rham and Hodge-Tate immersions yields maps
\[
j_{\dR}^*,j_{\mathrm{HT}}^*:\Gamma_{\mathcal{N}}(\mathcal{X})\to \Gamma_{\Prism}(\mathcal{X}),
\]
and we now have
\begin{align}\label{eqn:sections_first_desc}
\begin{diagram}
\Gamma_{\mathrm{syn}}(\mathcal{X}) &\rTo^{\simeq}& \mathrm{eq}\biggl(\Gamma_{\mathcal{N}}(\mathcal{X})&\pile{\rTo^{j_{\mathrm{dR}}^*}\\ \rTo_{j_{\mathrm{HT}}^*}}&\Gamma_{\Prism}(\mathcal{X}) \biggr).
\end{diagram}
\end{align}

\subsection{Some auxiliary stacks}
\label{subsec:aux_stacks}

Suppose that $\mathcal{X}= (\mathcal{X}^\preoneb,X^0)$ is a $1$-bounded stack over $R^{\mathrm{syn}}\otimes\Int/p^n\Int$. 

\subsubsection{}
Let $X^{(n)},X^{-,(n)},X^{-,0}$ be the prestacks over $R$ given by
\begin{align*}
X^{(n)}(C) &=\Map_{/R^{\mathrm{syn}}\otimes\Int/p^n\Int}(\Spec C/{}^{\mathbb{L}}p^n,\mathcal{X});\\
X^{-,(n)}(C) &=\Map_{/R^{\mathrm{syn}}\otimes\Int/p^n\Int}(\Aff^1/\Gm\times\Spec C/{}^{\mathbb{L}}p^n,\mathcal{X});\\
X^{0,(n)}(C) &=\Map_{/R^{\mathrm{syn}}\otimes\Int/p^n\Int}(B\Gm\times\Spec C/{}^{\mathbb{L}}p^n,\mathcal{X}).
\end{align*}
The last two stacks are simply the Weil restriction from $R/{}^{\mathbb{L}}p^n$ to $R$ of the attractor and fixed point locus of $(\mathcal{X}^\preoneb,X^0)$, respectively.

\subsubsection{}
If $R$ is in addition an $\Field_p$-algebra, then also define
\[
X^{+,(n)}(C) =\Map_{/R^{\mathrm{syn}}\otimes\Int/p^n\Int}(\Aff^1_+/\Gm\times\Spec C/{}^{\mathbb{L}}p^n,\mathcal{X}).
\]
This is the Weil restriction of the repeller associated with $\mathcal{X}$ and the map $\Aff^1_+/\Gm\times\Spec R/{}^{\mathbb{L}}p^n\to R^{\mathcal{N}}\otimes\Int/p^n\Int$.

\subsubsection{}
Suppose that $n=m$, so that $R$ is a $\Int/p^n\Int$-algebra. In this case, for any $R$-algebra $C$, we have a section $C/{}^{\mathbb{L}}p^n\to C$, and so we can also define prestacks $X,X^-,X^0,X^+$ over $R$ (the last only when $n=m=1$) by:
\begin{align*}
X(C) &=\Map_{/R^{\mathrm{syn}}\otimes\Int/p^n\Int}(\Spec C,\mathcal{X}^\preoneb);\\
X^{-}(C) &=\Map_{/R^{\mathrm{syn}}\otimes\Int/p^n\Int}(\Aff^1/\Gm\times\Spec C,\mathcal{X});\\
X^{0}(C) &=\Map_{/R^{\mathrm{syn}}\otimes\Int/p^n\Int}(B\Gm\times\Spec C,\mathcal{X});\\
X^{+}(C) &=\Map_{/R^{\mathrm{syn}}\otimes\Field_p}(\Aff^1_+/\Gm\times\Spec C,\mathcal{X}).
\end{align*}
The last three stacks are simply the attractor, fixed point locus and repeller for $\mathcal{X}$ base-changed from $C/{}^{\mathbb{L}}p^n$ to $C$. 

Note a subtlety about base-points: The base-point $B\Gm\times\Spec C\to \Aff^1_+/\Gm\times\Spec C$ is a lift of the Frobenius endomorphism of $B\Gm\times \Spec C$. In particular, we actually have
\[
X^+(C) = \Map_{/R^{\mathrm{syn}}\otimes\Field_p}(\Aff^1_+/\Gm\times\Spec C,\mathcal{X}^\preoneb)\times_{{}^\varphi X^{\preoneb,0}(C)}{}^\varphi X^0(C).
\]
Here, ${}^\varphi X^0$ is given by ${}^{\varphi}X^0 = X^0\times_{\Spec R,\varphi}\Spec R$.

\begin{proposition}
\label{prop:X_attractor_repble}
Suppose that $\mathcal{X}\to R^{\mathrm{syn}}\otimes\Int/p^n\Int$ is a $1$-bounded stack and suppose that $\pi_0(R)$ is a $G$-ring. Suppose also that $\mathcal{X}^\preoneb$ has quasi-affine diagonal over $ R^{\mathrm{syn}}\otimes\Int/p^n\Int$. Then $X^{(n)}$, $X^{-,(n)}$, $X^{0,(n)}$ and $X^{+,(n)}$ (if $m=1$)---and, if $n=m$, the prestacks $X$, $X^-$, $X^0$ and $X^+$ (if $n=m=1$)---are locally almost finitely presented derived Artin stacks over $R$.
\end{proposition} 
\begin{proof}
By Propositions~\ref{prop:weil_restriction} and~\ref{prop:t-integrable_repble} it follows that $X^{\preoneb,-,(n)}$, $X^{\preoneb,+,(n)}$, $X^{\preoneb,0,n}$ and $X^{0,(n)}$ are locally finitely presented derived Artin stacks. Since we have
\[
X^{\pm,(n)} = X^{\preoneb,\pm,(n)}\times_{X^{\preoneb,0,(n)}}X^{0,(n)},
\]
we see that these are also locally finitely presented derived Artin stacks.

The argument for $X^{(n)}$ is even simpler, since it is a mod-$p^n$ Weil restriction of a derived Artin stack by definition.

The argument for the remaining four prestacks (when $n=m$) is the same, but doesn't involve Weil restrictions.
\end{proof}

\begin{remark}
\label{rem:X_attractor_repble}
Suppose that $\mathcal{X}^\preoneb$ is graded integrable (this is in particular true if it has quasi-affine diagonal by Proposition~\ref{prop:1bounded_graded_deformation}). When $n=m=1$, for any $R$-algebra $C$, the map
\[
X^+(C)\to \Map_{/R^{\mathrm{syn}}\otimes\Field_p}((\Aff^1_+/\Gm)_{(u^2=0)}\times \Spec C,\mathcal{X}^\preoneb)\times_{{}^{\varphi}\mathcal{X}^{\preoneb,0}(C)}{}^\varphi X^0(C)
\]
is an equivalence. Suppose now that $\mathbb{L}^0_{\mathcal{X},\bullet}$ is the graded almost perfect complex over $X^0$ obtained via pullback from the relative cotangent complex of $\mathcal{X}^\preoneb$. Then, via graded deformation theory, we obtain a Cartesian diagram of stacks
\[
\begin{diagram}
X^+&\rTo& X^0\\
\dTo&&\dTo_0\\
{}^\varphi X^0&\rTo&\mathbf{V}(\mathbb{L}^0_{\mathcal{X},-1}[-1]).
\end{diagram}
\]
Therefore, as soon as $X^0$ is representable by an almost locally finitely presented derived Artin stack, it follows $X^+$ is also an almost locally finitely presented derived Artin stack. Note also that $X^+$ is quasicompact as soon as $X^0$ is so.
\end{remark}   

\subsection{D\'evissage to the $F$-zip stack}
\label{subsec:devissage_to_Fzips}

Suppose that $n=1$, and that we have a $1$-bounded stack $\mathcal{X}\to R^{\mathrm{syn}}\otimes\Field_p$ with $R \in \mathrm{CRing}_{\Field_p/}$.

\begin{construction}
\label{const:fzip_sections_for_devissage}
Define a prestack $\Gamma_{\Fzip}(\mathcal{X})$ over $R$ by
\begin{align*}
\Gamma_{\Fzip}(\mathcal{X}):\mathrm{CRing}_{R/}&\to \mathrm{Spc}\\
C&\mapsto \Map_{/R^{\mathrm{syn}}\otimes\Field_p}(C^{\Fzip},\mathcal{X}).
\end{align*}
The argument from Lemma~\ref{lem:gamma_fzip_description} shows that we have a presentation
\begin{align}\label{eqn:sections_fzips}
\begin{diagram}
\Gamma_{\Fzip}(\mathcal{X}) &\rTo^{\simeq}& \mathrm{eq}\biggl(X^-\times_XX^+&\pile{\rTo^{{}^{\varphi}\lambda_-^*\circ\mathrm{pr}_1}\\ \rTo_{\lambda_+^*\circ\mathrm{pr}_2}}&{}^\varphi X^0\biggr).
\end{diagram}
\end{align}
\end{construction}

\begin{lemma}
\label{lem:FZip_repblity}
Suppose that $\mathcal{X}^\preoneb$ is a relative $r$-stack with one of the following properties: 
\begin{itemize}
   \item $\mathcal{X}^\preoneb\to R^{\mathrm{syn}}\otimes\Field_p$ has quasi-affine diagonal and $\pi_0(R)$ is a $G$-ring;

   \item The stacks $X^-,X^+,X^0$ are almost locally finitely presented derived Artin $r$-stacks over $R$.
\end{itemize}
Then $\Gamma_{\Fzip}(\mathcal{X})$ is an almost locally finitely presented derived Artin $r$-stack over $R$.
\end{lemma}
\begin{proof}
Proposition~\ref{prop:X_attractor_repble} shows that the first condition implies the second. We now conclude using the presentation~\eqref{eqn:sections_fzips}.
\end{proof}

\begin{remark}
Let $\mathbb{L}_{\mathcal{X}}$ be the relative cotangent complex for $\mathcal{X}^\preoneb$ over $R^{\mathrm{syn}}\otimes\Int/p^n\Int$. For every section $x\in \Gamma_{\Fzip}(\mathcal{X})(C)$, the almost perfect $F$-zip $\bm{L}(\mathcal{X})_x = x^*\mathbb{L}_{\mathcal{X}}\in \mathrm{APerf}(C^{\Fzip})$ has Hodge-Tate weights bounded below by $-1$. We will view this as giving us an almost perfect $F$-zip $\bm{L}(\mathcal{X})$ over $\Gamma_{\mathrm{\Fzip}}(\mathcal{X})$. Since $\mathcal{X}$ is fibered in derived Artin $r$-stacks, this $F$-zip is $(-r)$-connective. 
\end{remark}

\begin{construction}
\label{const:S1_mathcal_X}
Given a section $x$ of $\mathcal{X}$ over $C^{\Fzip}$ (equivalently, of $\Gamma_{\Fzip}(\mathcal{X})$ over $C$), we obtain a $(-r)$-connective almost perfect $F$-zip $x^*\mathbb{L}_{\mathcal{X}}$, which corresponds to a tuple $(\Fil^\bullet_{\mathrm{Hdg}}L_x,\Fil^{\mathrm{conj}}_\bullet L_x,\eta,\alpha)$. We have a canonical map
\begin{align}
\label{eqn:phi_module_1bdd_definition}
\psi_x:\Fil^1_{\mathrm{Hdg}}L_x \to L_x \to \gr^{\mathrm{conj}}_1L_x \simeq \varphi^*\Fil^1_{\mathrm{Hdg}}L_x
\end{align}
By Theorem~\ref{thm:repbility_height_one}, for each $i\in \Int$, we now obtain a relative locally almost finitely presented derived Artin $(r+i)$-stack\footnote{If $r+i<0$, this will just be a relative derived scheme.} $\mathsf{S}_i(\mathcal{X})\to \Gamma_{\Fzip}(\mathcal{X})$, whose fiber over $x$ is given by $\mathsf{S}_{(\Fil^1_{\mathrm{Hdg}}L_x,\psi_x)[-i]}$.
\end{construction}

\begin{theorem}
\label{thm:devissage_to_fzips}
Suppose that $\mathcal{X}^\preoneb\to R^{\mathrm{syn}}\otimes\Field_p$ satisfies one of the following conditions:
\begin{itemize}
   \item It is strongly integrable\footnote{See Definition~\ref{defn:strongly_integrable}};
   \item It has quasi-affine diagonal.
\end{itemize}
Then:
\begin{enumerate}
   \item  There is a canonical Cartesian diagram of prestacks over $R$:
   \[
    \begin{diagram}
    \Gamma_{\mathrm{syn}}(\mathcal{X})&\rTo&\Gamma_{\Fzip}(\mathcal{X})\\
    \dTo&&\dTo_0\\
    \Gamma_{\Fzip}(\mathcal{X})&\rTo&\mathsf{S}_1(\mathcal{X}),
    \end{diagram}
   \]
   where the right vertical map is the zero section.
   \item Suppose that $\mathcal{X}^\preoneb$ is \emph{smooth} over $R^{\mathrm{syn}}\otimes\Int/p^n\Int$, and that its relative tangent complex is \emph{$1$-connective}. Then $\Gamma_{\mathrm{syn}}(\mathcal{X})\to \Gamma_{\Fzip}(\mathcal{X})$ is a quasisyntomic torsor under $\mathsf{S}_0(\mathcal{X})$.
\item  In particular, if $\mathcal{X}$ satisfies the hypotheses in Lemma~\ref{lem:FZip_repblity}, then $\Gamma_{\mathrm{syn}}(\mathcal{X})$ is represented by a locally almost finitely presented derived Artin $r$-stack over $R$.
\end{enumerate}
\end{theorem}

\begin{proof}
The last assertion is immediate from Lemma~\ref{lem:FZip_repblity}.

For the first two, begin by noting that the condition of having quasi-affine diagonal implies the other one by Remark~\ref{rem:use_of_filtered_completeness}. Therefore, we can assume that $\mathcal{X}$ is strongly integrable.

By quasisyntomic descent, we can reduce to checking the existence of the diagram in (1) for semiperfect inputs. Here, the result follows from Theorem~\ref{thm:semiperf_crys} and Proposition~\ref{prop:abstract_devissage_to_fzips}. 

Assertion (2) follows analogously from Corollary~\ref{cor:abstract_fzip_smooth_torsor}.

\end{proof}

\begin{corollary}
\label{cor:1_bounded_level_1_repble}
Suppose that $\mathcal{M}$ is an almost perfect $(-r)$-connective $F$-gauge of level $1$ over $R$ with Hodge-Tate weights bounded below by $-1$. If $\mathcal{M}$ is perfect with Tor amplitude in $[s,r]$, write $\mathcal{M}^\vee$ for the dual $F$-gauge: this has Tor amplitude in $[-r,-s]$ and has Hodge-Tate weights bounded above by $1$.
\begin{enumerate}
   \item The prestack $\Gamma^*_{\mathrm{syn}}(\mathcal{M})$ over $R$ is represented by an almost finitely presented derived Artin $r$-stack, which is in fact finitely presented if $\mathcal{M}$ is perfect.
   \item If $\mathcal{M}$ is perfect and $s\ge 1$, then $\Gamma_{\mathrm{syn}}(\mathcal{M}^\vee)\simeq \Gamma_{\mathrm{syn}}^*(\mathcal{M})$ is a smooth, faithfully flat derived Artin stack over $R$.
   \item If $r\leq 0$, then $\Gamma_{\mathrm{syn}}(\mathcal{M})$ is a derived affine scheme over $R$.
\end{enumerate}
\end{corollary}
\begin{proof}
As we saw in \S~\ref{subsec:vector stacks}, the vector stack $\mathcal{X}=\mathbf{V}(\mathcal{M})\to R^{\mathrm{syn}}\otimes\Field_p$ is an almost finitely presented $r$-stack, and Example~\ref{ex:1_bounded_perfect_complexes_stack} shows that it can be extended to a $1$-bounded stack---denoted by the same symbol---by bringing along the entire fixed point locus.

Moreover, the associated stacks $X^{-,(n)}$, $X^{+,(n)}$ and $X^{0,(n)}$ admit the following explicit description: Associated with $\mathcal{M}$ is an $F$-zip giving in particular the pair $(\Fil^\bullet_{\mathrm{Hdg}}M,\Fil^{\mathrm{conj}}_\bullet M)$; then $X^{-,(n)}$ (resp. $X^{+,(n)}$, $X^{0,(n)}$) is the mod-$p^n$ Weil restriction of the vector stack associated with $M/\Fil^1_{\mathrm{Hdg}}M$ (resp. $M/\Fil^{\mathrm{conj}}_{-1}M$, $\gr^0_{\mathrm{Hdg}}M$). In particular, all three stacks are representable, and so, if $\bm{M}$ is the underlying $F$-zip, then the discussion in Construction~\ref{const:fzip_cohomology} shows that $\Gamma^*_{\Fzip}(\bm{M})$ is representable by a locally finitely presented derived Artin $r$-stack.

By Theorem~\ref{thm:devissage_to_fzips} and Remark~\ref{rem:use_of_filtered_completeness}, we have a Cartesian diagram
\begin{align}\label{eqn:gamma_syn_to_gamma_fzip_cartesian}
\begin{diagram}
\Gamma^*_{\mathrm{syn}}(\mathcal{M})&\rTo&\Gamma^*_{\Fzip}(\bm{M})\\
\dTo&&\dTo\\
\Gamma^*_{\Fzip}(\bm{M})&\rTo&\mathsf{S}_{(N,\psi)[-1]}
\end{diagram}
\end{align}
where $\psi:N\to \varphi^*N$ is a map of almost perfect complexes over $\Gamma_{\Fzip}(\mathsf{M})$ with $N = \Fil^1_{\mathrm{Hdg}}M$.

Therefore, Theorem~\ref{thm:repbility_height_one} shows that $\Gamma_{\mathrm{syn}}(\mathbf{V}(\mathcal{M})) = \Gamma^*_{\mathrm{syn}}(\mathcal{M})$ is represented by a locally finitely presented derived Artin $r$-stack. 

For (2), first note that, under these hypotheses, $\Gamma^*_{\Fzip}(\mathcal{M})$ is smooth over $R$: This is because we can rewrite the formula~\eqref{eqn:section_fzips_perf_comp} as
\[
\Gamma^*_{\Fzip}(\bm{M})(R)\simeq R\Gamma_{\Fzip}(\Spec R,\bm{M}^\vee) \simeq \hcoker(\gr^{\mathrm{conj}}_0M^\vee[-1]\to \Fil^{\mathrm{conj}}_0M^\vee\times_{M^\vee} \Fil^0_{\mathrm{Hdg}}M^\vee),
\]
where both source and target on the right are values of smooth vector stacks, faithfully flat over $R$. The proof is now completed by the second assertion from Theorem~\ref{thm:repbility_height_one}.

We now come to (3): Under our hypotheses, $\bm{M}$ is a connective perfect $F$-zip, and so~\eqref{eqn:section_fzips_perf_comp} shows that $\Gamma^*_{\Fzip}(\bm{M})$ is the equalizer of a pair of maps between derived affine schemes over $R$, and is hence itself derived affine. The assertion now follows from~\eqref{eqn:gamma_syn_to_gamma_fzip_cartesian} and (3) of Corollary~\ref{cor:fppf_cohomology}.
\end{proof}

\subsection{Bootstrapping from characteristic $p$: coefficients}
\label{subsec:1-bounded_stacks_bootstrapping_coefficients}

Continue to assume that $R$ is an $\Field_p$-algebra. Suppose now that we have a $1$-bounded stack $\mathcal{X}\to R^{\mathrm{syn}}\otimes\Int/p^n\Int$. For $m\leq n$, write $\mathcal{X}_m$ for its restriction over $R^{\mathrm{syn}}\otimes\Int/p^m\Int$. 

Over $\Gamma_{\mathrm{syn}}(\mathcal{X}_1)$, we have an almost perfect $F$-gauge $\mathcal{L}_1(\mathcal{X})$ of level $1$ and Hodge-Tate weights bounded below by $-1$: this associates with every $x\in \Gamma_{\mathrm{syn}}(\mathcal{X}_1)(C)$ the pullback to $C^{\mathrm{syn}}\otimes\Field_p$ of the cotangent complex of $\mathcal{X}^\preoneb_1$ over $R^{\mathrm{syn}}\otimes\Field_p$. In turn, for any $i\in\Int$, this gives us via Corollary~\ref{cor:1_bounded_level_1_repble} a $\Mod[\mathrm{cn}]{\Field_p}$-valued relative locally finitely presented derived Artin stack $\Gamma^*(\mathcal{L}_1(\mathcal{X})[i])\to \Gamma_{\mathrm{syn}}(\mathcal{X}_1)$ associating with each $x$ the derived Artin stack
\[
\Gamma^*(\mathcal{L}_1(\mathcal{X})[i])_x = \Gamma^*_{\mathrm{syn}}(\mathcal{L}_1(\mathcal{X})_x[i])
\]
over $C$.

\begin{proposition}
\label{prop:bootstrapping_coeffs}
There is a canonical Cartesian diagram of prestacks
\[
\begin{diagram}
\Gamma_{\mathrm{syn}}(\mathcal{X}_{m+1})&\rTo&\Gamma_{\mathrm{syn}}(\mathcal{X}_m)\\
\dTo&&\dTo_0\\
\Gamma_{\mathrm{syn}}(\mathcal{X}_m)&\rTo&\Gamma^*(\mathcal{L}_1(\mathcal{X})[-1])\times_{\Gamma_{\mathrm{syn}}(\mathcal{X}_1)}\Gamma_{\mathrm{syn}}(\mathcal{X}_m),
\end{diagram}  
\]
In particular, if $\mathcal{X}_1$ satisfies the hypotheses of Theorem~\ref{thm:devissage_to_fzips}, then $\Gamma_{\mathrm{syn}}(\mathcal{X})$ is represented by a locally almost finitely presented derived Artin $r$-stack over $R$.
\end{proposition}
\begin{proof}
The last assertion follows from Theorem~\ref{thm:devissage_to_fzips} and Corollary~\ref{cor:1_bounded_level_1_repble}.

The existence of the Cartesian diagram of prestacks is an application of deformation theory. Let $\mathcal{L}_1^{\mathcal{N}}(\mathcal{X})$ (resp. $\mathcal{L}_1^{\Prism}(\mathcal{X})$) be the almost perfect complex over the Nygaard filtered prismatization of $\Gamma_{\mathcal{N}}(\mathcal{X}_1)$ (resp. over the prismatization of $\Gamma_{\Prism}(\mathcal{X}_1)$) obtained similarly to $\mathcal{L}_1(\mathcal{X})$ by pulling back the relative cotangent complex of $\mathcal{X}^\preoneb$ along the tautological map.

From the first relative vector stack, we obtain a prestack over $\Gamma_{\mathcal{N}}(\mathcal{X}_1)$ given on pairs $(C,x)$ with $x\in \Gamma_{\mathcal{N}}(\mathcal{X}_1)(C)$ by
\begin{align*}
\Gamma^*_{\mathcal{N}}(\mathcal{L}_1(\mathcal{X})[-1]):(C,x)&\mapsto \Map_{\mathrm{QCoh}(C^{\mathcal{N}}\otimes\Field_p)}(\mathcal{L}_1^{\mathcal{N}}(\mathcal{X})_x,\Reg{C^{\mathcal{N}}\otimes\Field_p}[1]).
\end{align*}
Similarly, over $\Gamma_{\Prism}(\mathcal{X}_1)$, we obtain a prestack given on pairs $(C,x)$ with $x\in \Gamma_{\Prism}(\mathcal{X}_1)(C)$ by
\begin{align*}
\Gamma^*_{\Prism}(\mathcal{L}_1(\mathcal{X})[-1]):(C,x)&\mapsto \Map_{\mathrm{QCoh}(C^{\Prism}\otimes\Field_p)}(\mathcal{L}_1^{\Prism}(\mathcal{X})_x,\Reg{C^{\Prism}\otimes\Field_p}[1]).
\end{align*}

Suppose that $C$ is semiperfect to fix ideas: we can reduce to considering only such inputs by quasi-syntomic descent.

Let us look at the fibers of the map
\[
\Gamma_{\mathcal{N}}(\mathcal{X}_{m+1})(C)\to  \Gamma_{\mathcal{N}}(\mathcal{X}_m)(C).
\]
By Theorem~\ref{thm:semiperf_crys}, $C^{\mathcal{N}}\otimes\Int/p^r\Int$ is the Rees stack $\Rees(\Fil^\bullet_{\mathcal{N}}\Prism_C)\otimes\Int/p^r\Int$. Then by the discussion in~\S\ref{subsec:filtered_square-zero} we have a canonical Cartesian square
\[
\begin{diagram}
\Gamma_{\mathcal{N}}(\mathcal{X}_{m+1})(C)&\rTo&\Gamma_{\mathcal{N}}(\mathcal{X}_m)(C)\\
\dTo&&\dTo_0\\
\Gamma_{\mathcal{N}}(\mathcal{X}_m)(C)&\rTo&\Map(\mathcal{R}(\Fil^\bullet_{\mathcal{N}}\Prism_C/{}^{\mathbb{L}}p^m \oplus \Fil^\bullet_{\mathcal{N}}\overline{\Prism}_C[1]),\mathcal{X}^\preoneb)\times_{X^{\preoneb,0,(1)}(C)}X^{0,(1)}(C).
\end{diagram}  
\]

Moreover, the prestack over $R$ given by
\[
C\mapsto\Map(\mathcal{R}(\Fil^\bullet_{\mathcal{N}}\Prism_C/{}^{\mathbb{L}}p^m \oplus \Fil^\bullet_{\mathcal{N}}\overline{\Prism}_C[1]),\mathcal{X}^\preoneb)\times_{X^{\preoneb,0,(1)}(C)}X^{0,(1)}(C)
\]
is canonically isomorphic over $\Gamma_{\mathcal{N}}(\mathcal{X}_m)$ to the pullback of the stack $\Gamma^*_{\mathcal{N}}(\mathcal{L}_1(\mathcal{X})[-1])$, showing that we have a Cartesian diagram of prestacks over $\Gamma_{\mathcal{N}}(\mathcal{X}_m)$:
\[
\begin{diagram}
\Gamma_{\mathcal{N}}(\mathcal{X}_{m+1})&\rTo&\Gamma_{\mathcal{N}}(\mathcal{X}_m)\\
\dTo&&\dTo_0\\
\Gamma_{\mathcal{N}}(\mathcal{X}_m)&\rTo&\Gamma^*_{\mathcal{N}}(\mathcal{L}_1(\mathcal{X})[-1])\times_{\Gamma_{\mathcal{N}}(\mathcal{X}_1)}\Gamma_{\mathcal{N}}(\mathcal{X}_m)
\end{diagram}  
\]

There exists an analogous Cartesian diagram with $\mathcal{N}$ replaced with $\Prism$. Combining these two diagrams with~\eqref{eqn:sections_first_desc} now proves the Proposition.
\end{proof}

\begin{remark}
\label{rem:bootstrapping_coeffs_no_need_for_char_p}
Note that the establishment of the Cartesian diagrams in the proposition above did not use the hypothesis that $R$ is an $\Field_p$-algebra: One simply has to replace the word `semiperfect' by `semiperfectoid' in the proof.
\end{remark}

Suppose that $\mathcal{X}^\preoneb$ is as in (2) of Theorem~\ref{thm:devissage_to_fzips}. Then we have the following more streamlined assertion, which can be deduced from Proposition~\ref{prop:bootstrapping_coeffs} and Corollary~\ref{cor:1_bounded_level_1_repble}.
\begin{corollary}
\label{cor:smooth_bootstrapping_coeffs}
With these hypotheses, $\Gamma_{\mathrm{syn}}(\mathcal{X}_{m+1})\to \Gamma_{\mathrm{syn}}(\mathcal{X}_{m})$ is a quasisyntomic torsor under the smooth Artin stack $\Gamma^*(\mathcal{L}_1(\mathcal{X}))$, and is in particular smooth and surjective.
\end{corollary}

\subsection{Deformation theory}
\label{subsec:1-bounded_stacks_def_theory}

We continue with the assumptions of the previous subsection. 

\begin{construction}
Suppose that we have a divided power thickening $(C'\twoheadrightarrow C,\gamma)$ of $R$-algebras. Then we have a canonical commuting square
\begin{align}\label{eqn:canonical_commuting_square}
\begin{diagram}
\Gamma_{\mathrm{syn}}(\mathcal{X})(C')&\rTo&X^{-,(n)}(C)\\
\dTo&&\dTo\\
\Gamma_{\mathrm{syn}}(\mathcal{X})(C)&\rTo&X^{-,(n)}(C)\times_{X^{(n)}(C)}X^{(n)}(C').
\end{diagram} 
\end{align}
This is obtained as follows: The top arrow and the first coordinate of the bottom arrow are obtained from the canonical map $\Gamma_{\mathrm{syn}}(\mathcal{X})\to X^{-,(n)}$ obtained via pullback along the mod-$p^n$ reduction of the map
\[
x^{\mathcal{N}}_{\dR}:\Aff^1/\Gm\times \Spec C\to C^{\mathcal{N}}
\]
for every $R$-algebra $C$. The second coordinate of the bottom arrow is obtained via pullback along the (mod-$p^n$ reduction of the) lift
\[
\tilde{x}_{\dR,C'}:\Spec C'\to C^{\Prism}
\]
from Lemma~\ref{lem:lift_divided_powers}. 
\end{construction}

\begin{theorem}
\label{thm:groth_messing_mod_p}
Suppose that $\mathcal{X}$ is $1$-bounded and strongly integrable, and that $\Gamma_{\mathrm{syn}}(\mathcal{X})$ is represented by a locally almost finitely presented derived $p$-adic formal Artin stack over $R$ \footnote{We will only need it to be locally almost finitely presented and infinitesimally cohesive.}. Let $(C'\twoheadrightarrow C,\gamma)$ be a \emph{nilpotent} divided power thickening. Then the commuting square~\eqref{eqn:canonical_commuting_square} is \emph{Cartesian}.
\end{theorem}
\begin{proof}
Write
\[
\alpha_{(C'\twoheadrightarrow C,\gamma)}:\Gamma_{\mathrm{syn}}(\mathcal{X})(C')\to \Gamma_{\mathrm{syn}}(\mathcal{X})(C)\times_{X^{-,(n)}(C)\times_{X^{(n)}(C)}X^{(n)}(C')}X^{-,(n)}(C')
\]
for the natural map. We would like to show that it is an equivalence when the divided power thickening is nilpotent.

By quasisyntomic descent, we can reduce to the case where $C'$ and $C$ are semiperfect. Note that the existence of divided powers on $C'\twoheadrightarrow C$ is equivalent to the existence of a commuting diagram
\[
\begin{diagram}
\Prism_{C'}=A_{\crys}(C')&\rTo&\Prism_C = A_{\crys}(C)\\
\dTo&\ldTo&\dTo\\
C'&\rTo&C
\end{diagram}
\]
Moreover, by Proposition~\ref{prop:dot_varphi_1_nilpotent} and its proof, the nilpotence of the divided powers is equivalent to the topological local nilpotence of the divided Frobenius $\dot{\varphi}_1$ on $\hker(\Prism_{C'}\to \Prism_C)$ arising from the diagonal lift in the diagram.

Now, first suppose that $\hker(C'\twoheadrightarrow C)$ is $1$-connective. In this case, we are in the situation of Proposition~\ref{prop:def_theory_frames} with $\underline{B} = \underline{\Prism}_{C'}$ and $\underline{A} = \underline{\Prism}_C$. Indeed, by Remark~\ref{rem:frames_surjective}, and the already observed topological local nilpotence of $\dot{\varphi}_1$ it only remains to check that the map $\Prism_{C'}\to \Prism_C$ is surjective. This follows from the argument used in~\cite[Remark 4.1.20]{bhatt2022absolute}: One reduces to checking that the map is surjective on the associated graded for the conjugate filtration on Hodge-Tate cohomology (that is, Hodge cohomology), and hence that the map $\wedge^k_{C'}\mathbb{L}_{C'/\Field_p}[-k]\to \wedge^k_{C}\mathbb{L}_{C/\Field_p}[-k]$ has connective fiber for all $k\ge 0$. This reduces in turn to knowing that the map
\[
  C'\otimes_C\wedge^k_{C'}\mathbb{L}_{C'/\Field_p}\to \wedge^k_{C}\mathbb{L}_{C/\Field_p}
\]
has $k$-connective fiber. However this fiber is filtered with associated gradeds isomorphic to
\[
   \wedge^i_{C}\mathbb{L}_{C/\Field_p}\otimes_C\wedge^j_C(\mathbb{L}_{C/C'}[-1])
\]
for $i+j=k$ and $j\ge 1$, and so it is enough to know that each of these pieces is $k$-connective. Now, $\mathbb{L}_{C/\Field_p}$ is $1$-connective, while $\mathbb{L}_{C/C'}$ is $2$-connective~\cite[Corollary 25.3.6.4]{Lurie2018-kh}, so the desired connectivity follows from d\'ecalage~\cite[Proposition 25.2.4.2]{Lurie2018-kh}.

In the general case, $\Prism_{C'}\to \Prism_C$ need not be surjective, so we cannot directly apply the arguments from \S~\ref{subsec:abstract_def_theory}. 

To proceed, let us make note of the following principles:
\begin{enumerate}
   \item If we have a factoring $C'\twoheadrightarrow C_1\twoheadrightarrow C$, equipped with compatible nilpotent divided power structures, then we can verify the Cartesianness of the square for $C'\twoheadrightarrow C$ by separately verifying it for each factoring map.
   \item  For every $m\ge 0$, $\tau_{\leq m}C'\twoheadrightarrow \tau_{\leq m}C$ inherits a divided power structure from $C'\twoheadrightarrow C$.
   \item If $C_1\to C$ is any map of semiperfect $\Field_p$-algebras, then $C_1\times_CC'\twoheadrightarrow C_1$ is equipped with a nilpotent divided power structure
\end{enumerate}

Applying these principles, we get a factoring 
\[
C'\twoheadrightarrow C\times_{\tau_{\ge 1}(C)}\tau_{\ge 1}(C')\twoheadrightarrow C
\] 
equipped with compatible divided powers. The first of these maps has $1$-connective homotopy kernel, and we have already verified that~\eqref{eqn:canonical_commuting_square} is Cartesian for such maps. The second is a base-change of the map $\tau_{\ge 1}(C)\twoheadrightarrow \tau_{\ge 1}(C')$, and using infinitesimal cohesiveness we can reduce the checking of the Cartesianness of the associated square to that of the square associated with the map of $1$-truncated rings. 

Therefore, we can assume that $C'$ and $C$ are $1$-truncated. Via the same technique, applied now to $0$-truncations, we can reduce to checking separately the following two cases:
\begin{itemize}
   \item $\pi_0(C')\xrightarrow{\simeq}\pi_0(C)$ is an isomorphism;
   \item $C'$ and $C$ are discrete.
\end{itemize}

Let us first deal with the second case. Here, we have $C = C'/I$, where $I\subset C'$ is an ideal equipped with nilpotent divided powers. Using local almost finite presentation, we see that we have 
\[
\Gamma_{\mathrm{syn}}(\mathcal{X})(C)= \colim_{J\subset I}\Gamma_{\mathrm{syn}}(\mathcal{X})(C'/J)
\]
where $J\subset I$ ranges over the finitely generated sub-ideals of $I$ that inherit a divided power structure from $J$.\footnote{Note that, since the divided powers are nilpotent, the divided power closure of any finitely generated ideal is once again finitely generated.} This lets us reduce to the case where $I$ is finitely generated, so that the divided powers $I^{[n]}$ of $I$ are eventually $0$.  By repeatedly applying the first principle above, we can reduce to the case where $I^{[2]} = 0$, so that the divided powers are trivial and $I\subset C'$ is a square-zero ideal. In this case, by the infinitesimal cohesiveness of $\Gamma_{\mathrm{syn}}(\mathcal{X})$, it suffices to verify the theorem in the situation where $C' = C\oplus I[1]$, where we are reduced to the already known case of $1$-connective fiber. 

Now, let us consider the first case where $\pi_0(C')\xrightarrow{\simeq}\pi_0(C)$ is an isomorphism. If we set $D = \pi_0(C)$, then $C'$ (resp. $C$) is a square-zero thickening by $M'[1]$ (resp. $M[1]$) for some discrete $D$-modules $M'$, $M$, and the map $C'\twoheadrightarrow C$ arises from a map $M'\to M$ of $D$-modules. If $M''\subset M$ is the image of $M$, we obtain a factoring $C'\twoheadrightarrow C''\twoheadrightarrow C$ into nilpotent divided power thickenings where $C''$ is a square-zero extension of $D$ by $M''[1]$. The first of these maps has $1$-connective fiber, so we are reduced to the situation where $M = M''$ is a submodule of $M'$. Set $N= M'/M$; then the fiber of $C'\twoheadrightarrow C$ is the discrete $D$-module $N$, and so $N$ is equipped with a divided power filtration given by the submodules $N^{[m]}\subset N$ generated by the images of $\gamma_k:N\to N$ for all $k\ge m$. Just as in the classical case, we can use this filtration to reduce to the case of a trivial square-zero thickening, and hence to the case of $1$-connective fiber.
\end{proof}

\begin{remark}
\label{rem:nilpotent_groth_messing}
Under further restrictions, one can dispense with the nilpotence of the divided powers on $C'\twoheadrightarrow C$.  Note that we have maps 
\[
\Gamma_{\mathrm{syn}}(\mathcal{X})(C) \to \Gamma_{\Fzip}(\mathcal{X})(C) \xrightarrow{x\mapsto \psi_x}\pi_0(\Map_{C}(\Fil^1_{\mathrm{Hdg}}L_x,\varphi^*\Fil^1_{\mathrm{Hdg}}L_x)).
\]
Here, $\psi_x$ was defined in Construction~\ref{const:S1_mathcal_X}. Within the $\pi_0(C)$-module on the right, we have the subspace of nilpotent operators $f$ satisfying
\[
0 = \Phi_n(f) = (\varphi^n)^*f\circ \cdots\circ\varphi^*f\circ f = 0\in \pi_0(\Map_{C}(\Fil^1_{\mathrm{Hdg}}L_x,(\varphi^{n+1})^*\Fil^1_{\mathrm{Hdg}}L_x))
\]
for some $n\ge 1$.

Let
\[
\Gamma^{\mathrm{nilp}}_{\mathrm{syn}}(\mathcal{X})(C) \to \Gamma_{\mathrm{syn}}(\mathcal{X})(C)
\]
be the pullback of this subspace. This is the \defnword{nilpotent locus} of $\Gamma_{\mathrm{syn}}(\mathcal{X})(C)$.

We then claim that, for any divided power thickening $(C'\twoheadrightarrow C,\gamma)$ as above, the commuting square~\eqref{eqn:canonical_commuting_square} restricts to a Cartesian square
\begin{align}\label{eqn:nilpotent_cartesian_square}
\begin{diagram}
\Gamma^{\mathrm{nilp}}_{\mathrm{syn}}(\mathcal{X})(C')&\rTo&X^{-,(n)}(C')\\
\dTo&&\dTo\\
\Gamma^{\mathrm{nilp}}_{\mathrm{syn}}(\mathcal{X})(C)&\rTo&X^{-,(n)}(C)\times_{X^{(n)}(C)}X^{(n)}(C')
\end{diagram}
\end{align}
Since the kernel of $\pi_0(C')\to \pi_0(C)$ is locally nilpotent, we see that $\Gamma^{\mathrm{nilp}}_{\mathrm{syn}}(\mathcal{X})(C')$ factors through $\Gamma^{\mathrm{nilp}}_{\mathrm{syn}}(\mathcal{X})(C)$, giving us the restricted commuting square. To see that this is Cartesian, we can use quasisyntomic descent to reduce to the case where $C'$ and $C$ are semiperfect. Here, if $\hker(C'\twoheadrightarrow C)$ is $1$-connective, then the argument used in the proof above still works: We have to replace the nilpotence of the map $\dot{\varphi}_1$ with the nilpotence of $\psi_x$ and the argument explained in Remark~\ref{rem:nilpotent_locus}.

For the general case, we only give a sketch of a proof: Let us consider the prestack $C\mapsto \Gamma_{\mathrm{Witt}}(\mathcal{X})(C)$ constructed using the Witt vector frames $\underline{W(C)}$ and the definitions from \S~\ref{subsec:abstract_def_theory}. We can similarly define the nilpotent locus
\[
\Gamma^{\mathrm{nilp}}_{\mathrm{Witt}}(\mathcal{X})(C)\to \Gamma_{\mathrm{Witt}}(\mathcal{X})(C).
\]
Now, Remark~\ref{rem:nilpotent_locus_witt} tells us that the natural map
\begin{align}\label{eqn:nilp_can_be_witted}
\Gamma^{\mathrm{nilp}}_{\mathrm{syn}}(\mathcal{X})(C)\to \Gamma^{\mathrm{nilp}}_{\mathrm{Witt}}(\mathcal{X})(C)
\end{align}
induced from the map of frames $\underline{\Prism}_C\to \underline{W(C)}$ is an equivalence.

To proceed, we need to recall the \defnword{relative Witt frame} $\underline{W(C'/C)}$ defined in~\cite[Example 2.1.9]{MR4355476}: All we need is that its underlying animated commutative ring is $W(C')$, but its zeroth associated graded is $C$, and $\underline{W(C')} \to \underline{W(C)}$ factors through a sequence of maps
\[
\underline{W(C')}\to \underline{W(C'/C)}\to \underline{W(C)}.
\]
Using Remark~\ref{rem:nilpotent_locus} again tells us that we can compute $\Gamma^{\mathrm{nilp}}_{\mathrm{Witt}}(\mathcal{X})(C)$ using $\underline{W(C'/C)}$ instead. Combining this with Proposition~\ref{prop:def_theory_frames_trivial} now gives us a Cartesian diagram
\[
\begin{diagram}
\Gamma^{\mathrm{nilp}}_{\mathrm{Witt}}(\mathcal{X})(C')&\rTo&X^{-,(n)}(C')\\
\dTo&&\dTo\\
\Gamma^{\mathrm{nilp}}_{\mathrm{Witt}}(\mathcal{X})(C)&\rTo&X^{-,(n)}(C)\times_{X^{(n)}(C)}X^{(n)}(C')
\end{diagram}.
\]
One checks that this square is compatible via the isomorphism~\eqref{eqn:nilp_can_be_witted} with the square~\eqref{eqn:nilpotent_cartesian_square}. This completes the proof.
\end{remark}   

\begin{remark}
\label{rem:nilpotent_sections_witt}
By quasisyntomic descent, the argument used above shows that we have a frame-theoretic interpretation of the nilpotent locus via the Witt vector frame for \emph{any} $R$-algebra $C$. That is, we always have
\[
\Gamma^{\mathrm{nilp}}_{\mathrm{syn}}(\mathcal{X})(C)\xrightarrow{\simeq}\Gamma^{\mathrm{nilp}}_{\mathrm{Witt}}(\mathcal{X})(C).
\]
\end{remark}  

\begin{remark}
\label{rem:nilpotent_sections-any_frame}
In fact, one can say more. Let $(\underline{A},\zeta)$ be any filtered prism and with $C = R_A\in \mathrm{CRing}_{R/}$, so that we have maps $\mathfrak{S}(\underline{A})\to C^{\mathrm{syn}}\to R^{\mathrm{syn}}$ given by Proposition~\ref{prop:frames_to_nygaard}. Then we can apply the definitions from \S\ref{subsec:abstract_def_theory} to the restriction of $\mathcal{X}$ over $\mathfrak{S}(\underline{A})\otimes\Field_p$ and obtain a space $\Gamma_{\underline{A}}(\mathcal{X})(C)$ equipped with natural maps
\[
\Gamma_{\mathrm{syn}}(\mathcal{X})(C) \to \Gamma_{\underline{A}}(\mathcal{X})(C)\to \Gamma_{\mathrm{Witt}}(\mathcal{X})(C).
\]
If we set 
\[
\Gamma^{\mathrm{nilp}}_{\underline{A}}(\mathcal{X})(C) = \Gamma_{\underline{A}}(\mathcal{X})(C)\times_{\Gamma_{\mathrm{Witt}}(\mathcal{X})(C)}\Gamma^{\mathrm{nilp}}_{\mathrm{Witt}}(\mathcal{X})(C)
\]
then the arguments used above show that we obtain natural maps
\[
\Gamma^{\mathrm{nilp}}_{\mathrm{syn}}(\mathcal{X})(C) \to \Gamma^{\mathrm{nilp}}_{\underline{A}}(\mathcal{X})(C)\to \Gamma^{\mathrm{nilp}}_{\mathrm{Witt}}(\mathcal{X})(C),
\]
where the composition of the two maps as well as the last map are equivalences. In particular, we can compute $\Gamma^{\mathrm{nilp}}_{\mathrm{syn}}(\mathcal{X})(C)$ using \emph{any} such frame $\underline{A}$ with $R_A \simeq C$.
\end{remark}   

\begin{remark}
\label{rem:classical_truncation_enough}
The argument used above, applied for any discrete semiperfect $C$, to the surjective map of frames $\underline{\Prism}_C\to \pi_0(\underline{\Prism}_C)$, shows that $\Gamma^{\mathrm{nilp}}_{\mathrm{syn}}(\mathcal{X})(C)$ depends only on the classical truncation of the syntomification $C^{\mathrm{syn}}$. 

We wonder if this is true even for the non-nilpotent locus. This would be implied by a topological local nilpotence result for the divided Frobenius endomorphism of $\hker(\Prism_R\to \pi_0(\Prism_R))$ for $R$ semiperfect. Concretely, we are looking at the fiber of the map between the derived crystalline cohomology and classical crystalline cohomology of $R$.

If this local nilpotence is true in general, this would show that the classical truncation of $\Gamma_{\mathrm{syn}}(\mathcal{X})$ depends only on the classical truncations of the stacks $C^{\mathrm{syn}}$ for varying $C$. This in turn would explain why, when $\mathcal{X}^\preoneb$ is fibered in derived algebraic spaces over $R^{\mathrm{syn}}\otimes\Field_p$, $\Gamma_{\mathrm{syn}}(\mathcal{X})$ is also a derived algebraic space over $R$.
\end{remark}

\subsection{Bootstrapping from characteristic $p$: the base}
\label{subsec:1-bounded_stacks_bootstrapping_base}
We will now take $R$ to be in $\mathrm{CRing}^{p\text{-comp}}$. For any $p$-adic formal prestack $Z$ over $R$ set $\mathsf{R}(Z) = Z^{(1)}$, so that $\mathsf{R}(Z)(C) = Z(C/{}^{\mathbb{L}}p)$ for any $C\in \mathrm{CRing}^{p\text{-comp}}_{R/}$. This gives an endomorphism of the $\infty$-category of $p$-adic formal prestacks over $R$, and so can be iterated: We have
\[
\mathsf{R}^t(Z)(C) = Z(C\otimes \Field_p^{\otimes_{\Int}t}).
\]

The key for us is the following systematic d\'evissage from characteristic $p$:
\begin{proposition}
\label{prop:devissage}
Let $\mathcal{X}\to R^{\mathrm{syn}}\otimes\Int/p^m\Int$ be a $1$-bounded stack that is strongly integrable. For any $C\in \mathrm{CRing}^{p\text{-comp}}_{R/}$, the canonical map
\[
\Gamma_{\mathrm{syn}}(\mathcal{X})(C) \to \mathrm{Tot}\left(\Gamma_{\mathrm{syn}}(\mathcal{X})(C\otimes_{\Int} \Field_p^{\otimes_{\Int}\bullet +1})\right)
\]
is an equivalence. That is, we have an equivalence of $p$-adic formal prestacks
\[
\Gamma_{\mathrm{syn}}(\mathcal{X})\xrightarrow{\simeq}\mathrm{Tot}\left(\mathsf{R}^{\bullet +1}(\Gamma_{\mathrm{syn}}(\mathcal{X}))\right).
\]
\end{proposition}

Now, if $(C'\twoheadrightarrow C,\gamma)$ is a divided power thickening of $R$-algebras, we obtain the canonical commuting square~\eqref{eqn:canonical_commuting_square}.\footnote{Strictly speaking, we had imposed the condition that $R$ be an $\Field_p$-algebra when we introduced this square; however, this hypothesis was not used in its construction.}
\begin{corollary}
[Grothendieck-Messing]
\label{cor:groth_messing}
Suppose that $\mathcal{X}\to R^{\mathrm{syn}}\otimes\Int/p^m\Int$ is $1$-bounded and strongly integrable, and that $\Gamma_{\mathrm{syn}}(\mathcal{X})\otimes\Field_p$ is representable. Then, if $(C'\twoheadrightarrow C,\gamma)$ is a \emph{nilpotent} divided power thickening, the commuting square~\eqref{eqn:canonical_commuting_square} is \emph{Cartesian}.
\end{corollary}
\begin{proof}
Nilpotence of divided powers is preserved under arbitrary base-change along maps $C'\to D'$. Therefore, for every $m\ge 1$, the map 
\[
C'\otimes\Field_p^{\otimes_{\Int}\bullet +1}\twoheadrightarrow C\otimes\Field_p^{\otimes_{\Int}\bullet + 1}
\]
of cosimplicial $R\otimes\Field_p$-algebras canonically lifts to a cosimplicial diagram of nilpotent divided power thickenings of $R\otimes\Field_p$-algebras. This gives us a cosimplicial diagram of commuting squares as in~\eqref{eqn:canonical_commuting_square}, which are all Cartesian by Theorem~\ref{thm:groth_messing_mod_p}. We conclude by Proposition~\ref{prop:devissage}, which now shows that the commuting square the corollary is concerned with is a limit of Cartesian ones.
\end{proof}

As an immediate consequence, we obtain:
\begin{corollary}
\label{cor:cotangent_complex}
With the hypotheses above, write $\varpi_{\mathcal{X}}:\Gamma_{\mathrm{syn}}(\mathcal{X})\to X^{-,(n)}$ for the canonical map. Then $\Gamma_{\mathrm{syn}}(\mathcal{X})$ admits an almost perfect cotangent complex over $X^{-,(n)}$, and we have canonical isomorphisms
\begin{align*}
\mathbb{L}_{\Gamma_{\mathrm{syn}}(\mathcal{X})/R}&\simeq \varpi_X^*\mathbb{L}_{X^{-,(n)}/X^{(n)}};\\
\mathbb{L}_{\Gamma_{\mathrm{syn}}(\mathcal{X})/X^{-,(n)}}&\simeq \varpi_X^*\left(\mathbb{L}_{X^{(n)}/R}\vert_{X^{-,(n)}}\right)[1].
\end{align*}
\end{corollary}

Assuming Proposition~\ref{prop:devissage}, we can now show:
\begin{theorem}
\label{thm:general_representability}
Suppose that $\mathcal{X}$ is a $1$-bounded $r$-stack over $R^{\mathrm{syn}}\otimes\Int/p^n\Int$. Suppose that one of the following holds:
\begin{enumerate}
   \item $\mathcal{X}^\preoneb$ has quasi-affine diagonal and $\pi_0(R)$ is a $G$-ring;
   \item $\mathcal{X}$ is strongly integrable, and the $p$-adic formal stacks $X^{-}$ and $X^{0}$ over $R/{}^{\mathbb{L}}p^n$ are representable.
\end{enumerate}
Then:
\begin{enumerate}
   \item $\Gamma_{\mathrm{syn}}(\mathcal{X})$ is represented by a $p$-adic formal locally almost finitely presented Artin $r$-stack over $R$.
   \item If $\mathcal{X}^\preoneb$ is flat over $R^{\mathrm{syn}}\otimes\Int/p^n\Int$ with (quasi-)affine diagonal, then $\Gamma_{\mathrm{syn}}(\mathcal{X})$ has (quasi-)affine diagonal.
   \item If $X^-$ and $X^0$ are quasi-compact, then $\Gamma_{\mathrm{syn}}(\mathcal{X})$ is quasi-compact.
\end{enumerate}
\end{theorem}
\begin{proof}
We can assume that $R$ is a $\Int/p^m\Int$-algebra for some $m\ge 1$. As usual, the first hypothesis implies the second, and so we will assume that (2) is valid. Moreover, by Remark~\ref{rem:X_attractor_repble}, when $n=m=1$, our hypotheses also imply that the repeller $X^+$ is also representable.

One can now prove assertion (1) using a rather general argument involving Artin-Lurie representability; see Remark~\ref{rem:artin-lurie_alternative} below. Here, we give a more direct proof using Grothendieck-Messing theory. 

By Proposition~\ref{prop:bootstrapping_coeffs}, we know that, under our hypotheses, $\Gamma_{\mathrm{syn}}(\mathcal{X})\otimes\Field_p$ is represented by a locally almost finitely presented Artin $r$-stack over $R/{}^{\mathbb{L}}p$.

If $p>2$, then applying Corollary~\ref{cor:groth_messing} to the natural nilpotent divided power structure on $R\twoheadrightarrow R/{}^{\mathbb{L}}p$, we obtain a Cartesian square of prestacks
\[
\Square{\Gamma_{\mathrm{syn}}(\mathcal{X})}{}{X^{-,(n)}}{}{}{\mathsf{R}(\Gamma_{\mathrm{syn}}(\mathcal{X}))}{}{X^{(n)}\times_{\mathsf{R}(X^{(n)})}\mathsf{R}(X^{-,(n)})}.
\] 
Assertion (1) now follows, since all the prestacks involved except for the one in the top left corner are known to be locally almost finitely presented derived $p$-adic formal Artin stacks over $R$.

If $p=2$, then we can do something similar, by first considering the nilpotent divided power thickening $R\twoheadrightarrow R/{}^{\mathbb{L}}4$ to reduce to showing that $\Gamma_{\mathrm{syn}}(\mathcal{X})\otimes\Int/4\Int$ is a locally finitely presented Artin stack over $\Int/4\Int$, and then using the trivial divided powers on the square zero extension $R\twoheadrightarrow R\otimes_{\Int/4\Int}\Field_2$ (for $R\in \mathrm{CRing}_{(\mathcal{O}/4)/}$) to reduce further to the known case of $n=1$.

Let us proceed to assertions (2) and (3): It is enough to prove them for the stack $\Gamma_{\mathrm{syn}}(\mathcal{X})\otimes\Field_p$. First, note that, under the hypotheses of (2)  (resp. of (3)), $\Gamma_{\Fzip}(\mathcal{X}_1)$ has (quasi-)affine diagonal (resp. is quasicompact): This follows from Proposition~\ref{prop:t-integrable_repble} and the presentation~\eqref{eqn:sections_fzips}. 

Now, we claim that the zero section $\Gamma_{\Fzip}(\mathcal{X}_1)\to \mathsf{S}_1(\mathcal{X}_1)$ is quasi-compact, and that it has affine diagonal under the hypotheses of (2). This reduces to knowing that $\mathsf{S}_0(\mathcal{X}_1)$ is quasi-compact over $\Gamma_{\Fzip}(\mathcal{X}_1)$, with affine diagonal under the hypotheses of (2). This can be deduced from the proof of Theorem~\ref{thm:repbility_height_one}: one has to make the additional observation that the assumption in (2) guarantees that the relative cotangent complex of $\mathcal{X}$ is $(-1)$-connective. Combined with Theorem~\ref{thm:devissage_to_fzips}, this shows that, under the hypotheses of (2) (resp. of (3)), $\Gamma_{\mathrm{syn}}(\mathcal{X}_1)\otimes\Field_p$ has (quasi-)affine diagonal (resp. is quasicompact). By Proposition~\ref{prop:bootstrapping_coeffs}, the assertions for $n\ge 2$ are now reduced to the following assertion: For any almost perfect $F$-gauge $\mathcal{M}$ over $R$ with Hodge-Tate weights bounded below by $-1$ the stack $\Gamma^*_{\mathrm{syn}}(\mathcal{M})\to \Spec R$ is quasi-compact, and has affine diagonal when $\mathcal{M}$ is $(-1)$-connective.  The quasi-compactness follows from Corollary~\ref{cor:1_bounded_level_1_repble}. The diagonal map for the stack is a torsor under $\Gamma^*_{\mathrm{syn}}(\mathcal{M}[1])$, so it suffices to now observe that---when $\mathcal{M}$ is $(-1)$-connective---$\Gamma^*_{\mathrm{syn}}(\mathcal{M}[1])$ is affine by the same corollary.
\end{proof}

\subsubsection{}
We now proceed towards the proof of Proposition~\ref{prop:devissage}. Let us say that a map $f:Z\to Y$ of $p$-adic formal prestacks over $R^{\mathrm{syn}}$ \defnword{satisfies Tot descent for $\mathcal{X}^\preoneb$} if the natural map
\[
\Map_{/R^{\mathrm{syn}}\otimes\Int/p^n\Int}(Y\otimes\Int/p^n\Int,\mathcal{X}^\preoneb)\to \mathrm{Tot}\left(\Map_{/R^{\mathrm{syn}}\otimes\Int/p^n\Int}(Z^{\times_Y(\bullet + 1)}\otimes\Int/p^n\Int,\mathcal{X}^\preoneb)\right)
\]
is an equivalence. The map \defnword{satisfies universal Tot descent for} $\mathcal{X}^\preoneb$ if, for any $Y'\to Y$, the base-change $Z\times_YY'\to Y'$ also satisfies Tot descent for $\mathcal{X}^\preoneb$

\begin{remark}
\label{rem:flat_tot_descent}
Any flat cover satisfies universal Tot descent for $\mathcal{X}^\preoneb$.
\end{remark}

\begin{remark}
\label{rem:composition_tot_descent}
A composition of maps satisfying (universal) Tot descent for $\mathcal{X}^\preoneb$ also satisfies (universal) Tot descent for $\mathcal{X}^\preoneb$. 
\end{remark}   

\begin{remark}
\label{rem:composition_tot_descent_flat}
Suppose that we have maps $Z\xrightarrow{f}Y\xrightarrow{g}V$ such that:
\begin{itemize}
   \item $f\circ g$ satisfies Tot descent for $\mathcal{X}^\preoneb$;
   \item $f$ satisfies universal Tot descent for $\mathcal{X}^\preoneb$
\end{itemize}
Then $g$ also satisfies Tot descent for $\mathcal{X}^\preoneb$. This follows because, from our assumption on $f$, we find that the map $f^{\times_Vm}: Z^{\times_V m}\to Y^{\times_V m}$ also satisfies Tot descent for $\mathcal{X}^\preoneb$ for all $m\ge 1$.
\end{remark}

\begin{remark}
\label{rem:hl_preygel}
We have the following observation of Halpern-Leistner and Preygel: Suppose that we have $A\in \mathrm{CRing}$ equipped with a map $\Int[T_1,\ldots,T_r]\to A$ such that $A$ is derived $J$-complete, where $J = (T_1,\ldots,T_r)\subset \Int[T_1,\ldots,T_r]$; set $\overline{A} = A/{}^{\mathbb{L}}(T_1,\ldots,T_r)$. Suppose that we have a derived $J$-adic formal Artin stack $\mathcal{Y}$ over $A$. Then, for $R\in \mathrm{CRing}_{A/}$ derived $J$-complete, the map
\[
\mathcal{Y}(R) \to \mathrm{Tot}\left(\mathcal{Y}(R\otimes_{A} \overline{A}^{\otimes^{\mathbb{L}}_{A}\bullet +1})\right)
\]
is an equivalence. In fact, one only needs for $\mathcal{Y}$ to be nilcomplete and infinitesimally cohesive; see~\cite[Cor. 3.1.4]{MR4560539}. So Proposition~\ref{prop:devissage} is certainly implied by Theorem~\ref{introthm:main}. Here we will use the former to complete the proof of the latter.
\end{remark}

\begin{lemma}
\label{lem:reduction_to_HT_locus}
The map $C^\Prism\times_{\Int_p^\Prism}\Int_p^{\mathrm{HT}}\to C^\Prism$ satisfies Tot descent for $\mathcal{X}^\preoneb$.
\end{lemma}
\begin{proof}
Via quasisyntomic descent once again, we reduce to the case where $C$ is semiperfectoid. Let $I\xrightarrow{t}\Prism_C$ be the generalized Cartier divisor on $\Prism_C$ underlying its structure of a prism, so that we have
\[
C^\Prism \simeq \Spf(\Prism_C)\;;\; C^\Prism\times_{\Int_p^\Prism}\Int_p^{\mathrm{HT}} \simeq \Spf(\Prism_C)_{(t=0)}.
\]

Now, $\Prism_C$ here is equipped with its $I$-adic topology with respect to which it is derived complete. Therefore, the lemma follows from Remark~\ref{rem:hl_preygel}.
\end{proof}

\begin{lemma}
\label{lem:on_HT_locus}
The map $\Field_p^{\mathrm{HT}}\to \Int_p^{\mathrm{HT}}$ satisfies universal Tot descent for $\mathcal{X}^\preoneb$
\end{lemma}
\begin{proof}
We will use Proposition~\ref{prop:hodge_tate_locus}, which shows (via Remark~\ref{rem:flat_tot_descent}) that the map $\Spf\Int_p\to \Int_p^{\mathrm{HT}}$ satisfies universal Tot descent for $\mathcal{X}$.

It is now enough to show (see Remark~\ref{rem:composition_tot_descent_flat}) that the map $\Spec\Field_p\to \Spf\Int_p$ satisfies universal Tot descent for $\mathcal{X}^\preoneb$. This follows from Remark~\ref{rem:hl_preygel} and the fact that we are dealing with $p$-adic formal stacks.
\end{proof}

\begin{lemma}
\label{lem:cartesian_square_bootstrapping_needed}
Suppose that we have $C\in \mathrm{CRing}^{p\text{-nilp}}_{R/}$. Then we have a canonical Cartesian square
\[
\begin{diagram}
\Map(C^{\mathcal{N}}\otimes\Int/p^n\Int,\mathcal{X})&\rTo&\Map(C^{\Prism}\otimes\Int/p^n\Int,\mathcal{X}^\preoneb)\\
\dTo&&\dTo\\
X^{-}(C/{}^{\mathbb{L}}p^n)&\rTo&X(C/{}^{\mathbb{L}}p^n).
\end{diagram}
\]
\end{lemma}
\begin{proof}
Via quasisyntomic descent, we reduce to the case of $C$ semiperfectoid, where this follows from Theorem~\ref{thm:semiperf_crys} and filtered integrability.
\end{proof}

\begin{proof}
[Proof of Proposition~\ref{prop:devissage}]
The limit preserving property of the  functors $\Spec C\mapsto C^{\Prism}$ and $\Spec C \mapsto C^{\mathcal{N}}$ shows that we have
\begin{align*}
C^\Prism\times_{\Int_p^{\mathrm{\Prism}}}\underbrace{\Field_p^{\mathrm{\Prism}}\times_{\Int_p^{\mathrm{\Prism}}}\times\cdots\times_{\Int_p^{\mathrm{\Prism}}}\Field_p^{\mathrm{\Prism}}}_{\bullet+1} &\xrightarrow{\simeq} (C\otimes \Field_p^{\otimes \bullet +1})^{\Prism};\\
C^{\mathcal{N}}\times_{\Int_p^{\mathcal{N}}}\underbrace{\Field_p^{\mathcal{N}}\times_{\Int_p^{\mathcal{N}}}\times\cdots\times_{\Int_p^{\mathcal{N}}}\Field_p^{\mathcal{N}}}_{\bullet+1}   &\xrightarrow{\simeq}(C\otimes \Field_p^{\otimes \bullet +1})^{\mathcal{N}}.
\end{align*}

Lemmas~\ref{lem:reduction_to_HT_locus} and~\ref{lem:on_HT_locus} together show that the composition
\[
 C^{\Prism}\times_{\Int_p^{\Prism}}\Field_p^{\mathrm{HT}}\to C^{\Prism}\times_{\Int_p^{\Prism}}\Field_p^{\Prism}\to C^{\Prism}
\]
satisfies Tot descent for $\mathcal{X}^\preoneb$, while the second map satisfies universal Tot descent for $\mathcal{X}^\preoneb$. Therefore, Remark~\ref{rem:composition_tot_descent_flat} now shows that $C^{\Prism}\to C^{\Prism}\times_{\Int_p^{\Prism}}\Field_p^{\Prism}$ satisfies Tot descent for $\mathcal{X}^\preoneb$.

This, combined with the discussion in the first paragraph, shows that we have
\[
\Map(C^{\Prism}\otimes\Int/p^n\Int,\mathcal{X}^\preoneb)\xrightarrow{\simeq}\mathrm{Tot} \Map\bigl((C\otimes\Field_p^{\otimes\bullet+1})^{\Prism}\otimes\Int/p^n\Int,\mathcal{X}^\preoneb\bigr).
\]

Now, Lemma~\ref{lem:cartesian_square_bootstrapping_needed} combined with Remark~\ref{rem:hl_preygel} tells us that we also have
\[
\Map(C^{\mathcal{N}}\otimes\Int/p^n\Int,\mathcal{X})\xrightarrow{\simeq}\mathrm{Tot} \Map\bigl((C\otimes\Field_p^{\otimes\bullet+1})^{\mathcal{N}}\otimes\Int/p^n\Int,\mathcal{X}\bigr).
\]

The proof is now concluded by contemplating the identity~\eqref{eqn:sections_first_desc}.
\end{proof}

\begin{remark}
\label{rem:artin-lurie_alternative}
When $\pi_0(R)$ is a $G$-ring, one can also deduce Theorem~\ref{thm:general_representability} from Proposition~\ref{prop:bootstrapping_coeffs} and the following general assertion: Suppose that $\mathfrak{Y}$ is a $p$-adic formal prestack over $R$ with the following properties:
\begin{enumerate}
   \item $\mathfrak{Y}\otimes\Field_p$ is represented by a locally almost finitely presented derived Artin stack over $R/{}^{\mathbb{L}}p$;
   \item $\mathfrak{Y}$ satisfies Tot descent with respect to $p$: For every $C\in \mathrm{CRing}^{p\text{-nilp}}_{R/}$, the natural map
   \[
    \mathfrak{Y}(C) \to \mathrm{Tot}\left(\mathfrak{Y}(C\otimes \Field_p^{\otimes \cdot+1})\right)
   \]
   is an equivalence.
\end{enumerate}
Then $\mathfrak{Y}$ is represented by a derived $p$-adic formal Artin stack over $R$. 

This is shown using Artin-Lurie representability~\cite[Theorem 7.1.6]{lurie_thesis}. All the criteria involving limits are easily checked using our hypotheses. The existence of a $p$-completely almost perfect cotangent complex for $\mathfrak{Y}$ follows from the existence of an almost perfect cotangent complex for $\mathfrak{Y}\otimes\Field_p$ and Tot descent along $p$ for $p$-completely almost perfect complexes. From this, one also deduces the local almost finite presentation, which completes the verification of all the criteria in \emph{loc. cit.}
\end{remark} 

\begin{remark}
\label{rem:nilpotent_locus_general}
For any $C\in \mathrm{CRing}^{p\text{-nilp}}_{R/}$, set
\[
\Gamma^{\mathrm{nilp}}_{\mathrm{syn}}(\mathcal{X})(C) = \Gamma^{\mathrm{nilp}}_{\mathrm{syn}}(\mathcal{X}_1)(C/{}^{\mathbb{L}}p)\times_{\Gamma_{\mathrm{syn}}(\mathcal{X}_1)(C/{}^{\mathbb{L}}p)}\Gamma_{\mathrm{syn}}(\mathcal{X})(C).
\]
The prestack $\Gamma^{\mathrm{nilp}}_{\mathrm{syn}}(\mathcal{X})\to \Gamma_{\mathrm{syn}}(\mathcal{X})$ is the \defnword{nilpotent locus}, and the argument from Corollary~\ref{cor:groth_messing} and Remark~\ref{rem:nilpotent_groth_messing} shows that, for any not necessarily nilpotent divided power thickening $(C'\twoheadrightarrow C,\gamma)$, we have a Cartesian diagram as in that remark. That is, we have a Grothendieck-Messing theory for such thickenings as long as we restrict to the nilpotent locus.

Moreover, via Remark~\ref{rem:nilpotent_sections_witt} and Tot descent, or by using Corollary~\ref{cor:def_theory_frames} and quasisyntomic descent, one sees that $\Gamma^{\mathrm{nilp}}_{\mathrm{syn}}(\mathcal{X})(C)$ can be computed using the Witt frame for \emph{any} $R$-algebra $C$. In fact, using the argument, one sees as in Remark~\ref{rem:nilpotent_sections-any_frame} that $\Gamma^{\mathrm{nilp}}_{\mathrm{syn}}(\mathcal{X})(C)$ can be computed using any filtered prism $(\underline{A},\zeta)$ satisfying $R_A\simeq C$. In particular, for discrete $C$, it depends only on the classical truncation of $C^{\mathrm{syn}}$.
\end{remark}  

\subsection{Functoriality}

Suppose that we have a map $f:\mathcal{X}_1\to \mathcal{X}_2$ of $1$-bounded stacks over $R^{\mathrm{syn}}\otimes\Int/p^n\Int$ satisfying the hypotheses of Theorem~\ref{thm:general_representability}. Then we obtain a map of derived stacks $\Gamma_{\mathrm{syn}}(f):\Gamma_{\mathrm{syn}}(\mathcal{X}_1)\to \Gamma_{\mathrm{syn}}(\mathcal{X}_2)$ over $R$. We also have the corresponding map of Weil restricted stacks $X_1^{-,(n)}\to X_2^{-,(n)}$ and $X_1^{(n)}\to X_2^{(n)}$.

The following is immediate from Corollary~\ref{cor:cotangent_complex}:
\begin{proposition}
\label{prop:relative_cotangent_complex}
Let $\varpi_1:\Gamma_{\mathrm{syn}}(\mathcal{X}_1)\to X_1^{-,(n)}$ be the canonical map. Then we have a natural isomorphism
\[
\mathbb{L}_{\Gamma_{\mathrm{syn}}(\mathcal{X}_1)/\Gamma_{\mathrm{syn}}(\mathcal{X}_2)}\xrightarrow{\simeq}\varpi_1^*\mathbb{L}_{X_1^{-,(n)}/(X_1^{(n)}\times_{X_2^{(n)}}X_2^{-,(n)})}.
\]
\end{proposition}

\subsection{Sections of perfect $F$-gauges with Hodge-Tate weights $\le 1$}
\label{subsec:sections_1-bounded_repble}

Suppose that $\mathcal{M}$ is a perfect $F$-gauge of level $n$ over $R\in \mathrm{CRing}^{p\text{-comp}}$ with Hodge-Tate weights bounded by $1$. Pulling $\mathcal{M}$ back along $x^{\mathcal{N}}_{\dR}$ yields an increasingly filtered perfect complex $\Fil^\bullet_{\mathrm{Hdg}} M_n$ over $R/{}^{\mathbb{L}}p^n$. The next theorem is immediate from Theorem~\ref{thm:general_representability} and Corollary~\ref{cor:1_bounded_level_1_repble} by observing that we have $\Gamma_{\mathrm{syn}}(\mathcal{M}) \simeq \Gamma^*_{\mathrm{syn}}(\mathcal{M}^\vee)$.\footnote{One can also, without any additional work, formulate and prove a version for almost perfect $F$-gauges wth Hodge-Tate weights $\ge -1$ by considering the prestack $\Gamma^*_{\mathrm{syn}}(\mathcal{M})$ instead.}

\begin{theorem}
\label{thm:sections_1-bounded_representable}
The prestack $\Gamma_{\mathrm{syn}}(\mathcal{M})$ is represented by a $p$-adic formal locally finitely presented derived Artin stack over $R$ with cotangent complex $\Reg{\Gamma_{\mathrm{syn}}(\mathcal{M})}\otimes_R(\gr^{-1}_{\mathrm{Hdg}}M_n)^\vee[1]$. Moreover, if $(C'\twoheadrightarrow C,\gamma)$ is a nilpotent divided power thickening in $\mathrm{CRing}^{p\text{-comp}}_{R/}$, then we have a Cartesian square
\[
\begin{diagram}
\Gamma_{\mathrm{syn}}(\mathcal{M})(C')&\rTo&C'\otimes_R\Fil^0_{\mathrm{Hdg}}M_n\\
\dTo&&\dTo\\
\Gamma_{\mathrm{syn}}(\mathcal{M})(C)&\rTo&(C\otimes_R\Fil^0_{\mathrm{Hdg}}M_n)\times_{C\otimes_RM_n}(C'\otimes_RM_n).
\end{diagram}
\]
Moreover: 
\begin{enumerate}
   \item If $\mathcal{M}$ has Tor amplitude in $(-\infty,-1]$, then $\Gamma_{\mathrm{syn}}(\mathcal{M})$ is a smooth faithfully flat $p$-adic formal stack over $R$.
   \item If $\mathcal{M}$ has Tor amplitude in $[0,\infty)$, then $\Gamma_{\mathrm{syn}}(\mathcal{M})$ is a derived affine $p$-adic formal scheme over $R$.
\end{enumerate}
\end{theorem} 
\begin{proof}
Only the numbered asssertions require proof. To show them, we can assume that $R$ is an $\Field_p$-algebra. We already know from Corollary~\ref{cor:1_bounded_level_1_repble} that the statements are true if $\mathcal{M}$ has level $1$ and we now use the usual d\'evissage by power of $p$ (say in the form of Proposition~\ref{prop:bootstrapping_coeffs}) to see that they are true in general.
\end{proof}

\subsection{Stacks of perfect $F$-zips of Hodge-Tate weights $0,1$}
\label{subsec:HTwts_01}

For every $n\ge 1$, let $\mathcal{X}\to \Int_p^{\mathrm{syn}}\otimes\Int/p^n\Int$ be the $1$-bounded stack obtained via base-change from $\mathcal{P}_{\{0,1\}}\to B\Gm$ as described in Example ~\ref{ex:perfect_HT_wts_01_stack}. 

We will denote the associated  formal prestack $\Gamma_{\mathrm{syn}}(\mathcal{X})\to \Int_p$ by $\mathrm{Perf}^{\mathrm{syn}}_{\{0,1\},n}$. Concretely, this associates with every $R\in \mathrm{CRing}^{p\text{-nilp}}$ the $\infty$-groupoid $\mathrm{Perf}_{\{0,1\}}(R^{\mathrm{syn}}\otimes\Int/p^n\Int)^{\simeq}$ of perfect $F$-gauges of level $n$ over $R$ with Hodge-Tate weights $0,1$. 

Over this prestack we have a canonical filtered perfect complex $\Fil^\bullet_{\mathrm{Hdg}}M_{\mathrm{taut}}$ obtained by viewing, for each $R$, the universal perfect $F$-gauge of level $n$ as a perfect complex over $R^{\mathrm{syn}}$, and pulling back along $x^{\mathcal{N}}_{\dR}$. We obtain the next theorem from Theorem~\ref{thm:general_representability} and the discussion in Example~\ref{ex:perfect_HT_wts_01_stack}.

\begin{theorem}
\label{thm:HTwts01_representable}
The prestack $\mathrm{Perf}^{\mathrm{syn}}_{\{0,1\},n}$ is represented by a $p$-adic formal locally finitely presented derived Artin stack over $\Int_p$ with cotangent complex $(\gr^{-1}_{\mathrm{Hdg}}M_{\mathrm{taut}})^\vee\otimes \Fil^0_{\mathrm{Hdg}}M_{\mathrm{taut}}$. Moreover, if $(C'\twoheadrightarrow C,\gamma)$ is a nilpotent divided power thickening of $p$-complete algebras in $\mathrm{CRing}$, then we have a Cartesian square
\[
\begin{diagram}
\mathrm{Perf}^{\mathrm{syn}}_{\{0,1\},n}(C')&\rTo&\mathrm{Perf}_{\{0,1\}}(\Aff^1/\Gm\times\Spec C'/{}^{\mathbb{L}}p^n)\\
\dTo&&\dTo\\
\mathrm{Perf}^{\mathrm{syn}}_{\{0,1\},n}(C)&\rTo&\mathrm{Perf}_{\{0,1\}}(\Aff^1/\Gm\times\Spec C/{}^{\mathbb{L}}p^n)\times_{\mathrm{Perf}(C/{}^{\mathbb{L}}p^n)}\mathrm{Perf}(C'/{}^{\mathbb{L}}p^n).
\end{diagram}
\]
\end{theorem}
\begin{proof}
The only thing that needs still to be verified is the assertion about the cotangent complex. For this, note that we have
\begin{align*}
X^{-,(n)}:C&\mapsto \mathrm{Perf}_{\{0,1\}}(\Aff^1/\Gm\times\Spec C/{}^{\mathbb{L}}p^n)\\
X^{(n)}:C&\mapsto \mathrm{Perf}(C/{}^{\mathbb{L}}p^n).
\end{align*}
The fiber of the map $X^{-,(n)}\to X^{(n)}$ over a perfect complex $M$ over $C/{}^{\mathbb{L}}p^n$ classifies two step filtrations $\Fil^\bullet M$ on $M$ with $\gr^iM$ perfect for all $i$, and $\gr^iM\simeq 0$ for $i\neq -1,0$. Giving such a datum is equivalent to specifying the map $f:\Fil^0M\to \Fil^{-1}M = M$ with $\Fil^0M$ perfect over $C/{}^{\mathbb{L}}p^n$. 

This shows that the tangent space of the map at $M$ is canonically isomorphic to the space of maps $\Fil^0M \to \gr^{-1}M$, which is of course the $C$-module $(\Fil^0M)^\vee\otimes_C\gr^{-1}M$. Taking duals and using Corollary~\ref{cor:cotangent_complex} now gives the desired cotangent complex.
\end{proof}

\begin{corollary}
\label{cor:HT_wts_0}
Given $R\in \mathrm{CRing}^{p\text{-nilp}}$, write $\mathrm{Perf}^{\mathrm{syn}}_{\{0\},n}(R)$ for the $\infty$-groupoid of perfect $F$-gauges with Hodge-Tate weights $0$. Then there is a canonical equivalence
\[
\mathrm{Perf}^{\mathrm{syn}}_{\{0\},n}(R)\xrightarrow{\simeq}D^b_{\mathrm{lisse}}(\Spec R,\Int/p^n\Int)
\]
where the right hand side is the bounded derived category of lisse $\Int/p^n\Int$-sheaves over $\Spec R$.
\end{corollary}
\begin{proof}
We can view $\mathrm{Perf}^{\mathrm{syn}}_{\{0\},n}$ as the open substack of $\mathrm{Perf}^{\mathrm{syn}}_{\{0,1\},n}$ parameterizing objects $\mathcal{M}$ such that $\gr^{-1}_{\mathrm{Hdg}}M_{\mathrm{taut}} \simeq 0$. Moreover, the description of the cotangent complex shows that this substack is \emph{\'etale} over $\Spf\Int_p$. In particular, it is determined completely by its restriction to perfect $\Field_p$-algebras $R$. 

For perfect $R$, the left hand side in the statement can be identified with the $\infty$-groupoid of perfect complexes $\mathsf{M}$ of $W_n(R)$-modules equipped with an isomorphism $\varphi^*\mathsf{M}\xrightarrow{\simeq}\mathsf{M}$. We can now conclude by a classical result of Katz, as formulated in~\cite[Proposition 3.6]{MR4600546}.
\end{proof} 

\begin{remark}
Explicitly, the equivalence is given by sending $\mathcal{M}$ to $R\Gamma_{\mathrm{syn}}(\mathcal{M})$, where the latter is seen to be an \'etale stack over $R$ by Theorem~\ref{thm:sections_1-bounded_representable}.
\end{remark}

\section{The algebraicity conjecture of Drinfeld}
\label{sec:main}

We can finally introduce the main protagonist of this paper. Fix a smooth affine group scheme $G$ over $\Int_p$ and a $1$-bounded cocharacter $\mu:\Gmh{\mathcal{O}}\to G_{\mathcal{O}}$ defined over the ring of integers $\mathcal{O}$ in a finite unramified extension of $\Rat_p$ with residue field $k$.

\subsection{Definitions}
\label{subsec:defns}
There is a canonical map $\Int_p^{\mathrm{syn}}\to B\Gm$ classifying the Breuil-Kisin twist from~\S\ref{subsec:breuil-kisin}. The restriction to $\mathcal{O}^{\mathcal{N}}$ lifts to a map $B\Gmh{\mathcal{O}}$, which does not descend to $\mathcal{O}^{\mathrm{syn}}$; however, we do have a map of pointed graded $p$-adic formal stacks
\[
\mathcal{O}^{\mathrm{syn}}\to (B\Gm\times\Spf\Int_p,\iota_{\mathcal{O}}).
\]
Therefore, we can pull the $1$-bounded stack $\mathcal{B}(G,\mu)$ from Definition~\eqref{defn:BGmu_1-bounded_stack} back over $\mathcal{O}^{\mathrm{syn}}$.

\begin{definition}
For any $R\in \mathrm{CRing}^{p\text{-comp}}_{\mathcal{O}/}$, we set $\BT{n}(R) = \Gamma_{\mathrm{syn}}(\mathcal{B}(G,\mu)\otimes\Int/p^n\Int)(R)$. Set 
\[
\BT{\infty} = \varprojlim_n\BT{n}.
\]

For any $R\in \mathrm{CRing}^{p\text{-comp}}_{\mathcal{O}/}$, an \defnword{$n$-truncated $(G,\mu)$-aperture} over $R$ is an object of the $\infty$-groupoid $\BT{n}(R)$.

A \defnword{$(G,\mu)$-aperture} over $R$ is an object of $\BT{\infty}(R)$.
\end{definition}

\begin{remark}
\label{rem:pointwise}
For $R\in \mathrm{CRing}_{\mathcal{O}/}^{f,p\text{-nilp}}$, $\BT{n}(R)$ can be described as the $\infty$-groupoid of $G$-torsors $\mathcal{Q}$ over $R^{\mathcal{N}}\otimes\Int/p^n\Int$ such that it is equipped with an equivalence $j_{dR}^*\mathcal{Q}\xrightarrow{\simeq}j_{HT}^*\mathcal{Q}$ of $G$-torsors over $R^\Prism\otimes\Int/p^n\Int$, satisfying the following condition: If $\mathcal{Q}^\mu$ is the associated $G\{\mu\}$-torsor, then, for every geometric point $R\to \kappa$ of $\Spf R$, the restriction of $(x^{\mathcal{N}}_{dR})^*\mathcal{Q}^\mu$ over $B\Gm\times \Spec \kappa$ is trivial.
\end{remark}

\begin{remark}
\label{rem:pullback_to_frames}
Suppose that $(\underline{A},\zeta)$ is a filtered prism. Then pullback along the map $\iota_{(\underline{A},\zeta)}$ from Proposition~\ref{prop:frames_to_nygaard}---or better along the map $\iota_{\mathrm{id}:\underline{A}\to \underline{A}}$ from Corollary~\ref{cor:frames_to_nygaard}---gives us a canonical arrow
\[
\BT{n}(R_A)\to \Wind{n}{\underline{A}}(R_A).
\]
More generally, if $\underline{B}\to \underline{A}\otimes\Int/p^n\Int$ is any map of frames, we obtain a functor from $\BT{n}(R_A)$ to the $\infty$-groupoid of windows over $\underline{B}$. In particular, when $\underline{A} = \underline{W(R)}$ is the Witt vector frame associated with $R\in \mathrm{CRing}_{\Field_p/}$, and $\underline{B} = \underline{W_n(R)}$ is the $n$-truncated Witt frame, we obtain a canonical map $\BT{n}(R)\to \Disp{n}(R)$, underlying a map of formal prestacks over $k$.
\end{remark}

\begin{remark}
\label{rem:drinfeld_definition}
When $\mu$ is defined over $\Int_p$, Drinfeld gives a different definition of $\BT{n}$ in~\cite{drinfeld2023shimurian}, which isn't quite the same as the one we give here, though it is \emph{isomorphic} to ours. 

To begin, he views $\mu$ as a cocharacter of the \emph{automorphism} group $\Aut(G)$ acting on $G$. Since $\mu$ is defined over $\Int_p$, $G\{\mu\}$ lives over $B\Gm$ and hence it makes sense to talk about $G\{\mu\}$-torsors over $R^{\mathrm{syn}}$ for any $R\in \mathrm{CRing}^{p\text{-nilp}}$. Drinfeld's definition of $\BT{n}(R)$ is as the $\infty$-groupoid of $G\{\mu\}$-torsors whose restriction to $B\Gm\times \Spec \kappa$ is trivial for any map $R\to \kappa$ to an algebraically closed field $\kappa$. 

We can include this in our framework, though the definition of the $1$-bounded stack $\mathcal{B}(G,\mu)$ will then have to be adjusted slightly: We will take it to be given by the pair $(BG\{\mu\},BM_\mu)$ living over the tautologically pointed graded stack $B\Gmh{\Int_p}$, where $BM_\mu$ is still the trivial $1$-bounded locus described in Lemma~\ref{lem:BGmu_trivial_locus}. If we define $\BT{n}$ in the same way with this adjusted definition, we now recover Drinfeld's definition and Theorem~\ref{thm:main_thm_body} and its proof below hold verbatim. In particular, we obtain a proof of~\cite[Conjecture C.3.1]{drinfeld2024shimurian}.

If we have a cocharacter $\mu$ of $G$, then the stacks obtained from our definition and from Drinfeld's, viewing $\mu$ as a cocharacter of $\Aut(G)$ instead, are canonically isomorphic, since we have an equivalence between $G\{\mu\}$-torsors and $G$-torsors preserving the corresponding $1$-bounded loci. 

However, there is the following phenomenon: When, for instance, $G$ is a torus, since a cocharacter of $G$ will always act trivially, Drinfeld's definition would yield a stack that is \emph{independent} of the cocharacter $\mu$, since $G\{\mu\}$ would always be $G\times B\Gm$. But, in $p$-adic Hodge theory, it is important to keep track of the cocharacter, since it knows about the Hodge-Tate weights of \'etale or Galois realizations. For this reason, we have chosen the definition and presentation used here. 

More importantly, as pointed out to us by the authors of~\cite{imai2023prismatic}, Drinfeld's description does not generalize directly when $\mu$ is only defined over some (unramified) ring of integers of $\Int_p$. See~\S\ref{subsec:tori} below for more discussion of the case of tori, and note the role played by the cocharacter in Proposition~\ref{prop:lubin-tate}.
\end{remark}

\subsection{The semiperfectoid case}
\label{subsec:semiperfectoid}
\begin{lemma}
\label{lem:f-semiperf_etale}
Suppose that $R$ is semiperfectoid and that $R\to S$ is \'etale. Then $S$ is also semiperfectoid. Moreover, the canonical map of frames $\underline{\Prism}_R\to \underline{\Prism}_S$ is $(p,I_R)$-completely \'etale, and we have
\[
\Prism_S\otimes_{\Prism_R}\Fil^\bullet_{\mathcal{N}}\Prism_R\xrightarrow{\simeq}\Fil^\bullet_{\mathcal{N}}\Prism_S.
\]
\end{lemma}
\begin{proof}
Choose a perfectoid ring $R_0$ and a surjection $R_0\twoheadrightarrow R$, lifting to a map $A_{\mathrm{inf}}(R_0)=\Prism_{R_0}\to \Prism_R$.

By~\cite[\href{https://stacks.math.columbia.edu/tag/04D1}{Tag 04D1}]{stacks-project}, and by the equivalence of the \'etale sites of $R$ and $\pi_0(R)$, there exists a $p$-completely \'etale map $R_0\to R'_0$ such that $S = R'_0\otimes_{R_0}R$. We therefore reduce the first statement to showing that $\Prism_{R_0}\to\Prism_{R'_0}$ is $(p,I_{R_0})$-completely \'etale. This follows from the fact that $R'_0$ is also perfectoid,\footnote{This is a \emph{much} easier assertion than the almost purity theorem for perfectoid algebras over fields!} and the associated map of tilts $R_0^\flat\to R_0^{',\flat}$ is also \'etale, which implies that the map 
\[
\Prism_{R_0} = W(R_0^\flat)\to W(R^{',\flat}_0) = \Prism_{R'_0}
\]
is $(p,I_{R_0})$-completely \'etale.

For the second assertion, first note that by Proposition~\ref{prop:etale_goes_to_etale}, $S^{\mathcal{N}}\to R^{\mathcal{N}}$ is $(p,I_R)$-completely \'etale. The proof is now completed by combining Proposition~\ref{prop:frame_props_hensel} and Theorem~\ref{thm:semiperf_crys}.
\end{proof}

The next two results tell us that $\BT{n}$ can be understood in terms of the $(G,\mu)$-windows from~\S\ref{subsec:G_mu_windows}

\begin{proposition}
[Quasisyntomic descent]
\label{prop:descent}
If $R\to R_{\infty}$ is as in Corollary~\ref{cor:semiperf_qsynt}, then we have:
\[
\BT{n}(R)\xrightarrow{\simeq}\mathrm{Tot}\left(\BT{n}(R_\infty^{\otimes_R \bullet+1})\right)
\]
\end{proposition}
\begin{proof}
This is essentially immediate from the definitions and Proposition~\ref{prop:flat_surjections}. 
\end{proof}

\begin{lemma}
\label{lem:quotient_stack_desc}
If $R$ is semiperfectoid, as sheaves on the small \'etale site $R_{\et}$ of $\Spf(R)$, we have a canonical equivalence
\[
\BT{n}\vert_{R_{\et}} \xrightarrow{\simeq}\Wind{n}{\underline{\Prism}_R}
\]
\end{lemma}
\begin{proof}
This is immediate from Lemma~\ref{lem:f-semiperf_etale}, Theorem~\ref{thm:semiperf_crys} and Remark~\ref{rem:pullback_to_frames}.
\end{proof}

\begin{remark}
\label{rem:pointwise_equiv_conditions}
Combining Proposition~\ref{prop:descent} and Lemma~\ref{lem:quotient_stack_desc} with Proposition~\ref{prop:fpqc_is_etale}, we see that $\BT{n}(R)$ can also be described as the $\infty$-groupoid of $G$-torsors $\mathcal{Q}$ over $R^{\mathcal{N}}\otimes\Int/p^n\Int$ equipped with an equivalence $j_{dR}^*\mathcal{Q}\xrightarrow{\simeq}j_{HT}^*\mathcal{Q}$ of $G$-torsors over $R^\Prism\otimes\Int/p^n\Int$, and satisfying the following equivalent conditions on the associated $G\{\mu\}$-torsor $\mathcal{Q}^\mu$:
\begin{enumerate}
   \item For every geometric point $R\to \kappa$ of $\Spf R$, the restriction of $(x^{\mathcal{N}}_{\dR})^*\mathcal{Q}$ over $B\Gm\times \Spec \kappa$ is isomorphic to $\mathcal{P}_\mu$, the torsor classified by $B\mu:B\Gmh{\mathcal{O}}\to BG_{\mathcal{O}}$;
   \item For every geometric point $R\to \kappa$ of $\Spf R$, the restriction of $(x^{\mathcal{N}}_{\dR})^*\mathcal{Q}^\mu$ over $B\Gm\times \Spec \kappa$ is trivial;
   \item  The restriction of $\mathcal{Q}^\mu$ over $R^{\mathcal{N}}\otimes\Int/p^n\Int$ is trivial locally in the $p$-quasisyntomic topology on $\Spf R$;
   \item  The restriction of $\mathcal{Q}^\mu$ over $R^{\mathcal{N}}\otimes\Int/p^n\Int$ is trivial flat locally on $\Spf R$.
\end{enumerate}
If $\Spf R$ is connected, then these conditions are also equivalent to: For \emph{some} geometric point $R\to \kappa$, the restriction of $(x^{\mathcal{N}}_{\dR})^* \mathcal{Q}$ over $B\Gm\times\Spec \kappa$ is isomorphic to $\mathcal{P}_\mu$.
\end{remark}

\subsection{Representability}
If we view $\mathcal{X} = \mathcal{B}(G,\mu)\otimes\Int/p^n\Int$ as a $1$-bounded stack over $\mathcal{O}^{\mathrm{syn}}\otimes\Int/p^n\Int$, then the associated attractor $X^{-,(n)}\to \Spec \mathcal{O}$ is $BP_\mu^{-,(n)}$, and the stack $X^{(n)}$ is the Weil restricted classifying stack $BG^{(n)}_\mu$. 

\subsubsection{}
If $\mathcal{X}_1 =\mathcal{B}(G,\mu)\otimes\Field_p$, then $\Gamma_{\Fzip}(\mathcal{X}_1)$ is a quasicompact smooth $0$-dimensional Artin stack over $k$ with affine diagonal. Indeed, it is not difficult to see using Remark~\ref{rem:BGmu_simple} and~\eqref{eqn:sections_fzips} that this is just the stack of \defnword{$F$-zips with $G$-structure of type $\mu$}; that is, of tuples $\bm{F} = (\mathcal{F},\mathcal{F}^+,\mathcal{F}^-,\eta^+,\eta^-,\alpha)$ where:
\begin{itemize}
   \item $\mathcal{F}$ is a $G$-torsor over $R$;
   \item $\mathcal{F}^+$ is a $P^+_{\mu^\varphi}$-torsor over $R$;
   \item $\mathcal{F}^-$ is a $P^-_\mu$-torsor over $R$;
   \item $\eta^+:\mathcal{F}^+\to \mathcal{F}$ is a $P^+_{\mu^\varphi}$-equivariant map;
   \item $\eta^-:\mathcal{F}^-\to \mathcal{F}$ is a $P^-_\mu$-equivariant map;
   \item $\alpha:\mathcal{F}^+/U^+_{\mu^\varphi}\xrightarrow{\simeq}\varphi^*\left(\mathcal{F}^-/U^-_{\mu}\right)$ is an isomorphism of $M_{\mu^\varphi}$-torsors.
\end{itemize}
The conclusion now follows from~\cite[\S 3.3]{Pink2015-ye}; see also the discussion in~\cite[\S 3.2]{drinfeld2023shimurian}, where it is denoted $\mathrm{Disp}^G_1$. In the notation of this paper, we have $\Gamma_{\Fzip}(\mathcal{X}_1) \simeq \Disp{1}$.

The next result proves Theorems~\ref{introthm:main} and~\ref{introthm:groth_mess}
\begin{theorem}
\label{thm:main_thm_body}
$\BT{n}$ is a quasicompact smooth $0$-dimensional $p$-adic formal Artin stack over $\mathcal{O}$ with affine diagonal. For every nilpotent divided power thickening $(C'\twoheadrightarrow C,\gamma)$ of $p$-complete $\mathcal{O}$-algebras, we have a Cartesian square
\[
\begin{diagram}
\BT{n}(C')&\rTo&BP^{-,(n)}_\mu(C')\\
\dTo&&\dTo\\
\BT{n}(C)&\rTo&BP^{-,(n)}_\mu(C)\times_{BG^{(n)}(C)}BG^{(n)}(C').
\end{diagram}
\]
Moreover, the transition maps $\BT{n+1}\to \BT{n}$ are smooth and surjective. 
\end{theorem}
\begin{proof}
Theorem~\ref{thm:devissage_to_fzips} and the discussion above show that $\BT{n}$ is a quasi-compact finitely presented derived $p$-adic formal Artin stack over $\mathcal{O}$ with affine diagonal. 

The existence of the stated Cartesian square for nilpotent divided power thickenings follows from Corollary~\ref{cor:groth_messing}. Since $BP^{-,(n)}_\mu\to BG^{(n)}$ is smooth, this also shows that $\BT{n}$ is a smooth $p$-adic formal Artin stack over $\mathcal{O}$.

The smooth surjectivity of the transition maps can be checked mod-$p$ and here it is a special case of Corollary~\ref{cor:smooth_bootstrapping_coeffs}. Note that we can actually say a bit more: $\BT{1}\otimes\Field_p$ is a torsor over $\Disp{1}$ under a group stack $\mathsf{S}_0(\mathcal{X}_1)$ by (2) of Theorem~\ref{thm:devissage_to_fzips}. In fact, from Corollary~\ref{cor:fppf_cohomology}, one sees that $\mathsf{S}_0(\mathcal{X}_1)$ is the classifying stack of a certain finite flat group scheme of height one. This is nothing but the \emph{Lau group scheme} from~\cite{drinfeld2024shimurian}. In particular, one sees that $\BT{1}\otimes\Field_p$ is a gerbe over $\Disp{1}$ banded by the Lau group scheme: this recovers the main result of \emph{loc. cit.} for smooth inputs.

Over $\BT{1}\otimes\Field_p$, we now have a canonical $F$-gauge $\mathcal{M}_1(\mathfrak{g})$ obtained by twisting the adjoint representation of $G$ by the tautological $G$-torsor over $(\BT{1}\otimes\Field_p)^{\mathrm{syn}}$. Then we find that $\BT{n}\otimes\Field_p\to \BT{n-1}\otimes\Field_p$ is a torsor under $\Gamma_{\mathrm{syn}}(\mathcal{M}_1(\mathfrak{g}[1]))$, which, by assertion (2) of Corollary~\ref{cor:1_bounded_level_1_repble} is a smooth, surjective Artin $1$-stack over $\BT{1}\otimes\Field_p$. 
\end{proof}

\begin{remark}
Remark~\ref{rem:independence_of_mu} shows that the stacks $\BT{n}$ depend only on the conjugacy class of $\mu$. 
\end{remark}

\begin{remark}
The smooth map $\BT{1}\otimes\Field_p\to \Disp{1}$ appearing in the proof above is nothing but the canonical one from Remark~\ref{rem:pullback_to_frames}. One can also consider the pullback maps $\BT{n}\otimes\Field_p\to \Disp{n}$ for all $n\ge 1$ associating an $n$-truncated $(G,\mu)$-display to every $n$-truncated $(G,\mu)$-aperture in characteristic $p$. Following Lau~\cite{MR2983008} and~\cite{drinfeld2024laugroupscheme}, one would expect that this map is a gerbe for a finite flat group scheme killed by the $n$-th power of Frobenius. 
\end{remark}

\begin{remark}
\label{rem:connectedness}
If $G$ is a connected group scheme, $\BT{n}$ is a connected $p$-adic formal Artin stack over $\mathcal{O}$. To begin, $\Disp{1}$ is connected by the quotient presentation from~\cite[Proposition 3.11]{Pink2015-ye}. Since $\BT{1}\otimes\Field_p$ is a gerbe over $\Disp{1}$ banded by a finite flat group scheme, we find that it is also connected. 

The description in the proof of Theorem~\ref{thm:main_thm_body} now shows that it is enough to know the following: If $\mathcal{F}$ is a vector bundle $F$-gauge over $R$ of level $1$ and Hodge-Tate weights $\{0,1\}$, then $\Gamma_{\mathrm{syn}}(\mathcal{F}[1])$ is a connected algebraic stack over $R$. Once again, it is a gerbe banded by a finite flat group scheme over $\Gamma_{\Fzip}(\bm{F}[1])$, where $\bm{F}$ is the associated vector bundle $F$-zip, and the desired connectedness can be checked directly for the latter stack.

In fact, the argument shows that in general we have a bijection
\[
\pi_0(\BT{n}\otimes\Field_p)\xrightarrow{\simeq}\pi_0(\Disp{1}),
\]
where the target is a certain quotient of $\pi_0(G\otimes\Field_p)$.
\end{remark}

\begin{remark}
\label{rem:adjoint_nilpotent_locus}
Let $\BT[G,\mu,\mathrm{nilp}]{n}\to \BT{n}$ be the nilpotent locus defined in Remark~\ref{rem:nilpotent_locus_general}: This satisfies a Grothendieck-Messing theory for not necssarily nilpotent divided power thickenings. One sees that this is the locus where the Lau group scheme has connected dual. Allowing $n\to \infty$, what one has here is the syntomic analogue of the \emph{adjoint nilpotent locus} considered in~\cite[\S 3.5]{MR4120808}.  

In fact, Remark~\ref{rem:nilpotent_locus_general} tells us that pulling back to the Witt frame produces a map of smooth formal stacks $\BT{\infty}\to \Disp{\infty}$ that is an isomorphism over the nilpotent locus. Therefore, the theory explained here recovers exactly that of B\"ultel-Pappas over the adjoint nilpotent locus.

More can be said: Let $(\underline{A},\zeta)$ be a filtered prism. Then we have natural maps (see Remark~\ref{rem:pullback_to_frames}) 
\[
\BT{n}(R_A)\to \Wind{n}{\underline{A}}(R_A)\to \Disp{n}(R_A).
\]
Remark~\ref{rem:nilpotent_sections-any_frame} shows that all these maps are equivalences when restricted to the adjoint nilpotent locus. One can use this observation to recover the main result of~\cite{bueltel2020gmuwindowsdeformationsgmudisplays}.
\end{remark}

\subsection{The case of trivial $\mu$}
\label{subsec:trivial_mu}

When $\mu = 0$ is the \emph{trivial} cocharacter $z\mapsto 1$, then we have $P^-_\mu = G$, and Theorem~\ref{thm:main_thm_body} tells us that $\BT[G,0]{n}$ is an \'etale $p$-adic formal stack over $\Int_p$ with affine \'etale diagonal.

\begin{proposition}
\label{prop:trivial_mu}
Suppose that $G$ is connected. Then there is a canonical equivalence
\[
\BT[G,0]{n}\xrightarrow{\simeq}B\underline{G(\Int/p^n\Int)},
\]
where the right hand side is the classifying stack of the locally constant group scheme $\underline{G(\Int/p^n\Int)}$.
\end{proposition}
\begin{proof}
Suppose that we have $\mathcal{Q}\in \BT[G,0]{n}(R)$. Consider the assignment $\Gamma_{\mathrm{syn}}(\mathcal{Q})$ on $\mathrm{CRing}^{p\text{-nilp}}_{R/}$ given by:
\[
C\mapsto \Map_{BG(C^{\mathrm{syn}}\otimes\Int/p^n\Int)}(G\rvert_{C^{\mathrm{syn}}\otimes\Int/p^n\Int},\mathcal{Q}).   
\]
This is represented over $R$ by
\[
\Spec R\times_{G,\BT[G,0]{n},\mathcal{Q}}\Spec R,
\]
which is the pullback along $(G,\mathcal{Q})$ of the diagonal of $\BT[G,0]{n}$ and is hence affine and \'etale over $R$.

We claim that this is an \'etale $G(\Int/p^n\Int)$-torsor over $R$, and hence gives a canonical map
\begin{align}\label{eqn:can_map_trivial_mu}
\BT[G,0]{n}\to B\underline{G(\Int/p^n\Int)}
\end{align}

To see this, it suffices to know that it is faithfully flat and that the natural map $\underline{G(\Int/p^n\Int)}\to \Gamma_{\mathrm{syn}}(G)$ of \'etale schemes is an equivalence.

Both assertions can be checked over algebraically closed fields $\kappa$ over $R$. Here, using Remark~\ref{rem:double_quotient_desc} and Lemma~\ref{lem:quotient_stack_desc}, we see that we have
\[
\BT[G,0]{n}(\kappa)\simeq [G(\Prism_{\kappa}/p^n\Prism_{\kappa})\dbqt{\sigma}{\mathrm{id}}G(\Prism_{\kappa}/p^n\Prism_{\kappa})],
\]
where $\sigma:G(\Prism_{\kappa}/p^n\Prism_{\kappa})\to G(\Prism_{\kappa}/p^n\Prism_{\kappa})$ is pullback under the Frobenius lift. Using the connectedness of $G$ and Lang's theorem, one sees that the natural map of groupoids
\[
[*/G(\Int/p^n\Int)]\to [G(\Prism_{\kappa}/p^n\Prism_{\kappa})\dbqt{\sigma}{\mathrm{id}}G(\Prism_{\kappa}/p^n\Prism_{\kappa})]
\]
is an equivalence. This proves both assertions, and also shows that the canonical map~\eqref{eqn:can_map_trivial_mu} is an equivalence.
\end{proof}

\subsection{Special \'etale loci}
\label{subsec:special_open}

In this subsection, we will assume that we have a smooth subgroup scheme $H\subset G$ such that $\mu$ factors through $H_{\mathcal{O}}$, so that we have maps $\BT[H,\mu]{n}\to \BT{n}$ and $\Disp[H,\mu]{n}\to \Disp{n}$. We will also make the following `versality' assumption:
\[
\Lie H + \Lie P^-_\mu = \Lie G \Leftrightarrow \Lie U^+_\mu \subset \Lie H.
\]

\begin{proposition}
\label{prop:special_open}
Under these assumptions:
\begin{enumerate}
   \item The map $\BT[H,\mu]{n}\to \BT{n}$ is \'etale.
   \item Suppose that $ P^+_\mu\subset H$. Then $\Disp[H,\mu]{1}\to \Disp{1}$ is an open immersion and the diagram
   \[
   \begin{diagram}
   \BT[H,\mu]{n}\otimes\Field_p&\rTo&\BT{n}\otimes\Field_p\\
     \dTo&&\dTo\\
     \Disp[H,\mu]{1}&\rTo&\Disp{1}
   \end{diagram} 
   \]
   is Cartesian. In particular, $\BT[H,\mu]{n}$ is an open substack of $\BT{n}$.
\end{enumerate}
\end{proposition} 
\begin{proof}
Let us check that $\BT[H,\mu]{n}\to \BT{n}$ is \'etale. By Grothendieck-Messing theory, this comes down to checking that the map
\[
B(H\cap P_\mu^-)\to BH\times_{BG}BP_\mu^-
\]
is \'etale. Looking at tangent complexes, this reduces to knowing that the map of complexes\footnote{The fiber product involved here is the \emph{homotopy} fiber product.}
\[
\Lie(H\cap P_\mu^-)[1] \to (\Lie H)[1]\times_{(\Lie G)[1]}(\Lie P_\mu^-)[1]
\]
is an isomorphism. This translates to the concrete claim that $\Lie H+\Lie P^-_\mu = \Lie G$, which is of course our assumption.

Now, consider (2): Suppose that $\bm{F} = (\mathcal{F},\mathcal{F}^+,\mathcal{F}^-,\eta^+,\eta^-,\alpha)$ is a tuple corresponding to a point of $\Disp{1}$. Since $P^+_{\mu^\varphi}\subset H$, we can consider the $H$-torsor $\mathcal{F}_H$ obtained from $\mathcal{F}^+$: this gives a reduction of structure group for $\mathcal{F}$. Unwinding definitions, one sees that the fiber of $\Disp[H,\mu]{1}$ over $\bm{F}$ parameterizes compatible reductions of structure group for $\mathcal{F}^-$ to an $H\cap P_\mu^{-}$-torsor. However, our assumption Lie algebras shows that the map
\[
\mathcal{F}^-/(H\cap P_\mu^-) \to \mathcal{F}/H
\]
is an open immersion. This shows that $\Disp[H,\mu]{1}$ is an open substack of $\Disp{1}$.

Now, the map $\BT[H,\mu]{1}\otimes\Field_p\to (\BT{1}\otimes\Field_p)\times_{\Disp{1}}\Disp[H,\mu]{1}$ is an isomorphism, since they are both gerbes over $\Disp[H,\mu]{1}$ banded by the same finite flat group scheme.

To finish, it suffices to know that the natural map
\[
\BT[H,\mu]{n+1}\otimes\Field_p\to (\BT{n+1}\otimes\Field_p)\times_{\BT{n}\otimes\Field_p}(\BT[H,\mu]{n}\otimes\Field_p)
\]
is an isomorphism for all $n\ge 1$. By the proof of Theorem~\ref{thm:main_thm_body}, the source (resp. target) is a torsor under $\Gamma_{\mathrm{syn}}(\mathcal{M}_1(\Lie H)[1])$ (resp. $\Gamma_{\mathrm{syn}}(\mathcal{M}_1(\Lie G)[1])$). Therefore, it suffices to know that the map
\[
\Gamma_{\mathrm{syn}}(\mathcal{M}_1(\Lie H)[1])\to \Gamma_{\mathrm{syn}}(\mathcal{M}_1(\Lie G)[1])
\]
of group stacks over $\BT[H,\mu]{1}\otimes\Field_p$ is an equivalence. In turn, this comes down to knowing that $\Gamma_{\mathrm{syn}}(\mathcal{M}_1(\Lie G/\Lie H)[1])$ is the trivial group stack. Now, the condition $\Lie P_\mu^+\subset \Lie H$ ensures that $\mathcal{M}_1(\Lie G/\Lie H)[1]$ is a perfect $F$-gauge over $\BT[H,\mu]{1}\otimes\Field_p$ of Hodge-Tate weights $\leq -1$, so we are done by Lemma~\ref{lem:ht_wts_le_-1} below.
\end{proof}

\begin{lemma}
\label{lem:ht_wts_le_-1}
Suppose that $\mathcal{M}$ is a perfect $F$-gauge of level $1$ over $R\in \mathrm{CRing}^{p\text{-nilp}}$ with Hodge-Tate weights $\leq -1$. Then $\Gamma_{\mathrm{syn}}(\mathcal{M})\simeq 0$.
\end{lemma} 
\begin{proof}
By Theorem~\ref{thm:sections_1-bounded_representable}, $\Gamma_{\mathrm{syn}}(\mathcal{M})$ is represented by a derived Artin stack with cotangent complex given by the pullback of $(\gr^{-1}_{\mathrm{Hdg}}M_1)^\vee$. Our conditions on the Hodge-Tate weights ensure that this last perfect complex is trivial, and so we find that $\Gamma_{\mathrm{syn}}(\mathcal{M})$ is \'etale over $R$. To check that it is trivial, it suffices now to check on points valued in an algebraically closed field $\kappa$. Here, the complex from which $\Gamma_{\mathrm{syn}}(\mathcal{M})(\kappa)$ is computed is of the form $M\xrightarrow{\varphi_0- \mathrm{id}}M$ where $M$ is a perfect complex over $\kappa$. The hypothesis on Hodge-Tate weights tells us that $\varphi_0$ is divisible by $p$, and is hence nullhomotopic. Therefore, $\Gamma_{\mathrm{syn}}(\mathcal{M})(\kappa)$ is trivial as desired.
\end{proof} 

\begin{remark}
\label{rem:ordinary_locus}
Here is a key example: Suppose that $\mu$ is defined over $\Int_p$. Then it is easy to see that $P^+_\mu$ satisfies the assumptions. The open locus $\BT[P^+_\mu,\mu]{n}\subset \BT{n}$ is the \defnword{ordinary locus} of $\BT{n}$, and assertion (2) of the proposition is the analogue of fact that the ordinariness of an $n$-truncated Barsotti-Tate group in characteristic $p$ can be checked by considering the $1$-truncated display associated with it. In fact, via Theorem~\ref{thm:g_functor} below, this last statement is basically a special case of the proposition. We should note however that the proof of that theorem uses \emph{a priori} knowledge of this density for the moduli of $F$-gauges; see Lemma~\ref{lem:generically_ordinary}.
\end{remark}

\begin{remark}
\label{rem:mu_ordinary_locus}
Suppose that $G$ is reductive, and that we have a maximal torus $T$ of $G$ along with a Borel $B$ containing $T$. If we choose $\mu$ to be a dominant (with respect to $B$) cocharacter of $T_{\mathcal{O}}$, then the sum of its Galois conjugates gives a cocharacter $\nu$ of $T$ defined over $\Int_p$. It can be checked now that the subgroup $P^+_\nu\subset G$ satisfies the assumptions above. The corresponding \'etale locus is the \defnword{$\mu$-ordinary locus}, studied in~\cite{Moonen2004-cc} in the context of $p$-divisible groups with additional endomorphisms. This is explored in more detail in~\cite{MadYoucis} in both the local and global contexts.
\end{remark}

\section{Explicit descriptions of $\BT{n}(R)$}
\label{sec:explicit}

In this section, we will see, following the deformation-theoretic method of Ito~\cite{ito2023deformation}, that the above theory yields explicit descriptions for $\BT{n}(R)$ in certain cases as the groupoid of $n$-truncated $(G,\mu)$-windows over some quite concrete frames. All objects in this subsection will be discrete, unless otherwise noted, so we are back on firm classical ground. 

We will also find that the deformation rings constructed by Faltings in~\cite{faltings:very_ramified}, and which play a key role in the construction of integral canonical models in~\cite{kisin:abelian} admit a clean interpretation as universal deformation rings for $\BT{\infty}$.

These results address conjectures formulated by Ito in~\cite[\S 7]{ito2023deformation}.

\subsection{An explicit description over some classical rings}
\label{subsec:bk_faltings_expl}

\subsubsection{}
We will put ourselves in the following situation (compare with~\cite[\S 6]{lau:displays}): 
\begin{itemize}
   \item $(S,I'=(E))$ will be an oriented prism, flat over $\Int_p$, with associated Frobenius lift $\varphi:S\to S$;
   \item $J\subset S$ will be a finitely generated ideal such that $\varphi(J) \subset J^2$;
   \item We will assume that that $S$ is $J$-adically complete, and that $E$ and $p$ map to non-zero divisors in $S/J^m$ for all $m\ge 1$. 
\end{itemize}
For $m\ge 1$, set $S_m = S/J^m$. Then $\varphi$ descends to an endomorphism of $S_m$.

We will associate with this data the following filtered prisms that are special cases of Example~\ref{example:bk_frames}: We set $I = \varphi(I')\subset S$. For each $m\ge 1$, we define $\underline{S}_m$ to be the frame with underlying non-negatively filtered ring $\Fil^\bullet_{I'} S_m$ with filtered Frobenius $\Fil^\bullet_{I'} S_m\to \Fil^\bullet_{I} S_m$. Set $R_m \defn S/(J^m+I) = S_m/\Fil^1_{I'}S_m$ and $R = S/I$, so that we have $R = \varprojlim_m R_m$.

Repeating the above construction with $S_m$ replaced with $S$ gives a frame $\underline{S}$ with $R_S = R$. We then have maps of frames $\underline{S}_{m+1}\to \underline{S}_m$ for each $m\ge 1$, and also
    \[
     \underline{S}\xrightarrow{\simeq}\varprojlim_{m}\underline{S}_m.
    \]

\begin{proposition}
\label{prop:bk_def_theory}
There is a canonical map
\[
\BT{n}(R_m)\to \Wind{\underline{S}_m}{n}(R_m),
\]
and the square
\[
\begin{diagram}
   \BT{n}(R_{m+1})&\rTo& \Wind{\underline{S}_{m+1}}{n}(R_{m+1})\\
    \dTo&&\dTo\\
   \BT{n}(R_m)&\rTo& \Wind{\underline{S}_{m}}{n}(R_{m})
\end{diagram}
\]
is Cartesian. In particular, we have
\[
\BT{n}(R)\simeq \BT{n}(R_1)\times_{\Wind{\underline{S}_{1}}{n}(S_{1})}\Wind{\underline{S}}{n}(R)
\]
\end{proposition}
\begin{proof}
For every $m\ge 1$, we have a canonical map $\BT{n}(R_m)\to \Wind{\underline{S}_m}{n}(R_m)$ obtained from Example~\ref{example:bk_frames} and Proposition~\ref{prop:frames_to_nygaard} via pullback along the map $\iota_{\underline{S}_m}$. We claim now that there exists a commuting diagram
\begin{align}\label{eqn:cartesian_bk_disp}
\begin{diagram}
\BT{n}(R_{m+1})&\rTo&\Wind{\underline{S}_{m+1}}{n}(R_{m+1})&\rTo&BP^{-,(n)}_\mu(R_m)\\
\dTo&&\dTo&&\dTo\\
\BT{n}(R_m)&\rTo& \Wind{\underline{S}_{m}}{n}(R_{m})&\rTo&BP^{-,(n)}(R_m)\times_{BG^{(n)}(R_m)}BG^{(n)}(R_{m+1})
\end{diagram}
\end{align}
where the outside rectangle is the Cartesian square obtained from Theorem~\ref{thm:main_thm_body} (and the trivial divided powers on $R_{m+1}\twoheadrightarrow R_m$), and where the square on the right is also Cartesian. This will of course prove the first part of the proposition. The last part will follow by taking the limit over $m$ on the equivalence
\[
\BT{n}(R_m)\xrightarrow{\simeq}\BT{n}(R_1)\times_{\Wind{\underline{S}_1}{n}(S_1)}\Wind{\underline{S}_m}{n}(R_m).
\]
That this limit yields the desired equivalence follows from~\cite[Corollary 1.5]{BHATT2017576}.

To prove the claim, first note the following explicit description of the Rees algebra for $\Fil^\bullet_IS_{m+1}$:
\[
S_{m+1}[u,t]/(ut-E)\xrightarrow[\simeq]{u\mapsto Et^{-1},t\mapsto t}\bigoplus_i\Fil^i_IS_{m+1}\cdot t^{-i}.
\]
Via this description, the map $\sigma:\Spf S_{m+1}\to \Rees(\Fil^\bullet_{I}S_{m+1})$ corresponds to the map of $S_{m+1}$-algebras given by
\[
S_{m+1}[u,t]/(ut-E)\xrightarrow{u\mapsto \varphi(E),t\mapsto 1}\varphi_*S_{m+1}.
\]
Since $\varphi(J)\subset J^2$, this map factors through $S_m[u,t]/(ut-E)$, which shows that $\sigma$ factors through a map
\[
\overline{\sigma}:\Spf S_{m+1}\to \Rees(\Fil^\bullet_I S_m).
\]

Therefore, the existence of the right Cartesian square in~\eqref{eqn:cartesian_bk_disp} is now due to Proposition~\ref{prop:hokey_def_theory}. 

To finish, we only need to verify that the map $\BT{n}(R_m)\to BG^{(n)}(R_{m+1})$ arising from the composition of the bottom horizontal arrows in~\eqref{eqn:cartesian_bk_disp} agrees with that showing up in Theorem~\ref{thm:main_thm_body}. In turn, this comes down to knowing that the composition
\[
\Spf R_{m+1}\to \Spf S_{m+1}\xrightarrow{\overline{\sigma}}\Rees(\Fil^\bullet_IS_m)\xrightarrow{\iota_{\underline{S}_m}}R_m^{\mathcal{N}} 
\]
agrees with the map
\[
\Spf R_{m+1}\xrightarrow{\tilde{x}_{\dR,R_{m+1}}}R_m^{\Prism}\xrightarrow{j_{\mathrm{HT}}} R_m^{\mathcal{N}}
\]
obtained from the trivial divided powers on $R_{m+1}\twoheadrightarrow R_m$. This is an easy check from the constructions.
\end{proof}

\begin{example}
\label{ex:BK_situation} 
 Let $\kappa$ be a perfect field in characteristic $p$, with associated ring of Witt vectors $W(\kappa)$. We set $S=W(\kappa)\pow{t_1,\ldots,t_r}$ for some $n\ge 0$, and take $\varphi$ to be the Frobenius lift with $\varphi(t_i) = t_i^p$.  Take $J=(t_1,\ldots,t_n)$, so that we have $\varphi(J)\subset J^p$. Suppose that $E$ satisfies $\varphi(E)\equiv E^p\pmod{p}$, and is such that $S/(E)$ is $p$-torsion free.

Here, $R_m = S/((E)+J^m)$ with $R_1 = S/((E)+J) = W(\kappa)/(p) = \kappa$, and $\underline{S}_1$ is isomorphic to the frame $\underline{\Prism}_\kappa$. In particular, the map
\[
\BT{n}(\kappa) = \BT{n}(R_1)\to \Wind{\underline{S}_1}{n}(\kappa)
\]
is an equivalence, and so we conclude that we have
\[
\BT{n}(R) \xrightarrow{\simeq} \Wind{\underline{S}}{n}(R).
\]
Note that the argument actually shows that we have
\[
\BT{n}(R_m)\xrightarrow{\simeq}\Wind{\underline{S}_m}{n}(R_m)
\]
for all $m\ge 1$.

Using Remark~\ref{rem:dilatations}, we can give a quite explicit description of $\Wind{\underline{S}_m}{\infty}(R_m)$: If $H_\mu$ is the dilatation of $G_S$ along $P^-_\mu\otimes R$, viewed as an \'etale sheaf over $S$, along with the natural map $\tau:H_\mu\to G_S$ as well as the map $\sigma = \varphi\circ\mathrm{int}(E)$, then an object in $\Wind{\underline{S}_m}{\infty}(R_m)$ is an $H_\mu$-torsor over $S_m$ along with an isomorphism of $G$-torsors $\sigma^*\mathcal{P}\xrightarrow{\simeq}\tau^* \mathcal{P}$.

The case where $n=1$ and $E$ is an Eisenstein polynomial is the context for the classical story of Breuil-Kisin modules. In this case, $R=S/(E)$ is a totally ramified ring of integers over $W(\kappa)$, and $R_m = R/(\pi^m)$ where $\pi\in R$ is a uniformizer with minimal polynomial $E$. 

This proves part (1) of~\cite[Conjecture 7.1.2]{ito2023deformation}.
\end{example}

\begin{example}
\label{ex:non_noetherian_situation}
We have a non-Noetherian analogue of the previous example by taking $\underline{S} = \underline{\Prism}_R$ for a perfectoid ring $R$, $E = \varphi^{-1}(\xi)$, where $\xi$ is a generator for $\ker(\theta:\Prism_R\to R)$, and $J = ([\varpi_1],[\varpi_2],\ldots,[\varpi_m])$ to be an ideal generated by Teichm\"uller lifts of topologically nilpotent elements $\varpi_i\in R^\flat$ that form a regular sequence. Suppose in addition that $[\varpi_1],\ldots,[\varpi_m],\xi$ also forms a regular sequence in $W(R^\flat)$. Then we are in a special case of the situation above with $S_1 = \Prism_R/J$, and $R_1 = R/\theta(J)$. 

In this case, we already know that $\BT{n}(R) \simeq \Wind{\underline{S}}{n}(R)$ by Lemma~\ref{lem:quotient_stack_desc}. But Proposition~\ref{prop:bk_def_theory} tells us that we also have
\[
\BT{n}(R)\simeq \BT{n}(R_m)\times_{\Wind{n}{\underline{S}_m}(R_m)}\Wind{n}{\underline{S}}(R).
\]
Since $\BT{n}(R)\to \BT{n}(R_m)$ is an effective epimorphism, this tells us that we in fact have
\[
\BT{n}(R_m)\simeq \Wind{n}{\underline{S}_m}(R_m).
\]
for all $m\ge 1$. This proves part (2) of~\cite[Conjecture 7.1.2]{ito2023deformation}

It also recovers---via Theorem~\ref{thm:dieudonne} below---a description of $p$-divisible groups over $R_m$ which was first observed by Ito~\cite[Theorem 6.3.6]{ito2023deformation} in the following situation: $R = \Reg{C}$ is the ring of integers in a perfectoid field $C$, $m=1$ and $\varpi = \varpi_1\in \Reg{C^\flat}$ is a topologically nilpotent non-zero element. See Remark~\ref{rem:non_noetherian_pdiv} below. 
\end{example} 

\begin{example}
\label{ex:faltings_situation}
  Let $S$ and $J$ be as in Example~\ref{ex:BK_situation}, but assume now that $E = p$, so that we have a \emph{crystalline} prism $(S,(p))$. In this case, and we obtain an equivalence
    \[
    \BT{n}(\kappa\pow{t_1,\ldots,t_r})\xrightarrow{\simeq}\Wind{\underline{S}}{n}(\kappa\pow{t_1,\ldots,t_r})\times_{\Wind{\underline{S}_1}{n}(\kappa)}\BT{n}(\kappa).
    \]
    But note that $\underline{S}_1$ is simply the frame $\underline{\Prism}_\kappa$, and so as in \emph{loc. cit.} the map $\BT{n}(\kappa)\to \Wind{\underline{S}_1}{n}(\kappa)$ is an equivalence. This gives us an equivalence:
   \[
     \BT{n}(\kappa\pow{t_1,\ldots,t_r})\xrightarrow{\simeq}\Wind{\underline{S}}{n}(\kappa\pow{t_1,\ldots,t_r}).
    \]

   By Remark~\ref{rem:dilatations}, we obtain a rather explicit description of the limiting groupoid $\Wind{\underline{S}}{\infty}(\kappa\pow{t_1,\ldots,t_r})$: Its objects are $H_\mu$-torsors $\mathcal{P}$ over $S$ equipped with an isomorphism $\sigma^* \mathcal{P}\xrightarrow{\simeq}\tau^* \mathcal{P}$ of $G$-torsors. Here, $H_\mu$ is the dilatation of $G_{\mathcal{O}}$ along $P^-_\mu\otimes k$ and $\sigma = \varphi\circ \mathrm{int}(\mu(p))$, while $\tau$ is the natural map as usual.
 \end{example}

\subsection{Relationship with Faltings deformation rings}
\label{subsec:deformation_rings}
Here we will find that the deformation rings of $\BT{\infty}$ can be described explicitly using results from the beginning of this subsection, combined with a construction of Faltings~\cite[\S 7]{faltings:very_ramified}. Compare with the main result of Ito in~\cite{ito2023deformation}, where one finds a version of such a result. There, however, the Faltings deformation space is only shown to have good descriptions for particular inputs from $\mathrm{Art}_{W(\kappa)}$; but, in~\cite{imai2023prismatic}, the authors combine Ito's work with ours here to give a different proof.

\subsubsection{}
Maintain the notation from the previous subsection, but assume now that $I = (p)$; for instance, this is the case in the situation of Example~\ref{ex:faltings_situation}.  We can then also define frames $\widetilde{\underline{S}}$ and $\widetilde{\underline{S}}_m$, where we take the underlying filtered commutative ring to be $S$ (resp. $S_m$) with the \emph{trivial} filtration. We then have maps of frames $\widetilde{\underline{S}}_{m+1}\to \widetilde{\underline{S}}_m$ for each $m\ge 1$, and also
    \[
     \widetilde{\underline{S}}\xrightarrow{\simeq}\varprojlim_{m}\widetilde{\underline{S}}_m.
    \]
In this case, we have $\widetilde{S}_m/\Fil^1\widetilde{S}_m = S_m$ and $\widetilde{S}/\Fil^1\widetilde{S} = S$.

Let us now use Remark~\ref{rem:trivial_filtration_displays}. It tells us that $\Wind{\infty}{\widetilde{\underline{S}}_m}(S_m)$ can be described quite explicitly: Giving an object here is equivalent to giving a $P^-_\mu$-torsor $\mathcal{P}'$ over $S_m$ along with an isomorphism $\sigma^* \mathcal{P}'\xrightarrow{\simeq}\tau^* \mathcal{P}'$ of $G$-torsors. Here, $\tau:P^-_\mu\to G_{\mathcal{O}}$ is the natural map and $\sigma = \varphi\circ \mathrm{int}(\mu(p))$.

In other words, we have
\begin{align}
\label{eqn:s_to_tilde_s_cartesian}
\Wind{\infty}{\underline{\widetilde{S}}_m}(S_m)\xrightarrow{\simeq}\Wind{\infty}{\underline{S}_m}(R_m)\times_{BP^-_\mu(R_m)\times_{BG(R_m)}BG(S_m)}BP^-_\mu(S_m).
\end{align}

\subsubsection{}
We now have a canonical map
\[
\BT{\infty}(S_m)\to \Wind{\infty}{\underline{S}_m}(R_m)\times_{BP^-_\mu(R_m)\times_{BG(R_m)}BG(S_m)}BP^-_\mu(S_m)\simeq \Wind{\infty}{\widetilde{\underline{S}}_m}(S_m)
\]
where the first coordinate is obtained from the composition
\[
\BT{\infty}(S_m)\to \BT{\infty}(R_m)\to \Wind{\infty}{\underline{S}_m}(R_m),
\]
while the second is pullback along $x^{\mathcal{N}}_{\dR,S_m}$. From Proposition~\ref{prop:bk_def_theory} and Grothendieck-Messing theory for $\BT{\infty}$, one finds that, for each $m\ge 1$, there is a Cartesian square:
\begin{align}
\label{eqn:faltings_cartesian_square}
\begin{diagram}
\BT{\infty}(S_{m+1})&\rTo& \Wind{\widetilde{\underline{S}}_{m+1}}{\infty}(S_{m+1})\\
    \dTo&&\dTo\\
   \BT{\infty}(S_m)&\rTo& \Wind{\widetilde{\underline{S}}_{m}}{\infty}(S_{m})
\end{diagram}
\end{align}

\begin{remark}
\label{rem:trivial_filt_groth_messing}
When $p>2$, one can use Grothendieck-Messing theory to show that in fact all the horizontal arrows in the above square are equivalences. If $p=2$, then this assertion fails already for $m=1$. 
\end{remark}

\subsubsection{}
Let $\kappa$ be a perfect field, and suppose that we have a point $x \in \BT{\infty}(\kappa)$. We can then consider the deformation problem on the usual category $\mathrm{Art}_{W(\kappa)}$ of Artin local $W(\kappa)$-algebras with residue field $\kappa$:
\begin{align*}
\mathrm{Def}_{x}:\mathrm{Art}_{W(\kappa)}&\to \mathrm{Spc}\\
A&\mapsto \mathrm{fib}_{x}(\BT{\infty}(A)\to \BT{\infty}(\kappa)).
\end{align*}

Grothendieck-Messing theory now tells us that, if $A'\twoheadrightarrow A$ is a square-zero thickening in $\mathrm{Art}_{W(\kappa)}$, then we have a Cartesian square
\[
\begin{diagram}
\BT{\infty}(A')&\rTo&BP_\mu^{-}(A')\\
\dTo&&\dTo\\
\BT{\infty}(A)&\rTo&BP_\mu^-(A)\times_{BG(A)}BG(A').
\end{diagram}
\]

Using this, we find:
\begin{lemma}
\label{lem:discreteness_def_theory}
For each $A\in \mathrm{Art}_{W(\kappa)}$, $\mathrm{Def}_x(A)$ is equivalent to a set, and $\mathrm{Def}_x$ is prorepresented by $\Spf R_x$ with $R_x\simeq W(\kappa)\pow{t_1,\ldots,t_d}$ where $d = \dim G - \dim P_\mu^-$.
\end{lemma}

\subsubsection{}
We begin by reformulating Faltings's construction in the language of torsors. Choose a lift $x'\in \BT{\infty}(W(\kappa))$, which in turn yields a window (or more precisely a compatible family of $n$-truncated windows) over the frame $\widetilde{\underline{S}}_1$ associated with the trivial filtration on $W(\kappa)$. 

Explicitly, this means the following: Let $H_\mu$ be the dilatation of $G_{\mathcal{O}}$ along $P^-_\mu\otimes k$. Under the equivalence 
\[
\BT{\infty}(\kappa) \xrightarrow{\simeq}\varprojlim_n \Wind{n}{\underline{\Prism}_\kappa}(\kappa),
\]
and Remark~\ref{rem:dilatations}, $x$ corresponds to an $H_\mu$-torsor $\mathcal{P}_x$ over $W(\kappa)$, equipped with an isomorphism $\sigma^* \mathcal{P}_x\xrightarrow{\simeq}\tau^*\mathcal{P}_x$ of $G$-torsors over $W(\kappa)$. The lift $x'$ gives rise to an object of $\Wind{\infty}{\widetilde{\underline{S}}_1}(W(\kappa))$, which, by Remark~\ref{rem:trivial_filtration_displays}, amounts to refining the $H_\mu$-torsor $\mathcal{P}_x$ to a $P^-_\mu$-torsor $\mathcal{P}_{x'}$ over $W(\kappa)$.

\subsubsection{}
Set $G_x = \Aut(\tau^*\mathcal{P}_x)$ and $P^-_x = \Aut(\mathcal{P}_{x'})$, so that $G_x$ is a pure inner form over $W(\kappa)$ of $G$, and $P^-_x\subset G_x$ is associated with a cocharacter $\mu_x:\Gm\to G_x$ that is conjugate to $\mu$, via the process explained in~\S\ref{subsec:cochar}. 

For such a choice of cocharacter, we can look at the `opposite' unipotent $U^+_x\subset G_x$: this is a commutative unipotent group scheme over $W(\kappa)$ (see Lemma~\ref{lem:1-bounded_exp}). We now define $R^{\mathrm{Fal}}_x$ to be the complete local ring of $U^+_x$ at the identity: this is abstractly isomorphic to $W(\kappa)\pow{t_1,\ldots,t_d}$ as a $W(\kappa)$-algebra.

We equip $R^{\mathrm{Fal}}_x$ with the Frobenius lift $\varphi$ arising from the $p$-power map on $U^+_x$, and take $J_x\subset R^{\mathrm{Fal}}_x$ to be the augmentation ideal: note that we have $\varphi(J_x)\subset J_x^p$. 

\subsubsection{}
We can now apply the setup from the beginning of the subsection with $(S,I') = (R^{\mathrm{Fal}}_x,(p))$ and $J = J_x$, and we find that we have:
\[
\mathrm{fib}_{x'}(\BT{\infty}(R^{\mathrm{Fal}}_x) \to \BT{\infty}(W(\kappa)))\xrightarrow{\simeq}\mathrm{fib}_{\mathcal{P}_{x'}}(\Wind{\widetilde{\underline{S}}}{\infty}(R^{\mathrm{Fal}}_x)\to \Wind{\widetilde{\underline{S}}_1}{\infty}(W(\kappa))).
\]

One way to get an object on the right is as follows: Let $j:W(\kappa)\to R^{\mathrm{Fal}}_x$ be the structure map: this actually underlies a map of frames $\underline{\widetilde{S}}_1\to \underline{\widetilde{S}}$, and so we can pull $\mathcal{P}_{x'}$ back to get the `constant' lift $\mathcal{P}^{\mathrm{con}}_{x'}$ over $\underline{\widetilde{S}}$. More precisely, this corresponds to the $P^-_\mu$-torsor $\mathcal{P}^{\mathrm{con}}_{x'}$ over $R^{\mathrm{Fal}}_x$, along with an isomorphism of $G$-torsors $\xi^{\mathrm{con}}_{x'}:\sigma^*\mathcal{P}^{\mathrm{con}}_{x'}\xrightarrow{\simeq}\tau^*\mathcal{P}^{\mathrm{con}}_{x'}$. All of this data is obtained simply via pullback from the corresponding data over $W(\kappa)$.

In $U^+_x(R^{\mathrm{Fal}}_x)$, we have the tautological element $g_x$. We now define a new display $\mathcal{P}^{\mathrm{Fal}}_{x'}$ by keeping the $P^-_\mu$-torsor $\mathcal{P}^{\mathrm{con}}_{x'}$, but replacing $\xi^{\mathrm{con}}_{x'}$ with the composition $\xi^{\mathrm{Fal}}_{x'} = g_x\circ \xi^{\mathrm{con}}_{x'}$.

As explained above, this yields an object $x^{\mathrm{Fal}}\in \BT{\infty}(R^{\mathrm{Fal}}_x)$ lifting $x'\in \BT{\infty}(W(\kappa))$, and so corresponds to a unique map $R_x\to R^{\mathrm{Fal}}_x$.

The next result implies~\cite[Conjecture 7.2.2]{ito2023deformation}.

\begin{proposition}
\label{prop:faltings_def_rings}
The map $R_x\to R^{\mathrm{Fal}}_x$ is an isomorphism.
\end{proposition} 
\begin{proof}
Let $\widehat{U}_x$ (resp. $\widehat{U}^{\mathrm{Fal}}_x$) be the deformation functor on $\mathrm{Art}_{W(\kappa)}$ represented by $R_x$ (resp. $R_x^{\mathrm{Fal}}$). If $\kappa[\epsilon]$ is the ring of dual numbers, we obtain maps of tangent spaces
\[
\widehat{U}_x^{\mathrm{Fal}}(\kappa[\epsilon])\to \widehat{U}_x(\kappa[\epsilon])\xrightarrow{\simeq}\mathrm{fib}_{(\mathcal{P}_{x'},\sigma^*\mathcal{P}_{x'})}(BP^-_\mu(\kappa[\epsilon])\to BP^-_\mu(\kappa)\times_{BG(\kappa)}BG(\kappa[\epsilon])),
\]
where the second arrow is the isomorphism from Grothendieck-Messing theory.

The source of this composition is simply $\epsilon \kappa[\epsilon]\otimes_{W(\kappa)}\Lie U^+_x$, and one can check that the map takes a tangent vector $\epsilon N$ to $\exp(-\epsilon N)\cdot \mathcal{P}_{x'}$. In particular, it is an isomorphism onto its image.

Since both complete local rings are normal of the same dimension, the proposition now follows from Nakayama's lemma.
\end{proof}

\subsection{The case of central $\mu$}
\label{subsec:central}

Suppose that $\mu$ is \emph{central} in $G$. For instance, this is the case whenever $G$ is a torus over $\Int_p$. In this case, we have $P^-_\mu = G$, and so Theorem~\ref{thm:main_thm_body} shows that $\BT[G,\mu]{n}$ is an \'etale $p$-adic formal stack over $\mathcal{O}$. 

\subsubsection{}
We will now show that $\BT[G,\mu]{\infty}(\mathcal{O})$ is non-empty. Objects here will be called \defnword{Lubin-Tate $(G,\mu)$-apertures}. This terminology is partially justified by  Proposition~\ref{prop:lubin-tate_groups} below.

By Example~\ref{ex:BK_situation}, this is equivalent to writing down objects in $\Wind[G,\mu]{\underline{S}}{\infty}(\mathcal{O})$ where $S = W(k)\pow{u}$ is equipped with the structure of a Breuil-Kisin frame associated with $I' = (u-p)$.

In the notation of that example, we have $H_\mu = G$, and so a choice of trivialization of the module $S\{1\}$ now further identifies this groupoid with the groupoid of $G$-torsors $\mathcal{Q}$ over $S$ equipped with an isomorphism $\varphi^* \mathcal{Q}\xrightarrow{\simeq}\mathcal{Q}$. Clearly, the trivial $G$-torsor has such structure, and so gives us an object in $\BT[G,\mu]{n}(\mathcal{O})$; note that the object that it corresponds to is not canonical and depends on all the choices we made, including that of the frame $\underline{S}$ as well as the trivialization of the module $S\{1\}$.

\begin{remark}
Alternatively, we could have first constructed objects in $\BT[G,\mu]{\infty}(k)$ and then used the formal \'etaleness of $\BT[G,\mu]{\infty}$ to obtain Lubin-Tate $(G,\mu)$-apertures.
\end{remark}

\begin{remark}
\label{rem:direct_construction_Tmu}
It would be interesting to give a direct construction of these $G$-torsors over $\mathcal{O}^{\mathrm{syn}}$. When $\mu$ is defined over $\Int_p$, one can use the composition
\[
\Int_p^{\mathrm{syn}}\to B\Gm\xrightarrow{B\mu}BG
\]
where the first map classifies the Breuil-Kisin twist. 

Note that this gives a \emph{canonical} Lubin-Tate $(G,\mu)$-aperture over $\Int_p$. This is related to the fact that over $\Int_p$ we have a canonical choice of a Lubin-Tate formal group given by $\mu_{p^\infty}$.
\end{remark}

\begin{proposition}
\label{prop:central_mu}
Suppose that $G$ is connected. Then there is a \emph{non}-canonical isomorphism
\[
\BT[G,\mu]{n}\xrightarrow{\simeq}B\underline{G(\Int/p^n\Int)}
\]
of $p$-adic formal stacks over $\Spf \mathcal{O}$. More precisely, $\BT[G,\mu]{n}$ is a gerbe over $\Spf\mathcal{O}$ banded by $G(\Int/p^n\Int)$ that is non-canonically trivial.
\end{proposition}
\begin{proof}
The proof is the same as that of Proposition~\ref{prop:trivial_mu}. Instead of the trivial $G$-torsor, one uses one of the Lubin-Tate $(G,\mu)$-apertures constructed above.
\end{proof}

\subsection{The case of tori}
\label{subsec:tori}
As mentioned above, one special case is when $G = T$ is a torus. Here is a reinterpretation of Proposition~\ref{prop:central_mu}:

\begin{proposition}
\label{prop:tori}
$\BT[T,\mu]{n}$ is a non-canonically trivial $B\underline{T(\Int/p^n\Int)}$-torsor over $\Spf \mathcal{O}$.
\end{proposition} 

\subsubsection{}
As is well-known, there is an initial instance of data $(T,\mu)$ with $\mu$ defined over $\mathcal{O}$. Take $T_0 = \Res_{\mathcal{O}/\Int_p}\Gm$ and $\mu_0:\Gmh{\mathcal{O}}\to T_{0,\mathcal{O}}$ obtained as follows: We have $T_{0,\mathcal{O}} \simeq \prod_{i=0}^{h-1}\Gmh{\mathcal{O}}$, where $h = [\mathcal{O}[1/p]:\Rat_p]$, and the isomorphism is obtained from the map of $\mathcal{O}$-algebras
\begin{align*}
\mathcal{O}\otimes_{\Int_p}\mathcal{O}&\xrightarrow{\simeq}\prod_{i=0}^{h-1}\mathcal{O}\\
a\otimes b&\mapsto (a\varphi^i(b))_{0\leq i\leq h-1}.
\end{align*}
We now take $\mu_0$ to be the inclusion in the first factor. 

For any other $\Int_p$-torus $T$ with cocharacter $\mu:\Gmh{\mathcal{O}}\to T_{\mathcal{O}}$, we now see that the composition
\[
\Gmh{\mathcal{O}}\xrightarrow{\mu_0}T_0 = \Res_{\mathcal{O}/\Int_p}\Gm\xrightarrow{\Res_{\mathcal{O}/\Int_p}\mu}\Res_{\mathcal{O}/\Int_p}T_{\mathcal{O}}\xrightarrow{\mathrm{Nm}_{\mathcal{O}/\Int_p}}T
\]
is equal to $\mu$. 

We can understand $\BT[T_0,\mu_0]{\infty}$ somewhat explicitly. The following will be used to reinterpret it in terms of Lubin-Tate $\mathcal{O}$-modules in Proposition~\ref{prop:lubin-tate_groups}.

\begin{proposition}
\label{prop:lubin-tate}
Giving a $(T_0,\mu_0)$-aperture over $R\in \mathrm{CRing}^{p\text{-nilp}}_{\mathcal{O}/}$ is equivalent to giving a line bundle $\mathcal{F}$ over $R^{\mathrm{syn}}\times \Spf \mathcal{O}$ with the following property: For any algebraically closed field $\kappa$ over $R$, the restriction of $\mathcal{F}$ to 
\[
B\Gm\times\Spec(\kappa\otimes_{\Int_p}\mathcal{O}) \simeq \prod_{i=0}^{h-1}B\Gm\times\Spec \kappa
\]
corresponds to a graded projective module of rank $1$ over $\prod_{i=0}^{h-1}\kappa$ that is in graded degree $1$ for $i=0$ and in graded degree $0$ for $i>0$.
\end{proposition} 
\begin{proof}
This is simply a reinterpretation of the definition using the fact that we have $BT_0 \simeq \Res_{\mathcal{O}/\Gm}B\Gm$. Note that the action of $B\underline{T_0(\Int_p)}$ under this optic is just given by tensor product of line bundles, where we use Proposition~\ref{prop:central_mu} to view $B\underline{T_0(\Int_p)}$ as the stack of line bundles over $R^{\mathrm{syn}}\times\Spf \mathcal{O}$ whose restrictions to $\prod_{i=0}^{h-1}B\Gm\times\Spec \kappa$ have graded degree $0$ in every coordinate.
\end{proof}

\section{The classification of truncated Barsotti-Tate groups}
\label{sec:bt_classification}

\subsection{The statement of the theorem}

\subsubsection{}
Recall that an \defnword{$n$-truncated Barsotti-Tate group scheme} over a discrete ring $R\in \mathrm{CRing}^{p\text{-nilp}}_\heartsuit$ is a finite flat commutative group scheme $G$ over $R$ with the following properties:
\begin{enumerate}
   \item $G$ is $p^n$-torsion;
   \item The sequence $G\xrightarrow{p^{n-1}}G\xrightarrow{p}G$ is exact in the middle;
   \item If $n = 1$, over $R/pR$, we have $\ker F = \im V\subset G\otimes\Field_p$, where $F:G\otimes\Field_p\to (G\otimes\Field_p)^{(p)}$ and $V:(G\otimes\Field_p)^{(p)}\to G\otimes\Field_p$ are the Frobenius and Verschiebung homomorphisms, respectively.
\end{enumerate}
See for instance~\cite[\S I]{MR0801922}.

These organize into a category $\mathcal{BT}_n(R)$, and we will write $\mathrm{BT}_n(R)$ for the underlying groupoid obtained by jettisoning the non-isomorphisms. For $1\leq r\leq n$, sending $G$ to $G[p^r]$ yields a functor $\mathcal{BT}_n(R)\to \mathcal{BT}_r(R)$.

An important role will be played by the following fundamental result of Grothendieck~\cite{MR0801922}:

\begin{theorem}
\label{thm:grothendieck}
The assignment $R\mapsto \mathrm{BT}_n(R)$ on $\mathrm{CRing}^{p\text{-nilp}}$ is represented by a finitely presented smooth $0$-dimensional $p$-adic formal Artin stack with affine diagonal.
\end{theorem}

\subsubsection{}
There is a canonical involution 
\[
\mathcal{BT}_n(R) \xrightarrow{G\mapsto G^*}\mathcal{BT}_n(R)
\]
induced by Cartier duality, with $G^* =\underline{\Hom}(G,\mup[p^n])$.

Let $\mathrm{Vect}_{\{0,1\}}(R^{\mathrm{syn}}\otimes\Int/p^n\Int)$ be the $\infty$-category of vector bundles on $R^{\mathrm{syn}}\otimes\Int/p^n\Int$ with Hodge-Tate weights $\{0,1\}$. There is once again a canonical involution
\[
\mathrm{Vect}_{\{0,1\}}(R^{\mathrm{syn}}\otimes\Int/p^n\Int)\xrightarrow{\mathcal{M}\mapsto \mathcal{M}^*}\mathrm{Vect}_{\{0,1\}}(R^{\mathrm{syn}}\otimes\Int/p^n\Int),
\]
with $\mathcal{M}^* = \mathcal{M}^\vee\{1\}$ is the twist of the dual vector bundle by the Breuil-Kisin twist $\mathcal{O}^{\mathrm{syn}}_n\{1\}$. In analogy with the involution on $\mathcal{BT}_n(R)$, we will refer to $\mathcal{M}^*$ as the \defnword{Cartier dual} of $\mathcal{M}$.

We can now state the main result of this section.
\begin{theorem}
\label{thm:dieudonne}
Suppose that $R$ belongs to $\mathrm{CRing}$. Then there is a canonical equivalence of $\infty$-categories
\[
\mathcal{G}_n:\mathrm{Vect}_{\{0,1\}}(R^{\mathrm{syn}}\otimes\Int/p^n\Int)\xrightarrow{\simeq}\mathcal{BT}_n(R)
\]
compatible with Cartier duality, so that for every $\mathcal{M}\in \mathrm{Vect}_{\{0,1\}}(R^{\mathrm{syn}}\otimes\Int/p^n\Int)$, there is a canonical isomorphism
\[
\mathcal{G}_n(\mathcal{M})^*\xrightarrow{\simeq}\mathcal{G}_n(\mathcal{M}^*).
\]
\end{theorem}

\begin{remark}
The theorem implies in particular that $\mathrm{Vect}_{\{0,1\}}(R^{\mathrm{syn}}\otimes\Int/p^n\Int)$ is a classical category. This fact is not evident from the definitions, since the derived stack $R^{\mathrm{syn}}\otimes\Int/p^n\Int$ is in general not classical.
\end{remark}

\begin{remark}
When $R$ is quasisyntomic, by sending $n\to \infty$, Theorem~\ref{thm:dieudonne} recovers the main result of Ansch\"utz and Le Bras from~\cite{MR4530092} classifying $p$-divisible groups over $R$ in terms of admissible prismatic Dieudonn\'e crystals over $R$ (see~\cite[Def. 4.10]{MR4530092}). Indeed, when $R$ is qrsp, we use Remark~\ref{rem:comparison_with_alb}, and the general quasisyntomic case then follows by descent.

Note however that, though our eventual argument for establishing the equivalence is fundamentally different, we still make use of the work of Ansch\"utz-Le Bras (via its reinterpretation by Mondal) to construct the inverse to our functor $\mathcal{G}_n$.
\end{remark}

\begin{remark}
\label{rem:power_series_BT}
When $R = \kappa\pow{x_1,\ldots,x_n}$, using Example~\ref{ex:faltings_situation} and an argument such as the one used in~\S\ref{subsec:alb}, one recovers de Jong's description~\cite{dejong:formal_rigid} of $p$-divisible groups over power series rings over perfect fields $k$ in terms of certain filtered $\varphi$-modules $\Fil^\bullet M$ over $W(\kappa)\pow{x_1,\ldots,x_m}$ equipped with the Frobenius lift satisfying $x_i\mapsto x_i^p$ (see also~\cite[\S 7]{faltings:very_ramified}). \emph{A priori}, de Jong's description is in terms of $F$-crystals, and so also requires a topologically nilpotent integrable connection on $M$ compatible with the $\varphi$-semilinear structure; however, as observed by Faltings~\cite[Theorem 10]{faltings:very_ramified}, with this choice of Frobenius, the integrable connection is actually uniquely determined by the rest of the data.
\end{remark}

\begin{remark}
 When $R = \Reg{K}$ is a totally ramified ring of integers over $W(\kappa)$ for some perfect field $\kappa$ with uniformizer $\pi$, then, combining the theorem with Example~\ref{ex:BK_situation}, one can recover a classification of $p$-divisible groups due to Kisin (the case $p>2$ for $\Reg{K}$)~\cite{kisin:f_crystals}, W. Kim (the case $p=2$ for $\Reg{K}$)~\cite{kim:2-adic}, and Lau (the general case of both $\Reg{K}$ and $\Reg{K}/\pi^m$)~\cite{lau:displays}.
\end{remark}

\begin{remark}
\label{rem:non_noetherian_pdiv}
Suppose that $R'$ is a perfectoid ring, and $J = ([\varpi_1],\ldots,[\varpi_m])\subset \Prism_{R'}$ an ideal satisfying the hypotheses in Example~\ref{ex:non_noetherian_situation}. If $R = R'/\theta(J)$, then combining \emph{loc. cit.} with Remark~\ref{ex:BK_situation} and Theorem~\ref{thm:dieudonne} shows that $p$-divisible groups over $R$ are classified by finite projective $\Prism_{R'}/J$-modules $\mathcal{N}$ equipped with a map $\varphi^* \mathcal{N}\to \mathcal{N}$ whose cokernel is finite projective over $R$. In fact, the argument gives similar descriptions of the category of truncated $p$-divisible groups over each quotient $R'/\theta(J)^n$. This is a special case of results of Lau~\cite[\S 5]{MR3867290}.
\end{remark}

\subsection{Height and dimension}

\subsubsection{}
The \defnword{height} of an $n$-truncated Barsotti-Tate group $G$ over $R$ is the $\Int_{\ge 0}$-valued locally constant function on $\Spec R$ such that $G[p]$ has degree $p^{h}$ over $R$. The \defnword{dimension} $d$ is the $\Int_{\ge 0}$-valued locally constant function such that $\ker F\subset G[p]\otimes\Field_p$ has degree $p^{d}$ over $R/pR$.

These are locally constant invariants of $G$, and yield decompositions
\[
\mathrm{BT}_n =\bigsqcup_{d\leq h}\mathrm{BT}^{h,d}_n,
\]
of $p$-adic formal Artin stacks, where $h$ ranges over the non-negative integers, $d$ over the non-negative integers bounded by $h$, and $\mathrm{BT}^{h,d}_n$ is the locus of $n$-truncated Barsotti-Tate groups $G$ of height $h$ and dimension $d$.

\subsubsection{}
On the $F$-gauge side of things, we have $\infty$-subgroupoids
\[
\mathrm{Vect}_{h,d}(R^{\mathrm{syn}}\otimes\Int/p^n\Int)\subset \mathrm{Vect}_{\{0,1\}}(R^{\mathrm{syn}}\otimes\Int/p^n\Int)^{\simeq}
\]
spanned by vector bundles $\mathcal{M}$ over $R^{\mathrm{syn}}\otimes\Int/p^n\Int$ of rank $h$  and Hodge-Tate weights $0,1$ such that the associated graded $R/{}^{\mathbb{L}}p^n$-module $\gr^{-1}_{\mathrm{Hdg}}M$ is locally free of rank $d$. 

Let us note the following:
\begin{lemma}
\label{lem:cartier_duality}
Cartier duality yields equivalences
\[
\mathrm{Vect}_{h,d}(R^{\mathrm{syn}}\otimes\Int/p^n\Int)\xrightarrow{\simeq}\mathrm{Vect}_{h,h-d}(R^{\mathrm{syn}}\otimes\Int/p^n\Int)
\]
\end{lemma}

\subsubsection{}
For $0\leq d\leq h$, let $\mu_d:\Gm\to \GL_h$ be the cocharacter given by the diagonal matrix
\[
\mu_d(z) = \mathrm{diag}(\underbrace{z,z,\ldots,z}_d,\underbrace{1,\ldots,1}_{h-d}).
\]
Associated with this, we have the smooth Artin stacks $\BT[\GL_h,\mu_d]{n}$ over $\Int_p$.

\begin{proposition}
\label{prop:vect_to_BT}
For $R\in \mathrm{CRing}^{p\text{-nilp}}$, there is a canonical equivalence of groupoids
\[
\BT[\GL_h,\mu_d]{n}(R)\xrightarrow{\simeq}\mathrm{Vect}_{h,d}(R^{\mathrm{syn}}\otimes\Int/p^n\Int).
\]
\end{proposition}
\begin{proof}
One can see this by combining Lemma~\ref{lem:quotient_stack_desc} with Propositions~\ref{prop:abstract_f-gauge_gl_n} and ~\ref{prop:descent}. 
\end{proof}

\begin{remark}
\label{rem:cartier_duality}
Via the above proposition and Lemma~\ref{lem:cartier_duality}, we find that there is a Cartier duality equivalence
\[
*:\BT[\GL_h,\mu_d]{n}\xrightarrow{\simeq}\BT[\GL_h,\mu_{h-d}]{n}.
\]
\end{remark}

We now have:
\begin{theorem}
\label{thm:BT_to_BT}
There is a canonical equivalence of smooth $p$-adic formal Artin stacks 
\[
\mathcal{G}:\BT[\GL_h,\mu_d]{n}\xrightarrow{\simeq}\mathrm{BT}^{h,d}_n
\]
such that the following diagram commutes up to canonical isomorphism:
\[
\begin{diagram}
\BT[\GL_h,\mu_d]{n}&\rTo^{\mathcal{G}}_\simeq&\mathrm{BT}^{h,d}_n\\
\dTo^*&&\dTo_{G\mapsto G^*}\\
\BT[\GL_h,\mu_{h-d}]{n}&\rTo^\simeq_{\mathcal{G}}&\mathrm{BT}^{h,h-d}_n.
\end{diagram}
\]
\end{theorem}  

Assuming this theorem, we can easily deduce Theorem~\ref{thm:dieudonne} via a standard argument.
\begin{proof}
[Proof of Theorem~\ref{thm:dieudonne}]
Theorem~\ref{thm:BT_to_BT}, combined with Proposition~\ref{prop:vect_to_BT}, gives us an isomorphism of $\infty$-groupoids
\[
\mathrm{Vect}_{\{0,1\}}(R^{\mathrm{syn}}\otimes\Int/p^n\Int)^{\simeq}\xrightarrow{\simeq}\mathrm{BT}_n(R)
\]
compatible with Cartier duality. To get an equivalence of $\infty$-categories, one now uses a graph construction: For $\mathcal{M}_1,\mathcal{M}_2$ in $\mathrm{Vect}_{\{0,1\}}(R^{\mathrm{syn}}\otimes\Int/p^n\Int)$, the space of maps $\mathcal{M}_1\to \mathcal{M}_2$ is equivalent to the space of isomorphisms $\mathcal{M}_1\oplus \mathcal{M}_2\xrightarrow{\simeq}\mathcal{M}_1\oplus \mathcal{M}_2$ that are `upper triangular' and project onto the identity endomorphisms of $\mathcal{M}_1$ and $\mathcal{M}_2$. A similar description holds for $\mathcal{G}_1$ and $\mathcal{G}_2$ in $\mathcal{BT}_n(R)$.
\end{proof} 

\begin{remark}
Combined with Remark~\ref{rem:adjoint_nilpotent_locus}, the proof above shows that we can recover Zink and Lau's classification of connected $p$-divisible groups (equivalently, $p$-divisible formal groups) by Witt vector displays~\cites{MR1827031,MR2983008} from our main theorem here.
\end{remark}

\subsection{From $F$-gauges to Barsotti-Tate groups}
\label{subsec:G_functor}

Suppose that we have $R\in \mathrm{CRing}^{p\text{-nilp}}$ and $\mathcal{M}$ in $\mathrm{Vect}_{\{0,1\}}(R^{\mathrm{syn}}\otimes\Int/p^n\Int)$. 

\subsubsection{}
Set $\mathcal{G}_n(\mathcal{M}) = \Gamma_{\mathrm{syn}}(\mathcal{M})$. Then by Theorem~\ref{thm:sections_1-bounded_representable} $\mathcal{G}_n(\mathcal{M})$ is locally finitely presented and \emph{quasi-smooth} over $R$ with cotangent complex given by $\Reg{\mathcal{G}_n(\mathcal{M})}\otimes_R\gr^{-1}_{\mathrm{Hdg}}M[1]$. More generally, for all $r\leq n$, set
\[
\mathcal{M}_r \defn \mathcal{M}\bigr\vert_{R^{\mathrm{syn}}\otimes\Int/p^{r}\Int}\;;\;\mathcal{G}_r(\mathcal{M}) = \Gamma_{\mathrm{syn}}(\mathcal{M}_r).
\]

\subsubsection{}
By tensoring $\mathcal{M}$ with the canonical short exact sequence
\[
0\to \Int/p^r\Int\xrightarrow{a\mapsto p^{n-r}a}\Int/p^n\Int\to \Int/p^{n-r}\Int\to 0,
\]
we obtain a fiber sequence in $\mathrm{Perf}(R^{\mathrm{syn}}\otimes\Int/p^n\Int)$: 
\[
\mathcal{M}_{r}\to \mathcal{M}\to \mathcal{M}_{n-r}.
\]
Taking derived global sections yields a fiber sequence
\begin{align}\label{eqn:rgamma_syn_sequence}
R\Gamma_{\mathrm{syn}}(\mathcal{M}_r)\to R\Gamma_{\mathrm{syn}}(\mathcal{M})\to R\Gamma_{\mathrm{syn}}(\mathcal{M}_{n-r})
\end{align}
of $\Mod{\Int/p^n\Int}$-valued quasisyntomic sheaves over $R$.

\begin{theorem}
\label{thm:g_functor}
Suppose that $R$ is discrete and that $\mathcal{M}$ is in $\mathrm{Vect}_{h,d}(R^{\mathrm{syn}}\otimes\Int/p^n\Int)$, then $\mathcal{G}_n(\mathcal{M})$ is a relative truncated Barsotti-Tate group scheme over $R$ of height $h$ and dimension $d$. In particular, for each $n\ge 1$, we have a canonical map of $p$-adic formal Artin stacks
\[
\mathcal{G}_n:\BT[\GL_h,\mu_d]{n}\to \mathrm{BT}^{h,d}_n.
\]
Moreover, if $r<n$, there is a canonical short exact sequence of truncated Barsotti-Tate group schemes
\[
0\to \mathcal{G}_{r}(\mathcal{M})\to \mathcal{G}_n(\mathcal{M})\to \mathcal{G}_{n-r}(\mathcal{M})\to 0
\]
obtained by taking the connective truncation of~\eqref{eqn:rgamma_syn_sequence}.
\end{theorem}   

The proof will need a little bit of preparation. First, let us record the following simple flatness criterion.

\begin{lemma}
\label{lem:classicality_criterion}
Suppose that $Y$ is a finitely presented quasi-smooth derived algebraic space over $R\in \mathrm{CRing}$ of virtual codimension $0$. Set $S = \Spec R$. Then the following are equivalent:
\begin{enumerate}
   \item $Y$ is flat over $R$;
   \item $Y\otimes_R\pi_0(R)$ is flat and classicaly lci over $\pi_0(R)$;
   \item For every $x\in S(\kappa)$ with $\kappa$ algebraically closed, $x^*Y\to \Spec \kappa$ is flat and classically lci;
   \item For every $x\in S(\kappa)$ with $\kappa$ algebraically closed, $\pi_0((x^*Y)(\kappa))$ is a finite set.
\end{enumerate}
\end{lemma}
\begin{proof}
Since $Y$ is quasi-smooth of virtual codimension $0$, \'etale locally on $Y$ we can present it as a derived complete intersection subscheme of some affine space $\Aff^n_R$ over $R$ cut out as the derived zero locus of $n$ polynomials $f_1,\ldots,f_n$. If $R$ is discrete, then such a derived zero locus is flat over $R$ if and only if the animated ring $R[x_1,\ldots,x_n]/{}^{\mathbb{L}}(f_1,\ldots,f_n)$ has no higher homotopy groups. This is precisely equivalent to this ring being a classical lci algebra over $R$. Since flatness can be tested after derived base-change over the classical truncation, we see that (1) and (2) are equivalent.

If $R = \kappa$ is an algebraically closed field, then $\pi_0((x^*Y)(\kappa))$ is finite precisely when the classical truncation $Y_{\mathrm{cl}}$ is $0$-dimensional and of finite type over $\kappa$. We now note that $\kappa[x_1,\ldots,x_n]/(f_1,\ldots,f_n)$ is $0$-dimensional precisely when $f_1,\ldots,f_n$ form a regular sequence. This shows the equivalence of (3) and (4).

To see the equivalence of (1) and (3), we now only have to make the additional observation that $M\in \Mod{R}$ is flat over $R$ if and only if its derived base-change over every algebraically closed field over $R$ is flat.
\end{proof}

\subsubsection{}
\label{subsubsec:fzip_underlying_fgauge_h_d}
Let us suppose that $R$ is an $\Field_p$-algebra. Here, we can consider the $F$-zip $\bm{M}$ underlying $\mathcal{M}$. This is given via restriction along the map $R^{\Fzip}\to R^{\mathrm{syn}}\otimes\Int/p^n\Int$, and amounts to specifying:
\begin{enumerate}
   \item A locally free $R$-module $M$ equipped with direct summands $\Fil^0_{\mathrm{Hdg}}M\subset M$ and $\Fil_1^{\mathrm{conj}}M\subset M$ of codimension $d$ and $h-d$, respectively;
   \item Isomorphisms
   \[
\xi_1:\Fil_1^{\mathrm{conj}}M\xrightarrow{\simeq}\varphi^*(M/\Fil^0_{\mathrm{Hdg}}M)\;;\; \xi_0:\gr_0^{\mathrm{conj}}M = M/\Fil^{\mathrm{conj}}_1M\xrightarrow{\simeq}\varphi^*\Fil^0_{\mathrm{Hdg}}M.
\]
\end{enumerate}
Now, consider the functor 
\[
\mathsf{G}(\bm{M}) = \tau^{\leq 0}R\Gamma_{\Fzip}(\bm{M}):\mathrm{CRing}_{R/}\to \Mod[\mathrm{cn}]{\Field_p}
\]

Unwinding definitions, one finds
\begin{align*}
\mathsf{G}(\bm{M})(C) &=\{m\in C\otimes_R\Fil^0_{\mathrm{Hdg}}M:\;\xi_0(\overline{m}) = \varphi^*m \},
\end{align*}
where $M\xrightarrow{m\mapsto \overline{m}}\gr^{\mathrm{conj}}_0M$ is the natural quotient map. Viewing $m\mapsto \xi_0(\overline{m})$ as a map $\overline{\xi}_0:\Fil^0_{\mathrm{Hdg}}M\to \varphi^*\Fil^0_{\mathrm{Hdg}}M$, we see that we have
\[
\mathsf{G}(\bm{M}) = \ker(\mathbf{V}((\Fil^0_{\mathrm{Hdg}}M)^\vee)\xrightarrow{\overline{\xi}_0-F}\mathbf{V}((\varphi^*\Fil^0_{\mathrm{Hdg}}M)^\vee)),
\]
so that $\mathsf{G}(\bm{M}) $ is the Cartier dual to a height one finite flat group scheme over $R$ of rank $p^{h-d}$.

\begin{proof}
[Proof of Theorem~\ref{thm:g_functor}]
Assume for now that $R$ is an $\Field_p$-algebra. By Theorem~\ref{thm:devissage_to_fzips}, with input from Corollary~\ref{cor:fppf_cohomology}, we see that we have a canonical Cartesian square:
\[
   \begin{diagram}
      \mathcal{G}_1(\mathcal{M})&\rTo&\mathsf{G}(\bm{M})\\
      \dTo&&\dTo_0\\
      \mathsf{G}(\bm{M})&\rTo&BG(\gr^{-1}_{\mathrm{Hdg}}M,\psi_{\bm{M}}),
   \end{diagram}
\]
where $G(\gr^{-1}_{\mathrm{Hdg}}M,\psi_{\bm{M}})$ is a finite flat height $1$ group scheme of rank $p^d$.

This description shows that the map $\mathcal{G}_1(\mathcal{M})\to \mathsf{G}(\bm{M})$---or more precisely, its base-change along any section $\Spec C\to \mathsf{G}(\bm{M})$---is quasi-smooth of virtual codimension $0$, and satisfies condition (4) of Lemma~\ref{lem:classicality_criterion}. Therefore, we see that the map is quasi-finite and flat.

If $R$ is \emph{perfect}, then, by  Proposition~\ref{prop:f_gauges_descent}, $\mathrm{Perf}(R^{\mathrm{syn}}\otimes\Field_p)$ is equivalent to the $\infty$-category of perfect $\underline{W_1(R)}$-gauges, and this in turn is the same as the $\infty$-category of perfect $F$-zips over $R$. Therefore, the map $\mathcal{G}_1(\mathcal{M})(R)\to \mathsf{G}(\bm{M})(R)$ is an \emph{isomorphism}. We conclude therefore that $\mathcal{G}_1(\mathcal{M})$ is \emph{finite} flat over $\mathsf{G}(\bm{M})$, and that we actually have a short exact sequence of finite flat group schemes
\begin{align}\label{eqn:height_coheight_sequence}
0\to G(\gr^{-1}_{\mathrm{Hdg}}M,\psi_{\bm{M}})\to \mathcal{G}_1(\mathcal{M})\to \mathsf{G}(\bm{M})\to 0.
\end{align}
In particular, we see that $\mathcal{G}_1(\mathcal{M})$ is finite flat over $R$ of rank $p^h$. 

We will now drop the hypothesis that $R$ is an $\Field_p$-algebra. Taking connective truncations of~\eqref{eqn:rgamma_syn_sequence} shows that we have an isomorphism
\begin{align}
\label{eqn:short_exact_sequence_devissage}
\mathcal{G}_r(\mathcal{M})\xrightarrow{\simeq}\hker^{\mathrm{cn}}\left(\mathcal{G}_n(\mathcal{M})\to \mathcal{G}_{n-r}(\mathcal{M})\right)
\end{align}
of $\Mod[\mathrm{cn}]{\Int/p^n\Int}$-valued quasisyntomic sheaves over $R$. 

If $R = \kappa$ is an algebraically closed field, then this isomorphism, combined with a simple induction on $r$ shows that $\mathcal{G}_n(\mathcal{M})(\kappa)$ is a finite set. We have already observed that $\mathcal{G}_n(\mathcal{M})$ is always a finitely presented quasi-smooth derived scheme of virtual codimension $0$ over $R$. Therefore, applying Lemma~\ref{lem:classicality_criterion} to $Y = \mathcal{G}_n(\mathcal{M})$ shows that $\mathcal{G}_n(\mathcal{M})$ is quasi-finite and flat over $R$

Now, by considering the isomorphism~\eqref{eqn:short_exact_sequence_devissage} twice, first as given, and then again with $r$ replaced by $n-r$, we find that we have
\[
 \mathrm{im}\left(\mathcal{G}_n(\mathcal{M})\xrightarrow{p^{n-r}}\mathcal{G}_n(\mathcal{M})\right) = \mathcal{G}_r(\mathcal{M}) = \ker\left(\mathcal{G}_n(\mathcal{M})\xrightarrow{p^r}\mathcal{G}_n(\mathcal{M})\right).
\]
In particular, for every $r$, $\mathcal{G}_n(\mathcal{M})$ is a $\mathcal{G}_r(\mathcal{M})$-torsor over $\mathcal{G}_{n-r}(\mathcal{M})$. An inductive argument now shows that $\mathcal{G}_n(\mathcal{M})$ is a finite flat group scheme over $R$ that is also a flat $\Int/p^n\Int$-module. 

To finish, we must know that when $R$ is an $\Field_p$-algebra, $\mathcal{G}_1(\mathcal{M})$ is a $1$-truncated Barsotti-Tate group scheme of dimension $d$. Since any $\mathcal{M}\in \mathrm{Vect}_{h,d}(R^{\mathrm{syn}}\otimes\Field_p)(R)$ can be lifted \'etale locally to an object in $\mathrm{Vect}_{h,d}(R^{\mathrm{syn}}\otimes\Int/p^2\Int)(R)$, the previous paragraph implies that $\mathcal{G}_1(\mathcal{M})$ is indeed Barsotti-Tate. Moreover, since the Verschiebung on $\mathsf{G}(\bm{M})$ is identically zero (its Cartier dual is the Frobenius homomorphism for a height one group scheme), we have 
\[
\im V \subset G(\gr^{-1}_{\mathrm{Hdg}}M,\psi_{\bm{M}})\subset \ker F.
\]
Since $\im V = \ker F$ by the Barsotti-Tate condition, we conclude that $\mathcal{G}_1(\mathcal{M})$ has dimension $d$, as desired.
 \end{proof}

\subsection{Cartier duality}

Let $\mathcal{O}_n$ be the structure sheaf of $R^{\mathrm{syn}}\otimes\Int/p^n\Int$. We will have use for the following result, which is due to Bhatt-Lurie:
\begin{proposition}
\label{prop:the_case_of_mupn_bhatt_lurie}
We have canonical isomorphisms
\[
\mathcal{G}_n(\mathcal{O}_n)\xrightarrow{\simeq}\Int/p^n\Int\;;\;\mathcal{G}_n(\mathcal{O}_n\{1\})\xrightarrow{\simeq}\mup[p^n]
\]
in $\mathrm{BT}_n(R)$.
\end{proposition}
\begin{proof}
Unwinding definitions, the first isomorphism follows from~\cite[Theorem 8.1.9]{bhatt2022absolute}, while the second follows from~\cite[Theorem 7.5.6]{bhatt2022absolute}

We can also give alternate proofs using the methods of this paper. As we already know from what we have seen above, $\mathcal{G}_n(\mathcal{O}_n)$ (resp.  $\mathcal{G}_n(\mathcal{O}_n\{1\})$) is an $n$-truncated Barsotti-Tate group schemes over $R$, height $1$ and dimension $0$ (resp. dimension $1$).

We can assume that $R = \Int/p^m\Int$ for some $m\ge 1$. Note that $\mathcal{G}_n(\mathcal{O}_n)$ is \'etale over $\Int/p^m\Int$, so it suffices to give a map $\Int/p^n\Int\to \mathcal{G}_n(\mathcal{O}_n)$ that is an isomorphism over $\overline{\Field}_p$. This is given by the structure map $\Int/p^n\Int\to \mathcal{O}_n$.

For the case of the Breuil-Kisin twist, note that $\mathcal{G}_n(\mathcal{O}_n\{1\})$ is of multiplicative type. Once again, it suffices to give a canonical map $\mup[p^n]\to \mathcal{G}_n(\mathcal{O}_n\{1\})$ that is an isomorphism over $\overline{\Field}_p$. In fact, since the scheme parameterizing such maps is finite \'etale over $\Int/p^m\Int$, we can assume that $m = 1$. In this case, by quasisyntomic descent, we only have to construct a canonical map 
\begin{align}\label{eqn:mupn_to_Acrys}
\mup[p^n](C)\to (\Fil^1\Prism_C/p^n)^{\varphi = p} = \mathcal{G}_n(\mathcal{O}_n\{1\})(C)
\end{align}
for qrsp $\Field_p$-algebras $C$. This is obtained using the isomorphism $\Prism_C\xrightarrow{\simeq}A_{\mathrm{crys}}(C)$ from~\eqref{eqn:semiperfect_prismatization}, and assigning to each $\alpha\in \mup[p^n](C)$ the image of the logarithm $\log([\tilde{\alpha}]^{p^n})\in A_{\mathrm{crys}}(C)$, where $\tilde{\alpha}\in C^\flat$ is a lift of $\alpha$ and $[\tilde{\alpha}]$ is its Teichm\"uller lift; see~\cite[\S 7.1]{bhatt2022absolute}.

To finish, it is enough to know that the map~\eqref{eqn:mupn_to_Acrys} is \emph{injective} for all qrsp $\Field_p$-algebras $C$. In fact, it suffices to verify that, when $C = \Field_p[x^{1/p^\infty}]/(x)$ and $n=1$, the element $\log([1-x])\in A_{\mathrm{crys}}(C)$ is not divisible by $p$. But one has (see displayed equation (37) in ~\cite[p. 168]{bhatt2022absolute}):
\[
\log([1-x]) \equiv -\sum_{d=1}^{p}\frac{[x]^d}{d}\pmod{p}.
\]
To see that the right hand side is non-zero, we only have to note that the element
\[
x + \frac{x^2}{2} + \ldots + \frac{x^{p-1}}{p-1} + (p-1)!\cdot x^{[p]}\in \Field_p\langle x\rangle
\]
in the standard divided power $\Field_p$-algebra is non-zero.

\end{proof}

\begin{construction}
For every vector bundle $\mathcal{M}$ over $R^{\mathrm{syn}}\otimes\Int/p^n\Int$ with Hodge-Tate weights $0,1$, we will now define a canonical map
\begin{align}
\label{eqn:cartier_dual_arrow}
\mathcal{G}_n(\mathcal{M}^*)\to \mathcal{G}_n(\mathcal{M})^*.
\end{align} 

This is obtained as follows: For every $x:R\to C$, we have
\begin{align*}
\mathcal{G}_n(\mathcal{M}^*)(C)&\simeq \Map_{\mathrm{QCoh}(C^{\mathrm{syn}}\otimes\Int/p^n\Int)}(x^*\mathcal{M}, \mathcal{O}_n\{1\}),
\end{align*}
and evaluation on global sections now yields a map
\[
\mathcal{G}_n(\mathcal{M}^*)(C)\to \Hom(\mathcal{G}_n(x^*\mathcal{M}),\mathcal{G}_n(\mathcal{O}_n\{1\}))\simeq \Hom(x^*\mathcal{G}_n(\mathcal{M}),\mup[p^n])\simeq \mathcal{G}_n(\mathcal{M})^*(C).
\]
Here, we have used Proposition~\ref{prop:the_case_of_mupn_bhatt_lurie} for the penultimate isomorphism.
\end{construction}

\begin{theorem}
\label{thm:cartier_duality}
The map~\eqref{eqn:cartier_dual_arrow} is an isomorphism.
\end{theorem}
\begin{proof}
Let's begin with the following easy observation that is immediate from Proposition~\ref{prop:the_case_of_mupn_bhatt_lurie}: The map is an isomorphism for any $\mathcal{M}$ that is an extension of $\mathcal{O}^{h-d}_n$ by $\mathcal{O}_n\{1\}^d$.

Denote the map~\eqref{eqn:cartier_dual_arrow} by $\alpha_n$, and view it as a map of $n$-truncated Barsotti-Tate group schemes over the smooth $p$-adic formal algebraic stack $\BT[\GL_h,\mu_d]{n}$ of the same height and dimension. Taking the limit over $n$ gives us a map $\alpha_\infty$ of $p$-divisible groups over $\BT[\GL_h,\mu_d]{\infty}$. We will actually show that $\alpha_\infty$ is an isomorphism: this is enough since $\BT[\GL_h,\mu_d]{\infty}\to \BT[\GL_h,\mu_d]{n}$ is a limit of smooth surjective maps.

Let $R$ be a universal deformation ring for $\BT[\GL_h,\mu_d]{\infty}\otimes\Field_p$ at a geometric point valued in a field $\kappa$ (see Lemma~\ref{lem:discreteness_def_theory}): this is isomorphic to $\kappa\pow{t_1,\ldots,t_{d(h-d)}}$ and is in particular normal. It is enough to know that the restriction of $\alpha_\infty$ over $\Spf R$ is an isomorphism. By~\cite[Lemma 2.4.4]{dejong:formal_rigid}, this restriction algebraizes to a map of $p$-divisible groups over $\Spec R$. The observation from the beginning of the proof---combined with Lemma~\ref{lem:generically_ordinary} below---shows that this map is generically an isomorphism. Therefore, we find from~\cite[Corollary 1.2]{de-Jong1998-ki}---and the subsequent discussion---that it is an isomorphism on the nose. 
\end{proof}

\begin{lemma}
\label{lem:generically_ordinary}
Let $P^+_{\mu_d}\subset \GL_h$ be the opposite parabolic associated with $\mu_d$.\footnote{ Explicitly, it consists of the upper triangular invertible matrices with respect to the decomposition $\Int_p^h = \Int_p^d\oplus \Int_p^{h-d}$.} Then, for all $n\ge 1$:
\begin{enumerate}
   \item $\BT[P^+_{\mu_d},\mu_d]{n}\to \BT[\GL_h,\mu_d]{n}$ is a dense open immersion parameterizing precisely those $F$-gauges $\mathcal{M}$ that are \'etale locally on the base extensions of $\mathcal{O}^{h-d}_n$ by $\mathcal{O}_n\{1\}^{d}$. 
   \item The square
   \[
    \begin{diagram}
    \BT[P^+_{\mu_d},\mu_d]{n}&\rTo&\BT[\GL_h,\mu_d]{n}\\
    \dTo&&\dTo\\
    \BT[P^+_{\mu_d},\mu_d]{1}&\rTo&\BT[\GL_h,\mu_d]{1}
    \end{diagram}
   \]
   is Cartesian.
\end{enumerate}
\end{lemma} 
\begin{proof}
This is of course the counterpart of the well-known statement that a versal $n$-truncated Barsotti-Tate group scheme is generically ordinary. 

Given the connectedness of $\BT[\GL_h,\mu_d]{n}$ (see Remark~\ref{rem:connectedness}), the lemma is essentially immediate from Proposition~\ref{prop:special_open}. The only thing to be remarked upon is the explicit description of the image of $\BT[P^+_{\mu_d},\mu_d]{n}$, but this is immediate from the definitions.
\end{proof}

\begin{remark}
\label{rem:dejong_use}
The proof of Theorem~\ref{thm:cartier_duality} ultimately relies on the results of de Jong in~\cite{de-Jong1998-ki} and hence also on his results from~\cite{dejong:formal_rigid} in the form of a fully faithfulness result for the classical Dieudonn\'e functor of Berthelot-Breen-Messing~\cite{bbm:cris_ii} for complete DVRs in characteristic $p$. This use can be circumvented by a more careful study of the map $\alpha_n$. 
\end{remark}

\subsection{From Barsotti-Tate groups to $F$-gauges}
\label{subsec:M_functor}

\subsubsection{}
Let $R$ be $p$-complete, $p$-quasisyntomic and either $p$-torsion free or an $\Field_p$-algebra. These conditions ensure that there exist quasisyntomic covers $R\to R'$ with $R'$ qrsp, and also that $\Prism_{R'}$ is $p$-completely flat for such covers. In particular, the derived stacks $R^{\mathcal{N}}\otimes\Int/p^n\Int$ and $R^{\mathrm{syn}}\otimes\Int/p^n\Int$ are actually \emph{classical}.

Let
\[
\epsilon_n:(R^{\mathrm{syn}}\otimes\Int/p^n\Int)_{\mathrm{fl}}\to R_{\mathrm{qsyn}}
\]
be the map of (classical) topoi arising via the functor $C\mapsto C^{\mathrm{syn}}\otimes\Int/p^n\Int$ on $p$-quasisyntomic $R$-algebras. Here, the left hand side (resp. the right hand side) is the topos of sheaves on the ind-fppf site over $R^{\mathrm{syn}}\otimes\Int/p^n\Int$ (resp. the small $p$-quasisyntomic site over $\Spf R$).

We can view $\mathcal{O}_n$ as a sheaf of rings on $(R^{\mathrm{syn}}\otimes\Int/p^n\Int)_{\mathrm{fl}}$, and $\mathcal{O}_n\{1\}$ as a quasicoherent sheaf with respect to $\mathcal{O}_n$.

\subsubsection{}
For $G\in \mathrm{BT}_n(R)$, set
\[
\mathcal{M}(G) = \underline{\Hom}_{(R^{\mathrm{syn}}\otimes\Int/p^n\Int)_{\mathrm{fl}}}(\epsilon_n^{-1}G^*,\mathcal{O}_n\{1\})
\]
where on the right we are considering the internal Hom sheaf over $(R^{\mathrm{syn}}\otimes\Int/p^n\Int)_{\mathrm{fl}}$. Note that by construction $\mathcal{M}(G)$ is a module over $\mathcal{O}_n$.

To alleviate notation, we will now drop the subscript $n$ and write $\mathcal{G}$ for the functor $\mathcal{G}_n$. The next result can be found in a certain form in Ansch\"utz-Le Bras~\cite{MR4530092}, though this formulation (and most of the proof, though we give a different approach to Cartier duality) is due to Mondal~\cite{Mondal2024-cy}:
\begin{proposition}
\label{prop:alb_mondal}
\begin{enumerate}
   \item $\mathcal{M}(G)$ is a vector bundle over $\mathcal{O}_n$ and yields an $F$-gauge in $\mathrm{Vect}_{\{0,1\}}(R^{\mathrm{syn}}\otimes\Int/p^n\Int)$.
   \item The functors 
\begin{align*}
   \mathcal{M}:\mathrm{BT}_n(R)&\to \mathrm{Vect}_{\{0,1\}}(R^{\mathrm{syn}}\otimes\Int/p^n\Int) ;\\
  \mathcal{G}: \mathrm{Vect}_{\{0,1\}}(R^{\mathrm{syn}}\otimes\Int/p^n\Int)&\xrightarrow{\text{Theorem~\ref{thm:g_functor}}} \mathrm{BT}_n(R)
\end{align*}
    form an adjoint pair.
    \item The unit $\mathrm{id}\to \mathcal{G}\circ\mathcal{M}$ is an isomorphism.
    \item There is a natural isomorphism $\mathcal{M}(G^*)^*\to \mathcal{M}(G)$.
\end{enumerate}
\end{proposition}
\begin{proof}
Most of the proof that we present here can be found in~\cite[\S 3]{Mondal2024-cy}. 

For claim (1), via quasisyntomic descent we reduce to the case where $R$ is qrsp. Here, the result follows from~\cite[Props. 3.56, 3.80, 3.81]{Mondal2024-cy}.

For the second claim, given $G\in \mathcal{BT}_n(R)$ and $\mathcal{M}\in \mathrm{Vect}_{\{0,1\}}(R^{\mathrm{syn}}\otimes\Int/p^n\Int) $, we find canonical isomorphisms:
\begin{align*}
\Hom_{\mathcal{O}_n}(\mathcal{M},\mathcal{M}(G))&\simeq  \Hom_{\mathcal{O}_n}\left(\mathcal{M},\underline{\Hom}_{(R^{\mathrm{syn}}\otimes\Int/p^n\Int)_{\mathrm{fl}}}(\epsilon_n^{-1}G^*,\mathcal{O}_n\{1\})\right)\\
&\simeq\Hom_{(R^{\mathrm{syn}}\otimes\Int/p^n\Int)_{\mathrm{fl}}}\left(\epsilon_n^{-1}G^*,\underline{\Hom}_{\mathcal{O}_n}(\mathcal{M},\mathcal{O}_n\{1\})\right)\\
&\simeq\Hom_{R_{\mathrm{qsyn}}}\left(G^*,\epsilon_{n,*}\underline{\Hom}_{\mathcal{O}_n}(\mathcal{O}_n,\mathcal{M}^*)\right)\\
&\simeq\Hom_{R_{\mathrm{qsyn}}}\left(G^*,\mathcal{G}(\mathcal{M}^*)\right)\\
&\simeq\Hom_{\mathcal{BT}_n(R)}(G^*,\mathcal{G}(\mathcal{M})^*) \\
&\simeq\Hom_{\mathcal{BT}_n(R)}(\mathcal{G}(\mathcal{M}),G).
 \end{align*}
Here, in the penultimate isomorphism, we have used Theorem~\ref{thm:cartier_duality}.

For claim (3), suppose that we are given $G\in \mathrm{BT}_n(R)$. We then find:
\begin{align*}
\mathcal{G}(\mathcal{M}(G))(R)&\simeq \Hom_{\mathcal{O}_n}(\mathcal{O}_n,\mathcal{M}(G))\\
&\simeq\Hom_{\mathcal{BT}_n(R)}(\mathcal{G}(\mathcal{O}_n),G)\\
&\simeq\Hom_{\mathcal{BT}_n(R)}(\Int/p^n\Int,G)\\
&\simeq G(R).
\end{align*}
Here, in the penultimate isomorphism, we have used Proposition~\ref{prop:the_case_of_mupn_bhatt_lurie}. Since this isomorphism is valid with $R$ replaced by any $p$-quasisyntomic $R$-algebra, claim (3) has been verified.

Finally, let us consider claim (4): We have
\begin{align*}
\Hom_{\mathcal{O}_n}(\mathcal{M}(G^*)^*,\mathcal{M}(G))&\simeq \Hom_{\mathcal{BT}_n(R)}(\mathcal{G}(\mathcal{M}(G^*)^*),G)\\
&\simeq \Hom_{\mathcal{BT}_n(R)}(\mathcal{G}(\mathcal{M}(G^*))^*,G)\\
&\simeq \Hom_{\mathcal{BT}_n(R)}((G^*)^*,G)\\
&\simeq \Hom_{\mathcal{BT}_n(R)}(G,G).
\end{align*}
Here, the first isomorphism uses claim (2), the second uses Theorem~\ref{thm:cartier_duality} and the third uses claim (3). The identity endomorphism of $G$ corresponds via these isomorphisms to the canonical arrow involved in claim (4).

That this arrow is an isomorphism is a consequence of the next lemma, applied with $\mathcal{M}_1 = \mathcal{M}(G^*)^*$ and $\mathcal{M}_2 = \mathcal{M}(G)$:
\begin{lemma}
\label{lem:G_functor_conservative}
Suppose that $\mathcal{M}_1$ and $\mathcal{M}_2$ are two perfect $F$-gauges of level $n$ over $R$ with Hodge-Tate weights $0,1$. Set $\mathcal{N}^* = \mathcal{N}^\vee\{1\}$ for any perfect $F$-gauge $\mathcal{N}$: this underlies an anti-involution on the $\infty$-category of perfect $F$-gauges of Hodge-Tate weights $0,1$. Suppose that $f:\mathcal{M}_1\to \mathcal{M}_2$ is an arrow such that 
\[
\Gamma_{\mathrm{syn}}(f):\Gamma_{\mathrm{syn}}(\mathcal{M}_1)\to \Gamma_{\mathrm{syn}}(\mathcal{M}_2)\;;\;\Gamma_{\mathrm{syn}}(f^*):\Gamma_{\mathrm{syn}}(\mathcal{M}^*_2)\to \Gamma_{\mathrm{syn}}(\mathcal{M}^*_1)
\]
are equivalences of derived stacks over $R$. Then $f$ is an isomorphism. 
\end{lemma} 
\begin{proof}
Set $\mathcal{N} = \hker(f)$; then we see that $\Gamma_{\mathrm{syn}}(\mathcal{N})(C) = 0$ for all $C\in \mathrm{CRing}_{R/}$. Similarly, $\Gamma_{\mathrm{syn}}(\mathcal{N}^*[-1])(C) = 0$ for all $C\in \mathrm{CRing}_{R/}$. Theorem~\ref{thm:sections_1-bounded_representable} now tells us that, if $\Fil^\bullet_{\mathrm{Hdg}}N$ and $\Fil^\bullet_{\mathrm{Hdg}}N^*$ are the filtered perfect complexes over $R$ obtained from $\mathcal{N}$ and $\mathcal{N}^*$ via pullback along $x^{\mathcal{N}}_{\dR}$, then  we have
\[
\gr^{-1}_{\mathrm{Hdg}}N \simeq 0\simeq \gr^{-1}_{\mathrm{Hdg}}N^*\simeq(\Fil^0_{\mathrm{Hdg}}N)^\vee.
\]

Since $\mathcal{N}$ has Hodge-Tate weights $0,1$, this shows that $\gr^i_{\mathrm{Hdg}}N \simeq 0$ for all $i$, and hence that $\Fil^\bullet_{\mathrm{Hdg}}N\simeq 0$. This implies that $\mathcal{N}\simeq 0$: To see this, we can assume that $R$ is semiperfectoid, in which case it follows from the observation that the map 
\[
\pi_0(\Prism_R/{}^{\mathbb{L}}(p,I))\to \pi_0(R/{}^{\mathbb{L}}(p,I))
\]
has nilpotent kernel; see for instance the end of the proof of Proposition~\ref{prop:1_bounded_cartesian}.
\end{proof}

\end{proof}

We are now ready to prove Theorem~\ref{thm:BT_to_BT}
\begin{proof}
[Proof of Theorem~\ref{thm:BT_to_BT}]
We first note that Proposition~\ref{prop:alb_mondal} gives us a left inverse $\mathcal{M}:\mathrm{BT}_n^{h,d}\to \BT[\GL_h,\mu_d]{n}$ to the map $\mathcal{G}$ from Theorem~\ref{thm:g_functor}. Indeed, by Theorem~\ref{thm:grothendieck}, $\mathrm{BT}_n^{h,d}$ is smooth, and therefore is a left Kan extension of its restriction to $p$-completely smooth $\Int_p$-algebras. This means that, to obtain the map $\mathcal{M}$ and verify that it is a left inverse to $\mathcal{G}$, it suffices to do so on such inputs, where it follows from the proposition.

From the same proposition, we find, for all $\mathcal{M}\in \BT[\GL_h,\mu_d]{n}(R)$, a canonical map of $F$-gauges $\mathcal{M}\to \mathcal{M}(\mathcal{G}(\mathcal{M}))$. To finish the proof of the theorem, we have to verify that this map is an isomorphism. 

For this, we can assume without loss of generality that $R$ is $p$-quasisyntomic. Now, we begin by observing that we also have a corresponding canonical map of Cartier dual $F$-gauges 
\[
\mathcal{M}^*\to \mathcal{M}(\mathcal{G}(\mathcal{M}^*))\simeq \mathcal{M}(\mathcal{G}(\mathcal{M})^*), 
\]
where the last isomorphism using Theorem~\ref{thm:cartier_duality}. Taking Cartier duals again yields a map
\[
\mathcal{M}(\mathcal{G}(\mathcal{M})^*)^* \to \mathcal{M}.
\]
and the composition
\[
\mathcal{M}(\mathcal{G}(\mathcal{M})^*)^*\to \mathcal{M}\to \mathcal{M}(\mathcal{G}(\mathcal{M}))
\]
is the canonical isomorphism in claim (4) of Proposition~\ref{prop:alb_mondal} applied with $G = \mathcal{G}(\mathcal{M})$. Alternatively, instead of using Cartier duality in this form, one can argue directly using Lemma~\ref{lem:G_functor_conservative}.

This shows that $\mathcal{M}\to \mathcal{M}(\mathcal{G}(\mathcal{M}))$ is an epimorphism, and since it is a map of vector bundles of the same rank over $R^{\mathrm{syn}}\otimes\Int/p^n\Int$, we conclude that it is in fact an isomorphism. 
\end{proof}

\begin{remark}
\label{rem:mondal_not_needed}
We have attempted, within reasonable bounds, to minimize our reliance on existing classification results: In essence, the only serious input we need along these lines is a full faithfulness result for the Dieudonn\'e functor over a complete DVR in characteristic $p$ with algebraically closed residue field. This can be found in~\cite[\S 4.1]{Berthelot1990-gp} and is used (indirectly) in the proof of Theorem~\ref{thm:cartier_duality}.

It is possible to streamline the proof quite a bit by making stronger use of existing results. 

To begin with, we could of course directly appeal to the equivalence established for quasisyntomic inputs in~\cite{MR4530092}.

In fact, since we already have the map $\mathcal{G}$, to verify that it is an equivalence, it suffices to do so over the special fiber. In particular, logically speaking, we only need Proposition~\ref{prop:alb_mondal} in the case where $R$ is an $\Field_p$-algebra. In this case, the verification of claim (1) of that proposition only requires the computations from~\cite{bbm:cris_ii} of the \emph{crystalline} cohomology of abelian schemes in characteristic $p$. Formulating the proof this way would make it completely independent of the results of~\cite{MR4530092} or~\cite{Mondal2024-cy}, though it wouldn't shed much light on what the inverse functor is, away from characteristic $p$.

Alternatively, one could use Grothendieck-Messing theory on both sides of the purported equivalence---via a truncated refinement of Remark~\ref{rem:groth_mess_compat} below---to further reduce to checking that $\mathcal{G}$ is an equivalence on points valued in algebraically closed fields. This would reduce us to classical Dieudonn\'e theory over perfect fields as explicated, say, in~\cite{MR0498610}. However, here one would have to compare our construction of finite flat group schemes here with that used by Fontaine, which involves Witt covectors. 

In fact, one could even abbreviate even this process quite a bit by directly appealing to the general classification results of Lau~\cite{lau2018divided}. This would establish the equivalence \emph{without} using the smoothness of the stack of truncated Barsotti-Tate groups, and so would give another proof of Grothenieck's theorem.

Another strategy is to reduce to the case of power series rings and then invoke the results of de Jong~\cite{dejong:formal_rigid}.
\end{remark} 

\begin{remark}
\label{rem:groth_messing_faltings}
Combined with Theorem~\ref{introthm:groth_mess}, we find that $\BT[h,d]{n}$ also enjoys a Grothendieck-Messing type deformation theory. This should be a special case of a very general lifting result of Faltings~\cite[Theorem 17]{Faltings2002-mp}, but we have not verified the requisite compatibility with the constructions of Faltings.
\end{remark}

\subsection{A polarized variant}
\label{subsec:polarized}

In this subsection, $R$ will always denote a derived $p$-complete discrete ring in $\mathrm{CRing}_{\heartsuit}$. 

\begin{definition}
\label{defn:twist_by_loc_system}
Suppose that $G$ is a finite flat commutative $p^n$-torsion group scheme over $R$. Then, given a $\Int/p^n\Int$-local system $L$ over $\Spec R$, we will write $G\otimes L$ for the finite flat group scheme obtained by tensor product of fppf sheaves of abelian groups: that this is indeed such an object is clear from \'etale descent.
\end{definition}

Alternatively we can also view $G\otimes L$ as the internal Hom sheaf $\underline{\Hom}(L^\vee,G)$ in fppf sheaves of abelian groups. This perspective shows:
\begin{observation}
\label{obs:cartier_dual_twisting}
We have a canonical isomorphism $(G\otimes L)^* \simeq G^*\otimes L^\vee$, where $*$ denotes the Cartier dual as usual, and $L^\vee$ is the dual local system.
\end{observation}

\begin{definition}
Fix an integer $g\ge 1$. A \defnword{principal quasi-polarization} on $G\in \BT[2g,g]{n}(R)$ is the provision of the following data: 
\begin{enumerate}
   \item A rank $1$ $\Int/p^n\Int$-local system $L$ over $\Spf R$;
   \item An isomorphism $\lambda:G\xrightarrow{\simeq}G^*\otimes L$ such that the composition
   \[
   G^*\xrightarrow{\simeq}(G\otimes L)^*\otimes L\xrightarrow{\lambda^*\otimes 1}G^*\otimes L
   \]
   is equal to $-\lambda$.
\end{enumerate}
We will call the tuple $(G,L,\lambda)$ a \defnword{principally quasi-polarized $n$-truncated Barsotti-Tate group} (of height $2g$), and denote the groupoid of such tuples by $\BT[2g,\mathrm{qpol}]{n}(R)$.

A \defnword{principal polarization on $G$} is a tuple $(G,L,\lambda)$ as above satisfying the following additional condition: \'Etale locally on $\Spf R$, there exists a lift $(\tilde{G},\tilde{L},\tilde{\lambda})$ to $\BT[2g,\mathrm{qpol}]{n+1}(R)$. We will call such tuples \defnword{principally polarized $n$-truncated Barsotti-Tate groups} and write $\BT[2g,\mathrm{ppol}]{n}(R)$ for the groupoid spanned by them.
\end{definition} 

\begin{remark}
\label{rem:different_polarizations}
In the literature---see for instance~\cite{MR2141704} or~\cite{Cais_Ellenberg_Zureick-Brown_2013}---one finds a definition of a principal quasi-polarization of $p$-divisible groups that is essentially the notion above (in the limit over $n$), except that $L$ is taken to be \emph{trivial}.

At finite level, we find a definition for the case $n=1$ in~\cite[\S 2.6]{Moonen2001-xm}, which uses Dieudonn\'e theory, and applies at the level of geometric points: it corresponds to our notion of a principal polarization given above. This is however a pointwise condition, and its moduli-theoretic interpretation is a little unclear. We have `solved' this issue above via the lifting-based condition, which is justified primarily by Proposition~\ref{prop:symplectic} below. See also Remark~\ref{rem:hopkins_lurie}.
\end{remark}

\subsubsection{}
Let $\GSp_{2g}$ be the generalized symplectic group over $\Int_p$ associated with the `standard' symplectic space over $\Int_p$ of rank $2g$, and let $\mu_g:\Gm\to \GSp_{2g}$ be the minuscule cocharacter splitting a Lagrangian subspace, and is such that the standard representation of $\GSp_{2g}$ yields a map of pairs $(\GSp_{2g},\mu_g)\to (\GL_{2g},\mu_g)$. Note that we have the similitude character $\nu:\GSp_{2g}\to \Gm=\GL_1$ satisfying $\nu\circ \mu_g = \mu_1$. This yields a map
\[
(\GSp_{2g},\mu_g)\to (\GL_g\times \Gm,\mu_g\times \mu_1).
\]

By Proposition~\ref{prop:lubin-tate} applied to the case $\mathcal{O} = \Int_p$, we see that $\BT[\Gm,\mu_1]{n}$ is the stack of line bundle $F$-gauges of the form $\mathcal{L}_0\{1\}$, where $\mathcal{L}_0$ is a line bundle $F$-gauge with Hodge-Tate weight $0$, associated with a $\Int/p^n\Int$-local system $L = \Gamma_{\mathrm{syn}}(\mathcal{L}_0)$ of rank $1$ over $\Spf R$. We will identify the stack of such local systems with the classifying stack $B\underline{(\Int/p^n\Int)^\times}$.

\begin{proposition}
\label{prop:symplectic}
There exists a canonical commuting diagram of $p$-adic formal classical Artin stacks over $\Int_p$ where the horizontal arrows are equivalences:
\[
\begin{diagram}
\BT[\GSp_{2g},\mu_g]{n}&\rTo^{\simeq}&\BT[2g,\mathrm{ppol}]{n}\\
\dTo&&\dTo_{(G,L,\lambda)\mapsto (G,L)}\\
\BT[\GL_{2g},\mu_g]{n}\times\BT[\Gm,\mu_1]{n}&\rTo_{\simeq}&\BT[2g,g]{n}\times B\underline{(\Int/p^n\Int)^\times}.
\end{diagram}
\]
\end{proposition}
\begin{proof}
All the formal stacks involved are smooth over $\Int_p$\footnote{For $\BT[2g,\mathrm{ppol}]{n}$, this is a `well-known' result, for which we could not find a specific reference, though one should certainly be able to extract it from the results of Faltings in~\cite[\S 7]{Faltings2002-mp}. In any case, the argument shows an unconditional result for smooth inputs, and so at least yields a canonical map $\BT[\GSp_{2g},\mu_g]{n}\to \BT[2g,\mathrm{ppol}]{n}$, which is an equivalence as soon as we know the smoothness of the target.}, so it suffices to construct such a diagram on $p$-torsion free qrsp inputs $R$. Since $R^{\mathrm{syn}}\otimes\Int/p^n\Int$ is a classical stack, using standard arguments, we find that $\BT[\GSp_{2g},\mu_g]{n}(R)$ is the groupoid of tuples $(\mathcal{M},\mathcal{L},\psi)$, where:
\begin{itemize}
   \item $\mathcal{M}$ is in $\BT[\GL_{2g},\mu_g]{n}(R)$;
   \item $\mathcal{L}$ is in $\BT[\Gm,\mu_1]{n}(R)$;
   \item $\psi:\wedge^2 \mathcal{M}\to \mathcal{L}$ is a map inducing an isomorphism of $F$-gauges
   \[
    f_\psi:\mathcal{M}\xrightarrow{\simeq} \mathcal{M}^\vee\otimes \mathcal{L}.
   \]
   Here, the exterior power $\wedge^2 \mathcal{M}$ is of vector bundles over the classical stack $R^{\mathrm{syn}}\otimes\Int/p^n\Int$, constructed as a quotient of $\mathcal{M}\otimes \mathcal{M}$ in the usual way.
\end{itemize}
Note that $f_\psi$ has the property that the composition
\[
 \mathcal{M}\xrightarrow{\simeq}(\mathcal{M}^\vee\otimes \mathcal{L})^\vee\otimes \mathcal{L}\xrightarrow[\simeq]{f_\psi^\vee\otimes 1}\mathcal{M}\otimes \mathcal{L}
\]
is equal to $-f_\psi$. When $p>2$, this condition on $f_\psi$ is enough to recover the symplectic form $\psi$. 

In general, we can say the following: Suppose that $f:\mathcal{M}\to \mathcal{M}^\vee\otimes \mathcal{L}$ can be lifted to an isomorphism 
\[
\tilde{f}:\tilde{\mathcal{M}}\xrightarrow{\simeq}\tilde{\mathcal{M}}^\vee\otimes\tilde{\mathcal{L}} 
\]
of vector bundle $F$-gauges of level $n+1$ satisfying the anti-symmetry property above. Then $f = f_{\psi}$ for a unique form $\psi:\wedge^2\mathcal{M}\to \mathcal{L}$. This is because, for any flat $\Int/p^{n+1}\Int$-algebra $A$, any $2$-torsion element $a\in A$ maps to $0$ in $\Int/p^n\Int$ (this is valid for all primes $p$!). Conversely, if $f = f_\psi$ for some $\psi$, then by the smoothness of $\BT[\GSp_{2g},\mu_g]{n+1}\to \BT[\GSp_{2g},\mu_g]{n}$ we see that it admits an \'etale local lift of level $n+1$.

Now, $\mathcal{L}\simeq \mathcal{L}_0\{1\}$ for some $\mathcal{L}_0$ of Hodge-Tate weight $0$ associated with a $\Int/p^n\Int$-local system $L$ of rank $1$. Therefore, we have
\[
\mathcal{M}^\vee\otimes \mathcal{L}\simeq \mathcal{M}^*\otimes \mathcal{L}_0.
\]
Moreover, we have a canonical isomorphism
\[
\mathcal{G}(\mathcal{M}^*\otimes \mathcal{L}_0)\xrightarrow{\simeq}\underline{\Hom}(L^\vee,\mathcal{G}(\mathcal{M})^*),
\]
which is clear from the full faithfulness part of Theorem~\ref{thm:dieudonne}.

Theorem~\ref{thm:BT_to_BT} now tells us that there is an equivalence of groupoids between tuples $(\mathcal{M},\mathcal{L},f')$ and $(G,L,\lambda')$, where $(\mathcal{M},\mathcal{L})$ and $(G,L)$ are as before, and $f':\mathcal{M}\xrightarrow{\simeq}\mathcal{M}^\vee\otimes \mathcal{L}$ and $\lambda':G \to G^*\otimes L$ are isomorphisms. As discussed above, the first kind of tuple lifts (uniquely) to $\BT[\GSp_{2g},\mu_g]{n}(R)$ if and only if it can be lifted, \'etale locally on $\Spf R$, to a tuple $(\tilde{\mathcal{M}},\tilde{\mathcal{L}},\tilde{f})$ of level $n+1$. We now conclude from the definition of $\BT[2g,\mathrm{ppol}]{n}(R)$.
\end{proof}

\begin{remark}
\label{rem:hopkins_lurie}
Frank Calegari has pointed out to us a \emph{different} notion of polarization due to Hopkins and Lurie~\cite[\S 3]{lurie_hopkins}, which gives a clean moduli-theoretic definition even when $p=2$. More precisely, for every finite flat group scheme $G$ over some ring $R$, they produce (see Definition 3.2.5 of \emph{loc. cit.}) a group scheme $\mathrm{Alt}^{(2)}_G$ of \emph{alternating $2$-forms} on $G$ that is a subscheme of the scheme $\mathrm{Skew}^{(2)}_G$ of alternating bilinear maps $G\times G\to \Gm$, and is equal to the latter when $2$ is invertible in $R$. We can then define a polarization on $G$ as a section of $\mathrm{Alt}^{(2)}_G$ whose associated alternating pairing is non-degenerate.

Now, suppose that $R$ is $p$-nilpotent and that $G$ is an $n$-truncated Barsotti-Tate group associated with an $F$-gauge $\mathcal{M}\in \BT[\GL_{2g},\mu_g]{n}$; then $\mathcal{F} \defn (\wedge^2 \mathcal{M})^\vee\{1\}$ has Hodge-Tate weights $\{-1,0,1\}$. Therefore, by Theorem~\ref{thm:sections_1-bounded_representable}, $\Gamma_{\mathrm{syn}}(\mathcal{F})$ is a derived affine scheme over $R$, and its classical truncation is a group scheme whose points parameterize maps of $F$-gauges $\wedge^2 \mathcal{M}\to \mathcal{O}_n\{1\}$. 

We now expect that this group scheme is isomorphic to $\mathrm{Alt}^{(2)}_G$.
\end{remark}

\begin{remark}
Here is a further aside: Hopkins and Lurie also define, for any $d\ge 2$, the group scheme of \emph{alternating $d$-forms} $\mathrm{Alt}^{(d)}_G$, and show in~\cite[Theorem 3.5.1]{lurie_hopkins} that, when $G$ is an $n$-truncated Barsotti-Tate group of height $h$ and dimension $1$, then $\mathrm{Alt}^{(d)}_G$ is represented by an $n$-truncated Barsotti-Tate group of height $\binom{h}{d}$ and dimension $\binom{h-1}{d}$. This should correspond simply to the fact that, with the notation of the previous remark, $(\wedge^d \mathcal{M})^\vee\{1\}$ is a vector bundle $F$-gauge of level $n$ and Hodge-Tate weights $\{0,1\}$ with height $\binom{h}{d}$ and dimension $\binom{h-1}{d}$.
\end{remark}   

\subsection{The crystalline and de Rham realizations}
\label{subsec:crystalline_comparison}

Suppose that $R$ is discrete and in $\mathrm{CRing}^{p\text{-comp}}$. Given an $F$-gauge $\mathcal{M}$ in $\BT[\GL_h,\mu_d]{\infty}(R)$, we obtain a $p$-divisible group $\mathcal{G}\defn \mathcal{G}(\mathcal{M})$ over $R$ of height $h$ and dimension $d$, and the crystalline Dieudonn\'e theory of~\cite{bbm:cris_ii} associates with the Cartier dual $\mathcal{G}^*$ a crystal of vector bundles $\mathbb{D}(\mathcal{G}^*)$ over the classical big crystalline site $(R/\Int_p)_{\mathrm{crys}}$ relative to the standard divided powers on $p\Int_p\subset\Int_p$. 

On the other hand, Remark~\ref{rem:crystalline_site} tells us that $\mathcal{M}$ gives us a crystal of vector bundles over the big (animated) crystalline site of $R$. In particular, we can restrict this to $(R/\Int_p)_{\mathrm{crys}}$ to obtain another crystal of vector bundles $\mathbb{D}(\mathcal{M})$ over $(R/\Int_p)_{\mathrm{crys}}$.

\begin{theorem}
\label{thm:crystalline_comparison}
There is a canonical isomorphism $\mathbb{D}(\mathcal{G}^*)\xrightarrow{\simeq}\mathbb{D}(\mathcal{M})$.
\end{theorem}
\begin{proof}
This is essentially a reinterpretation of the discussion in~\cite[\S 4.3]{MR4530092}. Here are the details.

Since $\BT[\GL_h,\mu_d]{\infty}$ is pro-smooth, we can reduce to the consideration of $p$-completely smooth $R$. Replacing $R$ with $R/pR$ does not affect the category of crystals, so we can assume that $R$ is a smooth $\Field_p$-algebra. 

Let $(R/\Int_p)_{\crys,\mathrm{pr}}$ be the big crystalline site for $\Spec R$ equipped with the pr topology given by extracting $p$-th roots: the definition of this site is due to Lau~\cite{lau2018divided} and is used in~\cite[\S 4.3]{MR4530092}. There is a map of topoi
\[
u^{\crys}:\mathrm{Shv}\left((R/\Int_p)_{\crys,\mathrm{pr}}\right)\to \mathrm{Shv}(R_{\mathrm{pr}})
\]
such that the classical Dieudonn\'e functor of~\cite{bbm:cris_ii} can be reinterpreted as the inner Ext
\[
\Dieu(\mathcal{G}^*) \simeq \varprojlim_n\underline{\Ext}^1((u^{\mathrm{crys}})^{-1}\mathcal{G}^*[p^n],\mathcal{O}_{\mathrm{crys}}),
\]
where $\mathcal{O}_{\crys}$ is the structure sheaf over $(R/\Int_p)_{\crys,\mathrm{pr}}$.

To compare this with our constructions, let us begin by invoking the explicit inverse constructed in Proposition~\ref{prop:alb_mondal}. This shows that $\mathcal{M}$ is given by
\[
\varprojlim_n\underline{\Hom}(\mathcal{G}^*[p^n],\Prism_{\_}/{}^{\mathbb{L}}p^n)\simeq \varprojlim_n\underline{\Ext}^1(\mathcal{G}^*[p^n],\Prism_{\_}).
\]
Here, we are looking at Hom (or Ext) between complexes of pr sheaves on $\Spec R$, where $\Prism_{\_}$ is the sheafification---as a derived $p$-complete complex---of the assignment $R' \mapsto \Prism_{R'}$ on quasisyntomic qrsp $R$-algebras. The final isomorphism holds because $\Prism_{\_}$ takes values in $p$-completely flat objects on qrsp inputs, while $\mathcal{G}^*[p^n]$ is of course killed by $p^n$.

To finish, therefore, we only need to know that we have a canonical isomorphism of complexes of pr sheaves
\[
Ru^{\mathrm{crys}}_* \mathcal{O}_{\crys}\xrightarrow{\simeq}\Prism_{\_}.
\]
It suffices to compare their values on qrsp quasisyntomic $R$-algebras. Here, this is a consequence of Remark~\ref{rem:derived_crystalline_semiperfect}: For any qrsp $R$-algebra $R'$, we have a canonical isomorphism $\Prism_{R'} \simeq A_{\crys}(R')$. On the other hand, $\theta_{R'}:A_{\crys}(R')\to R'$ with its divided powers is also the (pro-)initial object in $(R'/\Int_p)_{\crys,\mathrm{pr}}$.
\end{proof}

\subsubsection{}
Let $P^-_{h,d}\subset \GL_h$ be the parabolic subgroup associated the cocharacter $\mu_d$, so that $BP^-_{h,d}$ parameterizes filtered vector bundles $\Fil^\bullet \mathcal{V}$ with 
\[
\rank\gr^{i}\mathcal{V} = \begin{cases}
d&\text{if $i=-1$;}\\
h-d&\text{if $i=0$;}\\
0&\text{otherwise}.
\end{cases}
\]
Consider the natural map
\[
\BT[\GL_h,\mu_d]{\infty}\to BP^-_{h,d},
\]
obtained by taking the limit over $n$ from the diagram in Theorem~\ref{thm:main_thm_body}. 

This associates with every $F$-gauge $\mathcal{M}$ in $\BT[\GL_h,\mu_d]{\infty}(R)$ over a derived $p$-complete ring $R$ a filtered locally free $R$-module $\Fil^\bullet_{\mathrm{Hdg}}M$.

\subsubsection{}
On the other hand,  we have a short exact sequence
\begin{align}
\label{eqn:bbm_ses}
0\to \omega_{\mathcal{G}^*}\to \Dieu(\mathcal{G}^*)(R)\to \Lie(\mathcal{G})\to 0
\end{align}
of locally free $R$-modules, where $\omega_{\mathcal{G}^*}$ (resp. $\Lie(\mathcal{G})$) has rank $h-d$ (resp. $d$); see Corollaire 3.3.5 of \cite{bbm:cris_ii}.

\begin{proposition}
\label{prop:de_rham_realization}
There is a canonical isomorphism of $R$-modules $M\xrightarrow{\simeq}\Dieu(\mathcal{G}^*)(R)$ inducing isomorphisms
\[
\Fil^0_{\mathrm{Hdg}}M\xrightarrow{\simeq}\omega_{\mathcal{G}^*}\;;\; \gr^{-1}_{\mathrm{Hdg}}M\xrightarrow{\simeq}\Lie(\mathcal{G}).
\]
\end{proposition} 
\begin{proof}
We can view $\Dieu(\mathcal{G}^*)(R)$ as being a filtered locally free $R$-module supported in degrees $-1,0$. From this perspective, we see that we have written down two maps from $\BT[\GL_h,\mu_d]{\infty}$ to $BP^-_{h,d}$, and we would like to know that they are isomorphic. 

By derived descent and the pro-smoothness of the source, it suffices to verify this on qrsp $\Field_p$-algebras $R$. The proof now proceeds just as that of Theorem~\ref{thm:crystalline_comparison}, except one now keeps track of filtrations.

The Hodge filtration on the crystalline Dieudonn\'e module is given (up to degree shift) by
\[
\Fil^0\Dieu(\mathcal{G}^*)\simeq \varprojlim_n\underline{\Ext}^1((u^{\mathrm{crys}})^{-1}\mathcal{G}^*[p^n],\mathcal{J}_{\mathrm{crys}}),
\]
where $\mathcal{J}_{\mathrm{crys}}$ is the kernel of the natural surjection $\mathcal{O}_{\crys}\to (u^{\mathrm{crys}})^{-1}\Ga$. When evaluated on the trivial divided power thickening $R\xrightarrow{\mathrm{id}}R$, this yields the short exact sequence~\eqref{eqn:bbm_ses}.

Furthermore, the explicit inverse constructed in Proposition~\ref{prop:alb_mondal} shows that the restriction of $\mathcal{M}$ to $R^{\mathcal{N}}$ corresponds to a filtered $\Fil^\bullet_{\mathcal{N}}\Prism_R$-module $\Fil^\bullet \mathsf{M}$ admitting the following description: We have 
\[
\Fil^i\mathsf{M} = \varprojlim_n\Hom(\mathcal{G}^*[p^n],\Fil^{i+1}_{\mathcal{N}}\Prism_{\_}/{}^{\mathbb{L}}p^n)\simeq \varprojlim_n\Ext^1(\mathcal{G}^*[p^n],\Fil^{i+1}_{\mathcal{N}}\Prism_{\_}).
\]

The filtered module $\Fil^\bullet_{\mathrm{Hdg}}M$ is now obtained via filtered base-change along the map $\Fil^\bullet_{\mathcal{N}}\Prism_R\to \Fil^\bullet_{\mathrm{triv}}R$

To finish, therefore, we only need to know that we have canonical isomorphisms of complexes of pr sheaves
\[
Ru^{\mathrm{crys}}_* \mathcal{O}_{\crys}\xrightarrow{\simeq}\Prism_{\_}\;;\; Ru^{\mathrm{crys}}_*\mathcal{J}_{\crys} \xrightarrow{\simeq}\Fil^1_{\mathcal{N}}\Prism_{\_}.
\]
We have already verified the first. The second is a consequence of the discussion in \S~\ref{subsec:nygaard_filtered_affineness}: For any qrsp $R$-algebra $R'$, let $R'_0 = R^{',\flat}$ be the inverse perfection of $R'$ with prism structure given by $I_{R'_0}\subset W(R'_0) = \Prism_{R'_0}$. Then we have canonical isomorphisms
\[
\Fil^i_{\mathcal{N}}\Prism_{R'} \simeq \{x\in A_{\crys}(R'):\; \varphi(x)\in I_{R'_0}^iA_{\crys}(R')\}\subset A_{\crys}(R'). 
\]
On the other hand,  $\Fil^1_{\mathcal{N}}\Prism_{R'}$ is identified with the kernel of $\theta_{R'}:A_{\crys}(R')\to R'$ (see for instance~\cite[Proposition 5.3.6]{bhatt2022absolute}). 
\end{proof}

\begin{remark}
\label{rem:groth_mess_compat}
Combining Theorem~\ref{thm:crystalline_comparison} and Proposition~\ref{prop:de_rham_realization} shows that classical Grothendieck-Messing theory is compatible, via the equivalence in Theorem~\ref{thm:dieudonne}, with the Grothendieck-Messing theory for vector bundle $F$-gauges of Hodge-Tate weights $0,1$ given to us by Theorems~\ref{thm:main_thm_body} or~\ref{thm:HTwts01_representable}.
\end{remark}  

\begin{proposition}
\label{prop:lubin-tate_groups}
Let $\mathcal{O}$ and $(T_0,\mu_0)$ be as in \S~\ref{subsec:tori}. Then, for any discrete, derived $p$-complete $\mathcal{O}$-algebra $R$ in $\mathrm{CRing}$, $\BT[T_0,\mu_0]{\infty}(R)$ is canonically equivalent to the groupoid of Lubin-Tate formal $\mathcal{O}$-modules over $R$ of height $h$ and dimension $1$.
\end{proposition} 
\begin{proof}
It is well-known that the stack of Lubin-Tate formal $\mathcal{O}$-modules of height $h$ and dimension $1$ is represented by a pro-finite \'etale stack over $\Spf \mathcal{O}$ that is isomorphic to the trivial gerbe banded by $T_0(\Int_p)$, the automorphism group of any such Lubin-Tate formal $\mathcal{O}$-module over $\Spf \mathcal{O}$. 

Therefore, it suffices to establish the stated equivalence on $p$-completely \'etale $\mathcal{O}$-algebras $R$. Here, it is essentially immediate from Proposition~\ref{prop:lubin-tate}, except that we have to show the following: Given $\mathcal{F}$ as in that proposition, the Lie algebra $\Lie(\mathcal{G}(\mathcal{F}))$ of $\mathcal{G}(\mathcal{F}) = \Gamma_{\mathrm{syn}}(\mathcal{F})$ has rank $1$, and the resulting action of $\mathcal{O}\otimes_{\Int_p}\mathcal{O}$ on the Lie algebra factors through the structure map $\mathcal{O}\to R$. This is immediate from Proposition~\ref{prop:de_rham_realization}.
\end{proof}

\begin{remark}
More generally, Theorem~\ref{thm:BT_to_BT} can be applied to give a description of $p$-divisible groups of unramified `EL type' with a Kottwitz condition in terms of apertures for appropriate group-theoretical data.
\end{remark}

\subsection{The \'etale realization}
\label{subsec:etale}

For completeness, we will finish by reviewing a result of Mondal.

\subsubsection{}
Let $R$ be a quasisyntomic ring. Then Bhatt shows that there is a canonical functor~\cite[Construction 6.3.1]{bhatt_lectures}
\[
T_{\et}:\mathrm{Perf}(R^{\mathrm{syn}})\to D^b_{\mathrm{lisse}}(\mathrm{Spf}(R)_\eta^{\mathrm{ad}},\Int_p)
\]
where the right hand side is the bounded derived category of lisse $\Int_p$-sheaves on the adic generic fiber $\mathrm{Spf}(R)_\eta \defn \mathrm{Spa}(R[1/p],R)$.

In particular, given $\mathcal{M}\in \mathrm{Vect}_{\{0,1\},n}^{\mathrm{syn}}(R)$, viewed as a perfect complex over $R^{\mathrm{syn}}$, we obtain an object $T_{\et}(\mathcal{M})$ on the right hand side. On the other hand, we can also consider the finite flat $p^n$-torsion group scheme $G \defn \mathcal{G}(\mathcal{M})$ over $R$, and take its adic generic fiber $G^{\ad}_{\eta}$, which can also be viewed as a perfect complex of lisse $\Int_p$-sheaves. The next result follows from~\cite[Proposition 3.99]{Mondal2024-cy}; see also~\cite[Proposition 5.25]{MR4530092}.
\begin{proposition}
\label{prop:etale_realization}
There is a canonical isomorphism $T_{\et}(\mathcal{M})\xrightarrow{\simeq}G^{\ad}_\eta$.
\end{proposition}

\appendix

\section{Some completeness results}
\label{app:completeness}

\subsection{Tannaka duality}

We will need the following result of Bhatt and Halpern-Leistner~\cite[Theorem 5.1,Lemma 3.13]{BHATT2017576} (see also the discussion in Step 2 of the proof of~\cite[Proposition 5.1.13]{MR4560539}).
\begin{theorem}
[Tannakian reconstruction]
\label{thm:tannakian}
Suppose that $X$ is a \emph{classical} Noetherian Artin $n$-stack with quasi-affine diagonal; then, for any classical prestack $S$, there is an equivalence of $\infty$-groupoids
\[
\Map(S,X)\xrightarrow[\simeq]{f\mapsto f^*}\mathrm{Fun}^c_{\otimes}(\mathrm{APerf}^{\mathrm{cn}}(X),\mathrm{APerf}^{\mathrm{cn}}(S)).
\]
Here the right hand side is the space of symmetric monoidal functors that preserve finite colimits. In particular, if $S'\to S$ is a map of classical prestacks inducing an equivalence of symmetric monoidal $\infty$-categories
\[
\mathrm{APerf}^{\mathrm{cn}}(S)\xrightarrow{\simeq}\mathrm{APerf}^{\mathrm{cn}}(S'),
\]
then the natural map
\[
\Map(S,X)\to \Map(S',X)
\]
is an equivalence.
\end{theorem}

\subsection{Completeness}

\begin{proposition}
\label{prop:BGm_complete}
Suppose that $R$ is a complete Noetherian local ring with maximal ideal $\mx$. Then, for any relatively locally almost finitely presented derived Artin stack $\mathcal{X}\to B\Gm\times \Spec R$ with quasi-affine diagonal, the map
\[
\Map_{/B\Gm\times\Spec R}(B\Gm\times\Spec R,\mathcal{X})\to \varprojlim_m \Map_{/B\Gm\times\Spec R}(B\Gm\times\Spec R/\mx^m,\mathcal{X})
\]
is an equivalence.
\end{proposition}
\begin{proof}
We can replace $\mathcal{X}$ with its classical truncation $X \defn \mathcal{X}_{\mathrm{cl}}$. By Theorem~\ref{thm:tannakian}, we are now reduced to the easy observation that the functor
\[
\mathrm{APerf}^{\mathrm{cn}}(B\Gm\times \Spec R)\to \varprojlim_m\mathrm{APerf}^{\mathrm{cn}}(B\Gm\times \Spec R/\mathfrak{m}^m)
\]
is an equivalence.
\end{proof} 

\begin{proposition}
\label{prop:A1Gm_complete}
Suppose that we have $R\in \mathrm{CRing}$ and a relatively almost locally finitely presented derived Artin stack $\mathcal{X}\to \Aff^1/\Gm\times \Spec R$ with quasi-affine diagonal. Then for any Noetherian $B\in \mathrm{CRing}_{\heartsuit,R/}$ the natural map
\[
\Map_{/\Aff^1/\Gm\times \Spec R}(\Aff^1/\Gm \times \Spec B,\mathcal{X})\to\varprojlim_n\Map_{/\Aff^1/\Gm\times \Spec R}((\Aff^1/\Gm)_{(t^n=0)} \times \Spec B,\mathcal{X})
\]
is an equivalence.
\end{proposition}
\begin{proof}
Once again, we can replace $\mathcal{X}$ with its classical truncation, and so it suffices to know that the map
\[
\mathrm{APerf}^{\mathrm{cn}}(\Aff^1/\Gm\times \Spec B)\to \varprojlim_n\mathrm{APerf}^{\mathrm{cn}}((\Aff^1/\Gm)_{(t^n = 0)}\times \Spec B)
\]
is an equivalence. 

This boils down to the fact that filtered almost perfect complexes over $B$ are \emph{complete} for their filtration, since they are inverse limits of complexes $\Fil^\bullet M$ with $\Fil^iM \simeq 0$ for $i$ sufficiently large.
\end{proof}

\begin{lemma}
\label{lem:graded_weight_filtration}
Let $M_\bullet$ be a graded almost perfect module over a non-positively graded animated commutative ring $B_\bullet$. Write 
\[
\overline{M}_\bullet = B_0\otimes_{B_\bullet}M_\bullet
\]
for the graded base-change of $M_\bullet$. Then $M_\bullet$ admits a functorial increasing, exhaustive and complete filtration $\Fil_\bullet^{\mathrm{wt}}M_\bullet$ in $\mathrm{GrMod}_{B_\bullet}$ with
\[
\gr_{-i}^{\mathrm{wt}}M_\bullet \simeq B_\bullet(-i)\otimes_{B_0}\overline{M}_i.
\]
\end{lemma}
\begin{proof}
Once we have such an increasing and exhaustive filtration, it will automatically be complete: Indeed, the map
\[
M_\bullet \to \hcoker(\Fil^{\mathrm{wt}}_{-i}M_\bullet \to M_\bullet)
\]
is an isomorphism in graded degrees $\le (i-1)$.

To obtain the filtration, let's make the assertion a bit more precise. We have a functor from almost perfect graded $B_\bullet$-modules to doubly graded almost perfect $B_\bullet$-modules $M_\bullet \mapsto \bigoplus_i B_\bullet(-i)\otimes_{B_0}\overline{M}_i$, which geometrically can be seen as the functor
\[
\mathrm{APerf}((\Spec B_\bullet)/\Gm)\to \mathrm{APerf}(B\Gm\times (\Spec B_\bullet)/\Gm)
\]
via pullback along the maps
\[
B\Gm\times (\Spec B_\bullet)/\Gm\to B\Gm\times \Spec B_0\to (\Spec B_\bullet)/\Gm.
\]
Here, the first map is obtained from the structure morphism $(\Spec B_\bullet)/\Gm \to \Spec B_0$, and the second map from the map of graded rings $B_\bullet \to B_0$. The claim then is that this pullback functor admits a canonical lift
\[
\mathrm{APerf}((\Spec B_\bullet)/\Gm)\to \mathrm{APerf}(\Aff^1_+/\Gm\times (\Spec B_\bullet)/\Gm)
\]
whose composition with the functor
\[
\mathrm{APerf}(\Aff^1/\Gm\times (\Spec B_\bullet)/\Gm)\to \mathrm{APerf}((\Spec B_\bullet)/\Gm)
\]
induced by the open immersion $\Gm/\Gm\times (\Spec B_\bullet)/\Gm\to \Aff^1_+/\Gm\times (\Spec B_\bullet)/\Gm$ is equipped with an isomorphism to the identity.

It suffices to prove the claim with $\mathrm{APerf}$ replaced everywhere with $\Mod[\mathrm{cn}]{}$. Here, via the process of animation (or more precisely, the compatibility of the existing functors with sifted colimits), we can also replace $\Mod[\mathrm{cn}]{}$ with $\mathrm{Vect}$.

The proof now proceeds as in~\cite[Theorem 2.44]{guo2023frobenius}. Since we are dealing with finite locally free modules, there exist $a,b\in \Int$ such that $\overline{M}_k \simeq 0$ for $k\notin [a,b]$. We now claim that $M_k\simeq 0$ for $k>b$ and that the natural map $M_b \to \overline{M}_b$ is an equivalence. To see this, observe that we have a fiber sequence
\[
B_{\leq -1}\otimes_{B_\bullet}M_\bullet \to M_\bullet \to \overline{M}_{\bullet},
\]
and note that the left hand side is a graded module supported in degrees $\leq b-1$. This last observation can be checked using the bar resolution as in~\cite[Lemma 2.45]{guo2023frobenius}. 

The proof will now proceed by induction on $b-a$. We can use the equivalence $M_b\xrightarrow{\simeq}\overline{M}_b$ to obtain a canonical map
\[
B_\bullet\otimes_{B_0}\overline{M}_b(-b)\xrightarrow{\simeq}B_\bullet\otimes_{B_0}M_b(-b)\to M_\bullet 
\]
whose cofiber $M'_\bullet$ is supported in degrees $< b$. Graded base-change to $B_0$ yields a cofiber sequence
\[
\overline{M}_b(-b)\to \overline{M}\to \overline{M}',
\]
which shows that $\overline{M}_i\xrightarrow{\simeq}\overline{M}'_i$ for $i\leq b-1$ and $\overline{M}'_i\simeq 0$ for $i>b-1$. In particular, if $a=b$, then $\overline{M}'\simeq 0$, and the argument from the previous paragraph implies that in fact $M'\simeq 0$. This verifies the base step of the induction.

For the inductive step, our hypothesis implies that $M'_\bullet$ admits an increasing filtration $\Fil_\bullet^{\mathrm{wt}}M'_\bullet$ with 
\[
\gr_{-i}^{\mathrm{wt}}M'_\bullet \simeq B_\bullet\otimes_{B_0}\overline{M}'_{i}(-i) \simeq B_\bullet\otimes_{B_0}\overline{M}_{i}(-i)
\]
for $i\le b-1$.

We now obtain our desired filtration on $M_\bullet$ by setting
\[
\Fil_{-i}^{\mathrm{wt}}M_\bullet = \begin{cases}
0&\text{if $i> b$}\\
\Fil_{-i}^{\mathrm{wt}}M'_\bullet\times_{M'_\bullet}M_\bullet&\text{if $i\le b$}.
\end{cases} 
\]
\end{proof}

\begin{proposition}
\label{prop:graded_completeness}
Suppose that $B_\bullet$ is a non-positively graded discrete commutative ring and that $\mathcal{X}\to \mathcal{Y}\defn \Spec(B_\bullet)/\Gm$ is a relatively locally almost finitely presented derived Artin stack with quasi-affine diagonal. Then the natural map
\[
\Map_{/\mathcal{Y}}(\mathcal{Y},\mathcal{X})\to \varprojlim_m \Map_{/\mathcal{Y}}(\Spec(B_{\ge -m})/\Gm,\mathcal{X})
\]
is an equivalence.
\end{proposition}  
\begin{proof}
 We can replace $\mathcal{X}$ with its classical truncation, and by Noetherian approximation, we can also assume that $\mathcal{X}$ is the base-change over $\mathcal{Y}$ along a map $\mathcal{Y}\to \mathcal{Y}'$ of a locally finitely presented map of classical Noetherian stacks $\mathcal{X}'\to \mathcal{Y}'$ with quasi-affine diagonal. Thus, we are reduced via Theorem~\ref{thm:tannakian} to knowing that the functor
\[
\mathrm{APerf}^{\mathrm{cn}}(\mathcal{Y})\to \varprojlim_m \mathrm{APerf}^{\mathrm{cn}}(\mathcal{Y}_m)
\]
is an equivalence, where $\mathcal{Y}_m = (\Spec B_{\ge -m})/\Gm$. 

Let us first check full faithfulness: we need to know that, for $M_\bullet,N_\bullet$ on the left hand side, the natural maps
\begin{align*}
\Map_{B_\bullet}(M_\bullet,N_\bullet)&\to \varprojlim_m\Map_{B_{\ge -m}}(B_{\ge -m}\otimes_{B_\bullet}M_\bullet,B_{\ge -m}\otimes_{B_\bullet}N_\bullet)\\
&\xrightarrow{\simeq}\varprojlim_m\Map_{B_\bullet}(M_\bullet,B_{\ge -m}\otimes_{B_\bullet}N_\bullet)\\
&\xrightarrow{\simeq}\Map_{B_\bullet}(M_\bullet,\varprojlim_m B_{\ge -m}\otimes_{B_\bullet}N_\bullet)
\end{align*}
are isomorphisms. This comes down to knowing that the natural map
\[
N_\bullet \to \varprojlim_m B_{\ge -m}\otimes_{B_\bullet}N_\bullet
\]
is an isomorphism. Using Lemma~\ref{lem:graded_weight_filtration}, we can reduce to the case where $N_\bullet \simeq B_\bullet\otimes_{B_0}N'$, where $N'$ is an almost perfect connective complex of $B_0$-modules. This reduces us to knowing that
\[
B_{\bullet}\otimes_{B_0}N'\to  \varprojlim_{m} B_{\ge -m}\otimes_{B_0}N'
\]
is an equivalence. Comparing degree-by-degree, we are reduced to the case where $N'\simeq B_0$, where the assertion is clear.

For essential surjectivity, suppose that we are given an object on the right yielding a compatible sequence $(M^{(m)}_\bullet)_m$ of graded almost perfect connective $B_{\ge -m}$-modules $M^{(m)}_\bullet$. We need to check that $M_\bullet = \varprojlim_m M^{(m)}_\bullet$ is a graded almost perfect connective $B_\bullet$-module, and that its graded base-change to $B_{\ge -m}$ yields $M^{(m)}_\bullet$. Once again, using Lemma~\ref{lem:graded_weight_filtration}, we can reduce to the case where the compatible sequence is of the form $(B_{\ge -m}\otimes_{B_0}M)_m$ for some almost perfect connective complex of $B_0$-modules $M$. Here, the claim is obvious.
 \end{proof} 

\begin{definition}
\label{defn:graded_integrable}
Suppose that $\mathcal{X}\to \mathcal{Z}$ is a map of prestacks over $B\Gm$. We will say that $\mathcal{X}$ is \defnword{graded integrable} if, for all non-positive graded rings $B_\bullet$ as above, and for all maps $\mathcal{Y}\defn(\Spec B_\bullet)/\Gm\to \mathcal{Z}$ over $B\Gm$, the conclusion of Proposition~\ref{prop:graded_completeness} holds for the base-change $\mathcal{X}\times_{\mathcal{Z}}\mathcal{Y}\to \mathcal{Y}$.
\end{definition} 

\subsubsection{}
\label{subsubsec:filtered_completion}
Suppose that $\Fil^\bullet S$ is a non-negatively filtered animated commutative ring and that $\Fil^\bullet I\subset \Fil^\bullet S$ is a filtered ideal with the following properties:
\begin{itemize}
   \item $\Fil^iS$ is discrete for all $i\in \Int$;
   \item $\Fil^mI = 0$ for $m\le 0$; in particular, the map $\Fil^1I\to \Fil^1 S \to S$ is identically zero.
\end{itemize}
Consider the ideal $J\subset \mathrm{Rees}(\Fil^\bullet S)$ of the Rees algebra given by
\[
J = \bigoplus_{i=1}^\infty\Fil^iI\cdot t^{-i}\subset \bigoplus_{i\in \Int}\Fil^iS \cdot t^{-i}= \mathrm{Rees}(\Fil^\bullet S).
\]
This is a homogeneous ideal, and so, for each $m\ge 1$, we can consider the quotient $\mathrm{Rees}(\Fil^\bullet S)/J^m$: it corresponds to a filtered animated commutative ring structure $\Fil^\bullet_{(m)}S$ on $S$ with
\[
\Fil^i_{(m)}S = \Fil^i S/\left(\sum_{\stackrel{k_1+\ldots+k_m = i}{k_i\ge 1}}\mathrm{im}(\Fil^{k_1}I\otimes_S\cdots\otimes_S\Fil^{k_m}I)\right).
\]

\begin{proposition}
\label{prop:filtered_completeness}
Suppose that we have a relatively locally almost finitely presented derived Artin stack $\mathcal{X}\to \mathcal{Z})$ with quasi-affine diagonal. Then the natural map
\[
\Map_{/\mathcal{Z}}(\mathcal{Z},\mathcal{X})\to \varprojlim_m\Map_{/\mathcal{Z}}(\Rees(\Fil^\bullet_{(m)}S),\mathcal{X})
\]
is an equivalence.
\end{proposition} 
\begin{proof}
As above, we reduce to knowing that
\[
\mathrm{APerf}^{\mathrm{cn}}(\Rees(\Fil^\bullet S))\to \varprojlim_m\mathrm{APerf}^{\mathrm{cn}}(\Rees(\Fil^\bullet_{(m)}S))
\]
is an equivalence, which in turn follows easily from the observation that the map $\Fil^i S \to \Fil^i_{(m)}S$ is an equivalence for all $i<m$. 
\end{proof}

\begin{definition}
\label{defn:filtered_integrable}
Suppose that $\mathcal{X}\to \mathcal{Y}$ is a map of prestacks over $\Aff^1/\Gm$. We will say that $\mathcal{X}$ is \defnword{filtered integrable} if, for all pairs $(\Fil^\bullet S,\Fil^\bullet I)$ satisfying the above conditions, and for all maps $\mathcal{Z}\defn \Rees(\Fil^\bullet S)\to \mathcal{Y}$ over $\Aff^1/\Gm$, the conclusion of Proposition~\ref{prop:filtered_completeness} holds for the base-change $\mathcal{X}\times_{\mathcal{Y}}\mathcal{Z}\to \mathcal{Z}$.
\end{definition} 

\printbibliography

\end{document}